\documentclass[]{article}

\usepackage{graphicx}
\usepackage{amsmath}
\usepackage{amssymb}
\usepackage{amsthm}
\usepackage{pxfonts}
\usepackage{enumerate}
\usepackage{color}
\usepackage{mathdots}
\usepackage{sectsty}
\usepackage{tikz}
\usetikzlibrary{decorations.markings}
\usepackage{adjustbox}
\usepackage{enumitem}
\usepackage{caption}
\usepackage{bbold}
\usepackage[hidelinks]{hyperref}
\allowdisplaybreaks

 \hypersetup{
           breaklinks=true,   % splits links across lines
        }

\tikzset{middlearrow/.style={
        decoration={markings,
            mark= at position 0.5 with {\arrow{#1}} ,
        },
        postaction={decorate}
    }
}

\sectionfont{\scshape\centering\fontsize{11}{14}\selectfont}
\subsectionfont{\scshape\fontsize{11}{14}\selectfont}
\usepackage{fancyhdr}
\usepackage[nottoc,notlot,notlof]{tocbibind}

\newcommand\shorttitle{On some integrable models in inhomogeneous space}
\newcommand\authors{Theodoros Assiotis}

\newtheorem{thm}{Theorem}[section]

\newtheorem{lem}[thm]{Lemma}
\newtheorem{defn}[thm]{Definition}
\newtheorem{rmk}[thm]{Remark}
\newtheorem{prop}[thm]{Proposition}
\newtheorem*{theorem*}{Theorem}

\title{\large \bf ON SOME INTEGRABLE MODELS IN INHOMOGENEOUS SPACE}
\author{\small THEODOROS ASSIOTIS}
\date{}

\begin{document}

\maketitle

\begin{abstract} 
The purpose of this work is to build a framework that allows for an in-depth study of various generalisations to inhomogeneous space of models of Borodin-Ferrari \cite{BorodinFerrari}, Dieker-Warren \cite{DiekerWarren}, Nordenstam \cite{Nordenstam}, Warren-Windridge \cite{WarrenWindridge} of interacting particles in interlacing arrays, both in discrete and continuous time, involving both Bernoulli and geometric jumps. The models can in addition be either time-inhomogeneous or particle-inhomogeneous. We show that the correlation functions of these models are determinantal and compute explicitly the correlation kernel for the fully-packed initial condition. Using these formulae we prove a short-time asymptotic for these dynamics, in general inhomogeneous space, to the discrete Bessel determinantal point process with parameter depending on the inhomogeneous environment only through a kind of average. We moreover prove a number of closely related results including the following. We prove that the autonomous, inhomogeneous in space and time, TASEP-like and pushTASEP-like particle systems on the left and right edge of the array respectively have explicit transition kernels and that from any deterministic initial condition their distributions are marginals of a signed measure with determinantal correlation functions. We reinterpret the distribution of these dynamics in arrays in terms of coherent sequences of measures on a natural inhomogeneous generalisation of the celebrated Gelfand-Tsetlin graph \cite{BorodinOlshanskiBoundary} and prove that all of them are extreme points in the convex set of coherent measures. We connect some of our constructions of non-intersecting paths to independent walks with location-dependent jumps conditioned to never intersect by computing explicitly the intersection probabilities. We prove a novel duality relation between dynamics in inhomogeneous space and dynamics with inhomogeneities on the level of the array (particle inhomogeneities). We extend the work of Nordenstam \cite{Nordenstam} on the shuffling algorithm for domino tilings of the Aztec diamond and its relation to push-block dynamics in interlacing arrays, from the uniform weight to general weights on the tilings. We then connect this, for a special class of weights, back to our previous results. We also consider non-intersecting walks in inhomogeneous space and time with fixed starting and end points and obtain a formula for their correlation functions, involving among other ingredients, an explicit Riemann-Hilbert problem, generalising some of the results of Duits and Kuijlaars \cite{DuitsKuijlaars}. We then prove a limit theorem for the bottom lines in this line-ensemble, under some technical conditions, generalising some of the results of Berggren and Duits \cite{DuitsBeggren}. The main computational tool throughout this work is a natural generalisation of a Toeplitz matrix, that we call inhomogeneous Toeplitz-like matrix $\mathsf{T}_{\mathbf{f}}$ with (a possibly matrix-valued) symbol $\mathbf{f}$.
\end{abstract}

\tableofcontents

\section{Introduction}
The purpose of this work is an in-depth study of various integrable probabilistic models, of statistical mechanical nature, including systems of interacting particles, measures on non-intersecting paths and random tiling models in which the randomness depends in an inhomogeneous way on the underlying space. As far as we know, there is no precise definition of what an integrable probabilistic model really is. But an informal working definition could go as follows: a model that enjoys explicit formulae for the expectations of some class of observables that can be used to analyse the model further, especially asymptotically. In some cases this class of observables is very restricted, while in others one has explicit knowledge of workable formulae for all the correlation functions. Using these explicit formulae sometimes it is possible to completely analyse the asymptotic behaviour of the model, while sometimes one needs to combine this knowledge with the use of probabilistic or sometimes geometric arguments to perform such analysis. This has been particularly fruitful in studying models in the KPZ universality class \cite{CorwinKPZsurvey} which has been the main drive in the search for integrable models in the last two decades. We note however that not all models in the KPZ class are integrable, see for example \cite{QuastelSarkarKPZ} for a model where explicit formulae are completely absent, and vice-versa not all integrable models have necessarily KPZ behaviour.

In this paper we introduce and study space-inhomogeneous generalisations of, much-studied in the integrable probability literature, models of Borodin-Ferrari \cite{BorodinFerrari}, Dieker-Warren \cite{DiekerWarren}, Nordenstam \cite{Nordenstam} and Warren-Windridge \cite{WarrenWindridge} of interacting (through so-called push-block dynamics) particles in interlacing arrays, both in discrete and continuous time, involving both Bernoulli and geometric jumps. We also study in detail their one-dimensional Markovian marginals of TASEP-like and pushTASEP-like systems and related models of non-intersecting paths and models of tilings. A more detailed review of our contributions will follow shortly. The main message is that, after building the right tools, a lot of the integrable structures that exist in the homogeneous models can be shown to exist in the space inhomogeneous setting as well. Unsurprisingly, the explicit formulae are more involved but nevertheless still useful. Employing  them, certain scaling limits of these models are analysed already in this paper, although a general study of asymptotic behaviours is beyond the scope of the present work. Finally, some other models of interacting particles in inhomogeneous space have been studied in recent years \cite{BorodinPetrovInhomogeneous,KnizelPetrovSaenz,PetrovInhomogeneousPushTASEP,DeterminantalStructures}, some of them very closely related (in fact special cases) to our models, and a more detailed literature review will follow in Section \ref{SectionModels} where we state our results precisely.

We now turn to our main tool. The use of Toeplitz matrices, with both scalar and matrix symbols (in which case they are normally called block Toeplitz matrices), has been very important in the study of integrable probabilistic models, see \cite{JohanssonDetPP,JohanssonEdgeFluctuations, Schur,SchurDynamics,BorodinFerrari,DuitsKuijlaars,DuitsBeggren}. A suitable generalisation plays an importable role here and is also the common thread throughout our work. A natural-looking generalisation of a Toeplitz matrix with scalar symbol $f$ could be the following. Given two sequences of functions $\mathbf{u}(z)=\left(u_k(z)\right)_{k\in \mathbb{Z}_+}$ and $\mathbf{v}(z)=\left(v_k(z)\right)_{k\in \mathbb{Z}_+}$ and a contour $\mathfrak{C}_{\mathbf{u},\mathbf{v}}$ such that we have the ``biorthogonality" relation,
\begin{equation*}
\frac{1}{2\pi \textnormal{i}}\oint_{\mathfrak{C}_{\mathbf{u},\mathbf{v}}} \frac{u_k(z)}{v_j(z)}dz=\begin{cases}
1, \ \ &\textnormal{if } k=j,\\
0, \ \ &\textnormal{otherwise},
\end{cases}
\end{equation*}
we could define the matrix $[\mathcal{T}_f^{\mathbf{u},\mathbf{v}}(x,y)]_{x,y\in \mathbb{Z}_+}$ associated to the symbol $f$ by, with $x,y\in \mathbb{Z}_+$,
\begin{equation}\label{ToeplitzGeneralisation}
\mathcal{T}_f^{\mathbf{u},\mathbf{v}}(x,y)=\frac{1}{2\pi \textnormal{i}}\oint_{\mathfrak{C}_{\mathbf{u},\mathbf{v}}}\frac{u_x(z)}{v_y(z)}f(z)dz.
\end{equation}
Clearly, the Toeplitz matrix with symbol $f$ is the special case $u_k(z)=z^k, v_k(z)=z^{k+1}$ and $\mathfrak{C}_{\mathbf{u},\mathbf{v}}=\{z\in \mathbb{C}:|z|=1\}$.
Of course, the above is only interesting and useful if we can establish desirable properties for $\mathcal{T}_f^{\mathbf{u},\mathbf{v}}$ and find applications for it. With this in mind, the sequences of functions we take in this work will not be arbitrary, but rather they will be families of polynomials which depend on the inhomogeneous environment behind the probabilistic applications we are interested in. This allows us to define a generalisation of a Toeplitz matrix, that we call inhomogeneous Toeplitz-like matrix $\mathsf{T}_f$ associated to a symbol $f$, see Sections \ref{SectionNotation} and \ref{Section1D}, and  for matrix-valued symbols Section \ref{SectionLineEnsembles}, and develop a framework for probabilistic applications around it in the rest of the paper. As it turns out, $\mathsf{T}_f$ is actually similar to a standard Toeplitz matrix, with symbol $f(1-z)$, via an explicit, albeit complicated, change of basis matrix. However, we stress that we cannot just transfer results from the standard Toeplitz matrix setup to the inhomogeneous one through this similarity (with only some exceptions). See Section \ref{Section1D} for more details.

An important source of integrable probabilistic models, whose correlations functions can be computed explicitly, come from the Schur process \cite{Schur} and Schur dynamics on them \cite{SchurDynamics}. These constructions have their origin in symmetric function theory \cite{MacdonaldBook} and the Schur polynomials in particular \cite{MacdonaldBook}. A number of, but far from all,  the results in our paper could be rephrased and proven in terms of factorial Schur polynomials \cite{MacDonaldSchurvariations}, a certain generalisation of Schur polynomials, but we will not attempt to do it here. In some sense, some of our results constitute a factorial Schur generalisation of the Schur process and Schur dynamics \cite{Schur,SchurDynamics}. Instead, we have chosen to emphasize the inhomogeneous Toeplitz-like matrix/operator perspective, which seems to us better adapted for some of the probabilistic constructions we consider. For example, the formulae in Section \ref{SectionCouplings} for the so-called two-level couplings we do not know if they have have a natural interpretation in terms of factorial Schur polynomials, or any symmetric functions for that matter, but it would be very interesting if one existed. Finally, in \cite{FreeFermionVertex} a general fully inhomogeneous six vertex model and its associated symmetric functions, which degenerate to the factorial Schur polynomials, were studied in detail.  However, as far as we can tell, the results in our work do not follow as degenerations of results from \cite{FreeFermionVertex}. In fact, the paper \cite{FreeFermionVertex} considers a static model which is not evolving in time. It would be interesting though to understand whether and how the models here are related with the vertex model of \cite{FreeFermionVertex}.

The main novel contributions of this work, we believe, are the following.
\begin{itemize}
    \item We develop a framework, based on the inhomogeneous Toeplitz-like matrices $\mathsf{T}_{\mathbf{f}}$ with (possibly matrix) symbol $\mathbf{f}$ that we introduce, which allows us to prove intertwining relations between semigroups corresponding to non-intersecting paths and obtain explicit formulae for the distributions and correlation functions, including the explicit computation of correlation kernels, of the various types of dynamics we study in this paper, see Sections \ref{SectionSpaceTimeCorrIntro}, \ref{SpaceLevelInhomogeneousSectionIntro}, \ref{LineEnsemblSetionIntro} for an illustration of some of the results and Sections \ref{Section1D}, \ref{SectionIntertwining}, \ref{SectionCouplings}, \ref{SectionDynamicsOnArrays}, \ref{SectionComputationKernels} and \ref{SectionLineEnsembles} for more details on the techniques.
    \item We introduce couplings for the intertwined semigroups of non-intersecting paths mentioned above which have their origin in coalescing inhomogeneous in space and time Bernoulli and geometric walks, see Section \ref{SectionCouplings}. In the case of geometric walks, which is the most subtle, the dynamics coming from this coupling are different from what one gets if one uses a certain general recipe for couplings of intertwined semigroups developed by Borodin and Ferrari \cite{BorodinFerrari}. In particular, they have the desirable property that the projection on the left edge of the array is Markovian and is a kind of inhomogeneous-space (and time) geometric TASEP, unlike in \cite{BorodinFerrari} where the left-edge projection is not Markov, see Section \ref{BorodinFerrariCouplings} for more details.
    \item We obtain, via the use of intertwining relations, by developing analogues for inhomogeneous space of ideas of Dieker and Warren \cite{DiekerWarren}\footnote{The paper \cite{DiekerWarren} deals with the level/particle inhomogeneous setting.}, explicit formulae for the transition probabilities of the autonomous TASEP-like and pushTASEP-like particle systems in inhomogeneous space and time on the left and right edge respectively of our dynamics on arrays. Moreover, from the very structure of these formulae we can show, essentially without any computations at all, that the distributions of these particle systems, starting from any deterministic initial condition, are marginals of explicit (signed) measures with determinantal correlation functions. See Sections \ref{EdgeParticleSystemsIntro} and \ref{SectionEdgeTransitionKernels} for more details.

    \item We discover a novel duality relation for push-block dynamics in interlacing arrays which maps inhomogeneities in space to inhomogeneities of the level of the array (inhomogeneities on the particles) and vice versa, see Sections \ref{IntroDualitySection} and \ref{SectionDuality}. 
\end{itemize}

We also prove a number of other results, which to some readers may be more interesting than the above, including the following.
\begin{itemize}
    \item We extend the work of Nordenstam \cite{Nordenstam}, which dealt with the case of the uniform weight, explaining how the dynamics of the shuffling algorithm on domino tilings of the Aztec diamond with completely general weights are explicitly connected to Bernoulli push-block dynamics on interlacing arrays (with a certain time-shift). We then, for a special class of weights, connect this back to our earlier probabilistic results. See Sections \ref{ShufflingIntro} and \ref{SectionShuffling} for more details.
    \item We connect the measures coming from our dynamics on arrays to so-called coherent sequences of measures on a natural inhomogeneous generalisation of the classical Gelfand-Tsetlin graph \cite{OlshanskiHarmonic,BorodinOlshanskiHarmonic,BorodinOlshanskiBoundary,YoungBouquet} and prove that all these measures are extreme points in the convex set of coherent measures, see Sections \ref{SectionGraphIntro} and \ref{SectionGraph}. The ultimate goal here would be to have a complete classification of such extreme points, see \cite{BorodinOlshanskiBook,BorodinOlshanskiBoundary,YoungBouquet} for motivation for such problems, as was done in \cite{BorodinOlshanskiBoundary,VershikKerovUnitary,VershikKerovSymmetric,OlshanskiVershik,OlshanskiExtendedGelfandTsetlin,MatveevMacdonald} for allied models, but this appears to be a difficult task. 
    \item In another direction, we prove that independent walks in inhomogeneous space, each with a different drift, under conditions on the relative strengths of the drifts, and which are conditioned to never intersect have explicit transition probabilities, see Sections \ref{SpaceLevelInhomogeneousSectionIntro} and \ref{SectionConditionedWalks}. We note that this model does not fall into the general framework of \cite{DenisovWachtel} and subsequent works for which the increments of the walks are not location-dependent. Moreover, this Markov process matches the evolution of a fixed level/row of certain dynamics in arrays. It would be of particular interest to be able\footnote{We cannot do this yet. However, we hope to return to this problem in future work.} to take all the walks to be identical  since this would imply that the corresponding non-intersecting paths have a novel Gibbs resampling property in analogy to the Brownian Gibbs property \cite{CorwinHammond,BulkPropertiesAiry}, but with Brownian motion replaced by general inhomogeneous walks. The Brownian Gibbs property has been a key tool in studying KPZ universality class models \cite{CorwinKPZsurvey} and it would be interesting to see what could be done with the inhomogeneous-walk Gibbs property just alluded to.
    \item Using the explicit formulae of the correlation functions for the models we study, we prove two limit theorems. First, a short-time asymptotic result for the continuous-time dynamics we consider, see Sections \ref{SectionSpaceTimeCorrIntro} and \ref{SectionBesselConvergence}. In the limit we obtain the discrete Bessel determinantal point process \cite{BorodinOkounkovOlshanski,JohanssonDiscrete} whose parameter depends on the inhomogeneous environment only through a kind of average. Second, a limit theorem (under certain technical conditions) for the bottom paths of $N$ non-intersecting random walks in inhomogeneous space and time, with fixed starting and end points, see Sections \ref{LineEnsemblSetionIntro} and \ref{SectionLineEnsembles}. We strongly believe, that using our exact formulae, along with a more substantial analysis, more significant asymptotic theorems could be established, also in other scaling regimes, and the above results can be viewed as a proof of concept for this.

\end{itemize}

In the next section we introduce the models we study, state our main results precisely and discuss some of the ideas and techniques and relevant literature in more detail.

\paragraph{Acknowledgements} I am very grateful to Sunil Chhita for very useful discussions on the shuffling algorithm for the Aztec diamond. I am very grateful to Mustazee Rahman for very useful discussions on the transition probabilities and determinantal point processes for the edge particle systems and on how to invert certain matrices in Section \ref{SectionEdgeTransitionKernels}. A few of the ideas presented in that section arose from the discussions with Mustazee and I thank him very much. I am very grateful to Alexei Borodin for first telling me about the factorial Schur polynomials and for comments and pointers to the literature.

\section{Models and results}\label{SectionModels}

\subsection{General notation and terminology}\label{SectionNotation}
We introduce some notation and terminology that will be used throughout the paper. Let $\mathbb{R}_+=[0,\infty)$, $\mathbb{Z}_+=\{0,1,2,\dots\}$ and for $x_1,x_2\in \mathbb{Z}_+$, let $\llbracket x_1,x_2 \rrbracket=\{x_1,x_1+1,\dots,x_2\}$. Let $\mathbf{1}_\mathcal{A}$ be the indicator function of a set or event $\mathcal{A}$.

The most basic data in this work is a sequence/field of inhomogeneities $\mathbf{a}=\left(a_x\right)_{x\in \mathbb{Z}_+}$ on $\mathbb{Z}_+$. We assume throughout the paper that 
\begin{equation}\label{C1C2Def}
  \inf_{x\in \mathbb{Z}_+}a_x>0 \textnormal{ and } \sup_{x\in \mathbb{Z}_+} a_x <\infty.
\end{equation}

\begin{defn}
 Associated to $\mathbf{a}$ we define the sequence of ``characteristic polynomials" $p_x(z)=p_x(z;\mathbf{a})$ indexed by $x\in \mathbb{Z}_+$, having degree $x$, by $p_0(z)=1$ and 
\begin{equation*}
p_x(z)=p_x\left(z;\mathbf{a}\right)=\prod_{k=0}^{x-1}\left(1-\frac{z}{a_k}\right).
\end{equation*}   
\end{defn}

We define the Weyl chamber 
\begin{equation*}
 \mathbb{W}_N=\left\{\mathbf{x}=(x_1,x_2,\dots,x_N)\in \mathbb{Z}_+^N:x_1 < x_2 <\cdots < x_N\right\}.
\end{equation*}
We say that $\mathbf{x} \in \mathbb{W}_N$ and $\mathbf{y} \in \mathbb{W}_{N+1}$ interlace and denote this by $\mathbf{x}\prec \mathbf{y}$ if 
\begin{equation*}
y_1\le x_1 < y_2 \le x_2 < \cdots < y_{N} \le x_N < y_{N+1}.
\end{equation*}
%We note the positions of the strict inequalities. 
We also say that  $\mathbf{x} \in \mathbb{W}_N$ interlaces with $\mathbf{y} \in \mathbb{W}_{N}$ and abusing notation still write $\mathbf{x}\prec \mathbf{y}$ if 
\begin{equation*}
x_1 \le y_1 < x_2 \le y_2 < \cdots <x_N \le y_N.
\end{equation*}
\begin{defn}
 We define interlacing arrays of length $N$ by
\begin{equation*}
 \mathbb{IA}_N=\left\{\left(\mathbf{x}^{(1)},\mathbf{x}^{(2)},\dots,\mathbf{x}^{(N)}\right)\in \mathbb{W}_1\times \mathbb{W}_2\times \cdots \times \mathbb{W}_N:\mathbf{x}^{(i)}\prec \mathbf{x}^{(i+1)}, \  \textnormal{ for } i=1,\dots,N-1\right\}. 
\end{equation*}   
\end{defn}
We call $\mathbf{x}^{(N)}$ the $N$-th level of the array and individual coordinates $\mathsf{x}_i^{(N)}$ particles. Denote the set of infinite interlacing sequences $\left(\mathbf{x}^{(N)}\right)_{N\ge 1}$, $\mathbf{x}^{(1)}\prec \mathbf{x}^{(2)} \prec \mathbf{x}^{(3)} \prec \cdots$, by $\mathbb{IA}_\infty$. The following distinguished configuration will make its appearance often in the text.
\begin{defn}
 Define the fully-packed configuration in $\mathbb{IA}_\infty$ as the configuration $\left(\mathbf{x}^{(N)}\right)_{N\ge 1}$  with $\mathbf{x}^{(N)}=\left(0,1,\dots,N-1\right)$ for all $N\ge 1$.
\end{defn}

Given a function $f$ holomorphic in the half plane $\{z\in \mathbb{C}:\Re(z)>-\epsilon \}$ for some $\epsilon>0$ define $\mathsf{T}_f(x,y)=\mathsf{T}^\mathbf{a}_f(x,y)$ by
\begin{equation}
 \mathsf{T}_f(x,y)=-\frac{1}{2\pi \textnormal{i}} \frac{1}{a_y}\oint_{\mathsf{C}_\mathbf{a}}\frac{p_x(w)}{p_{y+1}(w)}f(w)dw, \ \ x,y \in \mathbb{Z}_+,   
\end{equation}
where the positively oriented contour $\mathsf{C}_{\mathbf{a}}\subset \{z\in \mathbb{C}:\Re(z)>-\epsilon \}$ is assumed to encircle all the points $\{a_x\}_{x\in \mathbb{Z}_+}$ (note that it does not contain any poles of $f$ by the analyticity assumption in the half plane). Observe that, this is a special case of $\mathcal{T}_f^{\mathbf{u},\mathbf{v}}$ from (\ref{ToeplitzGeneralisation}) with $u_k(z)=p_k(z)$, $v_k(z)=-a_kp_{k+1}(z)$, $\mathfrak{C}_{\mathbf{u},\mathbf{v}}=\mathsf{C}_\mathbf{a}$. When $a_x=1$, for all $x \in \mathbb{Z}_+$, then we can pick $\mathsf{C}_\mathbf{a}$ to be the circle $ \left\{z\in \mathbb{C}:|1-z|=1 \right\}$ and $[\mathsf{T}_f(x,y)]_{x,y\in \mathbb{Z}_+}$ is easily seen to be the Toeplitz matrix with symbol $f(1-z)$. 

We will be mainly interested in three particular choices of $f$ having probabilistic significance, $f(z)=e^{-tz}$ or $(1-\alpha z)$ or $(1+\beta z)^{-1}$, see Section \ref{Section1D} for more details. $\left(\mathsf{T}_{e^{-tz}}\right)_{t\ge 0}$ is the transition semigroup of a pure-birth chain with jump rate $a_x$ when at location $x\in \mathbb{Z}_+$. $\mathsf{T}_{1-\alpha z}$ is the single-step transition probability of a Bernoulli random walk with probability of moving to $x+1$, when at $x$, given by $\alpha  a_x$ and complementary probability $(1-\alpha a_x)$ for staying at $x$. $\mathsf{T}_{(1+\beta z)^{-1}}$ is the transition probability an inhomogeneous geometric walk with single-step probability to go from $x$ to $y\ge x$, given by
$(1+\beta a_y)^{-1}\prod_{k=x}^{y-1}\beta a_k (1+\beta a_k)^{-1}$.

 We normally denote a random configuration in $\mathbb{IA}_\infty$, whose law will be explicitly specified or be clear from context, by $\left(\mathsf{X}_i^{(n)}\right)_{1\le i \le n;n\ge 1}$. Similarly, we normally denote a stochastic process in $\mathbb{IA}_\infty$ either in discrete or continuous time, whose dynamics will be explicitly specified or be clear from context, by $\left(\mathsf{X}_i^{(n)}(t);t \ge 0\right)_{1\le i \le n;n\ge 1}$. At certain places, when there is no risk of confusion, we will drop the explicit dependence on time $t$ to ease notation. We will always, unless otherwise explicitly stated, denote the underlying law of the various random elements we encounter (how they are dependent on each other will always be specified or be clear from context) by $\mathbb{P}$.

\paragraph{Warning on notation} In two sections of this introductory part, Section \ref{IntroDualitySection} and \ref{LineEnsemblSetionIntro} only, and Sections \ref{SectionDuality} and \ref{SectionLineEnsembles} only, where the corresponding proofs can be found, it will be preferable, for reasons explained therein, to label levels of arrays starting from $0$ (which has one particle) instead of $1$ and coordinates of $\mathbb{W}_N$ starting from subscript $0$ instead of $1$. In particular, with this convention, the configuration of the $N+1$ particles at level $N$, which is in $\mathbb{W}_{N+1}$, will have coordinates $(x_0^{(N)},x_1^{(N)},\dots,x_N^{(N)})$. We will remind the reader of this at the relevant places.

\subsection{Space-time inhomogeneous dynamics and correlation kernels}\label{SectionSpaceTimeCorrIntro}

We introduce three types of Markov dynamics in $\mathbb{IA}_\infty$, one in continuous time and two in discrete time. We call these the continuous-time pure-birth push-block dynamics, sequential-update Bernoulli dynamics and Warren-Windridge geometric dynamics. These generalise the models considered in \cite{BorodinFerrari,WarrenWindridge,DiekerWarren,Nordenstam} to the space-inhomogeneous setting. Using our methods it would be possible to analyse other types of dynamics in discrete time as well, including inhomogeneous generalisations of the sequential-update Borodin-Ferrari geometric dynamics and parallel-update Bernoulli, see Section \ref{BorodinFerrariCouplings} for more details.

First, we define the dynamics in continuous time. This model first appeared in \cite{DeterminantalStructures}.

\begin{defn}[Continuous-time pure-birth push-block dynamics]\label{DefCtsTimeDynamics}
Each particle has an independent exponential clock with rate $a_x$ if the particle is at spatial location $x$. When the clock of the particle rings, say of $\mathsf{X}_i^{(n)}=x$, then it will attempt to jump to $x+1$. If $\mathsf{X}_i^{(n-1)}=x$ then the move is suppressed for otherwise the interlacing would break down. We say the particle is blocked. If $\mathsf{X}_i^{(n-1)}>x$ then $\mathsf{X}_i^{(n+1)}$ moves to $x+1$. If $\mathsf{X}_{i+1}^{(n+1)}=x+1$, then $\mathsf{X}_{i+1}^{(n+1)}$ is instantaneously moved to $x+2$ (we say it is pushed) so that the interlacing remains true. This pushing is propagated instantaneously to higher levels. See Figure \ref{CtsDynamics} for an illustration. Finally, we say that associated to these dynamics we have the function $f(z)=e^{-tz}$. 
\end{defn}

We now introduce the discrete-time dynamics. It is easy to see that we can also define these using certain recursive equations, which may be clearer to some readers compared to the descriptive definitions below, and we will make this explicit shortly.

\begin{defn}[Sequential-update Bernoulli dynamics]\label{DefBernoulliDynamics}
These dynamics are in discrete-time. Each time-step depends on an additional parameter $0\le \alpha \le (\sup_x a_x)^{-1}$ (which can change for each time-step). For each time-step locations are updated sequentially from lower levels to higher levels and from left to right within each level. Particle $\mathsf{X}_1^{(1)}$ moves as an inhomogeneous Bernoulli random walk with probability to go from $x$ to $x+1$ given by $\alpha a_x$ and complementary probability $1-\alpha a_x$ to stay at $x$. Suppose we have updated the first $n-1$ levels and the first $i-1$ particles on that level. Particle $\mathsf{X}_i^{(n)}$ checks if $\mathsf{X}_{i}^{(n-1)}=x$ in which case it is blocked and we move to update particle $\mathsf{X}_{i+1}^{(n)}$. Otherwise, it moves as an inhomogeneous Bernoulli walk with the above mentioned probabilities. If $\mathsf{X}_{i+1}^{(n+1)}=x+1$ then it is instantaneously moved to $x+2$ so that the interlacing remains and this pushing is propagated to higher levels. Particles that have been pushed do not attempt to move again (for example in the above scenario $\mathsf{X}_{i+1}^{(n+1)}$ does not attempt to move again when we update level $n+1$). See Figure \ref{BernoulliDynamics} for an illustration. Finally, we say that associated to this single time-step of sequential-update Bernoulli dynamics with parameter $\alpha$ we have the function $f(z)=1-\alpha z$.
\end{defn}

\begin{rmk}\label{RmkEquivDyn}
There is an alternative way to view the pushing move which may be more natural. There is no instantaneous pushing but instead when we try to update particle $\mathsf{X}_i^{(n)}$ at location $x$, in addition to checking if $\mathsf{X}_i^{(n-1)}=x$, in which case it is blocked, we also check if $\mathsf{X}_{i-1}^{(n-1)}=x$ in which case $\mathsf{X}_i^{(n)}$ is moved to $x+1$. If neither possibility occurs, then $\mathsf{X}_i^{(n)}$ simply moves as an inhomogeneous Bernoulli random walk with the probabilities above. It is clear that the resulting configuration at the end of the time-step is the same as the one obtained from Definition \ref{DefBernoulliDynamics}.
\end{rmk}

\begin{defn}[Warren-Windridge geometric dynamics]\label{DefGeometricDynamics} Particles move in discrete time and each time-step depends on an additional parameter $\beta \ge 0$ (which can change for each time-step). Particle locations are updated sequentially from lower levels to higher levels and from left to right. Particle $\mathsf{X}_1^{(1)}$ moves as an inhomogeneous geometric random walk with transition probability to go from $x$ to $x'\ge x$ given by $(1+\beta a_{x'})^{-1}\prod_{k=x}^{x'-1}\beta a_k(1+\beta a_k)^{-1}$. Suppose $\mathsf{X}_1^{(1)}$ moves to $x'$. We then update the next level. $\mathsf{X}_1^{(2)}=y$ will attempt to jump to location $y'$ with probability $(1+\beta a_{y'})^{-1}\prod_{k=y}^{y'-1}\beta a_k(1+\beta a_k)^{-1}$. However, any jumps past location $x$ will be suppressed. More precisely, with probability $\sum_{y'=x}^\infty (1+\beta a_{y'})^{-1}\prod_{k=y}^{y'-1}\beta a_k(1+\beta a_k)^{-1}$ the new location of $\mathsf{X}_1^{(2)}$ will be $x$. We note that $\mathsf{X}_1^{(2)}$ is blocked by the location $x$ of $\mathsf{X}_1^{(1)}$ at the beginning of the time-step (and not  its updated location $x'$). We then update $\mathsf{X}_2^{(2)}$ which is at location $z$. If $z>x'$, then $\mathsf{X}_2^{(2)}$ attempts to move to $z'\ge z$ with inhomogeneous geometric probability $(1+\beta a_{z'})^{-1}\prod_{k=z}^{z'-1}\beta a_k(1+\beta a_k)^{-1}$. If however $x'\ge z$ then $\mathsf{X}_2^{(2)}$ is moved/pushed to the intermediate position $x'+1$. From there it attempts to move to $z'\ge x'+1$ with probability $(1+\beta a_{z'})^{-1}\prod_{k=x'+1}^{z'-1}\beta a_k(1+\beta a_k)^{-1}$ Higher levels are updated in the same fashion. Note that, the push-block interactions ensure the process stays in $\mathbb{IA}_\infty$. See Figure \ref{GeometricDynamics} for an illustration of the dynamics. Finally, we say that associated to a single time-step of Warren-Windridge dynamics with parameter $\beta$ we have the function $f(z)=(1+\beta z)^{-1}$.
\end{defn}

\begin{figure}
\captionsetup{singlelinecheck = false, justification=justified}
\centering
\begin{tikzpicture}

\draw[dotted] (0,0) grid (5,2);

 \draw[fill] (1,0) circle [radius=0.1];
 
\draw[fill] (1,1) circle [radius=0.1];

\draw[fill] (0,2) circle [radius=0.1];

\draw[fill] (2,2) circle [radius=0.1];

\draw[fill] (3,1) circle [radius=0.1];

\draw[fill] (3,1) circle [radius=0.1];

\draw[fill] (4,2) circle [radius=0.1];

\draw[] (4,1) circle [radius=0.1];

\draw[] (5,2) circle [radius=0.1];

\draw[] (3,2) circle [radius=0.1];

\node[below] at (0,-0.3) {$x$};

\node[below] at (1,-0.2) {$x+1$};

\node[below] at (2,-0.2) {$x+2$};

\node[below] at (3,-0.2) {$x+3$};

\node[below] at (4,-0.2) {$x+4$};

\node[below] at (5,-0.2) {$x+5$};

\node[right] at (1,0) {$\mathsf{X}_1^{(1)}$};

\node[below left] at (1,1) {$\mathsf{X}_1^{(2)}$};

\node[above left] at (0,2) {$\mathsf{X}_1^{(3)}$};

\node[left] at (2,2) {$\mathsf{X}_2^{(3)}$};

\node[above] at (4,2.25) {$\mathsf{X}_3^{(3)}$};

\node[below right] at (3.1,1) {$\mathsf{X}_2^{(2)}$};

\draw[rotate around={135:(3.5,1.5)},dotted, thick] (3.5,1.5) ellipse (10pt and 30pt);

\draw[middlearrow={>}, very thick] (2.1,2.1) to [out=45, in=135] (2.9,2.1);

\draw[middlearrow={>}, very thick] (3.1,1.1) to [out=45, in=135] (3.9,1.1);

\draw[middlearrow={>}, very thick] (4.1,2.1) to [out=45, in=135] (4.9,2.1);

\draw[very thick, ->] (1.1,1.1) to [out=45, in=135] (1.9,1.1);

\draw[very thick,red] (1.3,1.5) to (1.7,1);

\draw[very thick,red] (1.3,1) to (1.7,1.5);

\node[above] at (2.5,2.5) {rate $a_{x+2}$};

\node[above left] at (1.2,1.2) {rate $a_{x+1}$};

\node[right] at (4,1.35) {rate $a_{x+3}$};

\end{tikzpicture}

\caption{A depiction of the continuous-time pure-birth chain push-block dynamics. Particles attempt to jump to the right by one at rate $a_x$ if at spatial location $x$. For example, in the figure $\mathsf{X}_2^{(3)}=x+2$ jumps at rate $a_{x+2}$ to $x+3$. The jump, at rate $a_{x+1}$, of particle $\mathsf{X}_1^{(2)}=x+1$ is blocked for otherwise it would violate the interlacing with particle $\mathsf{X}_1^{(1)}$. The jump, at rate $a_{x+3}$, of $\mathsf{X}_2^{(2)}$ to $x+4$ induces a simultaneous jump, namely it pushes particle $\mathsf{X}_3^{(3)}$ to $x+5$ to keep the interlacing. Observe that, the evolution of the particle systems $(\mathsf{X}_1^{(3)},\mathsf{X}_1^{(2)},\mathsf{X}_1^{(1)})$ and $(\mathsf{X}_1^{(1)},\mathsf{X}_2^{(2)},\mathsf{X}_3^{(3)})$ at the left and right edge respectively is autonomous and thus Markovian. }\label{CtsDynamics}
\end{figure}
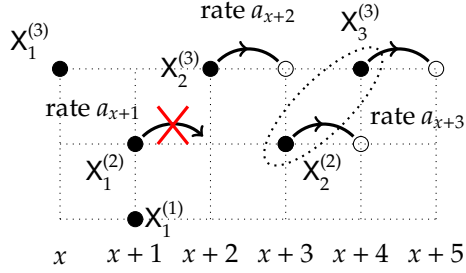

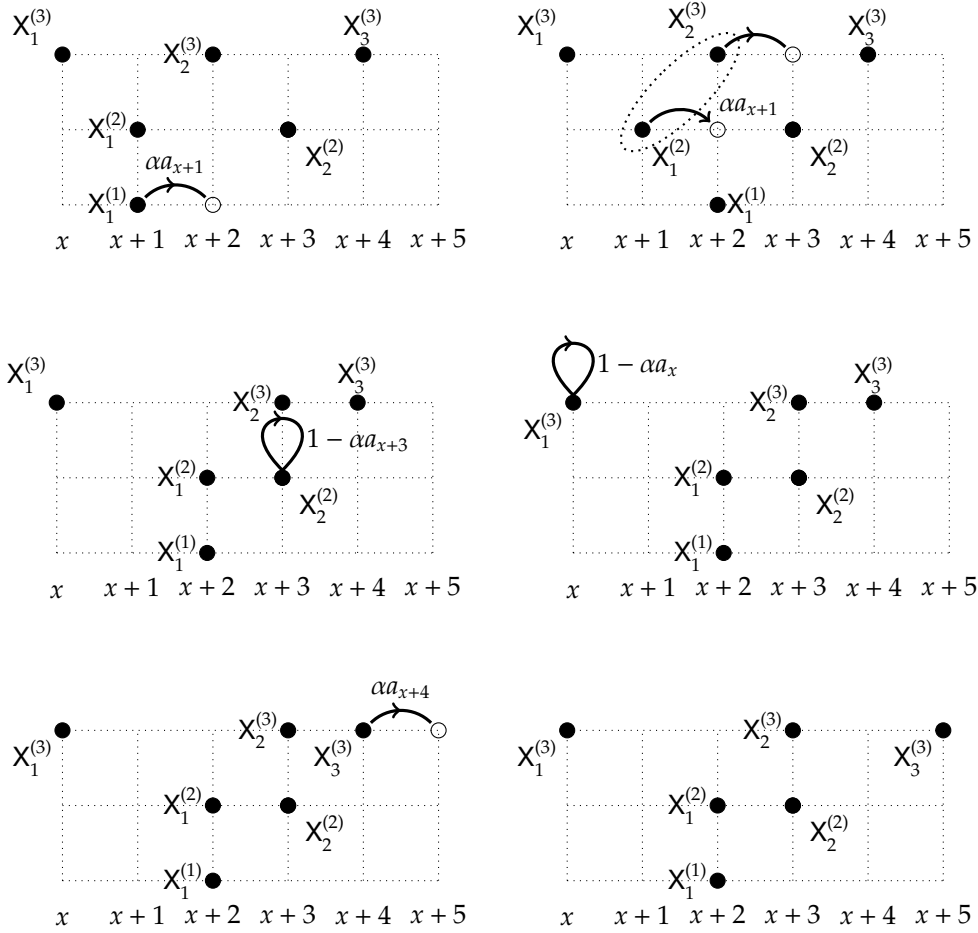
\begin{figure}
\captionsetup{singlelinecheck = false, justification=justified}
\centering

\begin{tikzpicture}

\draw[dotted] (0,0) grid (5,2);

 \draw[fill] (1,0) circle [radius=0.1];
 
\draw[fill] (1,1) circle [radius=0.1];

\draw[fill] (0,2) circle [radius=0.1];

\draw[fill] (2,2) circle [radius=0.1];

\draw[fill] (3,1) circle [radius=0.1];

\draw[fill] (3,1) circle [radius=0.1];

\draw[fill] (4,2) circle [radius=0.1];

\draw[] (2,0) circle [radius=0.1];

\node[below] at (0,-0.3) {$x$};

\node[below] at (1,-0.2) {$x+1$};

\node[below] at (2,-0.2) {$x+2$};

\node[below] at (3,-0.2) {$x+3$};

\node[below] at (4,-0.2) {$x+4$};

\node[below] at (5,-0.2) {$x+5$};

\node[left] at (1,0) {$\mathsf{X}_1^{(1)}$};

\node[ left] at (1,1) {$\mathsf{X}_1^{(2)}$};

\node[above left] at (0,2) {$\mathsf{X}_1^{(3)}$};

\node[left] at (2,2) {$\mathsf{X}_2^{(3)}$};

\node[above] at (4,2) {$\mathsf{X}_3^{(3)}$};

\node[below right] at (3.1,1) {$\mathsf{X}_2^{(2)}$};

\node[above] at (1.5,0.3) {$\alpha a_{x+1}$};

\draw[middlearrow={>}, very thick] (1.1,0.1) to [out=45, in=135] (1.9,0.1);

\end{tikzpicture}\ \ \ \ \ 
\begin{tikzpicture}

\draw[dotted] (0,0) grid (5,2);

 \draw[fill] (2,0) circle [radius=0.1];
 
\draw[fill] (1,1) circle [radius=0.1];

\draw[fill] (0,2) circle [radius=0.1];

\draw[fill] (2,2) circle [radius=0.1];

\draw[fill] (3,1) circle [radius=0.1];

\draw[fill] (3,1) circle [radius=0.1];

\draw[fill] (4,2) circle [radius=0.1];

\draw[] (3,2) circle [radius=0.1];

\draw[] (2,1) circle [radius=0.1];

\node[below] at (0,-0.3) {$x$};

\node[below] at (1,-0.2) {$x+1$};

\node[below] at (2,-0.2) {$x+2$};

\node[below] at (3,-0.2) {$x+3$};

\node[below] at (4,-0.2) {$x+4$};

\node[below] at (5,-0.2) {$x+5$};

\node[right] at (2,0) {$\mathsf{X}_1^{(1)}$};

\node[below right] at (1,1) {$\mathsf{X}_1^{(2)}$};

\node[above left] at (0,2) {$\mathsf{X}_1^{(3)}$};

\node[above left] at (2,2.1) {$\mathsf{X}_2^{(3)}$};

\node[above] at (4,2) {$\mathsf{X}_3^{(3)}$};

\node[below right] at (3.1,1) {$\mathsf{X}_2^{(2)}$};

\draw[middlearrow={>}, very thick] (2.1,2.1) to [out=45, in=135] (2.9,2.1);

\draw[very thick, ->] (1.1,1.1) to [out=45, in=135] (1.9,1.1);

\draw[rotate around={135:(1.5,1.5)},dotted, thick] (1.5,1.5) ellipse (10pt and 30pt);

\node[right] at (1.9,1.3) {$\alpha a_{x+1}$};

\end{tikzpicture}

\bigskip

\bigskip

\begin{tikzpicture}

\draw[dotted] (0,0) grid (5,2);

 \draw[fill] (2,0) circle [radius=0.1];
 
\draw[fill] (2,1) circle [radius=0.1];

\draw[fill] (0,2) circle [radius=0.1];

\draw[fill] (3,2) circle [radius=0.1];

\draw[fill] (3,1) circle [radius=0.1];

\draw[fill] (3,1) circle [radius=0.1];

\draw[fill] (4,2) circle [radius=0.1];

\node[below] at (0,-0.3) {$x$};

\node[below] at (1,-0.2) {$x+1$};

\node[below] at (2,-0.2) {$x+2$};

\node[below] at (3,-0.2) {$x+3$};

\node[below] at (4,-0.2) {$x+4$};

\node[below] at (5,-0.2) {$x+5$};

\node[left] at (2,0) {$\mathsf{X}_1^{(1)}$};

\node[left] at (2,1) {$\mathsf{X}_1^{(2)}$};

\node[above left] at (0,2) {$\mathsf{X}_1^{(3)}$};

\node[left] at (3,2) {$\mathsf{X}_2^{(3)}$};

\node[above] at (4,2) {$\mathsf{X}_3^{(3)}$};

\node[below right] at (3.1,1) {$\mathsf{X}_2^{(2)}$};

\draw[middlearrow={>}, very thick]  (3,1.1) to [out=135, in=45, distance=13mm] (3,1.1);

\node[right] at (3.2,1.5) {$1-\alpha a_{x+3}$};

\end{tikzpicture}\ \ \ \ \ 
\begin{tikzpicture}

\draw[dotted] (0,0) grid (5,2);

 \draw[fill] (2,0) circle [radius=0.1];
 
\draw[fill] (2,1) circle [radius=0.1];

\draw[fill] (0,2) circle [radius=0.1];

\draw[fill] (3,2) circle [radius=0.1];

\draw[fill] (3,1) circle [radius=0.1];

\draw[fill] (3,1) circle [radius=0.1];

\draw[fill] (4,2) circle [radius=0.1];

\node[below] at (0,-0.3) {$x$};

\node[below] at (1,-0.2) {$x+1$};

\node[below] at (2,-0.2) {$x+2$};

\node[below] at (3,-0.2) {$x+3$};

\node[below] at (4,-0.2) {$x+4$};

\node[below] at (5,-0.2) {$x+5$};

\node[left] at (2,0) {$\mathsf{X}_1^{(1)}$};

\node[left] at (2,1) {$\mathsf{X}_1^{(2)}$};

\node[below left] at (0,2) {$\mathsf{X}_1^{(3)}$};

\node[left] at (3,2) {$\mathsf{X}_2^{(3)}$};

\node[above] at (4,2) {$\mathsf{X}_3^{(3)}$};

\node[below right] at (3.1,1) {$\mathsf{X}_2^{(2)}$};

\draw[middlearrow={>}, very thick]  (0,2.1) to [out=135, in=45, distance=13mm] (0,2.1);

\node[right] at (0.2,2.5) {$1-\alpha a_{x}$};

\end{tikzpicture}

\bigskip

\bigskip

\begin{tikzpicture}

\draw[dotted] (0,0) grid (5,2);

 \draw[fill] (2,0) circle [radius=0.1];
 
\draw[fill] (2,1) circle [radius=0.1];

\draw[fill] (0,2) circle [radius=0.1];

\draw[fill] (3,2) circle [radius=0.1];

\draw[fill] (3,1) circle [radius=0.1];

\draw[fill] (3,1) circle [radius=0.1];

\draw[fill] (4,2) circle [radius=0.1];

\draw[] (5,2) circle [radius=0.1];

\node[below] at (0,-0.3) {$x$};

\node[below] at (1,-0.2) {$x+1$};

\node[below] at (2,-0.2) {$x+2$};

\node[below] at (3,-0.2) {$x+3$};

\node[below] at (4,-0.2) {$x+4$};

\node[below] at (5,-0.2) {$x+5$};

\node[left] at (2,0) {$\mathsf{X}_1^{(1)}$};

\node[left] at (2,1) {$\mathsf{X}_1^{(2)}$};

\node[below left] at (0,2) {$\mathsf{X}_1^{(3)}$};

\node[left] at (3,2) {$\mathsf{X}_2^{(3)}$};

\node[below left] at (4,2) {$\mathsf{X}_3^{(3)}$};

\node[below right] at (3.1,1) {$\mathsf{X}_2^{(2)}$};

\draw[middlearrow={>}, very thick] (4.1,2.1) to [out=45, in=135] (4.9,2.1);

\node[above] at (4.5,2.3) {$\alpha a_{x+4}$};

\end{tikzpicture}\ \ \ \ \
\begin{tikzpicture}

\draw[dotted] (0,0) grid (5,2);

 \draw[fill] (2,0) circle [radius=0.1];
 
\draw[fill] (2,1) circle [radius=0.1];

\draw[fill] (0,2) circle [radius=0.1];

\draw[fill] (3,2) circle [radius=0.1];

\draw[fill] (3,1) circle [radius=0.1];

\draw[fill] (3,1) circle [radius=0.1];

\draw[fill] (5,2) circle [radius=0.1];

\node[below] at (0,-0.3) {$x$};

\node[below] at (1,-0.2) {$x+1$};

\node[below] at (2,-0.2) {$x+2$};

\node[below] at (3,-0.2) {$x+3$};

\node[below] at (4,-0.2) {$x+4$};

\node[below] at (5,-0.2) {$x+5$};

\node[left] at (2,0) {$\mathsf{X}_1^{(1)}$};

\node[left] at (2,1) {$\mathsf{X}_1^{(2)}$};

\node[below left] at (0,2) {$\mathsf{X}_1^{(3)}$};

\node[left] at (3,2) {$\mathsf{X}_2^{(3)}$};

\node[below left] at (5,2) {$\mathsf{X}_3^{(3)}$};

\node[below right] at (3.1,1) {$\mathsf{X}_2^{(2)}$};

\end{tikzpicture}

\caption{A depiction of the Bernoulli sequential-update discrete-time dynamics. We update particles from lower levels to upper levels and left to right. Each particle has its own $0-1$ Bernoulli random variable of parameter $\alpha a_x$, when at spatial location $x$, to decide whether to (attempt to) jump to the right by one or to stay put. In the figure above next to each arrow the corresponding probability of that move is given. Note that, when we update the top row $\mathsf{X}_2^{(3)}$ does not attempt to move as it has already been pushed. An alternative, but obviously equivalent way to view the pushing interaction, see Remark \ref{RmkEquivDyn}, is the following: $\mathsf{X}_1^{(2)}=x$ moves to the right by one to $x+2$, $\mathsf{X}_2^{(3)}=x+2$ is not pushed immediately, but when we update the top row it is necessarily moved to $x+3$, irrespective of its own Bernoulli random variable, so that the interlacing remains. As for continuous-time dynamics, the evolution of the particle systems $(\mathsf{X}_1^{(3)},\mathsf{X}_1^{(2)},\mathsf{X}_1^{(1)})$ and $(\mathsf{X}_1^{(1)},\mathsf{X}_2^{(2)},\mathsf{X}_3^{(3)})$ at the left and right edge respectively is autonomous and thus Markovian.}\label{BernoulliDynamics}
\end{figure}

\begin{figure}
\captionsetup{singlelinecheck = false, justification=justified}
\centering

\begin{tikzpicture}
    \draw[fill] (2,0) circle [radius=0.1];

\draw[fill] (0,1) circle [radius=0.1];

\draw[fill] (4,1) circle [radius=0.1];

\draw[] (6,0) circle [radius=0.1];

\draw[blue] (6.5,1) circle [radius=0.1];

\draw[middlearrow={>}, very thick] (2,0) to (5.9,0);

\draw[middlearrow={>}, very thick,dotted] (4.1,1) to (6.4,1);

\node[above left] at (0,1) {$\mathsf{X}_1^{(2)}$};

\node[above left] at (2,0) {$\mathsf{X}_1^{(1)}$};

\node[above left] at (4,1) {$\mathsf{X}_2^{(2)}$};

\node[below] at (0,0.9) {$y$};

\node[below] at (2,-0.1) {$x$};

\node[below] at (4,0.9) {$z$};

\node[below] at (6,-0.1) {$x'$};

\node[below] at (6.5,0.9) {$x'+1$};

\node[below] at (4,-0.1) {$\frac{1}{1+\beta a_{x'}}\prod_{k=x}^{x'-1}\frac{\beta a_k}{1+\beta a_k}$};

\end{tikzpicture}

\bigskip

\begin{tikzpicture}
    \draw[] (2,0) circle [radius=0.1];

\draw[fill] (0,1) circle [radius=0.1];

\draw[fill] (6,0) circle [radius=0.1];

\draw[blue] (6.5,1) circle [radius=0.1];

\draw[middlearrow={>}, very thick] (0,1) to (1.9,1);

    \draw[] (2,1) circle [radius=0.1];

\draw[ very thick,dotted,red] (2,0.1) to (2,1.5);

\node[left] at (0,1) {$\mathsf{X}_1^{(2)}$};

\node[left] at (6,0) {$\mathsf{X}_1^{(1)}$};

%\node[above] at (6.5,1.1) {$\tilde{\mathsf{X}}_2^{(2)}$};

\node[below] at (0,0.9) {$y$};

\node[below] at (2,-0.1) {$x$};

\node[below] at (6,-0.1) {$x'$};

\node[below] at (6.5,0.9) {$x'+1$};

\node[above] at (1,1.2) {$\prod_{k=y}^{x-1}\frac{\beta a_k}{1+\beta a_k}$};

\draw[middlearrow={>}, very thick] (6.6,1) to (8.4,1);

    \draw[] (8.5,1) circle [radius=0.1];

    \node[below] at (8.5,0.9) {$z$};

       \node[above right] at (8.5,1) {$\mathsf{X}_2^{(2)}$};

 \node[below ] at (8,0.7) {$\frac{1}{1+\beta a_{z'}}\prod_{k=x'+1}^{z'-1}\frac{\beta a_k}{1+\beta a_k}$};

\end{tikzpicture}

\caption{A depiction of the Warren-Windridge geometric dynamics in discrete time. We update particles from lower levels to upper levels and from left to right. Particle $\mathsf{X}_1^{(1)}=x$ moves to $x'$ with the inhomogeneous geometric probability given in the figure. Since $x'\ge z$, $\mathsf{X}_2^{(2)}$ is moved to an intermediate position $x'+1$. We then move to the second level. $\mathsf{X}_1^{(2)}$ attempts to jump to a location $m\ge x$ but it is blocked by the location of $\mathsf{X}_1^{(1)}$ before it jumped, namely $x$, which becomes the new location of $\mathsf{X}_1^{(2)}$. We note here that,
\begin{equation*}
\sum_{m\ge x}\frac{1}{1+\beta a_m} \prod_{k=y}^{m-1} \frac{\beta a_k}{1+\beta a_k}= \prod_{k=y}^{x-1}\frac{\beta a_k}{1+\beta a_k}.
\end{equation*}
Finally, particle $\mathsf{X}_2^{(2)}$ jumps from its intermediate location $x'+1$ to $z'$ with inhomogeneous geometric probability given in the figure. As before, the evolution of the particle systems $(\mathsf{X}_1^{(2)},\mathsf{X}_1^{(1)})$ and $(\mathsf{X}_1^{(1)},\mathsf{X}_2^{(2)})$ at the left and right edge respectively is autonomous and thus Markovian.}\label{GeometricDynamics}
\end{figure}
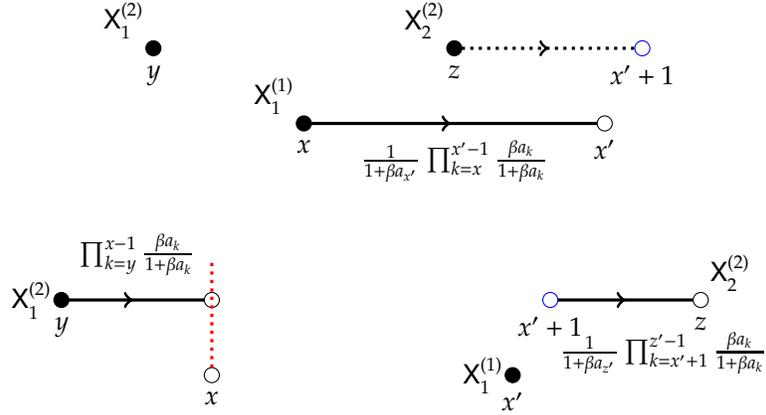

We note that in discrete time we can, and will, consider Markov processes which move at each time-step with either sequential-update Bernoulli or Warren-Windridge geometric dynamics and this will be conveniently encoded in the sequence of associated functions $\left(f_{s,s+1}(z)\right)_{s=0}^M$, with possibly $M=\infty$. We also observe that in all three types of dynamics the evolution of the first $N$ levels in $\mathbb{IA}_N$ is autonomous. In the rest of the paper we will consider generalisations/variations of these dynamics where the jump rates or transition probabilities of individual particles can depend in a more complicated way on time, space location and level of the particle. However, the interactions between particles will always be of exactly the form described in Definitions \ref{DefCtsTimeDynamics}, \ref{DefBernoulliDynamics} and \ref{DefGeometricDynamics} above.

\paragraph{Recursive equations for discrete-time dynamics} We now write down recursive equations describing the dynamics from Definitions \ref{DefBernoulliDynamics} and \ref{DefGeometricDynamics}. We will not make explicit use of these equations in the proofs but it is instructive to present them. In particular, they make the connection to various first/last passage percolation models corresponding to the one-dimensional Markovian projections on the left and right edge of the array that we will discuss next much clearer. We need some more notation.

For $0\le \alpha \le (\sup_x a_x)^{-1}$, define the random field $\mathsf{B}_\alpha=\left(\mathsf{B}_\alpha(x)\right)_{x\in \mathbb{Z_+}}$, by: for $x\in \mathbb{Z}_+$, $\mathsf{B}_\alpha(x)$ are independent and taking values in $\{0,1\}$ with probabilities
\begin{equation*}
\mathbb{P}\left(\mathsf{B}_\alpha(x)=1\right)=1-\mathbb{P}\left(\mathsf{B}_\alpha(x)=0\right)=\alpha a_x.
\end{equation*}
For $\beta \ge 0$, define the random field $\mathsf{G}_\beta=\left(\mathsf{G}_\beta(x)\right)_{x\in \mathbb{Z_+}}$, by: for $x\in \mathbb{Z}_+$, $\mathsf{G}_\alpha(x)$ are independent and taking values in $\mathbb{Z}_+$ with probabilities
\begin{equation*}
\mathbb{P}\left(\mathsf{G}_\beta(x)=n\right)=\left(1+\beta a_{x+n}\right)^{-1}\prod_{k=x}^{x+n-1}\beta a_k \left(1+\beta a_k\right)^{-1}.
\end{equation*}
For each triple $(i,n,t)$ with $1\le i \le n$, $n\ge 1$, $t\in \mathbb{Z}_+$ and parameters $\alpha,\beta$ satisfying the conditions above we denote by $\mathsf{B}^{i,(n)}_{t,\alpha}$ and $\mathsf{G}_{t,\beta}^{i,(n)}$ an independent copy of the fields $\mathsf{B}_\alpha$ and $\mathsf{G}_\beta$ respectively. Suppose we are given parameters $(\alpha_t)_{t\in \mathbb{Z}_+}$ and $(\beta_t)_{t\in \mathbb{Z}_+}$ as above. Then, a process $\left(\mathsf{X}_i^{(n)}(t);t\ge 0\right)_{1\le i \le n; n\ge 1}$ in discrete time following the sequential-update Bernoulli dynamics of Definition \ref{DefBernoulliDynamics} with the $t$-th step (namely the transition from time $t$ to $t+1$) taken with parameter $\alpha_t$ satisfies, and is determined by, the recursive equations:
\begin{equation}\label{BernoulliRecursion}
\mathsf{X}_i^{(n)}(t+1)=\min\left\{\mathsf{X}_i^{(n-1)}(t+1), \max\left\{\mathsf{X}_i^{(n)}(t)+\mathsf{B}_{t,\alpha_t}^{i,(n)}\left(\mathsf{X}_i^{(n)}(t)\right),\mathsf{X}_{i-1}^{(n-1)}(t+1)+1\right\} \right\}.
\end{equation}
Similarly, a process $\left(\mathsf{X}_i^{(n)}(t);t\ge 0\right)_{1\le i \le n; n\ge 1}$ in discrete time following the Warren-Windridge geometric dynamics of Definition \ref{DefGeometricDynamics} with the $t$-th step taken with parameter $\beta_t$ satisfies, and is determined by, the recursive equations:
\begin{align}\label{GeometricRecursion}
\mathsf{X}_i^{(n)}(t+1)=\min\bigg\{\mathsf{X}_i^{(n-1)}(t), &\max\left\{\mathsf{X}_i^{(n)}(t),\mathsf{X}_{i-1}^{(n-1)}(t+1)+1\right\}\nonumber\\
&+\mathsf{G}_{t,\beta_t}^{i,(n)}\left(\max\left\{\mathsf{X}_i^{(n)}(t),\mathsf{X}_{i-1}^{(n-1)}(t+1)+1\right\}\right)\bigg\}.
\end{align}
Of course, we can also take a mixture of Bernoulli and geometric steps and the recursive equations are modified in the obvious way. 

\begin{rmk}
 We note that the recursive equations (\ref{BernoulliRecursion}), (\ref{GeometricRecursion}) above and in particular their marginals on the left and right edge (\ref{BernoulliLeftRec}), (\ref{BernoulliRightRec}), (\ref{GeometricLeftRec}), (\ref{GeometricRightRec}) below are reminiscent but different\footnote{In the space-homogeneous case the recursive equations (\ref{BernoulliLeftRec}), (\ref{BernoulliRightRec}), (\ref{GeometricLeftRec}), (\ref{GeometricRightRec}) for the left and right edge projections do boil down to the RSK recursions, see \cite{DiekerWarren}.} to the recursive equations that come up in the Robinson-Schensted-Knuth (RSK) \cite{Fulton} correspondence. This combinatorial algorithm can also be used to study certain first/last passage percolation models in inhomogeneous space in a dynamical way in terms of interacting particles, see \cite{OConnellYorRepresentation, OConnellTAMS, OConnellConditionedWalk,DiekerWarren,JohanssonRahman,NikosSurvey,NikosTASEP}. However, with dynamics driven by RSK the inhomogeneities of the environment always become time and particle/level inhomogeneities for the particle system. In particular, particles are not moving in inhomogeneous space as in the models we can study with our methods.
   
\end{rmk}

\paragraph{One-dimensional Markov projections} Each of the three dynamics we have considered has two one-dimensional Markovian projections when restricted to the left and right edge of the array, namely to coordinates $\left(x_1^{(n)}\right)_{n\ge 1}$ and $\left(x_n^{(n)}\right)_{n\ge 1}$ respectively.

More precisely, the projection of the continuous-time dynamics on the right edge is given by inhomogeneous space push-TASEP \cite{DeterminantalStructures,PetrovInhomogeneousPushTASEP} in continuous time. The projection to the left edge is an inhomogeneous space zero-range process \cite{DeterminantalStructures}. Under the coordinate shift 
\begin{equation*}
 \left(x_1^{(n)}\right)_{n \ge 1} \mapsto \left(x_1^{(n)}-n+1\right)_{n \ge 1},   
\end{equation*}
it becomes equivalent to an inhomogeneous variant of TASEP (but not the standard and much more well-known inhomogeneous TASEP of \cite{SlowBond}).

For the sequential-update Bernoulli dynamics the right edge is given by the inhomogeneous space discrete-time Bernoulli push-TASEP. Its homogeneous version was studied in \cite{BorodinFerrari,DiekerWarren}. The projection on the left-edge is a sequential-update zero range process in discrete time with inhomogeneous Bernoulli jumps. Finally, the recursive equations for the left and right edge can be readily read out from (\ref{BernoulliRecursion})
(observe that the equations are closed as in they do not depend on any of the other coordinates of the array),
\begin{align}
\mathsf{X}_1^{(n)}(t+1)&=\min\left\{\mathsf{X}_1^{(n-1)}(t+1),\mathsf{X}_1^{(n)}(t)+\mathsf{B}_{t,\alpha_t}^{1,(n)}\left(\mathsf{X}_1^{(n)}(t)\right)\right\}\label{BernoulliLeftRec},\\
\mathsf{X}_n^{(n)}(t+1)&=\max\left\{\mathsf{X}_{n-1}^{(n-1)}(t+1)+1,\mathsf{X}_n^{(n)}(t)+\mathsf{B}_{t,\alpha_t}^{n,(n)}\left(\mathsf{X}_n^{(n)}(t)\right)\right\}.\label{BernoulliRightRec}
\end{align}

For the Warren-Windridge geometric dynamics the right edge is a geometric push-TASEP in inhomogeneous space with sequential update. The projection on the left edge is a parallel-update (not sequential!) zero-range process with inhomogeneous geometric jumps which under the above coordinate shift it becomes a parallel-update TASEP with inhomogeneous geometric jumps. The recursive equations for the left and right edge are easily seen from (\ref{GeometricRecursion}) to be given by (again observe that they are closed):
\begin{align}
\mathsf{X}_1^{(n)}(t+1)&=\min\left\{\mathsf{X}_1^{(n-1)}(t),\mathsf{X}_1^{(n)}(t)+\mathsf{G}_{t,\beta_t}^{1,(n)}\left(\mathsf{X}_1^{(n)}(t)\right)\right\},\label{GeometricLeftRec}\\
\mathsf{X}_n^{(n)}(t+1)&=\max\left\{\mathsf{X}_n^{(n)}(t),\mathsf{X}_{n-1}^{(n-1)}(t+1)+1\right\}+\mathsf{G}_{t,\beta_t}^{n,(n)}\left(\max\left\{\mathsf{X}_n^{(n)}(t),\mathsf{X}_{n-1}^{(n-1)}(t+1)+1\right\}\right).\label{GeometricRightRec}
\end{align}
The left-edge particle systems, in discrete time, are reminiscent, but as far as we can tell not quite the same, with the doubly geometric inhomogeneous corner growth model studied in \cite{KnizelPetrovSaenz}.

The following theorem is our first main result.

\begin{thm}\label{ThmCorrelationKernelArray} 
Suppose that, starting from the fully-packed configuration in $\mathbb{IA}_\infty$, we perform $M_1$ steps of sequential-update Bernoulli dynamics with parameters $\alpha_1,\dots,\alpha_{M_1}$, $M_2$ steps of Warren-Windridge geometric dynamics with parameters $\beta_1,\dots,\beta_{M_2}$ and finally continuous-time pure-birth dynamics for time $t$. The parameters $\alpha_i,\beta_i$ satisfy $0\le \alpha_i \le (\sup_{x\in \mathbb{Z}_+}a_x)^{-1}$ and $0\le \beta_i <(\sup_{x\in \mathbb{Z}_+}a_x-\inf_{x\in \mathbb{Z}_+}a_x)^{-1}$. We denote the resulting random configuration in $\mathbb{IA}_\infty$ by $\left(\mathsf{X}_i^{(n)}\right)_{1 \le i \le n;n\ge 1}$. Then, for any $m\ge 1$, and pairwise distinct points $(n_1,x_1),\dots,(n_m,x_m)$ in $\mathbb{N}\times \mathbb{Z}_+$ we have:
\begin{align*}
\mathbb{P}\left(\exists \   j_1,\dots,j_m \textnormal{ such that } \mathsf{X}_{j_i}^{(n_i)}=x_i \textnormal{ for } i=1,\dots,m\right)=\det\left(\mathfrak{K}_f\left[(n_i,x_i);(n_j,x_j)\right]\right)_{i,j=1}^m,
\end{align*}
where the correlation kernel $\mathfrak{K}_f$ is given by
\begin{align} 
\mathfrak{K}_f\left[(n_1,x_1);(n_2,x_2)\right]&=-\mathbf{1}_{n_2>n_1}\frac{1}{a_{x_1}}\frac{1}{2\pi \textnormal{i}}\oint_{\mathsf{C}_{\mathbf{a},0}} \frac{p_{x_2}(w)}{p_{x_1+1}(w)w^{n_2-n_1}}dw\nonumber\\
 &-\frac{1}{a_{x_1}}\frac{1}{(2\pi \textnormal{i})^2}\oint_{\mathsf{C}_{\mathbf{a},0}}dw\oint_{\mathsf{C}_0}du\frac{p_{x_2}(u)f(w)}{p_{x_1+1}(w)f(u)}\frac{w^{n_1}}{u^{n_2}}\frac{1}{w-u},
\end{align}
with the function $f$ given
\begin{equation}\label{FunctionfDef}
f(w)=\prod_{i=1}^{M_1}(1-\alpha_iw)\prod_{i=1}^{M_2}(1+\beta_iw)^{-1} \exp(-tw)    
\end{equation}
and where the contours $\mathsf{C}_{\mathbf{a},0}$ and $\mathsf{C}_0$ are as explained in the caption of Figure \ref{ContourTheorem}.

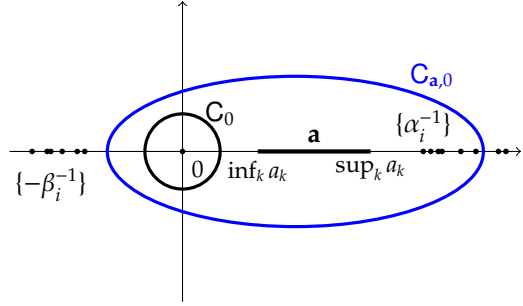
\begin{figure}
\captionsetup{singlelinecheck = false, justification=justified}
\centering
\begin{tikzpicture}

\draw[->] (-2.3,0) to (4.5,0);

\draw[->] (0,-2) to (0,2);

\node[right,blue] at (2.9,1) {$\mathsf{C}_{\mathbf{a},0}$};

\node[] at (0.5,0.5) {$\mathsf{C}_{0}$};

 \draw[fill] (0,0) circle [radius=0.03];

 \node[below right] at (0,0) {\small $0$};

  \node[below] at (1,0) {\small $\inf_{k} a_k$};

  \node[below] at (2.5,0) {\small $\sup_{k} a_k$};

  \draw[ultra thick] (1,0) to (2.5,0);

\draw[very thick,blue] (1.5,0) ellipse (2.5cm and 1cm);

\draw[very thick] (0,0) circle [radius=0.5];

\node[above] at (1.75,0) {$\mathbf{a}$};

 \draw[fill] (-1.3,0) circle [radius=0.03];

 \draw[fill] (-1.4,0) circle [radius=0.03];

 \draw[fill] (-1.6,0) circle [radius=0.03];

 \draw[fill] (-1.75,0) circle [radius=0.03];

 \draw[fill] (-1.8,0) circle [radius=0.03];

 \draw[fill] (-2,0) circle [radius=0.03];

 \node[below] at (-1.75,-0.1) {$\{-\beta_i^{-1}\}$};

 \draw[fill] (3.2,0) circle [radius=0.03];

 \draw[fill] (3.3,0) circle [radius=0.03];

 \draw[fill] (3.4,0) circle [radius=0.03];

 \draw[fill] (3.45,0) circle [radius=0.03];

 \draw[fill] (3.7,0) circle [radius=0.03];

 \draw[fill] (3.9,0) circle [radius=0.03];

  \draw[fill] (4.2,0) circle [radius=0.03];

 \draw[fill] (4.3,0) circle [radius=0.03];

 \node[above] at (3.2,0) {$\{\alpha_i^{-1}\}$};

\end{tikzpicture}

\caption{The contours from Theorem \ref{ThmCorrelationKernelArray}. The contour $\mathsf{C}_{\mathbf{a},0}$ must contain $0$ and the points $\{a_x\}$. The contour $\mathsf{C}_0$ must contain $0$ and is contained in $\mathsf{C}_{\mathbf{a},0}$. Moreover, $\mathsf{C}_{\mathbf{a},0}$ must not contain any of the points $\{-\beta_i^{-1}\}$ and the contour $\mathsf{C}_{0}$ must not contain any of the points $\{\alpha_i^{-1}\}$. Observe that, both $\mathsf{C}_{\mathbf{a},0}$ and $\mathsf{C}_{\mathbf{a},0}$ depend on $f$ through the above conditions on the parameters $\{-\beta_i^{-1}\}$ and $\{\alpha_i^{-1}\}$ but we suprress it from the notation.}\label{ContourTheorem}
\end{figure}

\end{thm}

\begin{rmk}
The order with which we perform the different types of dynamics is actually not important. For example we can perform in discrete-time either Bernoulli or geometric steps in any order and similarly the continuous-time dynamics. The law of the resulting configuration will be the same. This is not at all obvious a priori and is a consequence of our results. 
\end{rmk}

\begin{rmk}
The statement of the theorem is of course equivalent to the point process $\left\{\mathsf{X}_i^{(n)}\right\}$ being determinantal with correlation kernel $\mathfrak{K}_f$, see \cite{BorodinDeterminantal,JohanssonDeterminantal}.
\end{rmk}

\begin{rmk}
If we start the dynamics from certain more general initial conditions then the resulting point process will still have determinantal correlation functions. This follows from the results in the sequel. An explicit computation of the correlation kernel is more complicated however and we will not do it in this paper.
\end{rmk}

\begin{rmk}
The upper bound restriction on the $\beta_i$ parameters is  technical and we believe can be removed. Of course, in the homogeneous case it is vacuous. On the other hand, the restriction on the $\alpha_i$ parameters is necessary in order for the corresponding inhomogeneous Bernoulli jump to be well-defined.
\end{rmk}

Our next result is about the evolution of the projection on any single row of the array $\left(\mathsf{X}_i^{(N)}(t);t\ge 0\right)$, either in discrete or continuous time, and its correlations in time.

\begin{thm}\label{ThmCorrelationKernelNI}
Consider a process $\left(\mathsf{X}_k^{(n)}(t);t \ge 0\right)_{1\le k \le n; n\ge 1}$ in $\mathbb{IA}_\infty$, starting from the fully-packed configuration, either in discrete or continuous time $t$ evolving as follows:
\begin{itemize}
    \item In discrete time moves with sequential-update Bernoulli or Warren-Windridge geometric steps with the $k$-th step determined by a function $f_{k,k+1}(z)$ which is either of the form $(1-\alpha_{i_k}z)$ or $(1+\beta_{i_k}z)^{-1}$ satisfying $0\le \alpha_i \le \left(\sup_{x\in \mathbb{Z}_+} a_x\right)^{-1}$ and $0\le \beta_i <\left(\sup_{x\in \mathbb{Z}_+} a_x-\inf_{x\in \mathbb{Z}_+}a_x\right)^{-1}$.
    \item In continuous time moves with pure-birth push-block dynamics.
\end{itemize}
In discrete time we let $f_{s,t}(z)=\prod_{i=s}^{t-1} f_{i,i+1}(z)$ and in continuous time $f_{s,t}(z)=e^{-(t-s)z}$. Then, for any $N\ge 1$, the stochastic process $\left(\mathsf{X}_k^{(N)}(t);t \ge 0\right)_{1\le k \le N}$ is a Markov process in $\mathbb{W}_N$ with transition probabilities $\mathfrak{P}_{s,t}^{(N)}=\mathfrak{P}_{f_{s,t}}^{(N)}$, from time $s$ to time $t$, from $\mathbf{x}$ to $\mathbf{y}$, given by
\begin{equation}\label{SemigroupIntro}
\mathfrak{P}_{s,t}^{(N)}\left(\mathbf{x},\mathbf{y}\right)=\frac{\det\left(\partial_w^{i-1}p_{y_j}(w)\big|_{w=0}\right)_{i,j=1}^N}{\det\left(\partial_w^{i-1}p_{x_j}(w)\big|_{w=0}\right)_{i,j=1}^N}\det\left(\mathsf{T}_{f_{s,t}}\left(x_i,y_j\right)\right)_{i,j=1}^N.
\end{equation}
Moreover, for any $n\ge 1$, and any pairwise distinct points $(t_1,x_1),\dots, (t_n,x_n)$ in either $\mathbb{Z}_+ \times \mathbb{Z}_+$ or $\mathbb{R}_+ \times \mathbb{Z}_+$, depending on whether time is discrete or continuous, we have 
\begin{equation*}
 \mathbb{P}\left(\exists \ j_1,\dots,j_n \textnormal{ such that } \mathsf{X}_{j_i}^{(N)}(t_i)=x_i \textnormal{ for } 1 \le i \le n\right)= \det \left(\mathcal{K}_N\left[(t_i,x_i);(t_j,x_j)\right]\right)_{i,j=1}^n  
\end{equation*}
with the correlation kernel $\mathcal{K}_N$ given by 
\begin{align}\label{CorrKernelNonColliding}
\mathcal{K}_N\left[(s,x_1);(t,x_2)\right]=-\mathbf{1}_{t>s}\mathsf{T}_{f_{s,t}}\left(x_1,x_2\right)-\frac{1}{a_{x_2}}\frac{1}{(2 \pi \textnormal{i})^2}\oint_{\mathsf{C}_{\mathbf{a},0}} dw\oint_{\mathsf{C}_0} du \frac{p_{x_1}(u)f_{0,t}(w)}{p_{x_2+1}(w)f_{0,s}(u)} \frac{w^N}{u^N}\frac{1}{w-u},
\end{align}
where the contours $\mathsf{C}_{\mathbf{a},0}$ and $\mathsf{C}_0$ must again satisfy the conditions in the caption of Figure \ref{ContourTheorem}.
\end{thm}

\begin{rmk}  
The fact that $\mathfrak{P}_{s,t}^{(N)}$ is a transition probability is not immediately obvious from (\ref{SemigroupIntro}) and will be proven in the sequel.
\end{rmk}

\begin{rmk}\label{RemarkConditioning}
For the homogeneous model, for continuous-time dynamics (Poisson in this case) or discrete-time Bernoulli only steps, the following probabilistic interpretation of $\mathfrak{P}_{s,t}^{(N)}$, with $f_{s,t}(z)=e^{-(t-s)z}$ or $f_{s,t}(z)=(1-z)^{t-s}$, is known, see \cite{OConnellConditionedWalk,KonigOConnellRoch,RSKsurvey}. Namely, $\mathfrak{P}_{s,t}^{(N)}$ are the transition probabilities of $N$ independent Poisson processes or homogeneous Bernoulli walks respectively conditioned to never intersect. The case of geometric only steps can be a little more subtle depending on what we mean by first collision/intersection time, see \cite{OConnellConditionedWalk}, also Sections \ref{SpaceLevelInhomogeneousSection} and \ref{SectionConditionedWalks} for more details. By now there is a well-developed theory for studying collision times of random walks, equivalently exit times of random walks in cones/Weyl chambers,  with general increments, see for example \cite{DenisovWachtel}. However, in these models the distribution of the increments does not depend on the position of the walk as in our setup. As far as we can tell, very little is known for space-inhomogeneous walks. We believe, but do not prove here, that for a generic inhomogeneity sequence $\mathbf{a}$, $\mathfrak{P}_{s,t}^{(N)}$ has an analogous interpretation as independent inhomogeneous walks conditioned to never intersect. We will prove instead, see Sections \ref{SpaceLevelInhomogeneousSectionIntro} and \ref{SectionConditionedWalks}, that if one adds different drifts to each walk, with some further assumptions on the relative strengths of the drifts (which make the problem easier), then a corresponding statement is true, see Theorem \ref{ThmConditioning}.
\end{rmk}

\begin{rmk}
Proving that $\mathfrak{P}_{s,t}^{(N)}$ are the transition probabilities of independent inhomogeneous walks conditioned to never intersect would imply that the non-intersecting paths associated to $\mathfrak{P}_{s,t}^{(N)}$ have a novel Gibbs resampling property analogous to the Brownian Gibbs property, see \cite{CorwinKPZsurvey,BulkPropertiesAiry}, which has been instrumental in the study of KPZ universality class models \cite{CorwinKPZsurvey}, but with Brownian motion replaced by the one-dimensional walk with transition probability $\mathsf{T}_{f_{s,t}}$.
\end{rmk}

\begin{rmk}
In fact, the non-intersecting paths associated to the Markov process with transition probabilities $\mathfrak{P}_{s,t}^{(N)}$ give rise to a determinantal point process for general deterministic initial conditions beyond $(0,1,\dots,N-1)$. The explicit computation of the correlation kernel is an interesting problem. In the homogeneous case with only Bernoulli steps, this was done in \cite{GorinPetrov} to prove a universality result for local statistics for non-colliding Bernoulli walks. This was subsequently employed in \cite{aggarwal2019universality} to prove universality for local statistics of random uniformly distributed lozenge tilings.
\end{rmk}

\begin{rmk}
Observe that, if we identify the functions $f_{0,T}(w)$ and $f(w)$, from Theorem \ref{ThmCorrelationKernelNI} and (\ref{FunctionfDef}) respectively, we have
\begin{equation*}
\mathcal{K}_N\left[(T,x_1);(T,x_2)\right]=\mathfrak{K}_f\left[(N,x_2);(N,x_1)\right].
\end{equation*}
In particular, these correlation kernels give rise to the same point process on $\mathbb{Z}_+$ as they should be (which is clear from the probabilistic construction of the corresponding point processes).
\end{rmk}

We strongly believe that the theorems above can be used to investigate asymptotic questions regarding these dynamics. As an illustration of this, we prove next certain short-time asymptotics for the process $\left(\mathsf{X}_i^{(N)}(t);t\ge 0\right)_{1\le i \le N}$ in continuous-time, equivalently, by virtue of Theorem \ref{ThmCorrelationKernelNI}, for the process of non-intersecting paths with transition semigroup $\big(\mathfrak{P}^{(N)}_{\exp{(-t)}}\big)_{t\ge 0}$. These asymptotics work for any inhomogeneity sequence $\mathbf{a}$ as long as the average $\lim_{N\to \infty} N^{-1}\sum_{k=0}^{N-1} a_k$ exists. The limit process is given by the discrete Bessel determinantal point process \cite{BorodinOkounkovOlshanski,JohanssonDiscrete} whose parameters only depend on $\mathbf{a}$ through the average above. We do not have an intuitive reason explaining why this averaging phenomenon happens for times of order $N^{-1}$ but it would be interesting if one exists. For example in other scaling regimes the limit should depend on $\mathbf{a}$ in a more complicated way and the asymptotic analysis will also be more involved.

\begin{thm}\label{BesselTheorem} Consider $\left(\mathsf{X}_i^{(n)}(t);t\ge 0\right)_{1\le i \le n; n\ge 1}$ following the continuous-time pure-birth push-block dynamics in $\mathbb{IA}_\infty$ starting from the fully-packed configuration. Assume that the following average $\bar{a}=\lim_{N\to \infty} N^{-1}\sum_{x=0}^{N-1}a_x$ exists and let $\zeta>0$. Then, we have the following convergence in distribution for all $m\ge 1$, 
\begin{equation*}
\left(\mathsf{X}_N^{(N)}\left(\frac{\zeta}{N}\right)-N,\mathsf{X}_{N-1}^{(N)}\left(\frac{\zeta}{N}\right)-N,\dots,\mathsf{X}_{N-m+1}^{(N)}\left(\frac{\zeta}{N}\right)-N\right)\overset{\textnormal{d}}{\longrightarrow}\left(\mathfrak{B}_1^{\zeta \bar{a}},\mathfrak{B}_2^{\zeta \bar{a}},\dots,\mathfrak{B}_{m}^{\zeta \bar{a}}\right), \textnormal{ as } N \to \infty,
\end{equation*}
where, for $\sigma > 0$, $\mathfrak{B}_1^\sigma>\mathfrak{B}_2^\sigma>\mathfrak{B}_3^\sigma>\cdots $ are the ordered points of the discrete Bessel determinantal point process $\mathfrak{B}^\sigma$ on $\mathbb{Z}$. This point process satisfies and is determined by, for all $n\ge 1$, and $x_1,\dots,x_n \in \mathbb{Z}$,
\begin{equation*}
\mathbb{P}\left(\mathfrak{B}^\sigma \textnormal{ contains } x_1,\dots,x_n\right)=\det \left(\mathbf{J}_\sigma\left(x_i,x_j\right)\right)_{i,j=1}^n
\end{equation*}
where the correlation kernel $\mathbf{J}_\sigma(x,y)$ is given by, with $x,y\in \mathbb{Z}$,
\begin{equation*}
\mathbf{J}_\sigma\left(x,y\right)=\frac{1}{\left(2\pi \textnormal{i}\right)^2}\oint_{|z|=1^-}\oint_{|v|=1^+} \frac{z^{x}}{v^{y+1}}e^{z^{-1}-v^{-1}+\sigma v-\sigma z}\frac{1}{v-z}dz dv.
\end{equation*}
\end{thm}

\begin{rmk}
Here, the notations $\{z\in \mathbb{C}:|z|=1^-\}$ and $\{v\in \mathbb{C}:|v|=1^+\}$ for the contours mean any fixed circles about the origin with radii slightly smaller, respectively bigger, than $1$. In fact, we simply need the $z$-contour to be contained in the $v$-contour and to contain $0$.
\end{rmk}

\begin{rmk}
In the homogeneous case $a_x \equiv 1$ the above result boils down to a result from \cite{BorodinKuanPlancherel}. Although there is no discussion of dynamics there, the point process studied therein is equivalent to our case with $a_x \equiv 1$.
\end{rmk}

\subsection{Space-level inhomogeneous dynamics, parameter symmetry and conditioned walks}\label{SpaceLevelInhomogeneousSectionIntro}

In this section we consider dynamics where the jump rates/transition probabilities of particles are both space-dependent and also depend on their vertical coordinate/level (particle-dependent) through yet another parameter sequence $\boldsymbol{\gamma}=(\gamma_i)_{i=1}^\infty$. However, the models will depend on time in a homogeneous way so in particular they do not generalise the ones from the previous section. It would be interesting to find a model which retains the integrability and involves inhomogeneities of space, level and time (this is also related to the discussion in the penultimate paragraph of Section \ref{ShufflingIntro} on the domino shuffle connection) \footnote{An integrable vertex model (which is not determinantal) with three parameter sequences has recently been studied in \cite{Korotkikh}. Whether this model has a free fermion version, which presumably would give a TASEP-like model with three sequences of parameters, and whether this would be related to any of the inhomogeneous models of the present paper is an interesting question.}. Introducing particle-dependent parameters allows us, at least for a certain type of question, to go further in our analysis of the model, see Theorem \ref{ThmConditioning} and the discussion in Remark \ref{RemarkConditioning}.

We use the abbreviations pb, B, g for pure-birth, Bernoulli, geometric and the notation $\bullet \in \{\textnormal{pb, B, g}\}$. Suppose we are given a sequence $\boldsymbol{\gamma}=(\gamma_n)_{n=1}^\infty$ which satisfies\footnote{These parameter ranges ensure that the jump rates/probabilities encountered are well-defined. They are not optimal and can be extended. However, we stick to them for simplicity.}, depending on the type of dynamics to be defined next,
\begin{equation}\label{ParameterRanges}
\gamma_n\ge 0, \textnormal{ if } \bullet=\textnormal{pb}, \ \ 0<\gamma_n\le 1,  \textnormal{ if } \bullet=\textnormal{B}, \ \ \gamma_n \ge 1,  \textnormal{ if } \bullet=\textnormal{g}, \ \forall n \ge 1.
\end{equation}

\begin{defn}\label{DefSpaceLevelIhomogeneous}
 We denote by $\left(\mathsf{X}^{\bullet}(t);t\ge 0\right)=\left(\mathsf{X}_k^{(n),\bullet}(t);t \ge 0\right)_{1\le k \le n; n \ge 1}$ the stochastic process in $\mathbb{IA}_\infty$ which evolves according to the following dynamics, in either continuous or discrete time, depending on the choice of $\bullet \in \{\textnormal{pb, B, g}\} $:
\begin{enumerate}
    \item For $\bullet=\textnormal{pb}$, particles at level $n$ at position $x$ jump (in continuous time) to $x+1$ with rate $a_x+\gamma_n$, with interactions between particles being exactly as in Definition \ref{DefCtsTimeDynamics}.
    \item For $\bullet=\textnormal{B}$, particles at level $n$, at position $x$, jump (in discrete time) to $x+1$ with probability $\gamma_n a_x+1-\gamma_n$ and stay at $x$ with complementary probability $\gamma_n-\gamma_n a_x$, with interactions between particles being exactly as in Definition \ref{DefBernoulliDynamics}.
    \item For $\bullet=\textnormal{g}$, particles at level $n$, at position $x$, jump (in discrete time) to $y\ge x$,  with probability $(\gamma_n+\gamma_na_y)^{-1} \prod_{k=x}^{y-1}\left(\gamma_n(1+a_k)-1\right)(\gamma_n+\gamma_n a_k)^{-1}$, with interactions between particles being exactly as in Definition \ref{DefGeometricDynamics}.
\end{enumerate}
Here, we assume that the parameter sequence $\boldsymbol{\gamma}$ satisfies the corresponding conditions (\ref{ParameterRanges}) depending on the choice of $\bullet\in \{\textnormal{pb, B, g}\} $.
\end{defn}

Obviously the process $\left(\mathsf{X}^{\bullet}(t);t\ge 0\right)$ depends on $\boldsymbol{\gamma}$ but we suppress it from the notation as we do with $\mathbf{a}$. Write $\mathsf{T}_t^{\bullet}$ for $\mathsf{T}_{f_t^\bullet}$ with $f_t^{\textnormal{pb}}(z)=\exp(-tz)$ or $f_t^{\textnormal{B}}=(1-z)^t$ or $f_t^{\textnormal{g}}(z)=(1+z)^{-t}$. The following is the main result of this section.

\begin{thm}\label{ThmSpaceLevelInhomogeneous} Consider the stochastic process $\left(\mathsf{X}^{\bullet}(t);t\ge 0\right)$ in $\mathbb{IA}_\infty$
 from Definition \ref{DefSpaceLevelIhomogeneous} initialised from the fully-packed configuration. We assume that the sequence $\boldsymbol{\gamma}$ satisfies (\ref{ParameterRanges}) and moreover if $\bullet=\textnormal{B}$ then $\sup_{x\in \mathbb{Z}_+}a_x\le 1$ or if $\bullet=\textnormal{g}$ then $\sup_{x\in \mathbb{Z}_+}a_x-\inf_{x\in \mathbb{Z}_+}a_x<1$. Then, for any $N\ge 1$, the stochastic process $\left(\mathsf{X}_k^{(N),\bullet}(t);t \ge 0\right)_{1\le k \le N}$ is a Markov process in $\mathbb{W}_N$ with time-homogeneous transition semigroup $\left(\mathcal{P}_t^{\boldsymbol{\gamma},\bullet,N}\right)_{t\ge 0}$ determined by its explicit kernel,
\begin{equation}\label{ExplicitKernelDriftsIntro}
\mathcal{P}_t^{\boldsymbol{\gamma},\bullet,N}(\mathbf{x},\mathbf{y})=\prod_{j=1}^{N}\frac{1}{c_{t,\gamma_j}^\bullet} \times \frac{\det\left(h_{\gamma_i}^\bullet(y_j)\right)_{i,j=1}^N}{\det\left(h_{\gamma_i}^\bullet(x_j)\right)_{i,j=1}^N}\det \left(\mathsf{T}_{t}^\bullet (x_i,y_j)\right)_{i,j=1}^N,
\end{equation}
where the functions $h_\gamma^\bullet$ are given by
\begin{equation*}
 h_{\gamma}^{\textnormal{pb}}(x)=p_x(-\gamma), \ \ 
  h_{\gamma}^{\textnormal{B}}(x)=p_x\left(1-\gamma^{-1}\right),   \ \ 
 h_{\gamma}^{\textnormal{g}}(x)=p_x(\gamma^{-1}-1), 
\end{equation*}
and the constants $c_{t,\gamma}^\bullet$ by,
\begin{equation*}
  c_{t,\gamma}^{\textnormal{pb}}=e^{t\gamma}, \ \ c_{t,\gamma}^{\textnormal{B}}=\gamma^{-t}, \ \ c_{t,\gamma}^{\textnormal{g}}=\gamma^t.
\end{equation*}
In particular, from (\ref{ExplicitKernelDriftsIntro}) we obtain that the distribution of the projection on the $N$-th level of the array $\left(\mathsf{X}^{(N),\bullet}(t);t\ge 0\right)$, as a process, is symmetric in the parameters $\gamma_1,\gamma_2,\dots,\gamma_N$.
\end{thm}

\begin{rmk}
The fact that the distribution of a single row $\left(\mathsf{X}^{(N),\bullet}(t);t\ge 0\right)$ is symmetric in the level parameters $\gamma_1,\gamma_2,\dots,\gamma_N$ is non-trivial and we do not have an intuitive probabilistic reason explaining why this is true\footnote{Of course, most of these models are directly or indirectly related to symmetric functions and thus one may expect some sort of symmetry in the parameters to make its appearance but its exact form is hard to predict.}. An analogous result holds for interacting Brownian motions with different drifts in interlacing arrays, see for example \cite{InterlacingDiffusions} and the references therein.
\end{rmk}

%\begin{rmk}
%We note that $\mathcal{P}_t^{\boldsymbol{\gamma},\bullet, N}$ is only dependent, as it should by construction of $\left(\mathsf{X}^{\bullet}(t);t\ge 0\right)$, on $\boldsymbol{\gamma}$ through its first $N$ terms $(\gamma_1,\dots,\gamma_N)$.
%\end{rmk}

\begin{rmk}
From our proof of Theorem \ref{ThmSpaceLevelInhomogeneous} it will also follow, although we will not make this explicit, that the underlying point process is determinantal \cite{BorodinDeterminantal,JohanssonDeterminantal}. Computing the correlation kernel explicitly is an interesting question that we will not pursue further in this paper.
\end{rmk}

We now give a more probabilistic interpretation of the semigroup $\left(\mathcal{P}_t^{\boldsymbol{\gamma},\bullet, N}\right)_{t\ge 0}$. We require a definition which will also be useful later on.

\begin{defn}\label{Definition1DTransKernelwithdrift}
 We define $\mathsf{T}_{t,\gamma}^{\bullet}$ for $\bullet \in \{\textnormal{pb, B, g}\}$, with $\gamma$ satisfying (\ref{ParameterRanges}), as follows: $\mathsf{T}_{t,\gamma}^{\textnormal{pb}}$ is the transition probability in continuous time $t$ of a pure-birth chain with jump rate $a_k+\gamma$ at location $k$, while in the Bernoulli and geometric cases it is the transition probability in discrete time $t$ of the walk with single-step transition probabilities given by,
\begin{align*}
    \mathsf{T}_{1,\gamma}^{\textnormal{B}}(x,y)&= (\gamma a_x+1-\gamma) \mathbf{1}_{y=x+1}+(\gamma-\gamma a_x)\mathbf{1}_{y=x},\\
    \mathsf{T}_{1,\gamma}^{\textnormal{g}}(x,y)&=\frac{1}{\gamma(1+a_y)} \prod_{k=x}^{y-1}\frac{\gamma(1+a_k)-1}{\gamma(1+a_k)} \mathbf{1}_{y\ge x}.
\end{align*}   
\end{defn}

We then have:

\begin{thm}\label{ThmConditioning}  Let $N\ge 1$ be fixed and write $\boldsymbol{\gamma}=(\gamma_1,\dots,\gamma_N)$. Consider the stochastic process $\left(\mathsf{X}^{\bullet}_{\boldsymbol{\gamma}}(t);t\ge 0\right)=\left(\mathsf{x}_{\gamma_1}^\bullet(t),\dots,\mathsf{x}_{\gamma_N}^\bullet(t); t\ge 0\right)$, with $\mathsf{X}_{\boldsymbol{\gamma}}(0)=\mathbf{x}\in \mathbb{W}_N$, with coordinates $\left(\mathsf{x}_{\gamma_i}^\bullet(t);t\ge 0\right)$ being independent and having transition probabilities $\left(\mathsf{T}_{t,\gamma_i}^\bullet\right)_{t\ge 0}$. Assume that the parameters $\boldsymbol{\gamma}=(\gamma_1,\dots,\gamma_N)$, in addition to (\ref{ParameterRanges}),
satisfy the conditions, for $i=1,\dots,N-1$,
\begin{align}
\gamma_{i+1}-\gamma_i&>\sup_{k\in \mathbb{Z}_+} a_k-\inf_{k\in \mathbb{Z}_+}a_k, \textnormal{ if } \bullet=\textnormal{pb}, \label{ParameterConditions1}\\ 
\frac{\gamma_{i+1}}{\gamma_i}&<\frac{1-\sup_{k\in \mathbb{Z}_+} a_k}{1-\inf_{k\in \mathbb{Z}_+}a_k}, \textnormal{ if } \bullet=\textnormal{B},\label{ParameterConditions2}\\ 
\frac{\gamma_{i+1}}{\gamma_i}&>\frac{1+\sup_{k\in \mathbb{Z}_+}a_k}{1+\inf_{k\in \mathbb{Z}_+} a_k}, \textnormal{ if } \bullet=\textnormal{g}\label{ParameterConditions3}.
\end{align}
Consider the first collision or intersection time $\tau_{\textnormal{col}}^\bullet$ defined by
\begin{equation*}
\tau_{\textnormal{col}}^{\bullet}=\inf\{t>0:\mathsf{X}^{\bullet}_{\boldsymbol{\gamma}}(t-)\nprec \mathsf{X}^{\bullet}_{\boldsymbol{\gamma}}(t)\},
\end{equation*}
where if $\bullet\in \{\textnormal{B, g}\}$ then $\mathsf{X}^{\bullet}_{\boldsymbol{\gamma}}(t-)=\mathsf{X}^{\bullet}_{\boldsymbol{\gamma}}(t-1)$ while if $\bullet=\textnormal{pb}$ then $\mathsf{X}^{\bullet}_{\boldsymbol{\gamma}}(t-)=\lim_{s\uparrow t}\mathsf{X}^{\bullet}_{\boldsymbol{\gamma}}(s)$.
Let $\left(\mathsf{X}_{\boldsymbol{\gamma}}^{\bullet, \textnormal{n.c.}}(t);t\ge 0\right)$ be the process $\left(\mathsf{X}^{\bullet}_{\boldsymbol{\gamma}}(t);t\ge 0\right)$ conditioned on $\tau_{\textnormal{col}}^\bullet=\infty$. Then, the  transition probabilities of $\left(\mathsf{X}_{\boldsymbol{\gamma}}^{\bullet, \textnormal{n.c.}}(t);t\ge 0\right)$ are given by $\left(\mathcal{P}_t^{\boldsymbol{\gamma},\bullet, N}\right)_{t\ge 0}$ defined in (\ref{ExplicitKernelDriftsIntro}).
\end{thm}

\begin{rmk}
The use of the terminology collision or intersection time for $\tau_{\textnormal{col}}^\bullet$ may not be immediately clear. It is easy to see that in case $\bullet=\textnormal{pb}$ or $\bullet=\textnormal{B}$ then $\tau_{\textnormal{col}}^\bullet$ is actually equal to the more standard definition:
\begin{equation*}
\tilde{\tau}_{\textnormal{col}}^\bullet=\inf \{t>0:\mathsf{X}_{\boldsymbol{\gamma}}^\bullet(t)\notin \mathbb{W}_N\}.
\end{equation*}
For geometric jumps $\tau_{\textnormal{col}}^\textnormal{g}$ and  $\tilde{\tau}_{\textnormal{col}}^\textnormal{g}$ are however different, see \cite{OConnellConditionedWalk}, and it is easy to see that $\tau_{\textnormal{col}}^\textnormal{g}\le\tilde{\tau}_{\textnormal{col}}^\textnormal{g}$. In fact, $\tau_{\textnormal{col}}^\textnormal{g}$ is in some sense more natural. If one views the geometric random walk paths as paths in the corresponding Lindstrom-Gessel-Viennot (LGV) graph \cite{LGV}, see for example Figures \ref{PathsFixedStartingFinal} and \ref{LGVgraphsFigure}, then $\tau_{\textnormal{col}}^\textnormal{g}$ is exactly the first time these paths intersect.
\end{rmk}

\begin{rmk}
The prototypical result of the kind presented in Theorem \ref{ThmConditioning} is the explicit computation of the transition kernel of $N$ independent Brownian motions with ordered drifts in the Weyl chamber conditioned to never intersect, see \cite{BianeBougerolOConnell}. An analogous result holds for another special class of one-dimensional diffusions, the (generalised) squared Bessel process, see \cite{BESQdrifts}. In the discrete-space setting (and discrete time), for the homogeneous model $a_x\equiv 1$, the result above was first proven in \cite{OConnellConditionedWalk}. As far as we can tell, this was the only known case previously. It is worth noting that the problem in discrete compared to continuous space is more amenable to analysis as we have proven the result for random walks with essentially arbitrary inhomogeneity $\mathbf{a}$.
\end{rmk}

\begin{rmk}
The conditions (\ref{ParameterConditions1}), (\ref{ParameterConditions2}), (\ref{ParameterConditions3}) are an artefact of our proof which goes via a coupling to the homogeneous case. We believe that for generic sequence $\mathbf{a}$ these conditions are not necessary and moreover that we could then take the limit of the $\gamma_i$ parameters being equal. In particular, this would answer the question raised in Remark \ref{RemarkConditioning}.
\end{rmk}

\subsection{Transition kernels and determinantal processes for the edge particle systems for general initial condition}\label{EdgeParticleSystemsIntro}

In this section we consider the autonomous interacting particle systems, either in discrete or continuous time, $\left(\mathsf{X}_1^{(1)}(t),\mathsf{X}_2^{(2)}(t),\dots,\mathsf{X}_N^{(N)}(t);t\ge 0\right)$ and $\left(\mathsf{X}_1^{(1)}(t),\mathsf{X}_1^{(2)}(t),\dots,\mathsf{X}_1^{(N)}(t);t\ge 0\right)$ at the right and left edge of the interlacing array respectively, in the setting of Theorem \ref{ThmCorrelationKernelNI}. We follow the notations of that theorem. Since these systems are one-dimensional we actually only need a single index for each particle instead of both a subscript and superscript. However, we stick with the notation above to stress the connection with dynamics in arrays.

It will be more convenient notation-wise to consider $\left(\mathsf{X}_1^{(N)}(t),\mathsf{X}_1^{(N-1)}(t),\dots,\mathsf{X}_1^{(1)}(t);t\ge 0\right)$ instead of $\left(\mathsf{X}_1^{(1)}(t),\mathsf{X}_1^{(2)}(t),\dots,\mathsf{X}_1^{(N)}(t);t\ge 0\right)$. Towards this end, define the variation of the Weyl chamber $\overline{\mathbb{W}}_N$ where coordinates can be equal by
\begin{equation}\label{WeylChamberEqualcoordinates}
\overline{\mathbb{W}}_N=\left\{\mathbf{x}=(x_1,x_2,\dots,x_N)\in \mathbb{Z}_+^N:x_1 \le x_2\le \cdots \le x_N\right\},
\end{equation}
and observe that this is the state space of $\left(\mathsf{X}_1^{(N)}(t),\mathsf{X}_1^{(N-1)}(t),\dots,\mathsf{X}_1^{(1)}(t);t\ge 0\right)$.

Our interest here is to study these one-dimensional systems from arbitrary (deterministic) initial condition. Note that, the fully-packed configuration corresponds to the most special initial conditions $(0,1,\dots,N-1)$ and $(0,0,\dots,0)$ for $\left(\mathsf{X}_1^{(1)}(t),\mathsf{X}_2^{(2)}(t),\dots,\mathsf{X}_N^{(N)}(t);t\ge 0\right)$ and $\left(\mathsf{X}_1^{(N)}(t),\mathsf{X}_1^{(N-1)}(t),\dots,\mathsf{X}_1^{(1)}(t);t\ge 0\right)$ respectively. There has been intense activity in the last few years around the study of such one-dimensional systems starting from general initial condition \cite{KPZfixedpoint, NicaQuastelRemenik,MatetskiRemenik1, MatetskiRemenik2, matetski2022polynuclear, NikosTASEP,assiotis2023exact,arai2020kpz} beginning with the breakthrough work \cite{KPZfixedpoint} on TASEP where the KPZ fixed point was first constructed, the central object in the KPZ universality class \cite{CorwinKPZsurvey}. Despite this significant body of work on the topic, as far as we know, systems where the spatial motion of particles is inhomogeneous have only been studied, for general initial condition, in \cite{assiotis2023exact} where certain special interacting diffusions related to the classical ensembles of random matrices and classical orthogonal polynomials were considered. Theorem \ref{EdgeParticleSystemsThmIntro} below is a first step in developing the analogous framework of \cite{KPZfixedpoint,NicaQuastelRemenik,MatetskiRemenik1,MatetskiRemenik2,matetski2022polynuclear} for a  general class of one-dimensional space-inhomogeneous interacting particle systems.

The following theorem, informally stated, is the main result of this section. The precise statement can be found in Theorems \ref{EdgeExplicitTrans} and \ref{EdgeDeterminantal} in Section \ref{SectionEdgeTransitionKernels}. In order to prove these results we make essential use of the framework built to establish Theorems \ref{ThmCorrelationKernelArray} and \ref{ThmCorrelationKernelNI} and some additional non-trivial arguments that we briefly comment on in Section \ref{HistoryIdeasSection}.

\begin{thm}\label{EdgeParticleSystemsThmIntro}
In the setting of Theorem \ref{ThmCorrelationKernelNI}, with the notations and assumptions therein, the transition kernels $\mathfrak{E}_{f_{s,t},\textnormal{r}}^{(N)}$ and $\mathfrak{E}_{f_{s,t},\textnormal{l}}^{(N)}$ of the autonomous systems $\left(\mathsf{X}_1^{(1)}(t),\mathsf{X}_2^{(2)}(t),\dots,\mathsf{X}_N^{(N)}(t);t\ge 0\right)$ and $\left(\mathsf{X}_1^{(N)}(t),\mathsf{X}_1^{(N-1)}(t),\dots,\mathsf{X}_1^{(1)}(t);t\ge 0\right)$ on the right and left edge of the array respectively,
\begin{align*}
\mathfrak{E}_{f_{s,t},\textnormal{r}}^{(N)}\left(\mathbf{x},\mathbf{y}\right)&=\mathbb{P}\left(\left(\mathsf{X}_1^{(1)}(t),\dots,\mathsf{X}_N^{(N)}(t)\right)=\mathbf{y}\bigg|\left(\mathsf{X}_1^{(1)}(s),\dots,\mathsf{X}_N^{(N)}(s)\right)=\mathbf{x}\right),  \ \mathbf{x},\mathbf{y}\in \mathbb{W}_N,\\
\mathfrak{E}_{f_{s,t},\textnormal{l}}^{(N)}\left(\mathbf{x},\mathbf{y}\right)&=\mathbb{P}\left(\left(\mathsf{X}_1^{(N)}(t),\dots,\mathsf{X}_1^{(1)}(t)\right)=\mathbf{y}\bigg|\left(\mathsf{X}_1^{(N)}(s),\dots,\mathsf{X}_1^{(1)}(s)\right)=\mathbf{x}\right), \ \mathbf{x},\mathbf{y}\in \overline{\mathbb{W}}_N,
\end{align*}
are explicit. Moreover, for any $\mathbf{x}\in \mathbb{W}_N$ and $\mathbf{y}\in \overline{\mathbb{W}}_N$, and fixed $s\le t$, the probability measures $\mathfrak{E}_{f_{s,t},\textnormal{r}}^{(N)}(\mathbf{x},\cdot)$ and $\mathfrak{E}_{f_{s,t},\textnormal{l}}^{(N)}(\mathbf{y},\cdot)$  can be written as marginals of certain signed measures on $\mathbb{IA}_N$ with determinantal correlation functions.
\end{thm}

\begin{rmk}
A few days after this paper was first posted on $\mathsf{arXiv}$ the interesting and completely independent work \cite{iwao2023free} appeared. Among other things, the authors in \cite{iwao2023free} also study the space-inhomogeneous edge particle systems and obtain an explicit formula for their transition probabilities, namely an analogue of our Theorem \ref{EdgeExplicitTrans}. Their formula is different from ours and is given in terms of Grothendieck polynomials, which are symmetric functions used to study the K-theory of the Grassmannian  \cite{Lascoux1,Lascoux2}. Their approach to obtain such a formula is also different and goes via the algebraic combinatorics of the Grothendieck polynomials. It is interesting that the only symmetric functions which appear explicitly in our paper, and for another task altogether, are the factorial Schur polynomials and not the Grothendieck polynomials. Nevertheless, we believe there should be deep connections between the two works and it would be interesting to understand them.
\end{rmk}

We do not attempt here to compute explicitly the correlation kernel behind the determinantal correlations functions. This would involve solving a certain biorthogonalisation problem. In the level/particle inhomogeneous setting this is done systematically in the series of two papers \cite{MatetskiRemenik1,MatetskiRemenik2}, see also \cite{NikosTASEP}. Developing the analogous systematic theory for the space-inhomogeneous setting we have been considering would be a substantial task and we leave it for future work.

Moreover, the explicit transition kernels could possibly be used to analyse asymptotic multi-time distributions of these systems, see \cite{JohanssonMarkovChain,JohanssonTwoTime1,JohanssonTwotime2, JohanssonTwotime3, johansson2021multitime,JohanssonRahman} for example where this is done in the space-homogeneous\footnote{The paper \cite{JohanssonRahman} studies polynuclear growth in an inhomogeneous geometric environment. We note that, as mentioned earlier, in terms of the corresponding particle system the inhomogeneity transforms into inhomogeneity of the particles and time and not of space as it is here.} setting. Again, developing the analogous arguments in inhomogeneous space would require substantial work and we leave this for future investigations.

\subsection{Extremal measures for the inhomogeneous Gelfand-Tsetlin graph}\label{SectionGraphIntro}

In this introductory section and its continuation in Section \ref{SectionGraph} we will connect the previous constructions of interacting particle systems to a seemingly disparate object, an inhomogeneous generalisation of the famous Gelfand-Tsetlin graph \cite{OlshanskiHarmonic,BorodinOlshanskiHarmonic,BorodinOlshanskiBoundary,YoungBouquet} for which we prove a non-trivial structural result in Theorem \ref{TheoremGraph}. We need some notation and definitions.

\begin{defn}
 We define the weighted graded graph that we call the (non-negative) inhomogeneous Gelfand-Tseltin graph, with parameters $\mathbf{a}$, and denote by $\mathbf{GT}_+(\mathbf{a})$ as follows. It has vertex  set given by $\sqcup_{N=1}^\infty \mathbb{W}_N$. Two vertices $\mathbf{x}\in \mathbb{W}_N$ and $\mathbf{y}\in \mathbb{W}_{N+1}$ are connected by an edge if and only if they interlace $\mathbf{x}\prec \mathbf{y}$. Associated to each pair $(\mathbf{y},\mathbf{x})\in \mathbb{W}_{N+1}\times \mathbb{W}_N$ we have a weight $\textnormal{we}_{N+1,N}(\mathbf{y},\mathbf{x})$, which has non-zero value only if $\mathbf{y}$ and $\mathbf{x}$ are connected by an edge and is given by
\begin{equation*}
\textnormal{we}_{N+1,N}(\mathbf{y},\mathbf{x})=\prod_{i=1}^N \frac{1}{a_{x_i}}\mathbf{1}_{\mathbf{x}\prec \mathbf{y}}.
\end{equation*}   
\end{defn}

In the homogeneous case $a_x\equiv 1$, $\mathbf{GT}_+(\mathbf{a})$ is equivalent (after a shift of the co-ordinates on each level to go from $\mathbb{W}_N$ to so-called non-negative signatures \cite{BorodinOlshanskiBoundary}) to the non-negative Gelfand-Tsetlin graph, see \cite{BorodinOlshanskiBoundary}. We denote the homogeneous graph by $\mathbf{GT}_+$. This graph is of central importance in algebraic combinatorics and representation theory. It describes (more precisely the full graph where the co-ordinates of the vertices can also take negative values) the branching of irreducible representations of the inductive chain of unitary groups $\mathbb{U}(1) \subset \mathbb{U}(2) \subset \cdots \subset \mathbb{U}(N) \subset \cdots$, see \cite{OlshanskiHarmonic,BorodinOlshanskiHarmonic,BorodinOlshanskiBoundary,YoungBouquet}. We introduced the inhomogeneous graph $\mathbf{GT}_+(\mathbf{a})$ in \cite{DeterminantalStructures}. 

Define the dimension $\textnormal{dim}_N(\mathbf{y})$ of a vertex $\mathbf{y}\in \mathbb{W}_N$ by 
\begin{equation}
\textnormal{dim}_N(\mathbf{y})=\sum_{\mathbf{x}^{(1)}\prec \mathbf{x}^{(2)}\prec \cdots \prec \mathbf{x}^{(N-1)}\prec \mathbf{x}^{(N)}=\mathbf{y}}\prod_{i=1}^{N-1}\textnormal{we}_{i+1,i}(\mathbf{x}^{(i+1)},\mathbf{x}^{(i)}).
\end{equation}
Define the Markov kernels $\Lambda_{N+1,N}^{\mathbf{GT}_+(\mathbf{a})}$ from the vertex set $\mathbb{W}_{N+1}$ to the vertex set $\mathbb{W}_{N}$ by, with $\mathbf{x}\in \mathbb{W}_N$, $\mathbf{y}\in \mathbb{W}_{N+1}$,
\begin{equation*}
\Lambda_{N+1,N}^{\mathbf{GT}_+(\mathbf{a})}(\mathbf{y},\mathbf{x})=\frac{\textnormal{dim}_N(\mathbf{x})\textnormal{we}_{N+1,N}(\mathbf{y},\mathbf{x})}{\textnormal{dim}_{N+1}(\mathbf{y})}.
\end{equation*}
\begin{defn}
 We say that a sequence of probability measures $\left(\mu_N\right)_{N=1}^\infty$, with $\mu_N$ a probability measure on $\mathbb{W}_N$, is coherent (or consistent) if for all $N\ge 1$,
\begin{equation*}
\mu_{N+1}\Lambda_{N+1,N}^{\mathbf{GT}_+(\mathbf{a})}(\mathbf{x})=\mu_N(\mathbf{x}), \ \ \forall \mathbf{x}\in \mathbb{W}_N.
\end{equation*}   
\end{defn}
Sequences of coherent probability measures form a convex set. We say that a sequence is extremal if it is an extreme point in this set.

\begin{defn}
For an infinite sequence of parameters:
\begin{align}
\boldsymbol{\omega}=(\left(\alpha_i\right)_{i=1}^\infty,\left(\beta_i\right)_{i=1}^\infty,t) \in \mathbb{R}_+^{2\infty+1}, \textnormal{ such that } \alpha_1\ge \alpha_2 \ge \cdots ,\  \beta_1 \ge \beta_2  \ge \cdots,\nonumber\\ 
\sum_{i=1}^\infty \left(\alpha_i+\beta_i\right)<\infty \textnormal{ and } \alpha_1\le\left(\sup_{x\in\mathbb{Z}_+}a_x\right)^{-1} \textnormal{ and } \beta_1<\left(\sup_{k\in \mathbb{Z}_+}a_k-\inf_{k\in \mathbb{Z}_+}a_k\right)^{-1},\label{Defomega}
\end{align}
consider the function $f_{\boldsymbol{\omega}}(z)$ given by 
\begin{equation*}
f_{\boldsymbol{\omega}}(z)=e^{-tz}\prod_{i=1}^\infty\frac{1-\alpha_i z}{1+\beta_i z}.
\end{equation*}
This is holomorphic in $\{z\in \mathbb{C}:\Re(z)>-\beta_1^{-1}\}$. Moreover, for each $N\ge 1$, we define the following measure on $\mathbb{W}_N$ corresponding to $f_{\boldsymbol{\omega}}$:
\begin{equation}\label{GraphMeasures}
\mathcal{M}_N^{\boldsymbol{\omega}}(\mathbf{x};\mathbf{a})=\frac{\textnormal{dim}_N(\mathbf{x})}{a_0^{N-1}a_1^{N-2}\cdots a_{N-2}} \det\left(-\frac{1}{a_{x_j}}\frac{1}{2\pi \textnormal{i}}\oint_{\mathsf{C}_\mathbf{a}}\frac{p_{i-1}(z)f_{\boldsymbol{\omega}}(z)}{p_{x_j+1}(z)}dz\right)_{i,j=1}^N, \ \ \mathbf{x}\in \mathbb{W}_N.
\end{equation}

\end{defn}
We have the following result.

\begin{thm}\label{TheoremGraph} Consider the inhomogeneous Gelfand-Tsetlin graph $\mathbf{GT}_+(\mathbf{a})$ with $\inf_{x\in \mathbb{Z}_+} a_x\ge 1$.
Then, the sequence of measures $\left(\mathcal{M}^{\boldsymbol{\omega}}_N(\cdot;\mathbf{a})\right)_{N=1}^\infty$ forms an extremal coherent sequence of probability measures on $\mathbf{GT}_+(\mathbf{a})$.
\end{thm}

The result above answers a question from our previous paper \cite{DeterminantalStructures}, where we showed that the special case of measures corresponding to $\boldsymbol{\omega}=(\mathbf{0},\mathbf{0},t)$, equivalently $f_{\boldsymbol{\omega}}(z)=e^{-tz}$, is coherent and asked whether it was in fact extremal.

As we shall observe in Section \ref{SectionGraph}, the connection with the dynamics considered in this paper comes from the following equality
\begin{equation}\label{GraphMeasuresDynamicsConnection}
\mathcal{M}_N^{\boldsymbol{\omega}}\left(\mathbf{x};\mathbf{a}\right)=\mathfrak{P}_{f_{\boldsymbol{\omega}}}^{(N)}\left(\left(0,1,\dots,N-1\right),\mathbf{x}\right),\ \ \mathbf{x}\in \mathbb{W}_N,
\end{equation}
where we define $\mathfrak{P}_{f_{\boldsymbol{\omega}}}^{(N)}$ by the right hand side of (\ref{SemigroupIntro}) with $f_{s,t}$ replaced by $f_{\boldsymbol{\omega}}$. In probabilistic terms, see Remark \ref{RmkProbabilisticTermsGraph}, this means we run a mixture of sequential-update Bernoulli, geometric Warren-Windridge and continuous-time pure-birth dynamics (possibly an infinite number of discrete-time steps as long as the parameters $\alpha_i,\beta_i$ are summable), starting from the fully-packed configuration, and then look at the distribution of the $N$-th level of the array which is given by $\mathcal{M}_N^{\boldsymbol{\omega}}(\cdot;\mathbf{a})$. It will follow from the results in the sequel that these measures are coherent and more remarkably actually extremal which establishes Theorem \ref{TheoremGraph}.

The problem of classifying in an explicit way all coherent sequences of measures is known as the problem of determining the boundary of the graph, see \cite{OlshanskiHarmonic,BorodinOlshanskiHarmonic,BorodinOlshanskiBoundary,YoungBouquet} for general background and motivation. In the case of the Gelfand-Tsetlin graph  it is equivalent to the classification of extreme characters of the infinite-dimensional unitary group $\mathbb{U}(\infty)$ and also the classification of totally non-negative Toeplitz matrices, see \cite{BorodinOlshanskiBoundary}. The fact that an explicit classification sometimes exists is remarkable and has been achieved only for a handful of models, see for example \cite{BorodinOlshanskiBoundary,VershikKerovUnitary,VershikKerovSymmetric,OlshanskiVershik,OlshanskiExtendedGelfandTsetlin,VanPeskiIMRN,MatveevMacdonald, VadimqGelfandTsetlin,GorinOlshanski,OrbitalBeta}. In the case of $\mathbf{GT}_+$, recall $a_x\equiv 1$, all such measures are given by (\ref{GraphMeasures}), see \cite{BorodinOlshanskiBoundary}. It appears that the inhomogeneous graph 
 $\mathbf{GT}_+(\mathbf{a})$ may be another example where an explicit classification is possible. In particular, we believe that subject to some generic conditions on $\mathbf{a}$, and most likely after removing the condition on $\beta_1$, all extremal coherent sequences of measures for $\mathbf{GT}_+(\mathbf{a})$ should be of the form (\ref{GraphMeasures}). We state this as an open problem.

\paragraph{Problem} Classify extremal coherent sequences of measures on $\mathbf{GT}_+(\mathbf{a})$.

\subsection{Duality via a height function}\label{IntroDualitySection}
In this section we prove a new type of duality for the dynamics we have been considering\footnote{In fact, having even more general rates/transition probabilities.} via the application of a certain deterministic map. This map is closely related to the notion of a height function for an interlaced particle system \cite{BorodinFerrari} (in the case of the autonomous left-edge particle system viewed as a growth model this is exactly the standard height function, see Figure \ref{CornerGrowthHeightfunction}). However, as as far as we can tell, the specific choice below has not been used before. From a probabilistic standpoint this mapping essentially swaps space and level/particle inhomogeneities.

In order to present our results in their most symmetric and natural form, for this section and Section \ref{SectionDuality} where the proofs are given, we will be labelling levels of arrays starting from level $0$ instead of $1$ and coordinates of $\mathbb{W}_N$ using indices $0, 1,\dots,N-1$. In particular, the configuration of the $N+1$ particles at level $N$, which is in $\mathbb{W}_{N+1}$, will have coordinates $(x_0^{(N)},x_1^{(N)},\dots,x_N^{(N)})$. Moreover, with this notation convention, for a configuration $\left(x_k^{(n)}\right)_{0\le k \le n;n\ge 0}$ in $\mathbb{IA}_\infty$ we will call $\left(x_k^{(n)}\right)_{n\ge k}$ the $k$-th column of the configuration.

\begin{defn}
Define the set of configurations $\mathbb{IA}_\infty^*$ by:
\begin{equation*}
\mathbb{IA}_\infty^*=\left\{\left(x_k^{(n)}\right)_{0\le k \le n; n \ge 0}\in \mathbb{IA}_\infty: \forall i \ge 1, \exists j_i< \infty \textnormal{ such that } x_i^{(n)}=i, \textnormal{ for } n\ge j_i\right\}.
\end{equation*}
  
\end{defn}

In other words, particle configurations in $\mathbb{IA}_\infty^*$ are the ones that can be obtained from the fully-packed configuration by moving, for each column, a finite number of particles. 

\begin{defn}
 Define the following map $\mathsf{Hgt}$ (for height function) by,
\begin{align*}
\mathsf{Hgt}:\mathbb{IA}_\infty^* &\to \mathbb{IA}_\infty^*,\\
\left(x_i^{(j)}\right)_{0\le i \le j; j \ge 0}&\mapsto \left(\mathsf{h}_i(j)+i\right)_{0\le i \le j; j \ge 0},
\end{align*}
where $\mathsf{h}_i(j)$, for a given configuration $\left(x_i^{(j)}\right)_{0\le i \le j; j \ge 0}\in \mathbb{IA}_\infty^*$, is defined by 
\begin{equation*}
 \mathsf{h}_i(j)=\#\left\{n\ge i:x_i^{(n)}>j\right\}.
\end{equation*}   
\end{defn}

It is easy to see that $\mathsf{Hgt}$ is well-defined on $\mathbb{IA}_\infty^*$ and that $\mathsf{Hgt}$ maps the fully-packed configuration to itself. In fact, as we will prove in Section \ref{SectionDuality}, $\mathsf{Hgt}$ is an involution.

\begin{prop}\label{PropInvolution}
The map $\mathsf{Hgt}:\mathbb{IA}_\infty^*\to \mathbb{IA}_\infty^*$ is an involution.
\end{prop}

In particular, since $\mathsf{Hgt}$ is a bijection, given any Markov process $\left(\mathsf{X}(t);t\ge 0\right)$, in discrete or continuous time, in $\mathbb{IA}_\infty^*$ the stochastic process $\left(\mathsf{Hgt}(\mathsf{X}(t));t\ge 0\right)$, again taking values in $\mathbb{IA}_\infty^*$, is also Markovian.

Now, suppose that we are given a function $\boldsymbol{\theta}:\mathbb{Z}_+ \times\mathbb{Z}_+\to (0,\infty)$, that we call the environment, satisfying in the case of the continuous-time model in Definition \ref{DefDynamicsEnvironment} below
\begin{equation}\label{EnvironmentCond1}
\inf_{x,y\in \mathbb{Z}_+}\boldsymbol{\theta}(x,y)>0 \textnormal{ and }\sup_{x,y\in \mathbb{Z}_+}\boldsymbol{\theta}(x,y)<\infty
\end{equation}
and for the discrete-time models in Definition \ref{DefDynamicsEnvironment},
\begin{equation}\label{EnvironmentCond2}
\inf_{x,y\in \mathbb{Z}_+}\boldsymbol{\theta}(x,y)>0 \textnormal{ and }\sup_{x,y\in \mathbb{Z}_+}\boldsymbol{\theta}(x,y)<1.
\end{equation}
Given such a function $\boldsymbol{\theta}$ we define the following dynamics.

\begin{defn}\label{DefDynamicsEnvironment}
 We say that a process in $\mathbb{IA}_\infty$, in continuous-time, satisfies the pure-birth push-block dynamics in environment $\boldsymbol{\theta}$ if the jump rate of a particle at space location $x$, at level $y$, is given by $\boldsymbol{\theta}(x,y)$, with interactions between particles being exactly as in Definition \ref{DefCtsTimeDynamics}. 

We say that a process in $\mathbb{IA}_\infty$, in discrete-time, satisfies the sequential-update Bernoulli push-block dynamics in environment $\boldsymbol{\theta}$ if the jump probability of a particle at space location $x$, at level $y$, to position $x+1$ is given by $\boldsymbol{\theta}(x,y)$ and the probability to stay put at $x$ is $1-\boldsymbol{\theta}(x,y)$, with interactions between particles being exactly as in Definition \ref{DefBernoulliDynamics}.

We say that a process in $\mathbb{IA}_\infty$, in discrete-time, satisfies the Warren-Windridge geometric push-block dynamics in environment $\boldsymbol{\theta}$ if the jump probability of a particle at space location $x$, at level $y$, towards position $x+n$, for $n \in \mathbb{Z}_+$, is given by $\boldsymbol{\theta}(x,y)\boldsymbol{\theta}(x+1,y)\cdots \boldsymbol{\theta}(x+n-1,y)(1-\boldsymbol{\theta}(x+n,y))$, with interactions between particles being exactly as in Definition \ref{DefGeometricDynamics}.
\end{defn}

\begin{rmk}
Observe that, in continuous-time with $\boldsymbol{\theta}(x,y)=a_x$ we get back the dynamics from Definition \ref{DefCtsTimeDynamics}. In discrete-time, for Bernoulli jumps with $\boldsymbol{\theta}(x,y)=\alpha a_x$ we get back the dynamics from Definition \ref{DefBernoulliDynamics}, while for geometric jumps with $\boldsymbol{\theta}(x,y)=\beta a_x\left(1+\beta a_x\right)^{-1}$ we get back the dynamics from Definition \ref{DefGeometricDynamics}.
\end{rmk}

The following says that the dynamics above, if initialised in $\mathbb{IA}_\infty^*$, stay in $\mathbb{IA}_\infty^{*}$.

\begin{prop}\label{PropWellDefinedDynamics}
Let $\mathbf{x}\in \mathbb{IA}_\infty^*$. Assume the stochastic process $\left(\mathsf{X}(t);t\ge 0\right)$, with initial condition $\mathsf{X}(0)=\mathbf{x}$, evolves according to either the continuous time pure-birth, or discrete-time sequential-update Bernoulli or Warren-Windridge geometric push-block dynamics in environment $\boldsymbol{\theta}$ satisfying the corresponding conditions (\ref{EnvironmentCond1}), (\ref{EnvironmentCond2}) above. Then, almost surely, for all $t\ge 0$, $\mathsf{X}(t)\in \mathbb{IA}_\infty^*$.
\end{prop}

The following is the main result of this section. 

\begin{thm}\label{DualityThmIntro}
Let $\mathbf{x}\in \mathbb{IA}_\infty^*$ and environment $\boldsymbol{\theta}$ satisfying the corresponding conditions (\ref{EnvironmentCond1}), (\ref{EnvironmentCond2}) above. Suppose that the process $\left(\mathsf{X}(t);t \ge 0\right)$ evolves according to one of the following push-block dynamics:
\begin{enumerate}[label=(\alph*)]
    \item continuous-time pure-birth,
    \item sequential-update Bernoulli,
    \item Warren-Windridge geometric,
\end{enumerate}
 in environment $\boldsymbol{\theta}$ with initial condition $\mathsf{X}(0)=\mathbf{x}$. Then, the process $\left(\mathsf{Hgt}(\mathsf{X}(t));t\ge 0\right)$ 
follows respectively the push-block dynamics:
\begin{enumerate}[label=(\alph*)]
    \item continuous-time pure-birth,
    \item Warren-Windridge geometric,
    \item sequential-update Bernoulli,
\end{enumerate}
in environment $\hat{\boldsymbol{\theta}}$, where $\hat{\boldsymbol{\theta}}(x,y)=\boldsymbol{\theta}(y,x)$, and initial condition $\mathsf{Hgt}(\mathbf{x})$.
\end{thm}

In words, under $\mathsf{Hgt}$ the environment $\boldsymbol{\theta}$ always gets transformed to the environment $\hat{\boldsymbol{\theta}}$, continuous-time pure-birth dynamics stay of the same form and sequential-update Bernoulli dynamics get mapped to Warren-Windridge geometric dynamics and vice-versa. 

Although the statement of Theorem \ref{DualityThmIntro} is very simple, as far as we can tell, it is new even in the homogeneous case of constant $\boldsymbol{\theta}$ (at least we have not been able to locate any explicit statement in the literature). In the special case $\boldsymbol{\theta}(x,y)=a_y$, Theorem \ref{DualityThmIntro} shows how to map level/particle inhomogeneities to space inhomogeneities. For this special choice of $\boldsymbol{\theta}$, in the case of continuous time pure-birth dynamics with fully-packed initial configuration, and when projecting to the right-edge particles of the array, the correspondence above can be shown to be equivalent to a certain mapping used by Petrov in \cite{PetrovInhomogeneousPushTASEP} to obtain an expicit formula for the distribution of the height function of inhomogeneous space push-TASEP for the fully-packed initial condition. %Finally, for general $\lambda$ and general initial condition $\mathbf{x}\in \mathbb{IA}_\infty^*$ it is unclear whether there is any integrable structure behind these dynamics. In particular, having the determinantal property seems very unlikely in this generality.

We note that level/particle-inhomogeneous dynamics and space-inhomogeneous dynamics are not equivalent for finite systems. For example, even the motion of a single particle in the space-inhomogeneous setting can involve infinitely many parameters $\mathbf{a}=(a_x)_{x\in \mathbb{Z}_+}$. Thus, we would need infinitely many particles with particle-dependent and space-independent jump probabilities/rates to hope for any kind of correspondence.

\begin{rmk}
We could have taken the environment $\boldsymbol{\theta}$ to depend on time $t$ and the statement and proof of Theorem \ref{DualityThmIntro} would remain the same. 
\end{rmk}

\begin{rmk}
We could also consider Bernoulli dynamics with parallel-update, see \cite{BorodinFerrari}, also Section \ref{ParallelUpdateSection}. Under the map $\mathsf{Hgt}$ the Bernoulli parallel-update dynamics in environment $\boldsymbol{\theta}$ get mapped to Bernoulli parallel-update dynamics in environment $\hat{\boldsymbol{\theta}}$. The situation in the parallel-update setting is a little more subtle and we will not pursue it further in this paper.
\end{rmk}

\subsection{The domino tiling shuffling algorithm dynamics connection}\label{ShufflingIntro}

In this part, and its expansion in Section \ref{SectionShuffling}, we explain how a famous statistical mechanics model, domino tilings of the Aztec diamond with general (non-interacting) domino weights, is intimately connected to certain Bernoulli push-block dynamics with general parameters. We first give an example of the kind of result we can prove.

The object of study is the Aztec diamond, introduced in \cite{AlternatingSignDominoTilings1,AlternatingSignDominoTilings2}, which is a certain region in the square lattice with sawtooth boundary as in Figure \ref{AztecDiamond} and \ref{AztecDiamondParticles}, see Section \ref{SectionShuffling} for precise definitions. We can colour the squares of the Aztec diamond in black/white checkerboard fashion as in Figures \ref{AztecDiamond} and \ref{AztecDiamondParticles}. The Aztec diamond can be covered $1\times 2$ and $2\times 1$ dominos, see Figure \ref{AztecDiamondParticles} for an illustration. There are four types\footnote{The domino tiling of the Aztec diamond is equivalent to a certain dimer model on the corresponding Aztec diamond graph. Then, the terminology north, south, east, west becomes much more intuitive, see Section \ref{SectionShuffling}.} of dominos called north, south, east and west as shown in Figure \ref{AztecDiamond}. We put a certain coordinate system on the Aztec diamond as shown in Figures \ref{AztecDiamond}, \ref{AztecDiamondParticles} (the specific choice is so that it is consistent with the other  results in this paper), see Section \ref{SectionShuffling} for more details.

\begin{figure}
\captionsetup{singlelinecheck = false, justification=justified}
\centering
\begin{tikzpicture}
  \draw[ultra thick,->] (0,0) -- (1,1);

\draw[ultra thick,->] (0,0) -- (1,-1);

\node[above left] at (0.5,0.5) {space $x$};

\node[below left] at (0.5,-0.5) {level $n$};  
\end{tikzpicture}\ \ \ \
\begin{tikzpicture}

\fill[lightgray] (0,0) rectangle (0.5,0.5);
\fill[lightgray] (0.5,0.5) rectangle (1,1);
\draw[ultra thick] (0,0) rectangle (1,1);
\draw[dotted] (0,0) -- (1,1);
\draw[dotted] (0,1) -- (1,0);
\node[below left] at (0,0) {$1$}; 
\node[above left] at (0,1) {$0$};

\end{tikzpicture}\ \ \ \ \ \
\begin{tikzpicture}

\fill[lightgray] (0,0.5) rectangle (0.5,1);
\fill[lightgray] (0.5,1) rectangle (1,1.5);
\fill[lightgray] (1,1.5) rectangle (1.5,2);

\fill[lightgray] (0.5,0) rectangle (1,0.5);
\fill[lightgray] (1,0.5) rectangle (1.5,1);
\fill[lightgray] (1.5,1) rectangle (2,1.5);

\draw[ultra thick] (0,0.5) -- (0,1.5) -- (0.5,1.5) -- (0.5,2) -- (1.5,2) -- (1.5,1.5) -- (2,1.5) -- (2,0.5) -- (1.5,0.5) -- (1.5,0) -- (0.5,0) -- (0.5,0.5) -- (0,0.5);

\draw[dotted] (0,0.5) -- (1.5,2);
\draw[dotted] (0.5,0) -- (2,1.5);

\draw[dotted] (0,1.5) -- (1.5,0);
\draw[dotted] (0.5,2) -- (2,0.5);

\node[below left] at (0,0.5) {$1$}; 
\node[below left] at (0.5,0) {$2$}; 

\node[above left] at (0,1.5) {$0$}; 
\node[above left] at (0.5,2) {$1$};

\end{tikzpicture}\ \ \ \ \ 
\begin{tikzpicture}

\fill[lightgray] (1,0) rectangle (1.5,0.5);
\fill[lightgray] (1.5,0.5) rectangle (2,1);
\fill[lightgray] (2,1) rectangle (2.5,1.5);
\fill[lightgray] (2.5,1.5) rectangle (3,2);

\fill[lightgray] (0.5,0.5) rectangle (1,1);
\fill[lightgray] (1,1) rectangle (1.5,1.5);
\fill[lightgray] (1.5,1.5) rectangle (2,2);
\fill[lightgray] (2,2) rectangle (2.5,2.5);

\fill[lightgray] (0,1) rectangle (0.5,1.5);
\fill[lightgray] (0.5,1.5) rectangle (1,2);
\fill[lightgray] (1,2) rectangle (1.5,2.5);
\fill[lightgray] (1.5,2.5) rectangle (2,3);

\draw[ultra thick] (2,0) -- (1,0) -- (1,0.5) -- (0.5,0.5) -- (0.5,1)-- (0,1) -- (0,2) -- (0.5,2) -- (0.5,2.5)-- (1,2.5)--(1,3) -- (2,3) -- (2,2.5) -- (2.5,2.5) -- (2.5,2) -- (3,2) -- (3,1) -- (2.5,1) -- (2.5,0.5) -- (2,0.5) -- (2,0);

\node[below left] at (0,1) {$1$}; 
\node[below left] at (0.5,0.5) {$2$}; 
\node[below left] at (1,0) {$3$}; 

\node[above left] at (0,2) {$0$}; 
\node[above left] at (0.5,2.5) {$1$}; 
\node[above left] at (1,3) {$2$}; 

\draw[dotted] (0,1) -- (2,3);
\draw[dotted] (0.5,0.5) -- (2.5,2.5);
\draw[dotted] (1,0) -- (3,2);

\draw[dotted] (0,2) -- (2,0);
\draw[dotted] (0.5,2.5) -- (2.5,0.5);
\draw[dotted] (1,3) -- (3,1);

\end{tikzpicture}
\bigskip

\bigskip

\begin{tikzpicture}
\fill[lightgray] (0.5,0) rectangle (1,0.5);

\draw[ultra thick] (0,0) -- (1,0) -- (1,0.5) -- (0,0.5) -- (0,0);

\node[below] at (0.5,0) {North};

\end{tikzpicture}\ \ \ \ 
\begin{tikzpicture}
\fill[lightgray] (0,0) rectangle (0.5,0.5);

\draw[ultra thick] (0,0) -- (1,0) -- (1,0.5) -- (0,0.5) -- (0,0);

\node[below] at (0.5,0) {South};

\end{tikzpicture}\ \ \ \
\begin{tikzpicture}
\fill[lightgray] (0,0) rectangle (0.5,0.5);

\draw[ultra thick] (0,0) -- (0.5,0) -- (0.5,1) -- (0,1) -- (0,0);

\node[below] at (0.25,0) {West};

\end{tikzpicture}\ \ \ \
\begin{tikzpicture}
\fill[lightgray] (0,0.5) rectangle (0.5,1);

\draw[ultra thick] (0,0) -- (0.5,0) -- (0.5,1) -- (0,1) -- (0,0);

\node[below] at (0.25,0) {East};

\end{tikzpicture}

\caption{The Aztec diamond (sizes $1, 2, 3$) and the corresponding coordinate system in terms of space and level coordinates. Also depicted are the four types of dominoes.}\label{AztecDiamond}
\end{figure}
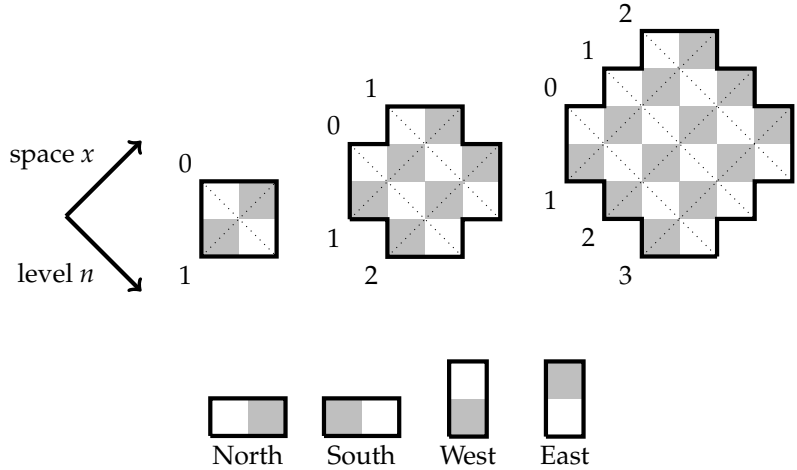

We can associate to each domino in a tiling a certain weight (the most general weights will be discussed in Section \ref{SectionShuffling}). Towards this end, suppose we are given two sequences $\mathbf{z}^{(1)}=(z_x^{(1)})_{x\in \mathbb{Z}_+},\mathbf{z}^{(2)}=(z_x^{(2)})_{x\in \mathbb{Z}_+}\in (0,\infty)^{\mathbb{Z}_+}$ and assume that 
\begin{equation}\label{AZrelationIntro}
\frac{z_x^{(1)}}{z_x^{(1)}+z_x^{(2)}}=a_x, \ \ \textnormal{ for all } x \in \mathbb{Z}_+. 
\end{equation}

\begin{defn}\label{AztecWeightsIntro}
  Given the two sequences $\mathbf{z}^{(1)}$ and $\mathbf{z}^{(2)}$ as above, we define the following weight of a domino tiling of the Aztec diamond, of any fixed size $k\ge 1$, as follows. East and north dominos get weight $1$. West dominos at horizontal location $x$ get weight $z_x^{(1)}$, while south dominos at horizontal location $x$ get weight $z_x^{(2)}$. The weight of the whole tiling is simply the product of weights of all individual dominos contained in the tiling. For any $k\ge 1$, this weighting gives rise, in the obvious way of normalising by the partition function, to a probability measure, that we denote by $\mathbb{P}^{(k),\mathbf{z}^{(1)},\mathbf{z}^{(2)}}$, on domino tilings of the Aztec diamond of size $k$.    
\end{defn}

\begin{figure}
\captionsetup{singlelinecheck = false, justification=justified}
\centering

\begin{tikzpicture}

\fill[lightgray] (1,0) rectangle (1.5,0.5);
\fill[lightgray] (1.5,0.5) rectangle (2,1);
\fill[lightgray] (2,1) rectangle (2.5,1.5);
\fill[lightgray] (2.5,1.5) rectangle (3,2);

\fill[lightgray] (0.5,0.5) rectangle (1,1);
\fill[lightgray] (1,1) rectangle (1.5,1.5);
\fill[lightgray] (1.5,1.5) rectangle (2,2);
\fill[lightgray] (2,2) rectangle (2.5,2.5);

\fill[lightgray] (0,1) rectangle (0.5,1.5);
\fill[lightgray] (0.5,1.5) rectangle (1,2);
\fill[lightgray] (1,2) rectangle (1.5,2.5);
\fill[lightgray] (1.5,2.5) rectangle (2,3);

\draw[ultra thick] (2,0) -- (1,0) -- (1,0.5) -- (0.5,0.5) -- (0.5,1)-- (0,1) -- (0,2) -- (0.5,2) -- (0.5,2.5)-- (1,2.5)--(1,3) -- (2,3) -- (2,2.5) -- (2.5,2.5) -- (2.5,2) -- (3,2) -- (3,1) -- (2.5,1) -- (2.5,0.5) -- (2,0.5) -- (2,0);

\node[below left] at (0,1) {$1$}; 
\node[below left] at (0.5,0.5) {$2$}; 
\node[below left] at (1,0) {$3$}; 

\node[above left] at (0,2) {$0$}; 
\node[above left] at (0.5,2.5) {$1$}; 
\node[above left] at (1,3) {$2$}; 

\draw[dotted] (0,1) -- (2,3);
\draw[dotted] (0.5,0.5) -- (2.5,2.5);
\draw[dotted] (1,0) -- (3,2);

\draw[dotted] (0,2) -- (2,0);
\draw[dotted] (0.5,2.5) -- (2.5,0.5);
\draw[dotted] (1,3) -- (3,1);

\draw[ultra thick] (0,1) rectangle (0.5,2);

\draw[ultra thick] (0.5,1.5) rectangle (1,2.5);

\draw[ultra thick] (1,0.5) rectangle (1.5,1.5);

\draw[ultra thick] (1,0) rectangle (2,0.5);

\draw[ultra thick] (1.5,0.5) rectangle (2.5,1);

\draw[ultra thick] (1.5,1) rectangle (2.5,1.5);

\draw[ultra thick] (2.5,1) rectangle (3,2);

\draw[ultra thick] (1,1.5) rectangle (2,2);

\draw[ultra thick] (1,2) rectangle (2,2.5);

\draw[ultra thick,red] (1,2) rectangle (2,2.5);

\draw[ultra thick,red] (2,1.5) rectangle (2.5,2.5);

\draw[ultra thick,red] (1,0.5) rectangle (1.5,1.5);

\draw[ultra thick,red] (1,0) rectangle (2,0.5);

\draw[ultra thick,red] (1.5,0.5) rectangle (2.5,1);

\draw[ultra thick,red] (2.5,1) rectangle (3,2);

\draw[ultra thick,->] (3.5,1.5) to (4.5,1.5);

 \draw[fill] (7,3) circle [radius=0.1];

 \node[above right] at (7,3) {$(2,1)$};

  \draw[fill] (8,2) circle [radius=0.1];

   \node[above right] at (8,2) {$(2,2)$};

     \draw[fill] (1.5,2.5) circle [radius=0.075];

     \draw[fill] (2,2) circle [radius=0.075];

     \draw[fill] (1,1) circle [radius=0.075];

     \draw[fill] (1.5,0.5) circle [radius=0.075];

    \draw[fill] (2,1) circle [radius=0.075];

    \draw[fill] (2.5,1.5) circle [radius=0.075];

  \draw[fill] (9,1) circle [radius=0.1];

   \node[above right] at (9,1) {$(2,3)$};

  \draw[fill] (8,0) circle [radius=0.1];

     \node[right] at (8,0) {$(1,3)$};

  \draw[fill] (7,-1) circle [radius=0.1];

     \node[right] at (7,-1) {$(0,3)$};

    \draw[fill] (6,0) circle [radius=0.1];

     \node[right] at (6,0) {$(0,2)$};

   \draw[ultra thick,->] (4.5,3) -- (5.5,4);

\draw[ultra thick,->] (4.5,3) -- (5.5,2);

\node[above left] at (5,3.5) {space $x$};

\node[below left] at (5,2.5) {level $n$};  

\node[right] at (4.7,3) {$(x,n)$};

\end{tikzpicture}

\caption{A depiction of the map from a domino tiling of the Aztec diamond to the particle system along with the corresponding coordinates. Recall we put a particle at the location of a south or east domino. Such dominoes are depicted in red. We note that the induced configuration of particles is not necessarily in $\mathbb{IA}_N$ (in the figure above it is not in $\mathbb{IA}_3$). However, if we replace the strict inequalities by just inequalities in the definition of interlacing then the particles would be interlaced.}\label{AztecDiamondParticles}
\end{figure}
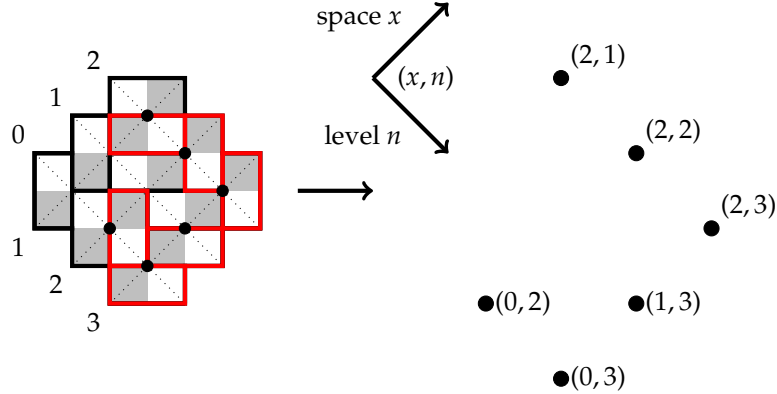

Now, given a domino tiling of the Aztec diamond (of any size $k\ge 1$) we can associate to it a particle configuration as follows. We put a particle whenever we see a south or east domino as in Figure \ref{AztecDiamondParticles}. The particles inherit the coordinates of the dominoes. It is a combinatorial fact that there are exactly $n$ particles on level $n$. We note that the particle configuration is not one-to-one with the domino tiling since some information is lost. However, an extension of this map can be made into a bijection by introducing an extra set of particles, see for example \cite{DuitsKuijlaars}, but we will not do it here.

We have the following theorem.

\begin{thm}\label{ShufflingThmIntro}
Let $\mathbf{a}$ be such that $\inf_{k\in \mathbb{Z}_+}a_k>0$ and $\sup_{k\in \mathbb{Z}_+}a_k<1$. Consider the  probability measures $\mathbb{P}^{(k),\mathbf{z}^{(1)},\mathbf{z}^{(2)}}$ on  domino tilings of the Aztec diamond of size $k$ defined above satisfying (\ref{AZrelationIntro}). Then, there exists a coupling $\mathbb{P}$ of the $\mathbb{P}^{(k),\mathbf{z}^{(1)},\mathbf{z}^{(2)}}$, for all $k\ge 1$, such that the following happens. If we denote by $\mathsf{x}_i^{(j)}(m)$, for $m\ge j$, the location of the $i$-th south or east domino (equivalently particle) on level $j$ of the random tiling of the size $k$ Aztec diamond distributed according to $\mathbb{P}^{(k),\mathbf{z}^{(1)},\mathbf{z}^{(2)}}$ in this coupling, then  for all $N\ge 1$, each discrete-time stochastic process 
\begin{equation}\label{TilingProcessIntro}
\left(\mathsf{x}_1^{(N)}(t+N),\mathsf{x}_2^{(N)}(t+N),\dots,\mathsf{x}_N^{(N)}(t+N);t \ge 0\right)
\end{equation}
evolves as a Markov process in $\mathbb{W}_N$, starting from $(0,1,\dots,N-1)$, with transition probabilities from time $t_1$ to time $t_2$ given by $\mathfrak{P}^{(N)}_{(1-z)^{t_2-t_1}}$. In particular, for any $N\ge 1$ and pairwise distinct time-space points $(t_1,x_1),\dots,(t_n,x_n)$ in $\mathbb{Z}_+\times \mathbb{Z}_+$,
\begin{align*}
\mathbb{P}\left(\exists \ j_1,\dots,j_n \textnormal{ such that } \mathsf{x}_{j_i}^{(N)}(t_i+N)=x_i \textnormal{ for } 1 \le i \le n\right)= \det \left(\mathcal{K}_N\left[(t_i,x_i);(t_j,x_j)\right]\right)_{i,j=1}^n
\end{align*}
where $f_{s,t}(z)=(1-z)^{t-s}$ in the definition of $\mathcal{K}_N$ from (\ref{CorrKernelNonColliding}).
\end{thm}

The probabilistic statement in the theorem above, in that (\ref{TilingProcessIntro}) is a Markov process with explicit transition probabilities, is new. The only known case before was the homogeneous one, $z^{(1)}_x\equiv z_x^{(2)}\equiv 1$, for all $x\in \mathbb{Z}_+$, which follows from the work of Nordenstam \cite{Nordenstam}. On the other hand, the fact that the model is determinantal (for general domino weights even) is well-known and follows from the classical work of Kasteleyn, see \cite{Kasteleyn,KenyonDimers}. However, an explicit computation of the correlation kernel in a form that is amenable to further analysis is highly non-trivial, see for example \cite{ChhitaYoung,ChhitaJohansson}, and for more general recent results (using different methods from ours for the computations, in fact related to the techniques of the next subsection) see \cite{DuitsKuijlaars,ChhitaDuits}.

The desired coupling in the theorem is obtained via the so-called domino shuffling algorithm, which was introduced in \cite{AlternatingSignDominoTilings1,AlternatingSignDominoTilings2} for the uniform weight and generalised in \cite{ProppShuffle}. This algorithm allows one to sample a random tiling of the fixed size Aztec diamond, corresponding to a weighting $\mathcal{W}$ of the dominos, via a sequence of local moves starting from an Aztec diamond of size $1$. How this algorithm works precisely is explained in detail in Section \ref{SectionShuffling}. Then, by extending the work of Nordenstam \cite{Nordenstam},
 we show that for any weighting $\mathcal{W}$ of the dominos the induced dynamics of the shuffling algorithm on the corresponding particle configuration is given by certain Bernoulli push-block dynamics on interlacing arrays with a time-shift (this time shift can already be anticipated from the form of Theorem \ref{ShufflingThmIntro}). Remarkably, the dependence of the weighting $\mathcal{W}$ only comes in the parameters of the 0-1 Bernoulli random variables governing the jumps and the actual interactions between particles are always the same (independent of $\mathcal{W}$). See Section \ref{SectionShuffling} for more details.

In order to obtain the coupling for all Aztec diamonds of different sizes the key, and as far as we can tell new, notion is that of a sequence of consistent weightings of Aztec diamonds of different sizes, see Section \ref{SectionShuffleMarkovConsistent}. This notion can be considered analogous to coherent sequences of probability measures on $\mathbf{GT}_+(\mathbf{a})$ from Section \ref{SectionGraphIntro}. This analogy is not perfect though. In the $\mathbf{GT}_+(\mathbf{a})$ graph case consistency is with respect to an application of the Markov kernels $\Lambda_{N+1,N}^{\mathbf{GT}_+(\mathbf{a})}$ while in the Aztec diamond case consistency is in terms of a certain deterministic dynamical system induced by maps $\left(\mathcal{UR}_k^{n}\right)_{k\le n}$ called urban renewal, see Section \ref{SectionShuffling} for more details.

As mentioned above, the connection to Bernoulli push-block dynamics on arrays works for any weighting $\mathcal{W}$, including the $2$-periodic and $k$-periodic weightings that have been much-studied in the past decade, see \cite{ChhitaYoung,ChhitaJohansson,DuitsKuijlaars,DuitsBeggren,BorodinDuits,ChhitaDuits,BerggrenAnnProb,BerggrenBorodin}. The only thing that changes are the parameters of the 0-1 Bernoulli random variables. Their dependence on time and space is different from $\alpha_t a_x$ (so they do not fall into the class of models that we studied in Section \ref{SectionSpaceTimeCorrIntro}) but a straightforward computation shows that these parameters are still relatively simple. It would be interesting to obtain probabilistic results (recall the determinantal property is always there) like the one in Theorem \ref{ShufflingThmIntro} for $2$-periodic weightings. It may be possible that this can be done by extending the results of Sections \ref{Section1D}, \ref{SectionIntertwining} and \ref{SectionCouplings} to inhomogeneous Toeplitz-like matrices with matrix symbol $\mathbf{f}$. A very small number of such results can be extended to the matrix symbol setting and this is used in Section \ref{SectionLineEnsembles} but it is not obvious how to do this for all of them (the naive extensions do not work). Another possibility is to extend the Schur dynamics formalism of Borodin \cite{SchurDynamics} which is based on symmetric function theory. We think that the right variant of Schur functions for this extension may be the loop Schur functions \cite{LoopSchur}, in part because of their connection to totally-positive block Toeplitz matrices \cite{LoopSchur} which come up in the study of the $2$-periodic Aztec diamond, see \cite{DuitsKuijlaars,DuitsBeggren}, also Section \ref{SectionLineEnsembles}. We will investigate this in the future.
 
Finally, it is worth mentioning that dynamics coming from the shuffling algorithm on dimer coverings of the full-plane $\mathbb{Z}^2$ with general weights have been studied in detail in \cite{ChhitaFerrari,ChhitaFerrariToninelli,ChhitaToninelli1,ChhitaToninelli2} but there is essentially no overlap in terms of results between those papers and ours.

\subsection{Line ensembles with fixed starting and final positions in discrete inhomogeneous space}\label{LineEnsemblSetionIntro}

In previous sections, see in particular Theorem \ref{ThmCorrelationKernelNI}, we studied $N$ non-intersecting,  for all times, random walks with inhomogeneous Bernoulli or geometric steps starting at time $t=0$ from fixed consecutive locations $(0,1,\dots,N-1)$. Now, we would like to study the model where the walks are conditioned to end at some fixed consecutive locations $(M,M+1,\dots,M+N-1)$, for some $M\in \mathbb{Z}_+$, at some fixed time $t=L$ (hence this model is only defined for times $0\le t \le L$). We can also think of this model as non-intersecting random walk bridges. Compared to the setting of walks non-intersecting for all times this is in general harder to study. 

Beyond its intrinsic probabilistic interest there is some significant motivation coming from statistical mechanics to study the above model. Namely, such Bernoulli walk paths and a mixture of Bernoulli and geometric walk paths are in bijection with lozenge tilings of the hexagon and domino tilings of the Aztec diamond respectively, see for example \cite{JohanssonDetPP,JohanssonNonIntersectingHahn,JohanssonNonIntersectingTilings,JohanssonEdgeFluctuations,GorinHahn,DuitsKuijlaars,DuitsBeggren}. In particular, an arbitrary probability measure on such tilings gets mapped to a corresponding probability measure on such non-intersecting paths which is better suited for further analysis.

Recently, Duits and Kuijlaars \cite{DuitsKuijlaars} and in subsequent work Berggren and Duits \cite{DuitsBeggren} have developed a theory that allows one to study such measures in great generality. Their main object of study is a probability measure given as a product of determinants involving (block) Toeplitz matrices with a matrix symbol $\mathbf{f}$. This is basically the probability measure on non-intersecting paths with fixed starting and final positions alluded to above. Via a connection to matrix-valued orthogonal polynomials and matrix-valued Riemann-Hilbert problems they are able to analyse this measure, both for finite $N$ and asymptotically.

The purpose of this part of the paper is to extend some of their results to the setting where we replace the Toeplitz matrices with matrix symbol\footnote{To be precise $\mathbf{f}(1-z)$.} $\mathbf{f}$ by an inhomogeneous Toeplitz-like matrix with matrix symbol $\mathbf{f}$, see Definition \ref{InhomogeneousBlockToeplitzDef}. In particular, the type of measure we study is defined in (\ref{ProbabilityMeasure}). We note that this is indeed a more general setting compared to measures given by Toeplitz matrices with matrix symbols. It is an interesting question to understand for which matrix-valued $\mathbf{f}$ the measure (\ref{ProbabilityMeasure}) has probabilistic meaning. Of course, for $\mathbf{a}=1^{\mathbb{Z}_+}$ we are back to the block Toeplitz matrix setting and such functions have been classified, see \cite{DuitsBeggren} and the references therein. More generally, if $\mathbf{a}$ is very close to $1^{\mathbb{Z}_+}$ (in a suitable sense), by continuity in the parameters, all the functions $\mathbf{f}$ that give rise to positive measures in \cite{DuitsBeggren} also do so in our inhomogeneous setting as well. For general $\mathbf{a}$ the answer is not clear (of course for general $\mathbf{a}$ but scalar $\mathbf{f}$ then products of $e^{-tz}$, $(1-\alpha z)$, $(1+\beta z)^{-1}$ would work).
 
Our main results are Theorems \ref{CorrKernelFixedStartEndpoint} and \ref{ThmLineEnsembleTopBottomLimit} in Section \ref{SectionLineEnsembles}. As these are far too technical to present in this introductory section, let us instead state a typical asymptotic result that can be proven within this framework (which in fact only requires scalar symbols $\mathbf{f}$). 

As with Section \ref{IntroDualitySection}, it will be notationally more convenient, and in order to be consistent with the works \cite{DuitsKuijlaars,DuitsBeggren}, to again label particles in $\mathbb{W}_N$ starting with index $0$ instead of $1$. Namely, the coordinates of an element in $\mathbb{W}_N$ will be denoted by $(x_0^{(N)},x_1^{(N)},\dots,x_{N-1}^{(N)})$. Similarly for random variables.

\begin{defn}\label{DefIntroFixedEndPoints}
 Let $L_1,L_2 \in \mathbb{Z}_+$ with $L_1+L_2=L$. Let $f_{r,r+1}(z)$ for $r=0,1,\dots,L-1$, be such that $L_1$ of them are of the form $1-\alpha_{i}z$ for some parameters $\alpha_1,\dots,\alpha_{L_1}$ and $L_2$ of them are of the form $(1+\beta_i z)^{-1}$ with parameters $\beta_1,\dots,\beta_{L_2}$. Consider $N$ independent identically distributed discrete-time random walks, with either inhomogeneous Bernoulli or geometric steps, with fixed inhomogeneity sequence $\mathbf{a}$, with the step at time $s$ following the transition probability $\mathsf{T}_{f_{s,s+1}}$, starting from locations $(0,1,\dots,N-1)$ at time $0$ and ending at locations $(M,M+1,\dots,M+N-1)$ at time $L$ and conditioned to not intersect in the intervening times $t=1,2,\dots,L-1$. Denote this stochastic process, which by construction stays in $\mathbb{W}_N$, by
\begin{equation}\label{LineEnsembleProcessIntro}
\left(\mathsf{X}_0^{N,L,M}(t), \mathsf{X}_1^{N,L,M}(t), \dots,\mathsf{X}_{N-1}^{N,L,M}(t);1\le t \le L-1\right).
\end{equation}   
\end{defn}

See Figure \ref{PathsFixedStartingFinal} for an illustration of the above setup. Then, we have the following limit theorem for the bottom paths in this path ensemble.

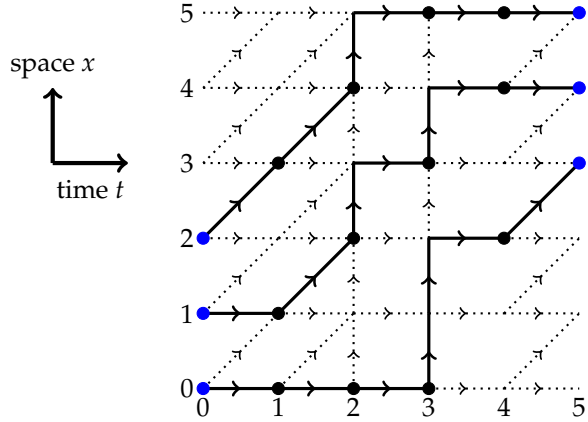
\begin{figure}
\captionsetup{singlelinecheck = false, justification=justified}
\centering
\begin{tikzpicture}

\node[left] at (0,0) {$0$};

\node[left] at (0,1) {$1$};

\node[left] at (0,2) {$2$};

\node[left] at (0,3) {$3$};

\node[left] at (0,4) {$4$};

\node[left] at (0,5) {$5$};

\node[below] at (0,0) {$0$};

\node[below] at (1,0) {$1$};

\node[below] at (2,0) {$2$};

\node[below] at (3,0) {$3$};

\node[below] at (4,0) {$4$};

\node[below] at (5,0) {$5$};

\draw[dotted,thick,middlearrow={>}] (0,0) -- (1,1);

\draw[very thick,middlearrow={>}] (0,0) -- (1,0);

\draw[dotted,thick,middlearrow={>}] (0,1) -- (1,2);

\draw[very thick,middlearrow={>}] (0,1) -- (1,1);

\draw[very thick,middlearrow={>}] (0,2) -- (1,3);

\draw[dotted,thick,middlearrow={>}] (0,2) -- (1,2);

\draw[dotted,thick,middlearrow={>}] (0,3) -- (1,4);

\draw[dotted, thick,middlearrow={>}] (0,3) -- (1,3);

\draw[dotted,thick,middlearrow={>}] (0,4) -- (1,5);

\draw[dotted, thick,middlearrow={>}] (0,4) -- (1,4);

\draw[dotted, thick,middlearrow={>}] (0,5) -- (1,5);

\draw[dotted, thick,middlearrow={>}] (1,0) -- (2,1);

\draw[very thick,middlearrow={>}] (1,0) -- (2,0);

\draw[very thick,middlearrow={>}] (1,1) -- (2,2);

\draw[dotted, thick,middlearrow={>}] (1,1) -- (2,1);

\draw[dotted, thick,middlearrow={>}] (1,2) -- (2,3);

\draw[dotted, thick,middlearrow={>}] (1,2) -- (2,2);

\draw[very thick,middlearrow={>}] (1,3) -- (2,4);

\draw[dotted, thick,middlearrow={>}] (1,3) -- (2,3);

\draw[dotted, thick,middlearrow={>}] (1,4) -- (2,5);

\draw[dotted, thick,middlearrow={>}] (1,4) -- (2,4);

\draw[dotted, thick,middlearrow={>}] (1,5) -- (2,5);

\draw[dotted, thick,middlearrow={>}] (4,0) -- (5,1);

\draw[dotted,thick,middlearrow={>}] (4,0) -- (5,0);

\draw[dotted,thick,middlearrow={>}] (4,1) -- (5,2);

\draw[dotted, thick,middlearrow={>}] (4,1) -- (5,1);

\draw[very thick,middlearrow={>}] (4,2) -- (5,3);

\draw[dotted, thick,middlearrow={>}] (4,2) -- (5,2);

\draw[dotted, thick,middlearrow={>}] (4,3) -- (5,4);

\draw[dotted, thick,middlearrow={>}] (4,3) -- (5,3);

\draw[dotted, thick,middlearrow={>}] (4,4) -- (5,5);

\draw[very thick,middlearrow={>}] (4,4) -- (5,4);

\draw[very thick,middlearrow={>}] (4,5) -- (5,5);

\draw[very thick,middlearrow={>}] (2,0) -- (3,0);

\draw[dotted,thick,middlearrow={>}] (2,1) -- (3,1);

\draw[dotted,thick,middlearrow={>}] (2,2) -- (3,2);

\draw[very thick,middlearrow={>}] (2,3) -- (3,3);

\draw[dotted, thick,middlearrow={>}] (2,4) -- (3,4);

\draw[very thick,middlearrow={>}] (2,5) -- (3,5);

\draw[dotted, thick,middlearrow={>}] (2,0) -- (2,1);

\draw[dotted, thick,middlearrow={>}] (2,1) -- (2,2);

\draw[very thick,middlearrow={>}] (2,2) -- (2,3);

\draw[dotted, thick,middlearrow={>}] (2,3) -- (2,4);

\draw[very thick,middlearrow={>}] (2,4) -- (2,5);

\draw[dotted, thick,middlearrow={>}] (3,0) -- (4,0);

\draw[dotted,thick,middlearrow={>}] (3,1) -- (4,1);

\draw[very thick,middlearrow={>}] (3,2) -- (4,2);

\draw[dotted, thick,middlearrow={>}] (3,3) -- (4,3);

\draw[very thick,middlearrow={>}] (3,4) -- (4,4);

\draw[very thick,middlearrow={>}] (3,5) -- (4,5);

\draw[very thick,middlearrow={>}] (3,0) -- (3,1);

\draw[very thick,middlearrow={>}] (3,1) -- (3,2);

\draw[dotted, thick,middlearrow={>}] (3,2) -- (3,3);

\draw[very thick,middlearrow={>}] (3,3) -- (3,4);

\draw[dotted, thick,middlearrow={>}] (3,4) -- (3,5);

\draw[fill,blue] (0,0) circle [radius=0.08];

\draw[fill,blue] (0,1) circle [radius=0.08];

\draw[fill,blue] (0,2) circle [radius=0.08];

\draw[fill,blue] (5,3) circle [radius=0.08];

\draw[fill,blue] (5,4) circle [radius=0.08];

\draw[fill,blue] (5,5) circle [radius=0.08];

\draw[fill] (1,0) circle [radius=0.08];

\draw[fill] (1,1) circle [radius=0.08];

\draw[fill] (1,3) circle [radius=0.08];

\draw[fill] (2,0) circle [radius=0.08];

\draw[fill] (2,2) circle [radius=0.08];

\draw[fill] (2,4) circle [radius=0.08];

\draw[fill] (3,0) circle [radius=0.08];

\draw[fill] (3,3) circle [radius=0.08];

\draw[fill] (3,5) circle [radius=0.08];

\draw[fill] (4,2) circle [radius=0.08];

\draw[fill] (4,4) circle [radius=0.08];

\draw[fill] (4,5) circle [radius=0.08];

\draw[ultra thick,->] (-2,3) -- (-2,4);

\draw[ultra thick,->] (-2,3) -- (-1,3);

\node[above] at (-2,4) {space $x$};

\node[below] at (-1.5,2.9) {time $t$};

\end{tikzpicture}

\caption{A depiction of the stochastic process (\ref{LineEnsembleProcessIntro}). Here $N=3$, $L=5$, $M=3$. We have $3$ Bernoulli jumps and $2$ geometric jumps. More precisely, the functions $f_{r,r+1}$ are given by $f_{0,1}(z)=(1-\alpha_1z)$, $f_{1,2}(z)=(1-\alpha_2z)$, $f_{2,3}(z)=(1+\beta_1z)^{-1}$, $f_{3,4}(z)=(1+\beta_2 z)^{-1}$, $f_{4,5}(z)=(1-\alpha_3 z)$. The fixed starting and final positions at $t=0$ and $t=5$ are depicted as blue filled circles. The values of the process at times $t=1,2,3,4$ are given by $(0,1,3)$, $(0,2,4)$, $(0,3,5)$, $(2,4,5)$. The three non-intersecting paths are depicted as solid lines. The dotted (directed) lines depict the possible trajectories of the random walks or equivalently the underlying directed Lindstrom-Gessel-Viennot (LGV) \cite{LGV}. The underlying weights on the edges coming from the transition probabilities of the walks are not depicted in the figure.}\label{PathsFixedStartingFinal}
\end{figure}

\begin{thm}\label{InhomogeneousLineEnsemblesIntro}
 In the setting of Definition \ref{DefIntroFixedEndPoints}, assume that the parameter sequence $\mathbf{a}$ satisfies
\begin{equation}
\inf_{x\in \mathbb{Z}_+}a_x\ge1-\mathfrak{c} \textnormal{ and } \sup_{x\in \mathbb{Z}_+}a_x\le1+\mathfrak{c},
\end{equation}
for some $0\le \mathfrak{c} <\frac{1}{3}$. Suppose there exist exactly $M$ indices $i_1,i_2,\dots,i_M$ such that the corresponding Bernoulli parameters satisfy
\begin{equation*}
(2-2\mathfrak{c})^{-1}<\alpha_{i_j}<(1+\mathfrak{c})^{-1}, \ \ j=1,\dots,M,
\end{equation*}
(in particular $L_1 \ge M$) and for $l\neq i_j$ we have $\alpha_l<\frac{1-2\mathfrak{c}}{2-2\mathfrak{c}}$. Finally, assume the geometric parameters satisfy $\beta_l<\frac{1}{2\mathfrak{c}}-1$. Then, for any $m\ge 1$, we have the following convergence in distribution for the bottom $m$ paths of the path ensemble (\ref{LineEnsembleProcessIntro}), as $N \to \infty$,
\begin{equation*}
\left(\mathsf{X}_0^{N,L,M}(t), \dots,\mathsf{X}_{m-1}^{N,L,M}(t);1\le t \le L-1\right)\overset{\textnormal{d}}{\longrightarrow}\left(\mathsf{X}_0^{\infty,L,M}(t),\dots,\mathsf{X}_{m-1}^{\infty,L,M}(t); 1\le t \le L-1\right),
\end{equation*}
where the limiting process $\left(\left(\mathsf{X}_i^{\infty,L,M}(t)\right)_{i=0}^\infty; t=1,\dots,L-1\right)$ is  determined through its determinantal correlation functions: for any $n\ge 1$ and pairwise distinct time-space points $(t_1,x_1), \dots, (t_1,x_n)$ in $\llbracket 1,L-1 \rrbracket \times \mathbb{Z}_+$, we have 
\begin{equation*}
\mathbb{P}\left(\exists \   j_1,\dots,j_n \textnormal{ such that } \mathsf{X}_{j_i}^{\infty,L,M}(t_i)=x_i \textnormal{ for } i=1,\dots,n\right)=\det\left(\mathsf{K}^{L,M}_\infty\left[(t_i,x_i);(t_j,x_j)\right]\right)_{i,j=1}^n,
\end{equation*}    
where the kernel $\mathsf{K}^{L,M}_\infty$ is given by:
\begin{align*}
&\mathsf{K}^{L,M}_\infty\left[(r,m);(r',k)\right]=-\mathbf{1}_{r>r'}\frac{1}{2\pi \textnormal{i}}\frac{1}{a_m}\oint_{|z|=1}\frac{p_k(1-z)}{p_{m+1}(1-z)}\prod_{l=r}^{r'-1}f_{l,l+1}(1-z)dz\nonumber\\
&-\frac{1}{(2\pi \textnormal{i})^2}\frac{1}{a_k}\oint_{|z|=1^{-}}\oint_{|w|=1^+} \prod_{j=1}^M \frac{1}{1-\alpha_{i_j}+\alpha_{i_j}w}\prod_{l=1;l\neq i_j}^{L_1}\frac{1}{1-\alpha_l+\alpha_l z}\prod_{l=1}^{L_2}(1+\beta_l-\beta_l z) \\
&\times\prod_{l=r'}^{L-1}f_{l,l+1}(1-w)\prod_{l=0}^{r-1}f_{l,l+1}(1-z)\frac{p_k(1-w)}{p_{m+1}(1-z)}\frac{dz dw}{z-w}.
\end{align*}
\end{thm}

\subsection{Relation to previous works, ideas and techniques}\label{HistoryIdeasSection}

We now go section by section and discuss briefly some of the ideas and techniques used therein and what seems to be the most directly relevant literature. Given the range of topics studied in this paper this appears to be the most organised way of doing this. For the same reason a complete literature review is unfortunately unfeasible.

We begin with Section \ref{Section1D}. The transition semigroup of a general pure-birth chain, as mentioned earlier, can be written as $\left(\mathsf{T}_{e^{-tz}}\right)_{t\ge 0}$ which has a nice contour integral expression. This form of the transition semigroup was the impetus behind our previous paper on the topic which only dealt with the continuous-time push-block dynamics \cite{DeterminantalStructures}. We then realised that the fundamental object is not actually the transition probability of a pure-birth chain but rather the $\mathsf{T}_f$ matrix/operator for general function $f$. In Section \ref{Section1D} we establish some basic properties of $\mathsf{T}_f$. Most of them are intuitive except maybe the most non-trivial property which is the duality relation from Lemma \ref{LemmaDuality}. This will be especially important in the multidimensional developments that appear in later sections. The results in Section \ref{Section1D} are basically all that is needed\footnote{We also include a couple of more results and comments about $\mathsf{T}_f$ which are of interest in themselves but not used subsequently.} to make subsequent computations work. Although additional ideas are required in each section, in terms of computations, the majority of them boil to down to properties established here. As already mentioned, $\mathsf{T}_f$, for general $\mathbf{a}$, is in fact similar to a standard Toeplitz matrix. However, even in the one-dimensional setting, with some exceptions, we cannot simply transfer over results for standard Toeplitz matrices (and in the multidimensional setting it is unclear whether this similarity is of any use at all). Finally, although the matrices/operators $\mathsf{T}_f$ are natural we have not been able to locate them in the Toeplitz matrix/operator literature (however the literature is truly vast so we may have missed something) and so surprisingly seem to be new. More importantly though, and this is the main message of this work, their probabilistic significance beyond the one-dimensional setting is, as far as we can tell, novel.

In Section \ref{SectionIntertwining} we introduce the transition kernels of the multidimensional versions of the one-dimensional dynamics we studied previously. These come from the Karlin-McGregor \cite{KarlinMcGregor} and Lindstrom-Gessel-Viennot (LGV) \cite{LGV} formula and give rise to non-intersecting paths. We then prove that the transition kernels of $N$ and $N+1$ particles are intertwined. This result generalises the setup of \cite{BorodinFerrari} which deals with transition probabilities coming from Toeplitz matrices. The key ingredient in the computation is the one-dimensional duality relation from Lemma \ref{LemmaDuality}. We also introduce in Section \ref{IntertwiningsDeterminantKernels} a more general setup for intertwinings of kernels that are given in terms of determinants and explain how our previous result fits into this framework. For some other intertwining relations that involve determinants, appearing in a different context, see for example \cite{Gateway,InterweavingRelations}.

In Section \ref{SectionCouplings} we introduce couplings between the intertwined transition kernels from Section \ref{SectionIntertwining} which are in some sense well-adapted to the intertwining. These couplings have their origin in coalescing random walks. The fact that there is a close\footnote{Although this connection is not really highlighted there.} connection between coalescing one-dimensional stochastic processes and dynamics in interlacing arrays originates with the work of Warren \cite{Warren} on Brownian motion. This was later developed in \cite{InterlacingDiffusions,BESQdrifts} for more general one-dimensional diffusions and in \cite{SurfaceGrowthKarlinMcGregor,DeterminantalStructures} for birth and death and pure-birth chains.
This section can be viewed as the correct discretisation in time (and space) of the results of \cite{Warren,InterlacingDiffusions,SurfaceGrowthKarlinMcGregor,DeterminantalStructures}. We note that discrete time hides some subtleties, for example when it comes to how particles are updated, and the explicit computations are trickier\footnote{On the other hand, technical issues such as well-posedness of the dynamics, existence and uniqueness of solutions to the corresponding Kolomogorov equations, are not present.}. We briefly explain what we do. Given a function $f$, so that $\mathsf{T}_f$ has probabilistic meaning, we build an explicit kernel $\mathsf{Q}_f^{N,N+1}$ on two-level interlacing configurations which comes from coalescing random walks (with motion governed by $\mathsf{T}_f$). Using the coalescing walk connection we can prove certain properties of $\mathsf{Q}_f^{N,N+1}$ including some intertwining relations from which the intertwining of Section \ref{SectionIntertwining} also follows (thus giving a different proof). However, as far as we can tell, the exact dynamics described by $\mathsf{Q}_f^{N,N+1}$ cannot be seen from the coalescing random walk representation. In the case of Bernoulli-only dynamics we prove directly by means of some recursive equations (the discrete-time Kolmogorov equation) that $\mathsf{Q}_f^{N,N+1}$, with $f=1-\alpha z$, describes a sequential-update Bernoulli dynamics step. In the case of geometric walks, which is the most subtle, we need to take a different approach altogether by developing the original idea of Warren and Windridge \cite{WarrenWindridge} to inhomogeneous space jumps. We believe\footnote{We have verified this explicitly by tediously checking all possibilities for $N=1$.}, but do not prove here, that $\mathsf{Q}_f^{N,N+1}$, with $f=\left(1+\beta z\right)^{-1}$, does describe a Warren-Windridge geometric dynamics step. Finally, we discuss connections with other related couplings of intertwined semigroups from the literature, see Section \ref{BorodinFerrariCouplings}.

%Although we have worked on this topic of a while, we believe a more conceptual understanding of the connection between coalescing walks and this type of push-block dynamics in arrays is still lacking.

In Section \ref{SectionDynamicsOnArrays}, we put these two-level couplings together in a consistent inductive fashion to consider multilevel dynamics in interlacing arrays in Propositions \ref{PropMultiLevelSpaceTime} and \ref{PropMultilevelSpaceLevel}. This proves Theorem \ref{ThmSpaceLevelInhomogeneous}, the probabilistic statement of Theorem \ref{ThmCorrelationKernelNI} and allows for the computation of the explicit correlation functions in Theorems \ref{ThmCorrelationKernelArray} and \ref{ThmCorrelationKernelNI} in the next section. The induction (given the two-level couplings), making use of the Markov functions theory of Rogers-Pitman \cite{RogersPitman}, by virtue of the intertwinings obtained previously, is standard and variants thereof can be found in multiple places in the literature \cite{Warren,BorodinFerrari,Toda,BorodinPetrov,InterlacingDiffusions,SurfaceGrowthKarlinMcGregor,InteractingDiffusionsPosDef,arista2023matrix} . The space-level inhomogeneous setting of Theorem \ref{ThmSpaceLevelInhomogeneous} is a little more involved to handle in a systematic way but still the main work was done in the preceding sections.

The computation of the correlation functions from Theorem \ref{ThmCorrelationKernelArray} and \ref{ThmCorrelationKernelNI} is done in Section \ref{SectionComputationKernels}. The fact that the point processes in question have determinantal correlation functions is a consequence of the results of Section \ref{SectionDynamicsOnArrays} and the celebrated Eynard-Mehta theorem \cite{EynardMehta,BorodinRains,BorodinFerrariPrahoferSasamoto}. The explicit computation of the correlation kernel then boils down to solving a certain biorthogonalisation problem. To do this we use in an essential way the contour integral formulae, for all the quantities involved, in terms of the polynomials $p_x(z)$. Making use of the results from Section \ref{Section1D} all the computations become rather neat. An analogous but simpler computation, in fact a special case, was performed in \cite{DeterminantalStructures}. That paper deal with continuous-time dynamics only and its main result is the case $f(w)=e^{-tw}$ of Theorem \ref{ThmCorrelationKernelArray}. Finally, the literature on different biorthogonalisation problems arising from the Eynard-Mehta theorem is vast, we list a very small sample \cite{BorodinKuanPlancherel,BorodinFerrari,BorodinFerrariPushASEP,BorodinFerrariPrahoferSasamoto,DeterminantalStructures,SurfaceGrowthKarlinMcGregor} which seems most relevant. 

In Section \ref{SectionConditionedWalks}, we prove the probabilistic representation of $\left(\mathcal{P}_t^{\boldsymbol{\gamma},\bullet, N}\right)_{t\ge 0}$ as independent walks conditioned to never intersect using a soft argument. In the special case of the space being homogeneous this boils down to essentially equivalent arguments (although presented a bit differently), which first appeared in \cite{OConnellConditionedWalk,BianeBougerolOConnell}. The proof relies in an essential way on the fact that walks with different drifts (realised through a Doob transform) become asymptotically ordered with probability one depending on the relative strength of their drifts, see (\ref{AsymptoticOrdering}). When the walks are identical this is of course no longer true and the argument breaks down. One would hope though that the model of identical walks can be recovered as the limit of removing the drifts, under possibly additional conditions on $\mathbf{a}$. There are some subtle points that need to be taken care of to make this rigorous and we leave it for future work. When the increments of the walks are not location-dependent there is vast literature on non-intersection probabilities, see \cite{DenisovWachtel,OrderedWalksKonig,OrderedWalksKonig2} for a small sample, but the idea used therein of coupling with Brownian motion does not appear useful here.

In Section \ref{SectionEdgeTransitionKernels}, we prove a precise version of Theorem \ref{EdgeParticleSystemsThmIntro}. First, to get the explicit form of the transition kernels $\mathfrak{E}_{f_{s,t},\textnormal{r}}^{(N)}$ and $\mathfrak{E}_{f_{s,t},\textnormal{l}}^{(N)}$ presented in Theorem \ref{EdgeExplicitTrans}, we develop a space-inhomogeneous extension of the original idea of Dieker and Warren \cite{DiekerWarren}, which itself deals with the level-inhomogeneous setting, see also \cite{NikosTASEP}. This approach makes use of intertwinings and inverting certain Markov kernels (when viewed as stochastic matrices) that are given as determinants. A non-intersecting path model and the LGV formalism \cite{LGV} to obtain equivalent combinatorial expressions for these kernels is essential for our argument. The resulting explicit formula is of what is usually referred to in the literature as Schutz-type, in reference to Schutz's original work on the transition probabilities of TASEP \cite{Schutz}. Schutz's original derivation \cite{Schutz} of the formula for TASEP used the Bethe ansatz instead. Once one has an explicit formula of this type, then to obtain determinantal correlation functions one goes through a procedure usually referred to as Sasamoto's trick \cite{Sasamoto,BorodinFerrariPrahoferSasamoto,BorodinFerrariPushASEP,BorodinFerrariSasamoto} of rewriting this formula as a sum over interlacing arrays. By virtue of the intertwining origins of our formula from Theorem \ref{EdgeExplicitTrans} such determinantal correlations are essentially an immediate consequence as shown in Theorem \ref{EdgeDeterminantal}. This is both conceptually more satisfying and computationally cleaner compared to the rewriting involved in the Sasamoto's trick approach \cite{Sasamoto,BorodinFerrariPrahoferSasamoto,BorodinFerrariPushASEP,BorodinFerrariSasamoto} but on the other hand requires a lot of machinery to have been developed already.

In Section \ref{SectionGraph}, we establish Theorem \ref{TheoremGraph}. We first prove consistency for $\left(\mathcal{M}^{\boldsymbol{\omega}}_N(\cdot;\mathbf{a})\right)_{N=1}^\infty$ and that they are indeed probability measures by observing the connection (\ref{GraphMeasuresDynamicsConnection}) with the dynamics we have been studying. Then, to prove extremality we make use of a well-adapted generating function for the measures $\left(\mathcal{M}^{\boldsymbol{\omega}}_N(\cdot;\mathbf{a})\right)_{N=1}^\infty$ and De-Finetti's theorem \cite{AldousDeFinetti}. As far as we can tell, in this context, this type of argument originates with \cite{OkounkovOlshanskiJack} and variants of it have been employed in the literature a few times. However, a number of rather special ingredients need to be present in order for it to work, as can be seen from the rather long proof, and the fact that it does in our case as well is not obvious a-priori. This section is also the only place where the factorial Schur functions \cite{MacDonaldSchurvariations} make their appearance explicitly. In fact, a number of, although far from all, the results in this paper can be phrased and proven in terms of factorial Schur functions which would be more in the spirit of developing the factorial Schur analogue of the Schur process \cite{Schur} and Schur dynamics \cite{SchurDynamics} framework. Instead we wanted to emphasize the inhomogeneous Toeplitz-matrix perspective which is somewhat more probabilistic in nature. For example, as far as we can tell, the couplings coming from coalescing walks in Section \ref{SectionCouplings} (for example $\mathsf{Q}_f^{N,N+1}$) do not have a natural interpretation in terms of symmetric functions.

In Section \ref{SectionDuality}, we prove Theorem \ref{DualityThmIntro}. As far as we can tell, no results of this kind have appeared in the literature before; even the fully homogeneous case appears to be new. Our initial motivation stemmed from the following. In an interesting paper \cite{PetrovInhomogeneousPushTASEP}, Petrov showed how one could obtain determinantal formulae for the inhomogeneous-space push-TASEP in continuous-time with fully-packed initial condition via a connection to the original level-inhomogeneous (but space-homogeneous) model of Borodin-Ferrari \cite{BorodinFerrari}. We then wanted to understand whether a duality of sorts between space and level inhomogeneities extended to dynamics on full arrays (and not just the right edge), for more general initial conditions and also for other types of dynamics in discrete time. This led to Theorem \ref{DualityThmIntro}. The proof itself is elementary and, as long as one takes the right perspective, not difficult.

Our results in Section \ref{SectionShuffling} generalise the work of Nordenstam \cite{Nordenstam}, where for the first time the dynamics of the shuffling algorithm on Aztec diamonds, see \cite{AlternatingSignDominoTilings1, AlternatingSignDominoTilings2,ProppShuffle}, in the case of uniform weights, were found to be connected to interlacing particle systems with push-block dynamics. In order to do this we reinterpret and combine the aforementioned work of Nordenstam \cite{Nordenstam} and Propp \cite{ProppShuffle} where the shuffling algorithm was introduced for general weights $\mathcal{W}$ on Aztec diamonds. Then, to prove Theorem \ref{ShufflingThmIntro} we first show that the weights from Definition \ref{AztecWeightsIntro} are consistent\footnote{The reader will have surely noticed by now that consistency of different types is an overarching theme of this paper.} in the sense of Definition \ref{DefConsistentWeightings} and then make use of our previous results. An extension of the work of Nordenstam in the language of line ensembles (equivalent to tilings of the Aztec diamond, see \cite{JohanssonNonIntersectingTilings,JohanssonArctic,BorodinFerrariTilings,DuitsKuijlaars}, also Sections \ref{LineEnsemblSetionIntro} and \ref{SectionLineEnsembles}) to some other weights beyond the uniform one can be found in \cite{BorodinFerrariTilings}. However, general weights $\mathcal{W}$ on Aztec diamonds, see Section \ref{SectionShuffling}, or even the more special weight from Theorem \ref{ShufflingThmIntro}, as far as we can tell, are not considered in \cite{BorodinFerrariTilings}. In particular, our results cannot be obtained from \cite{BorodinFerrariTilings}. It would be interesting to consider dynamics on line ensembles as in \cite{BorodinFerrariTilings} induced by a general weight $\mathcal{W}$ but we do not do it here. Finally, as already mentioned, tilings of Aztec diamonds and also dynamics coming from shuffling-type algorithms have been much studied in the literature \cite{JohanssonArctic,ChhitaYoung,ChhitaJohansson,BouttierChapuyCorteel, DimersRailYard,DuitsKuijlaars,DuitsBeggren,BorodinDuits,ChhitaDuits,BerggrenAnnProb,BerggrenBorodin,ChhitaFerrari,ChhitaFerrariToninelli,ChhitaToninelli1,ChhitaToninelli2} but the techniques and results of all these papers are different from what we do here. 

Our work in Section \ref{SectionLineEnsembles} on line ensembles with fixed starting and end points follows closely the methods of \cite{DuitsKuijlaars,DuitsBeggren}. We show how some of their results can be generalised to the inhomogeneous Toeplitz-like matrix setting with a matrix symbol in Theorems \ref{CorrKernelFixedStartEndpoint} and \ref{ThmLineEnsembleTopBottomLimit}, from which Theorem \ref{InhomogeneousLineEnsemblesIntro} easily follows. The exact fixed $N$ results from \cite{DuitsKuijlaars} allow for a rather clean adaptation, while the asymptotics relevant for the limit of the bottom paths are more involved due to the inhomogeneity $\mathbf{a}$. The main tools in this study are matrix-valued orthogonal polynomials and the analysis of the associated Riemann-Hilbert problems. This part of the paper is only a small step in developing the relevant theory (a proof of concept that something can be done) and many questions remain wide open: for example the probabilistic significance of $\mathsf{T}_\mathbf{f}$ for general matrix-valued $\mathbf{f}$, whether various multidimensional constructions for scalar $f$ studied earlier in the paper have analogues for matrix $\mathbf{f}$, and also asymptotics of the line ensembles in different scaling regimes which are considered in \cite{DuitsKuijlaars} (corresponding to the gas/smooth phase of the two-periodic Aztec diamond for example). There are also very interesting subsequent works \cite{BerggrenAnnProb,ChhitaDuits,BorodinDuits,BerggrenBorodin} which use some of the machinery of \cite{DuitsKuijlaars,DuitsBeggren}. There might be suitable adaptations of those results to our setup but this is purely speculative at present.

In Section \ref{SectionBesselConvergence}, the short-time convergence of the continuous-time dynamics, having general inhomogeneity $\mathbf{a}$, to the discrete Bessel determinantal point process is proven using some standard asymptotic analysis of the correlation kernel. The only noteworthy point is that, in this particular scaling, the rescaled polynomials $p_x(z)$ converge to an exponential function depending on $\mathbf{a}$ only through the average 
$\bar{a}$, see (\ref{CharPolyConv}), which is really what makes the proof work.

\paragraph{Organisation of the paper} 
In Section \ref{Section1D} we develop some theory for inhomogeneous Toeplitz-like matrices and the corresponding one-dimensional Markov dynamics. In Section \ref{SectionIntertwining} we discuss dynamics for non-intersecting paths (single levels of the arrays) and prove a key intertwining relation. In Section \ref{SectionCouplings} we consider various couplings for two sets of non-intersecting paths (two levels of the array). In Section \ref{SectionDynamicsOnArrays} we put everything together to consider consistent dynamics on arrays and also prove Theorem \ref{SpaceLevelInhomogeneousSection} on parameter symmetry. In Section \ref{SectionComputationKernels} we compute the correlation kernels from Section \ref{SectionSpaceTimeCorrIntro}. In Section \ref{SectionConditionedWalks} we prove Theorem \ref{ThmConditioning} on conditioned walks. In Section \ref{SectionEdgeTransitionKernels} we prove a precise version of Theorem \ref{EdgeParticleSystemsThmIntro}, for the edge particle systems, see Theorems \ref{EdgeExplicitTrans} and \ref{EdgeDeterminantal}. In Section \ref{SectionGraph} we prove Theorem \ref{TheoremGraph} on extremal coherent measures. In Section \ref{SectionDuality} we prove the duality results from Section \ref{IntroDualitySection}. In Section \ref{SectionShuffling} we explain the connection between the shuffling algorithm for sampling arbitrary weightings of the Aztec diamond and certain Bernoulli push-block dynamics and prove Theorem \ref{ShufflingThmIntro}. In Section \ref{SectionLineEnsembles}, we develop a framework generalising some of the results of \cite{DuitsKuijlaars} and \cite{DuitsBeggren} in order to prove Theorem \ref{InhomogeneousLineEnsemblesIntro}. Finally, in Section \ref{SectionBesselConvergence} we prove Theorem \ref{BesselTheorem} on convergence to the discrete Bessel point process.

\section{One-dimensional dynamics and inhomogeneous Toeplitz-like matrices}\label{Section1D}

\subsection{Inhomogeneous Toeplitz-like matrices}

Given $R>0$ we use the following notation for the half plane $\mathbb{H}_{-R}=\{z\in \mathbb{C}: \Re(z)>-R\}$. We write $\mathsf{Hol}\left(\mathbb{H}_{-R}\right)$ for the set of functions which are holomorphic in $\mathbb{H}_{-R}$. Recall that a function $f$ is entire if it is holomorphic in the whole of $\mathbb{C}$.

\begin{defn}
 For $f\in \mathsf{Hol}\left(\mathbb{H}_{-R}\right)$ we define the following quantities, where recall that the counterclockwise contour $\mathsf{C}_{\mathbf{a}}\subset \mathbb{H}_{-R}$ encircles all the points of the sequence $\mathbf{a}$,
\begin{align}
\mathsf{T}_f(x,y)&=-\frac{1}{2\pi \textnormal{i}} \frac{1}{a_y}\oint_{\mathsf{C}_\mathbf{a}}\frac{p_x(w)}{p_{y+1}(w)}f(w)dw, \ \ x,y \in \mathbb{Z}_+,\label{InhomogeneousToeplitzDisplay}\\
\mathsf{T}_f(x)&=\mathsf{T}_f(0,x)=-\frac{1}{2\pi \textnormal{i}} \frac{1}{a_x}\oint_{\mathsf{C}_\mathbf{a}}\frac{f(w)}{p_{x+1}(w)}dw, \ \ x\in \mathbb{Z}_+.
\end{align}   
\end{defn}
Clearly, $\mathsf{T}_f$ depends on $\mathbf{a}$ but we supress it from the notation since $\mathbf{a}$ is fixed. Observe that, in the homogeneous case $a_x\equiv 1$,  $[\mathsf{T}_f(x,y)]_{x,y\in \mathbb{Z}_+}$ is nothing but the Toeplitz matrix with symbol $f(1-z)$.  For this reason we shall call $\mathsf{T}_f$ the inhomogeneous Toeplitz-like matrix/operator with symbol\footnote{We prefer this terminology to the somewhat more precise ``... with symbol $f(1-z)$".} $f$. 

We note that the only poles for the integrand in the definition of $\mathsf{T}_f$ come from the inhomogeneity $\mathbf{a}$. In particular, we note that $\mathsf{T}_f(x,y)\equiv 0$, for $x>y$, namely the matrix $[\mathsf{T}_f(x,y)]_{x,y\in \mathbb{Z}_+}$ is upper-triangular. From a probabilistic standpoint, for certain special choices of $f$, $\mathsf{T}_f$ will be the transition probability, from $x$ to $y$, of a Markov chain on $\mathbb{Z}_+$ which moves only to the right.

More general functions $f$, which are not in $\mathsf{Hol}\left(\mathbb{H}_{-R}\right)$, can be used to define $\mathsf{T}_f$ and most of the results below have corresponding analogues. We restrict to $f\in \mathsf{Hol}\left(\mathbb{H}_{-R}\right)$ for simplicity since they suffice for the probabilistic applications we have in mind.

For functions $g:\mathbb{Z}_+ \to \mathbb{C}$ (equivalently sequences in $\mathbb{C}^{\mathbb{Z}_+}$), for which the sum converges, we write $\mathsf{T}_f g(x)=\sum_{y=0}^\infty\mathsf{T}_f(x,y)g(y)$. It is easy to see that the operator $\mathsf{T}_f$ is well-defined on bounded sequences but it is actually defined on a somewhat larger space, see Proposition \ref{OperatorWellDefined}. 

In fact, as we show in Proposition \ref{PropSimilarity}, the matrix $[\mathsf{T}_f(x,y)]_{x,y\in \mathbb{Z}_+}$, with general sequence $\mathbf{a}$, is similar to the standard Toeplitz matrix (namely with $a_x\equiv 1$) with symbol $f(1-z)$ using an explicit change of basis matrix. However, even in the one-dimensional setting of this section it is not possible, using this similarity, to simply translate results from the standard Toeplitz matrix setting to the inhomogeneous one, see Remark \ref{RemarkSimilarity}. In the multidimensional setting of Sections \ref{SectionIntertwining}, \ref{SectionCouplings} and \ref{SectionDynamicsOnArrays} it is unclear whether this matrix similarity is of any use at all but it would be interesting if something can be done using it.

We finally define the adjoint kernel by $\mathsf{T}_f^*(x,y)=\mathsf{T}_f(y,x)$ and corresponding operator $\mathsf{T}_f^*g(x)=\sum_{y=0}^\infty\mathsf{T}_f(y,x)g(y)$. Clearly, by its lower-triangular structure, $\mathsf{T}_f^*$ is well-defined on the whole of $\mathbb{C}^{\mathbb{Z}_+}$. From a probabilistic standpoint $\mathsf{T}_f^*$ evolves measures on $\mathbb{Z}_+$ according to the dynamics of the Markov chain governed by $\mathsf{T}_f$. %We also give a somewhat more abstract description of $\mathsf{T}_f^*$, when $f$ is entire, acting on a smaller sequence space, in Remark \ref{RemarkAbstract}.

The following constant will appear often throughout the paper.

\begin{defn}
Given an inhomogeneity sequence $\mathbf{a}$, define $R(\mathbf{a})$ by
\begin{equation}
R(\mathbf{a})=\sup_{k\in \mathbb{Z}_+}a_k-\inf_{k\in \mathbb{Z}_+}a_k.
\end{equation}    
\end{defn}

We have the following expansion of functions $f$ in terms of the polynomials $p_x(z)$. When $a_x \equiv 1$, $x\in \mathbb{Z}_+$, this is simply the Taylor expansion  of the function $f$ around 1. 

\begin{lem}\label{LemmaExpansion} Let $f \in \mathsf{Hol}\left(\mathbb{H}_{-R}\right)$ and $\mathcal{U}$ be a compact set in $\mathbb{H}_{-R}$ for some $R>0$. Suppose there exists a (counterclockwise) contour $\mathsf{C}_{\mathbf{a}}\subset \mathbb{H}_{-R}$ containing $\{a_x\}_{x\in \mathbb{Z}_+}$ such that
\begin{equation}\label{ContourCondition}
\sup_{k\in \mathbb{Z}_+}\frac{\sup_{u\in \mathcal{U}}|u-a_k|}{\inf_{w\in \mathsf{C}_\mathbf{a}}|w-a_k|}=r<1.
\end{equation}
Then, the following expansion converges absolutely and uniformly for $u \in \mathcal{U}$,
\begin{equation}\label{Expansion}
 f(u)=\sum_{x=0}^\infty p_x(u) \mathsf{T}_f(x).
\end{equation} 
If $f$ is entire then for any compact $\mathcal{U}\subset \mathbb{C}$ a contour $\mathsf{C}_{\mathbf{a}}$  satisfying (\ref{ContourCondition}) exists and so (\ref{Expansion}) converges uniformly on compact sets in $\mathbb{C}$. Finally, suppose $R>R(\mathbf{a})$. Then, for $\epsilon>0$ small enough and $\mathcal{U}=\mathcal{U}_\epsilon$ given by the rectangle
\begin{equation}\label{Rectangle}
 \mathcal{U}=\left\{u\in \mathbb{C}:-\epsilon \le \Re{u} \le \sup_{k\in \mathbb{Z}_+}a_k, -\epsilon \le \Im(u) \le \epsilon\right\}   
\end{equation}
a contour $\mathsf{C}_a$ satisfying (\ref{ContourCondition}) exists.

\end{lem}

\begin{proof}
Observe that,
\begin{align*}
-\frac{1}{a_y}\frac{1}{2\pi \textnormal{i}}\oint_{\mathsf{C}_\mathsf{a}}\frac{p_x(w)}{p_{y+1}(w)}dw=\mathbf{1}_{x=y}.
\end{align*}
Thus, we get that with $f(u)=p_y(u)$, and more generally by linearity for $f(u)$ any polynomial, the expansion (\ref{Expansion}), which is a finite sum, converges uniformly and absolutely for all $u\in \mathbb{C}$.

We now assume that $\mathsf{C}_{\mathbf{a}}$ is picked as in the statement. Recall that $f$ is analytic in $\mathbb{H}_{-R}$. Take $f_N(u)$ to be the truncated degree $N$ Taylor polynomial in the Taylor expansion of $f(u)$ about a point $u_{*}\in \mathbb{R}_+$. This expansion converges in the open disk of radius $R+u_*$ centred at $u_*$. By taking $u_*$ large enough we can get this disk to contain both $\mathcal{U}$ and $\mathsf{C}_\mathbf{a}$ and thus for the expansion to converge on $\mathcal{U}$ and $\mathsf{C}_\mathbf{a}$. Moreover note that, from the observation in the paragraph above we have,
\begin{equation}
 \sum_{x=0}^\infty p_x(u) \mathsf{T}_{f_N}(x)=\sum_{x=0}^\infty -\frac{1}{a_x}\frac{1}{2\pi \textnormal{i}}p_x(u)\oint_{\mathsf{C}_\mathbf{a}}\frac{f_N(w)}{p_{x+1}(w)}dw=f_N(u).
\end{equation}
Let $u\in \mathcal{U}$ be fixed. We then have the following bound, for any $N \in \mathbb{N}$ and  $(x,w)\in \mathbb{Z}_+\times \mathsf{C}_\mathbf{a}$, by using (\ref{ContourCondition}) and the fact that $f_N$ converges uniformly to $f$ in $\mathsf{C}_{\mathbf{a}}$ and thus it is uniformly bounded,
\begin{equation*}
\left|\frac{1}{a_x}\frac{p_x(u)f_N(w)}{p_{x+1}(w)}\right|  \lesssim \prod_{k=0}^x \frac{|u-a_k|}{|w-a_k|}\lesssim r^x.
\end{equation*}
Thus, by the dominated convergence theorem we have, for fixed $u \in \mathcal{U}$,
\begin{equation*}
 \sum_{x=0}^\infty p_x(u)\mathsf{T}_{f_N}(u) \overset{N \to \infty}{\longrightarrow} \sum_{x=0}^\infty p_x(u)\mathsf{T}_{f}(u).
\end{equation*}
On the other hand, we know that $f_N(u) \to f(u)$ uniformly for $u\in \mathcal{U}$ and this proves the expansion for fixed $u\in \mathcal{U}$. We now show that the series converges uniformly and absolutely in $\mathcal{U}$. We can bound for any $x\in \mathbb{Z}_+$  and for all $u\in \mathcal{U}$, using (\ref{ContourCondition}) and the fact that $f$ is uniformly bounded in $\mathcal{U}$,
\begin{align*}
\left|p_x(u)\mathsf{T}_f(x)\right| \lesssim \oint_{\mathsf{C}_\mathbf{a}} \prod_{k=0}^x \frac{|u-a_k|}{|w-a_k|} dw \lesssim r^x.
\end{align*}
By the Weirstrass M-test the desired conclusion follows.

When $f$ is entire then for any compact $\mathcal{U}\subset \mathbb{C}$ we can take $\mathsf{C}_{\mathbf{a}}$ a circle of very large radius centred at the origin and (\ref{ContourCondition}) is seen to hold.  Finally, suppose $R>R(\mathbf{a})$ and $\mathcal{U}=\mathcal{U}_\epsilon$ is as in (\ref{Rectangle}). We claim that taking $\mathsf{C}_{\mathbf{a}}\subset \mathbb{H}_{-R}$ a rectangular contour with sides parallel to the real and imaginary axes with the left side being part of the line $\Re(w)=-R+\epsilon$, the right side part of the line $\Re(w)=M$ and the upper and lower sides parts of the lines $\Im(w)=\pm M$ respectively for some large $M$ works. See Figure \ref{Contour1} for an illustration of this contour. For such $\mathsf{C}_\mathbf{a}$ we have, by taking $\epsilon$ small enough,
\begin{equation*}
 \sup_{k\in \mathbb{Z}_+}\frac{\sup_{u\in \mathcal{U}}|u-a_k|}{\inf_{w\in \mathsf{C}_\mathbf{a}}|w-a_k|}\le\frac{\sup_{k\in \mathbb{Z}_+}a_k+\mathcal{O}(\epsilon)}{\inf_{k\in \mathbb{Z}_+}a_k+R+\mathcal{O}(\epsilon)}<\frac{\sup_{k\in \mathbb{Z}_+}a_k}{\inf_{k\in \mathbb{Z}_+}a_k+R(\mathbf{a})}=1.
\end{equation*}
This concludes the proof.

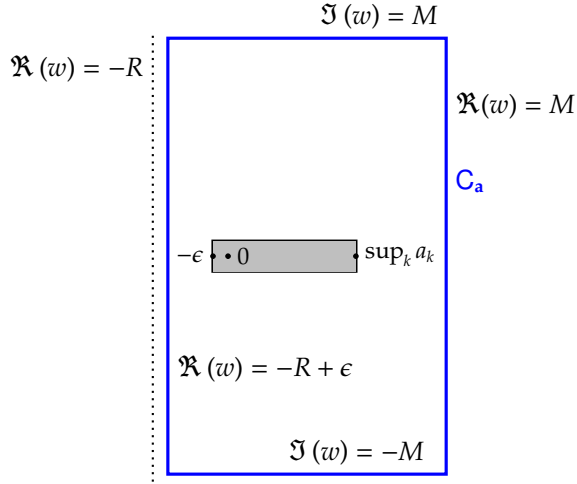
\begin{figure}
\captionsetup{singlelinecheck = false, justification=justified}
\centering
\begin{tikzpicture}
%\draw[dotted] (0,0) grid (2,4);

\draw[dotted, thick] (-1,-3) to (-1,3);

\node[left] at (-1,2.5) {$\Re\left(w\right)=-R$}; 

\node[right] at (-0.8,-1.5) {$\Re\left(w\right)=-R+\epsilon$}; 

\node[above] at (2,2.9) {$\Im\left(w\right)=M$}; 

\node[above] at (1.7,-2.9) {$\Im\left(w\right)=-M$}; 

\node[right] at (2.9,2) {$\Re(w)=M$};

\node[right,blue] at (2.9,1) {$\mathsf{C}_\mathbf{a}$};

\draw[blue, very thick] (-0.8,-2.9) rectangle (2.9,2.9);

\draw[very thick] (-0.2,-0.2) rectangle (1.7,0.2);

\fill[lightgray] (-0.2,-0.2) rectangle (1.7,0.2);

 \draw[fill] (0,0) circle [radius=0.03];

  \draw[fill] (-0.2,0) circle [radius=0.03];

 \draw[fill] (1.7,0) circle [radius=0.03];

 \node[right] at (0,0) {\small $0$};

  \node[left] at (-0.2,0) {\small $-\epsilon$};

  \node[right] at (1.7,0) {\small $\sup_{k} a_k$};

\end{tikzpicture}

\caption{The contour (not to scale) used in the proof of Lemma \ref{LemmaExpansion} associated to the rectangle $\mathcal{U}=\mathcal{U}_{\epsilon}$ from (\ref{Rectangle}).}\label{Contour1}
\end{figure}

\end{proof}

We have the following composition property for the $\mathsf{T}_f$ kernels. In the probabilistic setting this is simply the Chapman-Kolmogorov equation.

\begin{lem}\label{LemmaComposition}
Suppose $f,g \in \mathsf{Hol}\left(\mathbb{H}_{-R}\right)$ with $R>R(\mathbf{a})$. Then, we have 
\begin{equation}
\mathsf{T}_f \mathsf{T}_g=\mathsf{T}_{fg}.
\end{equation}
\end{lem}

\begin{proof}
We compute, by deforming the $u$ contour from $\mathsf{C}_{\mathbf{a}}$ to a contour $\tilde{\mathsf{C}}_{\mathbf{a}}\subset \mathcal{U}$ where $\mathcal{U}$ is the region defined in (\ref{Rectangle}), which can be done without crossing any poles, and making use of Lemma \ref{LemmaExpansion} and then deforming back to $\mathsf{C}_{\mathbf{a}}$,
\begin{align*}
\mathsf{T}_f\mathsf{T}_g(x,y)&=-\frac{1}{a_y}\frac{1}{2\pi \textnormal{i}}\oint_{\tilde{\mathsf{C}}_\mathbf{a}}\frac{g(u)}{p_{y+1}(u)}\sum_{m=0}^\infty p_m(u)\oint_{\mathsf{C}_{\mathbf{a}}}\frac{p_x(w)f(w)}{p_{m+1}(w)}dwdu\\
&=-\frac{1}{a_y}\frac{1}{2\pi \textnormal{i}}\oint_{\mathsf{C}_\mathbf{a}}\frac{p_x(u)}{p_{y+1}(u)} f(u)g(u) du=\mathsf{T}_{fg}(x,y),
\end{align*}
as desired.
\end{proof}

We have the following normalisation result for $\mathsf{T}_f$.

\begin{lem}\label{LemmaNormalisation}
Let $f\in \mathsf{Hol}\left(\mathbb{H}_{-\epsilon}\right)$ for some $\epsilon>0$. For any $x\in \mathbb{Z}_+$ we have
\begin{align*}
\sum_{y=0}^\infty \mathsf{T}_f(x,y)=f(0).
\end{align*}
\end{lem}

\begin{proof} Observe that since $f$ has no poles in $\mathsf{C}_{\mathbf{a}}$, $\mathsf{T}_f(x,y)=0$ for $y\le x$. Then, we deform $\mathsf{C}_{\mathbf{a}}$ to a contour $\tilde{\mathsf{C}}_{\mathbf{a}}\subset{\mathbb{H}_{-\epsilon}}$ such that 
\begin{equation}\label{SomeBound}
\sup_{w\in \tilde{\mathsf{C}}_{\mathbf{a}}}\sup_{k\in \mathbb{Z}_+}\left|\frac{a_k}{a_k-w}\right|=r<1.
\end{equation}
Such a contour is always possible to find. We can take a rectangular contour with sides parallel to the real and imaginary axes with the left side being part of the line $\Re(w)=-\epsilon'$ for some $0<\epsilon'<\epsilon$, the right side part of the line $\Re(w)=M$ and the upper and lower sides parts of the lines $\Im(w)=\pm M$ respectively for some very large $M$. See Figure \ref{Contour2} for an illustration. Then, we have uniformly for $w\in \tilde{\mathsf{C}}_{\mathbf{a}}$,
\begin{equation}
\sum_{y>x}\frac{1}{a_m p_{y+1}(w)}=-\frac{1}{wp_{x+1}(w)}.
\end{equation}
This follows, after some relabelling, by taking the $x \to \infty$ limit in the elementary identity, by virtue of the bound (\ref{SomeBound}) above for $w\in \tilde{\mathsf{C}}_{\mathbf{a}}$,
\begin{equation}\label{SeriesIdentity}
\sum_{k=0}^x\frac{1}{a_k}\prod_{i=0}^k\frac{a_i}{a_i-w}=-\frac{1}{w}\left(1-\prod_{i=0}^x\frac{a_i}{a_i-w}\right).
\end{equation}
Thus we obtain, where we deform to the $\mathsf{C}_{\mathbf{a},0}$ contour without crossing any poles,
\begin{align*}
\sum_{y=0}^\infty \mathsf{T}_f(x,y)=\sum_{y=x}^\infty \mathsf{T}_f(x,y)&=\frac{1}{2\pi \textnormal{i}}\oint_{\mathsf{C}_{\mathbf{a},0}}\left[\frac{1}{wp_{x+1}(w)}-\frac{1}{a_xp_{x+1}(w)}\right]p_x(w)f(w)dw\\
&=\frac{1}{2\pi \textnormal{i}}\oint_{\mathsf{C}_{\mathbf{a},0}}\frac{f(w)}{w}dw.
\end{align*}
Evaluating the integral by picking the residue at $0$ gives the result.

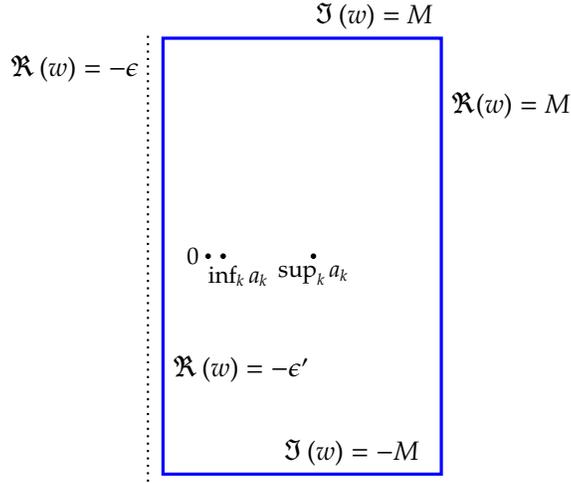
\begin{figure}
\captionsetup{singlelinecheck = false, justification=justified}
\centering
\begin{tikzpicture}
%\draw[dotted] (0,0) grid (2,4);

\draw[dotted, thick] (-1,-3) to (-1,3);

\node[left] at (-1,2.5) {$\Re\left(w\right)=-\epsilon$}; 

\node[right] at (-0.8,-1.5) {$\Re\left(w\right)=-\epsilon'$}; 

\node[above] at (2,2.9) {$\Im\left(w\right)=M$}; 

\node[above] at (1.7,-2.9) {$\Im\left(w\right)=-M$}; 

\node[right] at (2.9,2) {$\Re(w)=M$};

\draw[blue, very thick] (-0.8,-2.9) rectangle (2.9,2.9);

 \draw[fill] (0,0) circle [radius=0.03];

  \draw[fill] (-0.2,0) circle [radius=0.03];

 \draw[fill] (1.2,0) circle [radius=0.03];

 \node[below] at (0.2,0) {\small $\inf_k a_k$};

  \node[left] at (-0.2,0) {\small $0$};

  \node[below] at (1.2,0) {\small $\sup_{k} a_k$};

\end{tikzpicture}

\caption{The contour (not to scale) used in the proof of Lemma \ref{LemmaNormalisation}.}\label{Contour2}
\end{figure}

\end{proof}

For $g:\mathbb{Z}_+ \to \mathbb{C}$ define the forward and backward discrete derivatives $\nabla^+$ and $\nabla^-$ by
\begin{equation*}
\nabla^+g(x)=g(x+1)-g(x), \ \ \nabla^- g(x)=g(x-1)-g(x).
\end{equation*}
We have the following relation for $\mathsf{T}_f$ that we call the duality relation (the terminology is because this relation is somewhat reminiscent to dualities for Markov processes \cite{Duality} and in particular the Siegmund duality \cite{Siegmund}; however it is actually not a case of a Siegmund duality or any Markov duality for that matter).

\begin{lem}\label{LemmaDuality}
 Let $f\in \mathsf{Hol}\left(\mathbb{H}_{-\epsilon}\right)$ for some $\epsilon>0$. We then have
\begin{equation}
-\frac{a_x}{a_y}\nabla_x^+\mathsf{T}_f\mathbf{1}_{\llbracket 0,y \rrbracket}(x)=\mathsf{T}_f(x,y).
\end{equation}
\end{lem}

\begin{proof}
Observe that, from Lemma \ref{LemmaNormalisation} we have
\begin{equation*}
\nabla_x^+\mathsf{T}_f\mathbf{1}_{\llbracket 0,y \rrbracket}(x)=-\nabla_x^+\sum_{m>y}\mathsf{T}_f(x,m).
\end{equation*}
On the other hand, using identity (\ref{SeriesIdentity}), after deforming first to the contour $\tilde{\mathsf{C}}_{\mathbf{a}}$ as in the proof of Lemma \ref{LemmaNormalisation} and then to $\mathsf{C}_{\mathbf{a},0}$, we get that the last display is equal to
\begin{equation*}
-\nabla_x^+ \frac{1}{2\pi \textnormal{i}}\oint_{\mathsf{C}_{\mathbf{a},0}}\frac{p_x(w)f(w)}{wp_{y+1}(w)}dw.
\end{equation*}
Bringing the discrete derivative $\nabla^+_x$ inside the integral, using that
\begin{equation*}
a_x\nabla_x^+\frac{p_x(w)}{w}=-p_x(w)
\end{equation*}
and then deforming back to $\mathsf{C}_{\mathbf{a}}$ gives the result.
\end{proof}

Finally, we note that, for fixed $\lambda$, the function $x\mapsto p_x(\lambda)$ is an eigenfunction of $\mathsf{T}_f$ with explicit eigenvalue.

\begin{lem}\label{LemmaEigenfunction} Let $f \in \mathsf{Hol}\left(\mathbb{H}_{-R}\right)$ with $R>R(\mathbf{a})$ and $\lambda \in \mathcal{U}$ the set defined in (\ref{Rectangle}). Then, the function $h_\lambda(x)=p_x(\lambda)$ is an eigenfunction of $\mathsf{T}_f$ with eigenvalue $f(\lambda)$,
\begin{equation}
\mathsf{T}_fh_\lambda(x)=f(\lambda)h_\lambda(x), \ \ \forall x \in \mathbb{Z}_+.
\end{equation}
If $f$ is entire then the above holds for all $\lambda \in \mathbb{C}$.
\end{lem}

\begin{proof}
This is an application of Lemma \ref{LemmaExpansion} with the function $f(z)h_z(x)=f(z)p_x(z)$.
\end{proof}

In the rest of this subsection we collect some general results and comments about the operator/matrix $\mathsf{T}_f$. These will not be used in any of our probabilistic applications but they are interesting in their own right so we present them here. Define the following spaces $\ell_{\textnormal{exp},M}\left(\mathbb{Z}_+\right)$, where $M \in (1,\infty)$, and $\ell_{\textnormal{exp}}\left(\mathbb{Z}_+\right)$ by,
\begin{align*}
\ell_{\textnormal{exp},M}\left(\mathbb{Z}_+\right)&=\left\{\left(g(k)\right)_{y\in\mathbb{Z}_+}\in \mathbb{C}^{\mathbb{Z}_+}:\exists \  C_M<\infty \textnormal{ such that for all } y\in \mathbb{Z}_+, |g(y)|\le C_M M^y \right\},\\
\ell_{\textnormal{exp}}\left(\mathbb{Z}_+\right)&=\left\{\left(g(k)\right)_{y\in\mathbb{Z}_+}\in \mathbb{C}^{\mathbb{Z}_+}:\exists \  M \in (0,\infty) \textnormal{ such that } \left(g(k)\right)_{y\in\mathbb{Z}_+}\in \ell_{\textnormal{exp},M}\left(\mathbb{Z}_+\right)\right\}.
\end{align*}

\begin{prop}\label{OperatorWellDefined}
Suppose $f\in \mathsf{Hol}\left(\mathbb{H}_{-\epsilon}\right)$ for some $\epsilon$. Then, there exists $r<1$ and $C_{f,r}<\infty$ such that for all $x,y\in \mathbb{Z}_+$ we have
\begin{equation*}
\left|\mathsf{T}_f\left(x,y\right)\right|\le C_{f,r} r^{|y-x|}.
\end{equation*}
In particular, $\mathsf{T}_f$ is well-defined acting on $\ell_{\textnormal{exp},M}\left(\mathbb{Z}_+\right)$ for $M<r^{-1}$. If $f$ is entire then for any $0<\delta<1$ there exists $C_{f,\delta}<\infty$ such that for all $x,y\in \mathbb{Z}_+$,
\begin{equation*}
\left|\mathsf{T}_f\left(x,y\right)\right|\le C_{f,\delta} \delta^{|y-x|},
\end{equation*}
and so in this case $\mathsf{T}_f$ is well-defined acting on $\ell_{\textnormal{exp}}\left(\mathbb{Z}_+\right)$.
\end{prop}

\begin{proof}
For the first statement we can deform the $\mathsf{C}_\mathbf{a}$ contour to the rectangle from the proof of Lemma \ref{LemmaNormalisation}, see also Figure \ref{Contour2}. By taking the absolute value inside the integral we obtain the desired conclusion. When $f$ is entire we can deform $\mathsf{C}_\mathbf{a}$ to a very large circle from which, after bringing the absolute value inside the integral, the result follows.
\end{proof}

\begin{rmk}\label{RemarkAbstract}
We give a more abstract interpretation of $\mathsf{T}_f^{*}$ viewed as an operator from a certain sequence space to itself when $f$ is entire. Define the space $\ell_{\textnormal{hol}}$ by,
\begin{equation*}
\ell_{\textnormal{hol}}=\left\{\left(g(k)\right)_{y\in\mathbb{Z}_+}\in \mathbb{C}^{\mathbb{Z}_+}: \forall \  0<r< 1, \  \exists \  C_r<\infty \textnormal{ such that for all } y\in \mathbb{Z}_+, |g(y)|\le C_r r^y \right\}.
\end{equation*}
Observe that, we have the inclusions, with $M,q \ge 1$, where $\ell^{q}$ are the usual sequence spaces,
\begin{equation*}
\ell_{\textnormal{hol}}\left(\mathbb{Z}_+\right)\subset \ell^q\left(\mathbb{Z}_+\right)
\subset \ell_{\textnormal{exp},M}\left(\mathbb{Z}_+\right)
\subset \ell_{\textnormal{exp}}\left(\mathbb{Z}_+\right).
\end{equation*}
Now, we consider the map $\mathcal{V}$ given by, where $\mathsf{Hol}\left(\mathbb{C}\right)$ denotes the space of entire functions,
\begin{align*}
\mathcal{V}:\ell_{\textnormal{hol}}\left(\mathbb{Z}_+\right) &\to \mathsf{Hol}\left(\mathbb{C}\right),\\
\left(g(y)\right)_{y\in \mathbb{Z}_+}&\mapsto h(z)=\sum_{y\in \mathbb{Z}_+}g(y) p_y(z),
\end{align*}
which by the preceding results of this section it is well-defined and moreover it is in fact a bijection with inverse given by,
\begin{align*}
\mathcal{V}^{-1}:\mathsf{Hol}\left(\mathbb{C}\right) &\to \ell_{\textnormal{hol}}\left(\mathbb{Z}_+\right),\\
h &\mapsto \left(-\frac{1}{a_y}\frac{1}{2\pi \textnormal{i}}\oint_{\mathsf{C}_\mathbf{a}} \frac{h(z)}{p_{y+1}(z)}dz\right)_{y\in \mathbb{Z}_+}=\left(\mathsf{T}_h(y)\right)_{y\in \mathbb{Z}_+}.
\end{align*}
The fact that this sequence has the correct decay properties again follows by deforming $\mathsf{C}_\mathbf{a}$ to a very large circle. Given $f\in \mathsf{Hol}\left(\mathbb{C}\right)$ define the multiplication operator $\mathsf{Mult}_f:\mathsf{Hol}\left(\mathbb{C}\right)\to \mathsf{Hol}\left(\mathbb{C}\right)$ by $\mathsf{Mult}_f\left(h\right)=fh$. Then, we can see that we have the following representation of $\mathsf{T}_f^*:\ell_{\textnormal{hol}}\left(\mathbb{Z}_+\right) \to \ell_{\textnormal{hol}}\left(\mathbb{Z}_+\right)$:
\begin{equation}\label{SimilarityAdjoingOperator}
\mathsf{T}_f^*=\mathcal{V}^{-1}\mathsf{Mult}_f\mathcal{V}.
\end{equation}
\end{rmk}

We now prove that the matrix $[\mathsf{T}_f(x,y)]_{x,y\in \mathbb{Z}_+}$, with general sequence $\mathbf{a}$, is similar to the standard Toeplitz matrix (namely with $a_x\equiv 1$) with symbol $f(1-z)$ using an explicit change of basis matrix. We restrict to entire $f$ for simplicity. The result below can also be derived from (\ref{SimilarityAdjoingOperator}) but we give a direct computational proof which is instructive.

\begin{prop}\label{PropSimilarity}
Let $f$ be entire. Let $\mathsf{T}_f(x,y)$ be as in (\ref{InhomogeneousToeplitzDisplay}) with general and fixed $\mathbf{a}$ and denote by $\tilde{\mathsf{T}}_f(x,y)$ the homogeneous case (standard Toeplitz case) of (\ref{InhomogeneousToeplitzDisplay}):
\begin{equation*}
\tilde{\mathsf{T}}_f(x,y)=\tilde{\mathsf{T}}_f(y-x)=-\frac{1}{2\pi \textnormal{i}} \oint_{|1-z|=1} \frac{f(z)(1-z)^x}{(1-z)^{y+1}}dz=\frac{1}{2\pi \textnormal{i}} \oint_{|z|=1}f(1-z)z^{x-y-1}dz.
\end{equation*}
Moreover, consider the matrix $\mathbf{A}(\mathbf{a})$ with entries given by, with $k,m \in \mathbb{Z}_+$,
\begin{equation}\label{ADef}
\mathbf{A}_{km}(\mathbf{a})=(-1)^{k-m}\left(\prod_{l=0}^{k-1}a_l^{-1} \right)e_{k-m}\left(1-a_0,1-a_1,\dots,1-a_{k-1}\right)
\end{equation}
where $e_l$ is the $l$-th elementary symmetric polynomial:
\begin{equation*}
e_l(z_1,z_2,\dots,z_N)= \sum_{1\le j_1< j_2 < \cdots < j_l \le N} z_{j_1} z_{j_2}\cdots z_{j_l}.
\end{equation*}
Then, for any $x,y \in \mathbb{Z}_+$, we have
\begin{equation}
\mathsf{T}_f\left(x,y\right)=\left[\mathbf{A}(\mathbf{a})\tilde{\mathsf{T}}_f\mathbf{A}^{-1}(\mathbf{a})\right](x,y).
\end{equation}
\end{prop}

\begin{proof}
Observe that, both $\left\{(1-z)^x\right\}_{x\in \mathbb{Z}_+}$ and $\left\{p_x(z)\right\}_{x\in \mathbb{Z}_+}$ are bases in the ring of polynomials $\mathbb{C}[z]$. By virtue of Vieta's formulae and (\ref{ADef}) we can see that $\mathbf{A}(\mathbf{a})$ is actually the change of basis matrix, and in particular also invertible,
\begin{align}
    p_k(z)&=\sum_{m\in \mathbb{Z}_+}\mathbf{A}_{km}(\mathbf{a})(1-z)^m,\label{ChangeofBasis1}\\
    (1-z)^k&=\sum_{m \in \mathbb{Z}_+}\mathbf{A}^{-1}_{km}(\mathbf{a})p_m(z).\label{ChangeofBasis2}
\end{align}
Then, using (\ref{ChangeofBasis1}), we can write
\begin{equation*}
\mathsf{T}_f(x,y)=-\frac{1}{2\pi \textnormal{i}}\frac{1}{a_y} \oint_{\mathsf{C}_{\mathbf{a}}}\frac{f(z)p_x(z)}{p_{y+1}(z)}dz=\sum_{m \in \mathbb{Z}_+}\mathbf{A}_{xm}(\mathbf{a})\left(-\frac{1}{2\pi \textnormal{i}}\frac{1}{a_y}\oint_{\mathsf{C}_\mathbf{a}} \frac{f(z)(1-z)^m}{p_{y+1}(z)}dz\right).
\end{equation*}
By virtue of Lemma \ref{LemmaExpansion}, since $f$ is entire, expanding $f(z)(1-z)^m$ as a series in two ways
\begin{align*}
   f(z)(1-z)^m&=\sum_{k\in \mathbb{Z}_+}\left(-\frac{1}{2\pi \textnormal{i}} \oint_{|1-z|=1} \frac{f(z)(1-z)^m}{(1-z)^{k+1}}dz\right) (1-z)^k=\sum_{k\in \mathbb{Z}_+}\tilde{\mathsf{T}}_f(m,k)(1-z)^k,\\
   f(z)(1-z)^m&=\sum_{k\in \mathbb{Z}_+}\left(-\frac{1}{2\pi \textnormal{i}}\frac{1}{a_k}\oint_{\mathsf{C}_\mathbf{a}} \frac{f(z)(1-z)^m}{p_{k+1}(z)}dz\right)p_k(z),
\end{align*}
using (\ref{ChangeofBasis2}) and comparing coefficients for $p_y(z)$ we get
\begin{equation*}
-\frac{1}{2\pi \textnormal{i}}\frac{1}{a_y}\oint_{\mathsf{C}_\mathbf{a}}\frac{f(z)(1-z)^m}{p_{y+1}(z)}dz=\sum_{l\in \mathbb{Z}_+} \tilde{\mathsf{T}}_f(m,l)\mathbf{A}_{ly}^{-1}(\mathbf{a})
\end{equation*}
and this completes the proof.
\end{proof}

\begin{rmk}\label{RemarkSimilarity}
Using Proposition \ref{PropSimilarity} and assuming the corresponding result for the standard Toeplitz setting, namely $a_x\equiv 1$, we can give quick alternative proofs of Lemmas \ref{LemmaComposition} and \ref{LemmaEigenfunction} for general $\mathbf{a}$ (and entire $f$). It is also possible to prove Lemma \ref{LemmaNormalisation} for general $\mathbf{a}$ in the same way if we observe that $\sum_{m=0}^{\infty}\mathbf{A}_{xm}^{-1}(\mathbf{a})=1$, for all $x\in \mathbb{Z}_+$. However, the important (later on) duality formula from Lemma \ref{LemmaDuality} does not seem to follow along these lines.
\end{rmk}

\subsection{One-dimensional dynamics and Markov transition kernels}

The following choices for the function $f$ in $\mathsf{T}_f$ in Lemmas \ref{LemmaBernoulliTrans}, \ref{LemmaGeometricTrans}, \ref{LemmaPureBirthTrans} below are the basic building blocks for the models we study. They correspond to transition probabilities for an inhomogeneous space Bernoulli walk, geometric walk and continuous time pure-birth chain.

\begin{lem}\label{LemmaBernoulliTrans}
Let $f(z)=1-\alpha z$. We have, with $x,y\in \mathbb{Z}_+$,
\begin{equation}
\mathsf{T}_{f}(x,y)=\alpha a_x \mathbf{1}_{y=x+1}+(1-\alpha a_x)\mathbf{1}_{y=x}.
\end{equation}
\end{lem}
\begin{proof}
Immediate evaluation of the contour integral using the residue formula.
\end{proof}

Note that, for the above expression to be positive and thus have probabilistic meaning we need $0\le \alpha\le \left(\sup_k a_k\right)^{-1}$.

\begin{lem}\label{LemmaGeometricTrans}
Let $f(z)=(1+\beta z)^{-1}$. We have, with $x,y \in \mathbb{Z}_+$,
\begin{equation}\label{GeomDensity}
\mathsf{T}_{f}(x,y)=\frac{1}{1+\beta a_y} \prod_{k=x}^{y-1}\frac{\beta a_k}{1+\beta a_k} \mathbf{1}_{y\ge x}.
\end{equation}
\end{lem}
\begin{proof}
This is again a direct derivation but a little less trivial so we give the details. For the computation we assume that all the $a_k$'s are distinct and then remove this restriction by continuity. After evaluating the contour integral in terms of residues (with all the $a_k$'s distinct) and some relabelling, in order to prove (\ref{GeomDensity}), we are required to show that
\begin{equation*}
\sum_{k=0}^n \prod_{j\neq k} \frac{1}{a_j-a_k}\frac{1}{1+\beta a_k}=\beta^n \prod_{k=0}^n\frac{1}{1+\beta a_k}.
\end{equation*}
Clearing the denominators we need to show
\begin{equation*}
\sum_{k=0}^n(-1)^k\prod_{m>l, l\neq k} (a_m-a_l) \prod_{l \neq k} (1+\beta a_l)=\beta^n \prod_{m>l} (a_m-a_l).
\end{equation*}
Now, observe that the left hand side can be written as the determinant
\begin{equation*}
\det \begin{bmatrix}
\prod_{l\neq 0} (1+\beta a_l) & 1 & a_0 &\cdots & a_0^{n-2}\\
\vdots & \vdots & \vdots &\ddots  & \vdots \\
 \prod_{l\neq n} (1+\beta a_l) & 1 & a_n & \cdots & a_n^{n-2}
\end{bmatrix}.
\end{equation*}
Replace $a_n$ by a variable $w$ and consider the following polynomial of degree $n-1$ in $w$
\begin{align*}
w\mapsto \det \begin{bmatrix}
\prod_{l\neq 0}^{n-1} (1+\beta a_l)(1+\beta w) & 1 & a_0 &\cdots & a_0^{n-2}\\
\vdots & \vdots & \vdots &\ddots  & \vdots \\
 \prod_{l=0}^{n-1} (1+\beta a_l) & 1 & w & \cdots & w^{n-2}
\end{bmatrix}.
\end{align*}
It has roots at $w=a_0, a_1, \dots, a_{n-1}$. So it is equal to 
\begin{equation*}
C_n^\beta\prod_{i=0}^{n-1}(w-a_i) 
\end{equation*}
where $C_n^{\beta}$ is the coefficient of the $w^{n-1}$ term. This term is obtained from 
\begin{align*}
w^{n-2}\det \begin{bmatrix}
\prod_{l\neq 0} (1+\beta a_l)(1+\beta w) & 1 & a_0 &\cdots & a_0^{n-2}\\
\vdots & \vdots & \vdots &\ddots  & \vdots \\
 \prod_{l\neq n-1} (1+\beta a_l)(1+\beta w) & 1 & w & \cdots & w^{n-2}
\end{bmatrix}\\
=w^{n-2}(1+\beta w)\det  \begin{bmatrix}
\prod_{l\neq 0} (1+\beta a_l) & 1 & a_0 &\cdots & a_0^{n-3}\\
\vdots & \vdots & \vdots &\ddots  & \vdots \\
 \prod_{l\neq n-1} (1+\beta a_l) & 1 & a_{n-1} & \cdots & a_{n-1}^{n-3}
 \end{bmatrix}.
\end{align*}
Hence, we get 
\begin{equation*}
C_n^{\beta}=\beta\det \begin{bmatrix}
\prod_{l\neq 0} (1+\beta a_l) & 1 & a_0 &\cdots & a_0^{n-3}\\
\vdots & \vdots & \vdots &\ddots  & \vdots \\
 \prod_{l\neq n-1} (1+\beta a_l) & 1 & a_{n-1} & \cdots & a_{n-1}^{n-3}
\end{bmatrix}
\end{equation*}
and the claim follows by induction.
\end{proof}

Again, for the above expression for $\mathsf{T}_f(x,y)$ to be positive we need $\beta \ge 0$.

\begin{lem}\label{LemmaPureBirthTrans}
The probability that a pure-birth chain having jump rate $a_k$ at location $k\in \mathbb{Z}_+$ goes from $x\in \mathbb{Z}_+$ to $y\in \mathbb{Z}_+$ in time $t\in \mathbb{R}_+$ is given by $\mathsf{T}_f(x,y)$ with $f(z)=e^{-tz}$.
\end{lem}
\begin{proof}
This is well-known. One easily shows that $\mathsf{T}_{e^{-tz}}(x,y)$ solves the corresponding Kolmogorov equation. A derivation can be found, for example, in \cite{WangWaugh}.
\end{proof}

\section{An intertwining}\label{SectionIntertwining}

\subsection{Intertwining for non-intersecting paths}

In this section, in a sense that will be clearer in the sequel, we prove that the dynamics on single levels of the interlacing arrays are consistent from $N$ to $N+1$. We need some definitions and notation.

\begin{defn}\label{DefKMsemigroup}
 Given $f\in \mathsf{Hol}\left(\mathbb{H}_{-\epsilon}\right)$ for some $\epsilon>0$ define the following kernel $\mathsf{P}^{(N)}_f$ on $\mathbb{W}_N$ by 
 \begin{equation}
\mathsf{P}_f^{(N)}(\mathbf{x},\mathbf{y})=\det \left(\mathsf{T}_f(x_i,y_j)\right)_{i,j=1}^N.
 \end{equation}
\end{defn}

\begin{defn}\label{PreMarkovKernelDef}
We define the following non-negative kernel $\Lambda_{N+1,N}$ from $\mathbb{W}_{N+1}$ to $\mathbb{W}_N$ by 
\begin{equation}\label{PreMarkovKernelDisplay}
\Lambda_{N+1,N}(\mathbf{y},\mathbf{x})=\prod_{i=1}^N\frac{1}{a_{x_i}}\mathbf{1}_{\mathbf{x}\prec \mathbf{y}}, \ \ \mathbf{x}\in \mathbb{W}_N,  \ \mathbf{y}\in \mathbb{W}_{N+1}.
\end{equation}
\end{defn}

We have the following alternative description of $\Lambda_{N+1,N}$ using the well-known fact that $\mathbf{1}_{\mathbf{x}\prec \mathbf{y}}$ can be written as a determinant with indicator function entries \cite{Warren,BorodinFerrariPrahoferSasamoto}. To do this, we will use the standard notational device in the setting of extending the set $\mathbb{Z}_+$ by an extra symbol that we denote by $\mathsf{virt}$.

\begin{lem}\label{LemmaPhiRep}
Let $N\ge 1$. Define the function $\phi: \left(\mathbb{Z}_+\cup \{\mathsf{virt}\}\right)\times \mathbb{Z}_+$ by, with $y\in \mathbb{Z}_+$,
\begin{align*}
\phi(x,y)=\begin{cases} -a_x^{-1}\mathbf{1}_{y>x}, \ \ &x\in \mathbb{Z}_+\\
1, \ \ &x=\mathsf{virt}.
\end{cases}
\end{align*}
Then, with $x_{N+1}=\mathsf{virt}$ we have
\begin{equation}
\Lambda_{N+1,N}(\mathbf{y},\mathbf{x})=\det\left(\phi(x_i,y_j)\right)_{i,j=1}^{N+1}.
\end{equation}
\end{lem}

We observe the composition property.

\begin{prop}\label{PropKMcomposition}
Let $f,g \in \mathsf{Hol}\left(\mathbb{H}_{-R}\right)$ with $R>R(\mathbf{a})$. Then, we have 
\begin{equation}
\mathsf{P}_f^{(N)}\mathsf{P}_g^{(N)}=\mathsf{P}_{fg}^{(N)}. 
\end{equation}
\end{prop}
\begin{proof}
This is a direct application of the Cauchy-Binet formula and Lemma \ref{LemmaComposition}.
\end{proof}

For certain choices of functions $f$ we obtain that $\mathsf{P}_f^{(N)}$ has a probabilistic interpretation, for any $N\ge 1$, in terms of non-intersecting paths. Of course, for $N=1$ this is already a consequence of Lemmas \ref{LemmaBernoulliTrans}, \ref{LemmaGeometricTrans} and \ref{LemmaPureBirthTrans}.

\begin{prop}\label{LGV-KMprop}
Let $f(z)$ be a (possibly infinite) product of factors of the form $1-\alpha_iz, (1+\beta_i z)^{-1}, e^{-tz}$, where $0\le \alpha_i \le \left(\sup_k a_k\right)^{-1}$, $0\le \beta_i<\infty$ and $t\ge 0$, such that $f \in \mathsf{Hol}\left(\mathbb{H}_{-R}\right)$. Then, $\mathsf{P}^{(N)}_f$ is non-negative.
\end{prop}

\begin{proof}
If $f(z)=e^{-tz}$, then $\mathsf{P}^{(N)}_f$, by virtue of Lemma \ref{LemmaPureBirthTrans}, corresponds to the Karlin-McGregor \cite{KarlinMcGregor} transition kernel of independent pure-birth chains killed when they intersect. If instead $f(z)$ is a finite product of factors of the form $1-\alpha_iz, (1+\beta_i z)^{-1}$ satisfying $0\le \alpha_i \le \left(\sup_k a_k\right)^{-1}$, $0\le \beta_i<\infty$, then $\mathsf{P}^{(N)}_f$, by virtue of Lemma \ref{LemmaBernoulliTrans} and Lemma \ref{LemmaGeometricTrans} is given by the Lindstrom-Gessel-Viennot (LGV) formula for non-intersecting walks on a weighted directed acyclic graph \cite{LGV}, see Figure \ref{LGVgraphsFigure} for an illustration. It is the transition probability of independent walks, taking either Bernoulli or geometric steps, in discrete-time, killed when they intersect. Finally, the general product case follows from Proposition \ref{PropKMcomposition} and a limit in the number of factors.
\end{proof}

\begin{figure}
\captionsetup{singlelinecheck = false, justification=justified}
\centering
\begin{tikzpicture}
%\draw[dotted] (0,0) grid (2,4);

\node[left] at (0,0) {$x$};

\node[left] at (0,1) {$x+1$};

\node[left] at (0,2) {$x+2$};

\node[left] at (0,3) {$x+3$};

\node[left] at (0,4) {$x+4$};

\draw[fill] (0,0) circle [radius=0.05];

\draw[fill] (0,1) circle [radius=0.05];

\draw[fill] (0,2) circle [radius=0.05];

\draw[fill] (0,3) circle [radius=0.05];

\draw[fill] (0,4) circle [radius=0.05];

\node[right] at (2,0) {$x$};

\node[right] at (2,1) {$x+1$};

\node[right] at (2,2) {$x+2$};

\node[right] at (2,3) {$x+3$};

\node[right] at (2,4) {$x+4$};

\draw[fill] (0,0) circle [radius=0.05];

\draw[fill] (0,1) circle [radius=0.05];

\draw[fill] (0,2) circle [radius=0.05];

\draw[fill] (0,3) circle [radius=0.05];

\draw[fill] (0,4) circle [radius=0.05];

\draw[very thick,middlearrow={>}] (0,0) -- (2,0);

\draw[very thick,middlearrow={>}] (0,0) -- (2,1);

\draw[very thick,middlearrow={>}] (0,1) -- (2,1);

\draw[very thick,middlearrow={>}] (0,1) -- (2,2);

\draw[very thick,middlearrow={>}] (0,2) -- (2,2);

\draw[ very thick,middlearrow={>}] (0,2) -- (2,3);

\draw[very thick,middlearrow={>}] (0,3) -- (2,3);

\draw[very thick,middlearrow={>}] (0,3) -- (2,4);

\draw[very thick,middlearrow={>}] (0,4) -- (2,4);

\node[above right] at (1,0) {$1-\alpha a_x$};

\node[above ] at (0.8,0.5) {$\alpha a_x$};

\node[above right] at (1,1) {$1-\alpha a_{x+1}$};

\node[above ] at (0.8,1.5) {$\alpha a_{x+1}$};

\node[above right] at (1,2) {$1-\alpha a_{x+2}$};

\node[above ] at (0.8,2.5) {$\alpha a_{x+2}$};

\node[above right] at (1,3) {$1-\alpha a_{x+3}$};

\node[above ] at (0.8,3.5) {$\alpha a_{x+3}$};

\node[above right] at (1,4) {$1-\alpha a_{x+4}$};

\draw[fill] (2,0) circle [radius=0.05];

\draw[fill] (2,1) circle [radius=0.05];

\draw[fill] (2,2) circle [radius=0.05];

\draw[fill] (2,3) circle [radius=0.05];

\draw[fill] (2,4) circle [radius=0.05];

\end{tikzpicture}\ \ \ \ \ \
\begin{tikzpicture}
\node[left] at (0,0) {$x$};

\node[left] at (0,1) {$x+1$};

\node[left] at (0,2) {$x+2$};

\node[left] at (0,3) {$x+3$};

\node[left] at (0,4) {$x+4$};

\draw[fill] (0,0) circle [radius=0.05];

\draw[fill] (0,1) circle [radius=0.05];

\draw[fill] (0,2) circle [radius=0.05];

\draw[fill] (0,3) circle [radius=0.05];

\draw[fill] (0,4) circle [radius=0.05];

\node[right] at (2,0) {$x$};

\node[right] at (2,1) {$x+1$};

\node[right] at (2,2) {$x+2$};

\node[right] at (2,3) {$x+3$};

\node[right] at (2,4) {$x+4$};

\draw[fill] (0,0) circle [radius=0.05];

\draw[fill] (0,1) circle [radius=0.05];

\draw[fill] (0,2) circle [radius=0.05];

\draw[fill] (0,3) circle [radius=0.05];

\draw[fill] (0,4) circle [radius=0.05];

\draw[very thick,middlearrow={>}] (0,0) -- (2,0);

\draw[very thick,middlearrow={>}] (0,0) -- (0,1);

\draw[very thick,middlearrow={>}] (0,1) -- (2,1);

\draw[very thick,middlearrow={>}] (0,1) -- (0,2);

\draw[very thick,middlearrow={>}] (0,2) -- (2,2);

\draw[very thick,middlearrow={>}] (0,2) -- (0,3);

\draw[very thick,middlearrow={>}] (0,3) -- (2,3);

\draw[very thick,middlearrow={>}] (0,3) -- (0,4);

\draw[very thick,middlearrow={>}] (0,4) -- (2,4);

\node[above] at (1,0) {$(1+\beta a_x)^{-1}$};

\node[left] at (0,0.5) {$\frac{\beta a_x}{1+\beta a_x}$};

\node[above] at (1,1) {$(1+\beta a_{x+1})^{-1}$};

\node[left] at (0,1.5) {$\frac{\beta a_{x+1}}{1+\beta a_{x+1}}$};

\node[above] at (1,2) {$(1+\beta a_{x+2})^{-1}$};

\node[left] at (0,2.5) {$\frac{\beta a_{x+2}}{1+\beta a_{x+2}}$};

\node[above] at (1,3) {$(1+\beta a_{x+3})^{-1}$};

\node[left] at (0,3.5) {$\frac{\beta a_{x+3}}{1+\beta a_{x+3}}$};

\node[above] at (1,4) {$(1+\beta a_{x+4})^{-1}$};

\draw[fill] (2,0) circle [radius=0.05];

\draw[fill] (2,1) circle [radius=0.05];

\draw[fill] (2,2) circle [radius=0.05];

\draw[fill] (2,3) circle [radius=0.05];

\draw[fill] (2,4) circle [radius=0.05];

\end{tikzpicture}

\caption{An illustration of part of the LGV graphs corresponding to a single Bernoulli step with parameter $\alpha$, namely corresponding to the function $f(z)=1-\alpha z$ and a single geometric step with parameter $\beta$, corresponding to the function $f(z)=(1+\beta z)^{-1}$. We can join such elementary graphs one after the other.}\label{LGVgraphsFigure}
\end{figure}
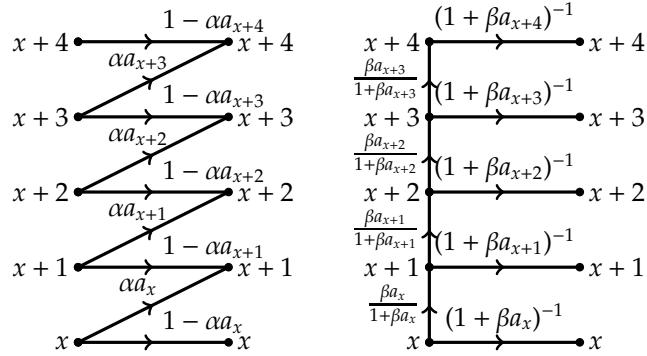

The following intertwining relation between the transition kernels  (although note that at this stage positivity is not yet required) $\mathsf{P}_f^{(N)}$ and $\mathsf{P}_f^{(N+1)}$ is at the heart of many of our results. We give a direct proof but different arguments for a proof will also be presented in the sequel.

\begin{thm}\label{IntertwiningSingleLevel} Let $N\ge 1$. Let $f\in \mathsf{Hol}\left(\mathbb{H}_{-\epsilon}\right)$ for some $\epsilon>0$, with $f(0)=1$. Then, we have the intertwining 
\begin{equation}
\mathsf{P}^{(N+1)}_{f}\Lambda_{N+1,N}=\Lambda_{N+1,N}\mathsf{P}^{(N)}_{f}.
\end{equation}    
\end{thm}

\begin{proof}

We first compute the left hand side using Cauchy-Binet formula to obtain, where we use Lemma \ref{LemmaNormalisation} and the fact that $f(0)=1$,
\begin{align*}
\mathsf{P}^{(N+1)}_{f}\Lambda_{N+1,N}(\mathbf{y},\mathbf{x})&=\sum_{\mathbf{z}\in \mathbb{W}_{N+1}} \det \left(\mathsf{T}_f(y_i,z_j)\right)_{i,j=1}^{N+1} \det \left(\phi(x_i,z_j)\right)_{i,j=1}^{N+1}\\
&=\det\begin{bmatrix}
\mathsf{T}_f\phi^*(y_1,x_1) & \cdots & \mathsf{T}_f\phi^*(y_1,x_{N}) & 1\\
\vdots & \vdots & \vdots &\vdots \\
 \mathsf{T}_f\phi^*(y_{N+1},x_1) & \cdots & \mathsf{T}_f\phi^*(y_{N+1},x_N)  & 1
\end{bmatrix}.
 \end{align*}
 Here, and below, $\phi^*(x,y)=\phi(y,x)$ and $\phi^* \mathsf{T}_f$ and $\mathsf{T}_f\phi^*$ denotes convolution (not multiplication). We now compute the right hand side, by expanding along the last column and using Cauchy-Binet again
\begin{align*}
 \Lambda_{N+1,N}\mathsf{P}^{(N)}_{f}(\mathbf{y},\mathbf{x})&=\sum_{\mathbf{z}\in \mathbb{W}_{N}} \det \left(\phi(z_i,y_j)\right)_{i,j=1}^{N+1} \det \left(\mathsf{T}_f(z_i,x_j)\right)_{i,j=1}^{N}\\
 &=\sum_{l=0}^{N+1} (-1)^{N+1-l}\det\left(\phi^* \mathsf{T}_f(y_i,x_j)\right)_{i=1,\dots,N+1,i\neq l;j=1,\dots,N}\\
&=\det\begin{bmatrix}
\phi^*\mathsf{T}_f(y_1,x_1) & \cdots & \phi^*\mathsf{T}_f(y_1,x_{N}) & 1\\
\vdots & \vdots & \vdots &\vdots \\
 \phi^*\mathsf{T}_f(y_{N+1},x_1) & \cdots & \phi^*\mathsf{T}_f(y_{N+1},x_N)  & 1
\end{bmatrix}.
 \end{align*}
We note that from Lemma \ref{LemmaDuality} we have 
\begin{equation*}
a_y^{-1}\mathsf{T}_f(y,x)=a_x^{-1}\nabla_y^+ \sum_{z>x} \mathsf{T}_f(y,z)
\end{equation*}
and thus by summing over $y$ we obtain
\begin{align}\label{InterInterRelation}
 \phi^*\mathsf{T}_f(y,x)=\mathsf{T}_f\phi^*(y,x)-\mathsf{T}_f\phi^*(0,x).
\end{align}
Now, using the identity above and column operations we obtain the desired equality.
\end{proof}

We now go on to obtain an analogous intertwining for the normalised versions of $\mathsf{P}_f^{(N)}$ and $\Lambda_{N+1,N}$ to be defined shortly. We need the following definition.

\begin{defn}\label{RecursiveDefh_n}
For any $N\ge 1$ we define the following strictly positive function $\mathfrak{h}_N(\mathbf{x})=\mathfrak{h}_N(\mathbf{x};\mathbf{a})$ for $\mathbf{x}\in \mathbb{W}_N$ recursively by $\mathfrak{h}_1(x)=1$ and 
\begin{equation*}
\mathfrak{h}_{N+1}(\mathbf{x})=\Lambda_{N+1,N}\mathfrak{h}_N(\mathbf{x}).
\end{equation*}
\end{defn}
\begin{rmk}
  Observe that, by comparing the definitions of the two functions we get  that $\mathfrak{h}_N(\mathbf{x})=\textnormal{dim}_N(\mathbf{x})$ from Section \ref{SectionGraphIntro} on the inhomogeneous Gelfand-Tsetlin graph.  
\end{rmk}
From Theorem \ref{IntertwiningSingleLevel} we immediately obtain the following.

\begin{prop}\label{EigenfunctionKM}
Let $N\ge 1$.  Let $f\in \mathsf{Hol}\left(\mathbb{H}_{-\epsilon}\right)$ for some $\epsilon>0$, with $f(0)=1$. Then
$\mathsf{P}_f^{(N)}\mathfrak{h}_N=\mathfrak{h}_N$.
\end{prop}

Thus, by a Doob $h$-transform, see \cite{Doob,RevuzYor}, by $\mathfrak{h}_N$ we can define the following kernels $\mathfrak{L}_{N+1,N}$ and $\mathfrak{P}_f^{(N)}$ from $\mathbb{W}_{N+1}$ to $\mathbb{W}_N$ and from $\mathbb{W}_N$ to itself respectively. By construction, $\mathfrak{L}_{N+1,N}$ is Markov. Similarly, as long as $f$ is so that $\mathsf{P}_f^{(N)}$ is non-negative, then $\mathfrak{P}_f^{(N)}$ is also Markov.

\begin{defn}
 Let $N\ge 1$.  Let $f\in \mathsf{Hol}\left(\mathbb{H}_{-\epsilon}\right)$ for some $\epsilon>0$, with $f(0)=1$. Define 
\begin{align}
 \mathfrak{P}_f^{(N)}(\mathbf{x},\mathbf{y})&=\frac{\mathfrak{h}_N(\mathbf{y})}{\mathfrak{h}_N(\mathbf{x})}\mathsf{P}_f^{(N)}(\mathbf{x},\mathbf{y}), &\mathbf{x},\mathbf{y}\in \mathbb{W}_N,\label{Semigroupppp}\\
  \mathfrak{L}_{N+1,N}(\mathbf{y},\mathbf{x})&=\frac{\mathfrak{h}_{N}(\mathbf{x})}{\mathfrak{h}_{N+1}(\mathbf{y})}\Lambda_{N+1,N}(\mathbf{y},\mathbf{x}),  &\mathbf{x}\in \mathbb{W}_N, \mathbf{y}\in \mathbb{W}_{N+1}.
\end{align}
\end{defn}
Here, we are slightly abusing notation as we have already defined $\mathfrak{P}^{(N)}_f(\mathbf{x},\mathbf{y})$ for special choices of $f=f_{s,t}$ in (\ref{SemigroupIntro}). By virtue of Lemma \ref{H_NRep} the expressions (\ref{SemigroupIntro}) and (\ref{Semigroupppp}) are one and the same. Observe that, the following theorem is then a direct consequence of Theorem \ref{IntertwiningSingleLevel} and the above definitions.

\begin{thm}\label{IntertwiningSingleLevelNorm}
Let $N\ge 1$.  Let $f\in \mathsf{Hol}\left(\mathbb{H}_{-\epsilon}\right)$ for some $\epsilon>0$, with $f(0)=1$. Then,
\begin{equation}
\mathfrak{P}^{(N+1)}_{f}\mathfrak{L}_{N+1,N}=\mathfrak{L}_{N+1,N}\mathfrak{P}^{(N)}_{f}.
\end{equation}  
\end{thm}

\subsection{Intertwinings for determinantal kernels}\label{IntertwiningsDeterminantKernels}

In this subsection we introduce a slightly more general framework for intertwinings involving determinants which may be of independent interest. As we explain below, from this setup it would be possible to anticipate Theorem \ref{IntertwiningSingleLevel}, although we do not give the details required for a complete independent proof. 

Fix $N\ge 1$. We are given the following data: for $1\le j \le N+1$, parameters $\lambda_j \in \mathbb{C}$, $\lambda_j \neq \lambda_k$ and corresponding functions $h_{\lambda_j}:\mathbb{Z}_+\to \mathbb{C}$, for $i=1,2,3, 4$, kernels 
$A^{(i)}:\mathbb{Z}_+ \times (\mathbb{Z}_+\cup \{\mathsf{virt}\})  \to \mathbb{C}$, such that, with the series assumed to be converging absolutely,
\begin{equation}\label{EigenfunctionRelInter}
 A^{(i)}h_{\lambda_j}(x)=\sum_{y=0}^\infty A^{(i)}(x,y)h_{\lambda_j}(y)=c_{\lambda_j}^{(i)}h_{\lambda_j}(x)   
\end{equation}
and moreover $A^{(i)}(x,\mathsf{virt})=h_{\lambda_{N+1}}(x)$. Finally, assume that for any $\mathbf{x}\in \mathbb{W}_{n}$, with $1\le n\le N$,
\begin{equation*}
\det\left(h_{\lambda_j}(x_i)\right)_{i,j=1}^n\neq 0.
\end{equation*}

\begin{defn}
Given the data in the above paragraph, define the kernels, for $i=1,2,3,4,$
\begin{align*}
A_{n}^{(i)}(\mathbf{x},\mathbf{y})&=\prod_{j-1}^n\frac{1}{c_{\lambda_j}^{(i)}}\frac{\det\left(h_{\lambda_j}(y_i)\right)_{i,j=1}^n}{\det\left(h_{\lambda_j}(x_i)\right)_{i,j=1}^n}\det\left(A^{(i)}(x_i,y_j)\right)_{i,j=1}^n, \ \ \mathbf{x},\mathbf{y}\in \mathbb{W}_n, \ \ n=1,\dots, N+1,\\
A_{N+1,N}^{(i)}\left(\mathbf{y},\mathbf{x}\right)&=\prod_{j=1}^{N-1}\frac{1}{c_{\lambda_j}^{(i)}}\frac{\det\left(h_{\lambda_j}(y_i)\right)_{i,j=1}^N}{\det\left(h_{\lambda_j}(x_i)\right)_{i,j=1}^{N+1}}\det\left(A^{(i)}(y_i,x_j)\right)_{i,j=1}^{N+1}, \ \ \mathbf{x}\in \mathbb{W}_N,\mathbf{y}\in \mathbb{W}_{N+1}.
\end{align*}

\end{defn}
Then, we have the following proposition.

\begin{prop} Assume the above, then we have, for $i=1,2,3,4,$
\begin{align}\label{NormalizationPropInter}
 \sum_{\mathbf{y}\in \mathbb{W}_n}A_n^{(i)}(\mathbf{x},\mathbf{y})=1, \ \ \sum_{\mathbf{x}\in \mathbb{W}_N}A^{(i)}_{N+1,N}(\mathbf{y},\mathbf{x})=1,
\end{align}
where we write $x_{N+1}=\mathsf{virt}$. Moreover, if we assume, for $x,y\in \mathbb{Z}_+$, 
\begin{equation}\label{AlmostCommutation}
  A^{(1)}A^{(2)}(x,y)=A^{(3)}A^{(4)}(x,y)+\tilde{h}(y)h_{\lambda_{N+1}}(x)
\end{equation}
for some (possibly identically zero) function $\tilde{h}:\mathbb{Z}_+\to \mathbb{C}$ we have the intertwining 
\begin{equation}
A_{N+1}^{(1)}A^{(2)}_{N+1,N}=A^{(3)}_{N+1,N}A^{(4)}_{N}.
\end{equation}
\end{prop}

\begin{proof}
In the case of $A_n^{(i)}$, the normalization (\ref{NormalizationPropInter}) is a direct consequence of the Cauchy-Binet formula and the eigenfunction relation (\ref{EigenfunctionRelInter}). In the case of $A_{N,N-1}^{(i)}$ we expand along the last column, use Cauchy-Binet and the eigenfunction relation (\ref{EigenfunctionRelInter}), from which (\ref{NormalizationPropInter}) follows.

For the intertwining, we compute the left hand side first, with $x_{N+1}=\mathsf{virt}$ using the Cauchy-Binet formula 
\begin{align*}
\left[A_{N+1}^{(1)}A_{N+1,N}^{(2)}\right](\mathbf{y},\mathbf{x})=\prod_{j=1}^{N+1} \frac{1}{c_{\lambda_j}^{(1)}}\prod_{j=1}^{N}\frac{1}{c_{\lambda_j}^{(2)}}\frac{\det\left(h_{\lambda_j}(x_i)\right)_{i,j=1}^{N}}{\det\left(h_{\lambda_j}(y_i)\right)_{i,j=1}^{N+1}}\det\left(A^{(1)}A^{(2)}(y_i,x_j)\right)_{i,j=1}^{N+1}.
\end{align*}
Now, we work on the right hand side, we expand along the last column of the size-$(N+1)$ determinant, and use Cauchy-Binet to obtain
\begin{align*}
\left[A_{N+1,N}^{(3)}A_{N}^{(4)}\right](\mathbf{y},\mathbf{x})= \prod_{j=1}^{N+1} \frac{1}{c_{\lambda_j}^{(3)}}\prod_{j=1}^{N}\frac{1}{c_{\lambda_j}^{(4)}}\frac{\det\left(h_{\lambda_j}(x_i)\right)_{i,j=1}^{N}}{\det\left(h_{\lambda_j}(y_i)\right)_{i,j=1}^{N+1}}\det\left(A^{(3)}A^{(4)}(y_i,x_j)\right)_{i,j=1}^{N+1}.
\end{align*}
Now, observe that the last column of the determinants on the left and right hand sides involving the kernels $A^{(1)}A^{(2)}$ and $A^{(3)}A^{(4)}$ has entries in the $i$-th row given by
$c_{\lambda_{N+1}}^{(2)}h_{\lambda_{N+1}}(x_i)$ and $c_{\lambda_{N+1}}^{(3)}h_{\lambda_{N+1}}(x_i)$ respectively. Using relation (\ref{AlmostCommutation}) and column operations we then obtain
\begin{equation*}
  \det\left(A^{(1)}A^{(2)}(y_i,x_j)\right)_{i,j=1}^{N+1}  = \frac{c_{\lambda_{N+1}}^{(1)}}{c_{\lambda_{N+1}}^{(3)}}\det\left(A^{(3)}A^{(4)}(y_i,x_j)\right)_{i,j=1}^{N+1}.
\end{equation*}
Thus, we get 
\begin{equation*}
   \left[A_{N+1}^{(1)}A_{N+1,N}^{(2)}\right](\mathbf{y},\mathbf{x}) = \prod_{j=1}^N\frac{c_{\lambda_j}^{(3)}c_{\lambda_j}^{(4)}}{c_{\lambda_j}^{(1)}c_{\lambda_j}^{(2)}}\left[A_{N+1,N}^{(3)}A_{N}^{(4)}\right](\mathbf{y},\mathbf{x}).
\end{equation*}
Since, as we have proven earlier, both sides sum over $\mathbf{x}\in \mathbb{W}_N$ to $1$ we obtain that the ratio of constants must be equal to $1$ (for example, when $\tilde{h}\equiv 0$ then this is trivial to see from the eigenfunction relation) and this gives the desired intertwining.
\end{proof}

Of course, for the above to have any probabilistic meaning one needs to address the non-trivial question of positivity of the various kernels. 

We  now explain, without giving the rigorous details, how to get Theorem \ref{IntertwiningSingleLevel} from this framework. Let $f\in \mathsf{Hol}\left(\mathbb{H}_{-R}\right)$, with $R>R(\mathbf{a})$ and $f(0)=1$, be arbitrary. Take $A^{(1)}=A^{(4)}=\mathsf{T}_f$. Let $g_{\lambda_{N+1}}(w)=(w-\lambda_{N+1})^{-1}$. Take 
\begin{align*}
 A^{(2)}(x,y)=A^{(3)}(x,y)=-\frac{1}{a_y}\frac{1}{2\pi \textnormal{i}}\oint_{\mathsf{C}_{\mathbf{a},0}}\frac{p_x(w)g_{\lambda_{N+1}}(w)}{p_{y+1}(w)}dw.
\end{align*}
Note that, $g_{\lambda_{N+1}} \in \mathsf{Hol}\left(\mathbb{H}_{\lambda_{N+1}}\right) $ and observe that $A^{(2)}=A^{(3)}=\mathsf{T}_{g_{\lambda_{N+1}}}$ for $\lambda_{N+1}$ large and negative (since then the pole at $\lambda_{N+1}$ is not contained in the contour $\mathsf{C}_{\mathbf{a},0}$). Take, as $h_{\lambda}(x)=p_x(\lambda)$ which, from Lemma \ref{LemmaEigenfunction}, is an eigenfunction of the $A^{(i)}$ with eigenvalues $c^{(1)}_{\lambda_{}}=c^{(4)}_{\lambda}=f(\lambda)$ and $c^{(2)}_{\lambda}=c_{\lambda}^{(3)}=(\lambda-\lambda_{N+1})^{-1}$. Finally, from Lemma \ref{LemmaComposition} all the $A^{(i)}$ operators commute when $\lambda_{N+1}<-R(\mathbf{a})$.  We now take $\lambda_1,\dots, \lambda_{N+1} \to 0$. First, note that $h_{\lambda_j}(x) \to 1$ and as $\lambda_{N+1}\to 0$, $A^{(2)}(x,y)=A^{(3)}(x,y) \to \phi^*(x,y)=\phi(y,x)$ using the representation of $\phi$ in Lemma \ref{LemmaPhiRep}. Moreover, although the operators $A^{(i)}$ no longer commute if $\lambda_{N+1}$ is small (the reason is that we have an extra residue coming from the pole of $g_{\lambda_{N+1}}$ at $\lambda_{N+1}$ in $\mathsf{C}_{\mathbf{a},0}$) we still get a relation of the form (\ref{AlmostCommutation}) which essentially becomes (\ref{InterInterRelation}) in the limit. Furthermore, ratios of functions $\det\left(h_{\lambda_{j}}(x_i)\right)_{i,j=1}^{n}$ combined with the eigevalues $c^{(i)}_{\lambda}$ converge to ratios of functions $\mathfrak{h}_n$, by virtue of the representation of the $\mathfrak{h}_n$ given in Lemma \ref{H_NRep}. Putting everything together, and after some manipulations we formally obtain the intertwining in Theorem \ref{IntertwiningSingleLevel}.

\section{Some couplings}\label{SectionCouplings}
In this section we introduce certain couplings between the intertwined semigroups from Section \ref{SectionIntertwining} that give rise to the push-block type dynamics introduced in Section \ref{SectionSpaceTimeCorrIntro}. These constructions have their origin, in some sense, in the most basic coalescing random walk model that we discuss next. The following space of two-level interlaced configurations will make its appearance often:
\begin{equation}
\mathbb{W}_{N,N+1}=\left\{(\mathbf{x},\mathbf{y})\in \mathbb{W}_N\times \mathbb{W}_{N+1}:\mathbf{x}\prec \mathbf{y}\right\}. 
\end{equation}

\subsection{Couplings from coalescing random walks}
Suppose we are given a sequence of functions $\left(f_{t,t+1}(z)\right)_{t=0}^\infty$, with either $f_{t,t+1}(z)=1-\alpha_t z$ or $f_{t,t+1}(z)=(1+\beta_tz)^{-1}$, where for all $t\ge 0$, the parameters satisfy $0 \le \alpha_t\le (\sup_{x\in\mathbb{Z}_+}a_x)^{-1}$ and $0\le \beta_t <\infty$. In particular, $\mathsf{T}_{f_{t,t+1}}$ corresponds to either an inhomogeneous Bernoulli or geometric jump.  For $s \le t$, define the function $f_{s,t}$ (with $f_{t,t}=1$) by
\begin{equation*}
f_{s,t}(z)=f_{s,s+1}(z)\cdots f_{t-1,t}(z).
\end{equation*}
At each space-time point $(x,t)\in \mathbb{Z}_+^2$ we put an independent Bernoulli random variable $\mathsf{U}_{x,t}$ which can be one of two kinds depending on $f_{t,t+1}$:
\begin{align*}
\mathsf{U}_{x,t}=
 \begin{cases}
 \mathsf{U}_{x,t}^{(\alpha)}, \ \ &\textnormal{ if } f_{t,t+1}(z)=1-\alpha_tz,\\
 \mathsf{U}_{x,t}^{(\beta)}, \ \ &\textnormal{ if } f_{t,t+1}(z)=(1+\beta_tz)^{-1},
 \end{cases}
\end{align*}
where the Bernoulli variables $\mathsf{U}^{(\alpha)}_{x,t}, \mathsf{U}^{(\beta)}_{x,t}$ satisfy
\begin{align*}
\mathbb{P}\left(\mathsf{U}_{x,t}^{(\alpha)}=1\right)&=1-\mathbb{P}\left(\mathsf{U}_{x,t}^{(\alpha)}=0\right)=\alpha_t a_x,\\
\mathbb{P}\left(\mathsf{U}_{x,t}^{(\beta)}=1\right)&=1-\mathbb{P}\left(\mathsf{U}_{x,t}^{(\beta)}=0\right)=\frac{\beta_t a_x}{1+\beta_t a_x}.
\end{align*}
Then, we do the following:
\begin{itemize}
    \item If $\mathsf{U}_{x,t}^{(\alpha)}=1$ we put an arrow going from $(x,t)$ to $(x+1,t+1)$.
    \item If $\mathsf{U}_{x,t}^{(\alpha)}=0$ we put an arrow going from $(x,t)$ to $(x,t+1)$.
    \item If $\mathsf{U}_{x,t}^{(\beta)}=1$ we put an arrow going from $(x,t)$ to $(x+1,t)$.
    \item If $\mathsf{U}_{x,t}^{(\beta)}=0$ we put an arrow going from $(x,t)$ to $(x,t+1)$.
\end{itemize}

Note that, every space-time point has exactly one outgoing arrow but possibly multiple incoming arrows. See Figure \ref{CoalescingFlowFigure} for an illustration.
\begin{defn}
 For any times $s \le t$,  we define the random map $\mathcal{Z}_{s,t}:\mathbb{Z}_+\to \mathbb{Z}_+$ as follows: $\mathcal{Z}_{s,t}(x)$ is the location at time $t$ of the (deterministic) motion starting from $x\in \mathbb{Z}_+$ at time $s$ which follows the random arrows in the above construction.   
\end{defn}
See Figure \ref{CoalescingFlowFigure} for an illustration of $\mathcal{Z}_{s,t}$.

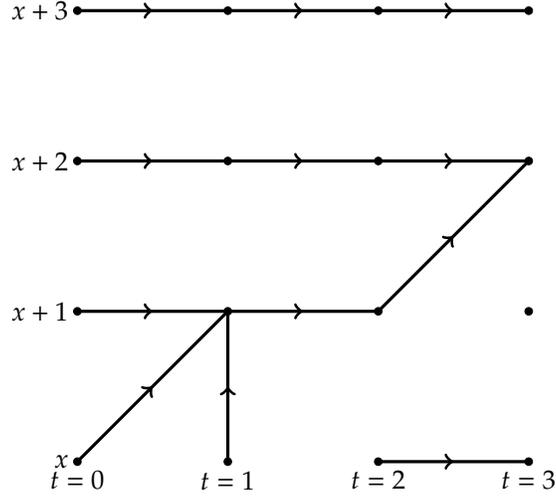
\begin{figure}
\captionsetup{singlelinecheck = false, justification=justified}
\centering
\begin{tikzpicture}
%\draw[dotted] (0,0) grid (2,4);

\node[left] at (0,0) {$x$};

\node[left] at (0,2) {$x+1$};

\node[left] at (0,4) {$x+2$};

\node[left] at (0,6) {$x+3$};

\node[below] at (0,0) {$t=0$};

\node[below] at (2,0) {$t=1$};

\node[below] at (4,0) {$t=2$};

\node[below] at (6,0) {$t=3$};

\draw[fill] (0,0) circle [radius=0.05];

\draw[fill] (0,2) circle [radius=0.05];

\draw[fill] (0,4) circle [radius=0.05];

\draw[fill] (0,6) circle [radius=0.05];

\draw[fill] (2,0) circle [radius=0.05];

\draw[fill] (2,2) circle [radius=0.05];

\draw[fill] (2,4) circle [radius=0.05];

\draw[fill] (2,6) circle [radius=0.05];

\draw[fill] (4,0) circle [radius=0.05];

\draw[fill] (4,2) circle [radius=0.05];

\draw[fill] (4,4) circle [radius=0.05];

\draw[fill] (4,6) circle [radius=0.05];

\draw[fill] (6,0) circle [radius=0.05];

\draw[fill] (6,2) circle [radius=0.05];

\draw[fill] (6,4) circle [radius=0.05];

\draw[fill] (6,6) circle [radius=0.05];

\draw[very thick,middlearrow={>}] (0,0) -- (2,2);

\draw[very thick,middlearrow={>}] (0,2) -- (2,2);

\draw[very thick,middlearrow={>}] (0,4) -- (2,4);

\draw[very thick,middlearrow={>}] (0,6) -- (2,6);

\draw[very thick,middlearrow={>}] (2,0) -- (2,2);

\draw[very thick,middlearrow={>}] (4,0) -- (6,0);

\draw[very thick,middlearrow={>}] (2,2) -- (4,2);

\draw[very thick,middlearrow={>}] (2,4) -- (4,4);

\draw[very thick,middlearrow={>}] (2,6) -- (4,6);

\draw[very thick,middlearrow={>}] (4,2) -- (6,4);

\draw[very thick,middlearrow={>}] (4,4) -- (6,4);

\draw[very thick,middlearrow={>}] (4,6) -- (6,6);

\end{tikzpicture}

\caption{An illustration of the construction of $\mathcal{Z}_{s,t}$. The arrows are random constructed from the independent random variables $\mathsf{U}_{x,t}$ at each space-time point $(x,t)$ as explained in the text. The first step from time $0$ to time $1$ is a Bernoulli step, the second step is a geometric step and the third step is again a Bernoulli step. The random walks starting at $x$ and $x+1$ move independently for one step and then they coalesce and move together. For example, $\mathcal{Z}_{0, 1}(x)=\mathcal{Z}_{0, 1}(x+1)=x+1$, $\mathcal{Z}_{0, 2}(x)=\mathcal{Z}_{0, 2}(x+1)=x+1$, $\mathcal{Z}_{0, 3}(x)=\mathcal{Z}_{0, 3}(x+1)=x+2$.}\label{CoalescingFlowFigure}
\end{figure}

 Observe that, by construction, almost surely for all $t_1 \le t_2 \le t_3$ we have
\begin{equation*}
    \mathcal{Z}_{t_2,t_3}\left(\mathcal{Z}_{t_1,t_2}(x)\right)=\mathcal{Z}_{t_1,t_3}(x), \ \ \forall x \in \mathbb{Z}_+.
\end{equation*}
We can thus think of $\left(\mathcal{Z}_{s,t}\right)_{s\le t}$ as a stochastic flow of maps. Fixing the starting time $s$ and locations $x_1,\dots,x_N$ the following process is called the $N$-point motion of the flow
\begin{equation*}
 t\mapsto \left(\mathcal{Z}_{s,t}(x_1),\dots,\mathcal{Z}_{s,t}(x_N)\right), \ t\ge s.
\end{equation*}
By its very construction this motion is distributed as $N$ random walks starting from locations $x_1,\dots,x_N$ at time $s$, moving independently with Bernoulli or geometric jumps (depending on the corresponding functions $f_{u,u+1}$ for $u\ge s$) until any two walks meet at a space-time point $(x,t)$ from which time onwards they move together (as the same random variables are used to drive both their evolutions).

We can compute the distribution of this process explicitly.
\begin{prop}\label{FlowDistribution}
Let $N\ge 1$ and $s\le t$. Let $x_1\le \cdots \le x_N$ and $y_1\le \cdots \le y_N$. Then,
\begin{equation}\label{FlowDistributionDisplay}
\mathbb{P}\left(\mathcal{Z}_{s,t}(x_1)\le y_1,\dots,\mathcal{Z}_{s,t}(x_N)\le y_N\right)=\det\left(\mathsf{T}_{f_{s,t}}\mathbf{1}_{\llbracket 0,y_j\rrbracket}(x_i)-\mathbf{1}_{i<j}\right)_{i,j=1}^N.
\end{equation}
\end{prop}

\begin{proof} We first assume $\mathbf{x}=(x_1,\dots,x_N)\in \mathbb{W}_N$ and then remove this restriction below. Observe that by summing over the LGV-formula we obtain
\begin{equation*}
\mathbb{P}\left(\mathcal{Z}_{s,t}(x_1)\le y_1<\mathcal{Z}_{s,t}(x_2)\le y_2<\cdots <\mathcal{Z}_{s,t}(x_N)\le y_N\right)=\det\left(\mathsf{T}_{f_{s,t}}\mathbf{1}_{\llbracket 0,y_j\rrbracket}(x_i)\right)_{i,j=1}^N.
\end{equation*}
Observe that this formula would have simply been zero and thus not immediately useful if the $x_i$ were not distinct. Then, by writing  the indicator
\begin{equation*}
\mathbf{1}_{\mathcal{Z}_{s,t}(x_1)\le y_1,\dots,\mathcal{Z}_{s,t}(x_N)\le y_N}
\end{equation*}
in terms of indicators
\begin{equation*}
\mathbf{1}_{\mathcal{Z}_{s,t}(x_{i_1})\le y_{j_1}<\mathcal{Z}_{s,t}(x_{i_2})\le y_{i_2}<\cdots<\mathcal{Z}_{s,t}(x_{i_k})\le y_{j_k}}
\end{equation*}
for increasing sequences $i_1,\dots,i_k$ and $j_1,\dots,j_k$ as explained in detail in Proposition 9 of \cite{Warren} we obtain the desired formula for $\mathbf{x}\in \mathbb{W}_N$. We now explain how to remove the restriction that the coordinates of $\mathbf{x}$ are distinct. Suppose we have $m$ distinct values
\begin{align*}
x_1=\cdots=x_{i_1-1}, x_{i_1}=\cdots=x_{i_2-1},\dots,
x_{i_{m-1}}=\cdots =x_N.
\end{align*}
By the coalescence property we have, with $i_0=1$,
\begin{equation*}
\mathbb{P}\left(\mathcal{Z}_{s,t}(x_i)\le y_i, \ \ \textnormal{ for } 1 \le i \le N\right)=\mathbb{P}\left(\mathcal{Z}_{s,t}(x_{i_{j-1}})\le y_{i_{j-1}}, \ \ \textnormal{ for } 1 \le j \le m\right).
\end{equation*}
 Thus, we need to show that the size $N$ determinant on the right hand side of (\ref{FlowDistributionDisplay}) boils down to the corresponding size $m$ determinant. Using the special structure of the determinant this is easy to prove. For $j=0,\dots,m-1$ subtract row $i_j$ from rows $i_j+1,\dots,i_{j+1}-1$. Then, by swapping consecutive rows bring rows $i_0,i_1,\dots,i_{m-1}$ to positions $1,\dots, m$ and similarly by swapping consecutive columns bring columns $i_0,i_1,\dots,i_{m-1}$ to positions $1,\dots,m$. We then get the determinant of a block triangular matrix whose top-left corner is the $m\times m$ matrix of interest and whose bottom-right corner is the identity matrix. The conclusion follows.
\end{proof}

The following definition is central in our argument.

\begin{defn}\label{DefinitionQmatrix} Consider the following kernel, defined as the $(2N+1)\times (2N+1)$ determinant, with $s\le t$ and $(\mathbf{x},\mathbf{y}),(\mathbf{x}',\mathbf{y}') \in \mathbb{W}_{N,N+1}$,
\begin{align*}
\mathcal{Q}_{s,t}^{N,N+1}\left[(\mathbf{x},\mathbf{y}),(\mathbf{x}',\mathbf{y}')\right]=\det\begin{bmatrix}
 \mathcal{A}_{s,t}(\mathbf{y},\mathbf{y}') & \mathcal{B}_{s,t}(\mathbf{y},\mathbf{x}')\\ \mathcal{C}_{s,t}(\mathbf{x},\mathbf{y}') &  \mathcal{D}_{s,t}(\mathbf{c},\mathbf{x}') 
\end{bmatrix},
\end{align*} 
where the entries of the matrices $\mathcal{A}_{s,t}(\mathbf{y},\mathbf{y}')$,  $\mathcal{B}_{s,t}(\mathbf{y},\mathbf{x}')$, $\mathcal{C}_{s,t}(\mathbf{x},\mathbf{y}')$ and $\mathcal{D}_{s,t}(\mathbf{x},\mathbf{x}')$ of sizes $(N+1)\times (N+1)$, $(N+1)\times N$, $N\times (N+1)$ and $N\times N$ respectively are given by 
\begin{align*}
\mathcal{A}_{s,t}(\mathbf{y},\mathbf{y}')_{ij}&=-\nabla_{y_j'}^-\mathsf{T}_{f_{s,t}}\mathbf{1}_{\llbracket0,y_j'\rrbracket}(y_i)=\mathsf{T}_{f_{s,t}}(y_i,y_j),\\
\mathcal{B}_{s,t}(\mathbf{y},\mathbf{x}')_{ij}&=\frac{1}{a_{x_j'}}\left(\mathsf{T}_{f_{s,t}}\mathbf{1}_{\llbracket 0,y_j'\rrbracket}(y_i)-\mathbf{1}_{j\ge i}\right),\\
\mathcal{C}_{s,t}(\mathbf{x},\mathbf{y}')_{ij}&=a_{x_i}\nabla_{x_i}^+\nabla_{y_j'}^-\mathsf{T}_{f_{s,t}}\mathbf{1}_{\llbracket0,y_j'\rrbracket}(x_i),\\
\mathcal{D}_{s,t}(\mathbf{x},\mathbf{x}')_{ij}&=-\frac{a_{x_i}}{a_{x_j'}}\nabla_{x_i}^+\mathsf{T}_{f_{s,t}}\mathbf{1}_{\llbracket0,x_j'\rrbracket}(x_i)=\mathsf{T}_{f_{s,t}}(x_i,x_j').
\end{align*}
\end{defn}

\begin{rmk}\label{RemarkTwoLevelfunction}
One could replace the function $f_{s,t}$ in the definition above by a general $f\in \mathsf{Hol}\left(\mathbb{H}_{-\epsilon}\right)$ for some $\epsilon>0$ to define a kernel $\mathcal{Q}_f^{N,N+1}$. Then, from an analytic standpoint the results in Propositions \ref{First2levelInterProp} and \ref{Second2LevelInterProp} extend to this setting as well. What is not clear however is whether they have any probabilistic meaning for more general $f$. 
\end{rmk}

\begin{prop}\label{FlowRepProp}
We have, where $(\mathbf{x},\mathbf{y}),(\mathbf{x}',\mathbf{y}') \in \mathbb{W}_{N,N+1}$,
\begin{align*}
\mathcal{Q}_{s,t}^{N,N+1}\left[(\mathbf{x},\mathbf{y}),(\mathbf{x}',\mathbf{y}')\right]=\prod_{i=1}^N\frac{a_{x_i}}{a_{x_i'}}(-\nabla_{x_i}^+)\prod_{i=1}^{N+1}(-\nabla_{y_{i}'}^-)\mathbb{P}\left(\mathcal{Z}_{s,t}(y_i)\le y_i', \mathcal{Z}_{s,t}(x_j)\le x_j', \textnormal{ for all } i,j\right).
\end{align*}
\end{prop}
\begin{proof}
This follows by virtue of Proposition \ref{FlowDistribution} and the explicit Definition \ref{DefinitionQmatrix} of $\mathcal{Q}_{s,t}^{N,N+1}$  and careful inspection of the two formulae.
\end{proof}

\begin{prop}\label{QMatrixProperties}
Let $s\le t$. Viewed as a kernel from $\mathbb{W}_{N,N+1}$ to itself, $\mathcal{Q}^{N,N+1}_{s,t}$ is sub-Markov:
\begin{align*}
 \mathcal{Q}_{s,t}^{N,N+1}\left[(\mathbf{x},\mathbf{y}),(\mathbf{x}',\mathbf{y}')\right]&\ge 0, \ \ \forall (\mathbf{x},\mathbf{y}), (\mathbf{x}',\mathbf{y}')\in \mathbb{W}_{N,N+1},\\ 
 \sum_{(\mathbf{x}',\mathbf{y}')\in \mathbb{w}_{N,N+1}} \mathcal{Q}_{s,t}^{N,N+1}\left[(\mathbf{x},\mathbf{y}),(\mathbf{x}',\mathbf{y}')\right] &\le 1, \ \ \forall (\mathbf{x},\mathbf{y})\in \mathbb{W}_{N,N+1}.
\end{align*}
\end{prop}

\begin{proof}
 Proving non-negativity directly from the determinant formula is actually tricky. We use the connection to the flow $\mathcal{Z}_{s,t}$ via Proposition \ref{FlowRepProp} instead. Observe that, by the very construction of the coalescing flow $\mathcal{Z}_{s,t}$ the events
 \begin{align*}
 \left\{\mathcal{Z}_{s,t}(x_i)\le x_i',\mathcal{Z}_{s,t}(y_i)\le y_j', \ \textnormal{ for all } i,j\right\}
 \end{align*}
 are increasing as the variables $x_i$ decrease and the variables $y_j'$ increase. The claim then follows from the representation given in Proposition \ref{FlowRepProp} of $\mathcal{Q}^{N,N+1}_{s,t}$ in terms of the probability of such events. 
 
 We now move to the second item of the statement. We show that 
\begin{equation*} \sum_{\mathbf{y}':(\mathbf{x}',\mathbf{y}')\in \mathbb{W}_{N,N+1}} \mathcal{Q}_{s,t}^{N,N+1}\left[(\mathbf{x},\mathbf{y}),(\mathbf{x}',\mathbf{y}')\right]=\det\left(\mathcal{D}_{s,t}(x_i,x_j)\right)_{i,j=1}^N=\mathsf{P}^{(N)}_{f_{s,t}}\left(\mathbf{x},\mathbf{x}'\right).
\end{equation*}
Note that, this is nothing else but the LGV formula \cite{LGV} for non-intersecting paths discussed in Proposition \ref{LGV-KMprop} so it is sub-Markov and hence the sum over $\mathbb{W}_{N,N+1}$ of $\mathcal{Q}^{N,N+1}_{s,t}$ is indeed less than $1$. The above claim can be seen in two ways (which are essentially equivalent). First, by direct computation. Use multinearity to bring the sum over $\mathbf{y}'$ inside the determinant definition of $\mathcal{Q}^{N,N+1}_{s,t}$ and then use the relations (we use the convention $x_{N+1}'\equiv \infty$ and $x_{0}'\equiv -1$)
\begin{align*}
 \sum_{y_j'=x_{j-1}'+1}^{x_j'} \mathcal{A}_{s,t}(\mathbf{y},\mathbf{y}')_{ij}&=\mathsf{T}_{f_{s,t}}\mathbf{1}_{\llbracket 0,x_j'\rrbracket}(y_i)-\mathsf{T}_{f_{s,t}}\mathbf{1}_{\llbracket 0,x_{j-1}'\rrbracket}(y_i),\\
 \sum_{y_j'=x_{j-1}'+1}^{x_j'}\mathcal{C}_{s,t}(\mathbf{x},\mathbf{y}')_{ij}&=-a_{x_i}\nabla_{x_i}^+\mathsf{T}_{f_{s,t}}\mathbf{1}_{\llbracket 0,x_j'\rrbracket}(x_i) +a_{x_i}\nabla_{x_i}^+\mathsf{T}_{f_{s,t}}\mathbf{1}_{\llbracket 0,x_{j-1}'\rrbracket}(x_i).
\end{align*}
By some simple row-column operations the claim follows. Second, using the flow
representation in Proposition \ref{FlowRepProp} take the sum over $\mathbf{y}'$ to get rid of the discrete derivatives in $y_j'$, then use the formula in Proposition \ref{FlowDistribution} and conclude by virtue of Lemma \ref{LemmaDuality}.
\end{proof}

A more involved argument, that introduces an inverse $\mathcal{Z}_{s,t}^{-1}$ to the flow $\mathcal{Z}_{s,t}$ can be used to establish the semigroup property $\mathcal{Q}^{N,N+1}_{t_1,t_2}\mathcal{Q}^{N,N+1}_{t_2,t_3}=\mathcal{Q}^{N,N+1}_{t_1,t_3}$, for $t_1\le t_2 \le t_3$. The actual dynamics of the induced process on $\mathbb{W}_{N,N+1}$, as far as we can tell, cannot be obtained directly using this connection and we need to take a different approach in the next subsections. We believe that $\mathcal{Q}_{s,t}^{N,N+1}$ should correspond to the transition probabilities of two-level dynamics taking sequential-update Bernoulli steps or Warren-Windridge geometric steps of certain parameters depending on whether each factor $f_{u,u+1}(z)$ is given by $(1-\alpha_uz)$ or $(1+\beta_uz)^{-1}$. We will prove this in the case of Bernoulli jumps in Proposition \ref{PropBernoulliDynamics} and take yet another approach to consider geometric jumps in Section \ref{GeometricCouplingSection}.

Before discussing such dynamics further we prove a number of intertwining relations (which in fact hold beyond the probabilistic setting, see Remark \ref{RemarkTwoLevelfunction}). Towards this end, define the projection $\Pi_N:\mathbb{W}_{N,N+1} \to \mathbb{W}_N$ by projecting on the $\mathbf{x}$-coordinate. In particular, for a function $g$ on $\mathbb{W}_{N}$ we thus define a function $\Pi_{N}g(\mathbf{x},\mathbf{y})=g(\mathbf{x})$ on $\mathbb{W}_{N,N+1}$. 

\begin{prop}\label{First2levelInterProp}
 Let $s\le t$. We have 
 \begin{align}  \left[\mathcal{Q}_{s,t}^{N,N+1}\Pi_N\right]\left((\mathbf{x},\mathbf{y}),\mathbf{x}'\right)&=\left[\Pi_N\mathsf{P}_{f_{s,t}}^{(N)}\right]\left((\mathbf{x},\mathbf{y}),\mathbf{x}'\right), \ &(\mathbf{x},\mathbf{y})\in \mathbb{W}_{N,N+1}, \mathbf{x}'\in \mathbb{W}_{N},\label{First2levelInter1}\\
\left[\mathsf{P}_{f_{s,t}}^{(N+1)}\Lambda_{N+1,N}\right]\left(\mathbf{y},(\mathbf{x}',\mathbf{y}')\right) &=\left[\Lambda_{N+1,N}\mathcal{Q}_{s,t}^{N,N+1}\right]\left(\mathbf{y},(\mathbf{x}',\mathbf{y}')\right), \ &\mathbf{y}\in \mathbb{W}_{N+1}, \left(\mathbf{x}',\mathbf{y}'\right)\in \mathbb{W}_{N,N+1},\label{First2levelInter2}
 \end{align}
 where we view $\Lambda_{N+1,N}$ given in Definition \ref{PreMarkovKernelDef} as a kernel from $\mathbb{W}_{N+1}$ to $\mathbb{W}_{N,N+1}$ in the obvious way
 \begin{equation*}
\Lambda_{N+1,N}\left(\mathbf{y},(\mathbf{x},\mathbf{k})\right)=\prod_{i=1}^N \frac{1}{a_{x_i}}\mathbf{1}_{\mathbf{x}\prec \mathbf{k}}\mathbf{1}_{\mathbf{k}=\mathbf{y}}.
 \end{equation*}
\end{prop}

\begin{proof}
The equality (\ref{First2levelInter1}) is shown using the computation at the end of the proof of Proposition \ref{QMatrixProperties}. For equality (\ref{First2levelInter2}), we need to sum over $\mathbf{x}$ such that $ \mathbf{x}\prec \mathbf{y}$. We use multinearity to bring the sum inside the determinant and the relations
\begin{align*}
\sum_{x_i={y_i}}^{y_{i+1}-1} \frac{1}{a_{x_i}} \mathcal{C}_{s,t}(\mathbf{x},\mathbf{y}')_{ij}&=\nabla_{y_j'}^+\mathsf{T}_{f_{s,t}}\mathbf{1}_{\llbracket0,y_j'\rrbracket}(y_{i+1})-\nabla_{y_j'}^+\mathsf{T}_{f_{s,t}}\mathbf{1}_{\llbracket0,y_j'\rrbracket}(y_{i}),\\
 \sum_{x_i={y_i}}^{y_{i+1}-1} \frac{1}{a_{x_i}}\mathcal{D}_{s,t}(\mathbf{x},\mathbf{x}')_{ij}&=-\frac{1}{a_{x_j'}}\mathsf{T}_{f_{s,t}}\mathbf{1}_{\llbracket0,x_j'\rrbracket}(y_{i+1})+\frac{1}{a_{x_j}'}\mathsf{T}_{f_{s,t}}\mathbf{1}_{\llbracket0,x_j'\rrbracket}(y_{i}).
\end{align*}
Then, the statement follows from simple row-column operations.
\end{proof}

Observe that, by a combination of displays (\ref{First2levelInter1}) and (\ref{First2levelInter2}) we obtain yet another proof of Theorem \ref{IntertwiningSingleLevel}.

\begin{prop}\label{DefNormalisedKernel} Let $s\le t$. The kernel defined by 
\begin{align}\label{TwoLevelNormalised}
\mathsf{Q}^{N,N+1}_{s,t}\left[(\mathbf{x},\mathbf{y}),(\mathbf{x}',\mathbf{y}')\right]=\frac{\mathfrak{h}_N(\mathbf{x}')}{\mathfrak{h}_N(\mathbf{x})}\mathcal{Q}^{N,N+1}_{s,t}\left[(\mathbf{x},\mathbf{y}),(\mathbf{x}',\mathbf{y}')\right]
\end{align}
is a Markov kernel from $\mathbb{W}_{N,N+1}$ to itself.
\end{prop}

\begin{proof}
By combining Propositions \ref{EigenfunctionKM} and \ref{First2levelInterProp} we obtain 
\begin{align*}
\left[\mathcal{Q}^{N,N+1}_{s,t}\Pi_N\mathfrak{h}_N\right](\mathbf{x},\mathbf{y})=\Pi_N\mathfrak{h}_N(\mathbf{x},\mathbf{y}),
\end{align*}
from which the result follows.
\end{proof}

If by abusing notation we view $\mathfrak{L}_{N+1,N}$ as a Markov kernel from $\mathbb{W}_{N+1}$ to $\mathbb{W}_{N,N+1}$ as we did with $\Lambda_{N+1,N}$, by combining the results above we readily obtain the intertwinings of Markov kernels.

\begin{prop}\label{Second2LevelInterProp} Let $s\le t$. Then, we have
 \begin{align}  \left[\mathsf{Q}_{s,t}^{N,N+1}\Pi_N\right]\left((\mathbf{x},\mathbf{y}),\mathbf{x}'\right)&=\left[\Pi_N\mathfrak{P}_{f_{s,t}}^{(N)}\right]\left((\mathbf{x},\mathbf{y}),\mathbf{x}'\right), \ &(\mathbf{x},\mathbf{y})\in \mathbb{W}_{N,N+1}, \mathbf{x}'\in \mathbb{W}_{N},\\
\left[\mathfrak{P}_{f_{s,t}}^{(N+1)}\mathfrak{L}_{N+1,N}\right]\left(\mathbf{y},(\mathbf{x}',\mathbf{y}')\right) &=\left[\mathfrak{L}_{N+1,N}\mathsf{Q}_{s,t}^{N,N+1}\right]\left(\mathbf{y},(\mathbf{x}',\mathbf{y}')\right), \ &\mathbf{y}\in \mathbb{W}_{N+1}, \left(\mathbf{x}',\mathbf{y}'\right)\in \mathbb{W}_{N,N+1}.
 \end{align}
\end{prop}

\subsection{Sequential-update Bernoulli coupling}

\iffalse
Given a discrete time process $\mathsf{X}(t)\in \mathbb{Z}_+^N$ with initial condition $\mathsf{X}(0)\in \mathbb{W}_{N}$
\begin{equation}
\tau=\inf \left\{n:\mathsf{X}(n-1)\nprec \mathsf{X}(n)\right\}.
\end{equation}
Note that, in case of Bernoulli jumps we have
\begin{equation}
\tau= \inf \left\{n:\mathsf{X}(n)\notin \mathbb{W}_N\right\}.
\end{equation}
This is no longer the case for geometric jumps. There is a subtle difference between the two stopping times although they are connected, see for example [], [].

\fi

In this section we prove that the formula for $\mathcal{Q}_{s,t}^{N,N+1}$ given in Definition \ref{DefinitionQmatrix} with $f_{u,u+1}(z)=1-\alpha_u z$, where $0\le \alpha_i \le \left(\sup_k a_k\right)^{-1}$, actually corresponds to the sequential-update Bernoulli push-block dynamics. More precisely, we have the following result. Observe that, in the definition of the stochastic process $\left(\mathsf{X}(t),\mathsf{Y}(t);t\ge 0\right)$ below the interaction of the $\mathsf{Y}$-components with the $\mathsf{X}$-components is exactly given by the interaction between consecutive levels in the sequential-update Bernoulli dynamics of Definition \ref{DefBernoulliDynamics}.

\begin{prop} \label{PropBernoulliDynamics}
Consider the following stochastic process $\left(\mathsf{X}(t),\mathsf{Y}(t);t\ge 0\right)$ in $\mathbb{W}_N\times \mathbb{W}_{N+1}$ in discrete time initialised at $(\mathbf{x},\mathbf{y})\in \mathbb{W}_{N,N+1}$. The component $\left(\mathsf{X}(t);t\ge 0\right)$ evolves autonomously as $N$ independent Bernoulli walks, each with transition probabilities from time $s$ to time $s+1$ given by $\mathsf{T}_{(1-\alpha_s z)}$. Then, given the updated component $\mathsf{X}(s+1)$ we update $\mathsf{Y}(s)$ to $\mathsf{Y}(s+1)$ as follows:
\begin{itemize}
    \item If $\mathsf{Y}_i(s)=x$ and $\mathsf{X}_i(s+1)=x$ then $\mathsf{Y}_i(s+1)=x$ (block).
    \item If $\mathsf{Y}_i(s)=x$ and $\mathsf{X}_{i-1}(s+1)=x$ then $\mathsf{Y}_i(s+1)=x+1$ (push).
    \item Otherwise, $\mathsf{Y}_i$ moves as a Bernoulli random walk with transition probability $\mathsf{T}_{(1-\alpha_s z)}$ independent of the other coordinates $\mathsf{Y}_j$.
\end{itemize}
The whole process is killed at the first collision time $\tau=\inf\{t\ge 0:\mathsf{X}(t)\notin \mathbb{W}_N\}$ of the $\mathsf{X}$ coordinates, so that in particular $\left(\mathsf{X}(t),\mathsf{Y}(t)\right)\in \mathbb{W}_{N,N+1}$ for all $t<\tau $. Then, the transition probabilities from time $s$ to time $t$ of $\left(\mathsf{X}(t),\mathsf{Y}(t);t\ge 0\right)$ killed when it exits $\mathbb{W}_{N,N+1}$ are given by $\mathcal{Q}_{s,t}^{N,N+1}$ with the underlying functions given by $f_{i,i+1}(z)=1-\alpha_i z$, where $0\le \alpha_i \le \left(\sup_k a_k\right)^{-1}$.
\end{prop}

\begin{proof}
Take a test function $g:\mathbb{W}_{N,N+1}\to \mathbb{R}_+$. Fix an arbitrary time $T\ge 0$. Define for any $0\le t \le T$, the functions (we drop dependence on $N$)
\begin{align*}
F\left(t,T,\left(\mathbf{x},\mathbf{y}\right)\right)  &=\sum_{\left(\mathbf{x}',\mathbf{y}'\right)\in \mathbb{W}_{N,N+1}}\mathcal{Q}^{N,N+1}_{t,T}\left[(\mathbf{x},\mathbf{y}),(\mathbf{x}',\mathbf{y}')\right]g(\mathbf{x}',\mathbf{y}'),\\
G\left(t,T,\left(\mathbf{x},\mathbf{y}\right)\right)&=\mathbb{E}\left[g\left(\mathsf{X}(T),\mathsf{Y}(T)\right)\mathbf{1}_{T<\tau}\big|\left(\mathsf{X}(t),\mathsf{Y}(t)\right)=(\mathbf{x},\mathbf{y}), t\le \tau\right].
\end{align*}
We now show, by backward induction in $t$, that $F$ and $G$ are equal from which the proposition follows. Observe that,
\begin{equation*}
F\left(T,T,(\mathbf{x},\mathbf{y})\right)=G\left(T,T,(\mathbf{x},\mathbf{y})\right)=g(\mathbf{x},\mathbf{y}), \ \ (\mathbf{x},\mathbf{y})\in \mathbb{W}_{N,N+1}.
\end{equation*}
Then, observe that by the way our dynamics work $G$ satisfies the following set of equations (the last three equations (\ref{BernoulliEq2}), (\ref{BernoulliEq3}), (\ref{BernoulliEq4}) define the value of $G\left(t+1,T,\left(\mathbf{x}+\boldsymbol{\delta},\mathbf{y}+\boldsymbol{\epsilon}\right)\right)$ in (\ref{BernoulliEq1}) below when $(\mathbf{x}+\boldsymbol{\delta},\mathbf{y}+\boldsymbol{\epsilon})\notin \mathbb{W}_{N,N+1}$), with the first equation holding for $0\le t\le T-1$ and the rest for $0\le t \le T$:
\begin{align}
&G\left(t,T,\left(\mathbf{x},\mathbf{y}\right)\right)=\sum_{\delta_i,\epsilon_i \in \{0,1\}}\prod_{i=1}^N \mathsf{T}_{f_{t,t+1}}(x_i,x_i+\delta_i)\prod_{i=1}^{N+1}\mathsf{T}_{f_{t,t+1}}(y_i,y_i+\epsilon_i)G\left(t+1,T,\left(\mathbf{x}+\boldsymbol{\delta},\mathbf{y}+\boldsymbol{\epsilon}\right)\right)\label{BernoulliEq1},\\
&\nabla_{y_i}^+G\left(t,T,\left(\mathbf{x},\mathbf{y}\right)\right)\big|_{y_{i}=x_i}=0  \ \ (\textnormal{i-th particle is blocked}),\label{BernoulliEq2}\\
&\nabla_{y_{i+1}}^+G\left(t,T,\left(\mathbf{x},\mathbf{y}\right)\right)\big|_{y_{i+1}=x_i}=0  \ \ (\textnormal{(i+1)-th particle is pushed}),\label{BernoulliEq3}\\
&G\left(t,T,\left(\mathbf{x},\mathbf{y}\right)\right)\big|_{x_{i}=x_{i+1}}=0 \ \ (\textnormal{process killed at } \tau).\label{BernoulliEq4}
\end{align}
These determine $G(t,T,(\mathbf{x},\mathbf{y}))$ starting from $G\left(T,T,(\mathbf{x},\mathbf{y})\right)=g(\mathbf{x},\mathbf{y})$. We now show that $F(t,T,(\mathbf{x},\mathbf{y}))$ satisfies the same equations which completes the proof. 

First, it readily follows from the form of the entries $\mathcal{C}_{s,t}(\mathbf{x},\mathbf{y}')_{ij}$ and $\mathcal{D}_{s,t}(\mathbf{x},\mathbf{x}')_{ij}$ that, for $0\le t \le T$:
\begin{equation*}
F\left(t,T,(\mathbf{x},\mathbf{y})\right)\big|_{x_i=x_{i+1}}=0.
\end{equation*}
The boundary condition 
\begin{equation*}
\nabla_{y_i}^+F\left(t,T,\left(\mathbf{x},\mathbf{y}\right)\right)\big|_{y_{i}=x_i}=0,
\end{equation*}
follows from the relations
\begin{align*}
\nabla_{y_i}^+\mathcal{A}_{t,T}\left(\mathbf{y},\mathbf{y}'\right)_{ij}\big|_{y_i=x_i}&=-a_{x_i}^{-1}\mathcal{C}_{t,T}\left(\mathbf{x},\mathbf{y}'\right)_{ij},\\
\nabla_{y_i}^+\mathcal{B}_{t,T}\left(\mathbf{y},\mathbf{x}'\right)_{ij}\big|_{y_i=x_i}&=-a_{x_i}^{-1}\mathcal{D}_{t,T}\left(\mathbf{x},\mathbf{x}'\right)_{ij},
\end{align*}
while the boundary condition
\begin{equation*}
\nabla_{y_{i+1}}^+F\left(t,T,\left(\mathbf{x},\mathbf{y}\right)\right)\big|_{y_{i+1}=x_i}=0,
\end{equation*}
follows from the relations
\begin{align*}
\nabla_{y_{i+1}}^+\mathcal{A}_{t,T}\left(\mathbf{y},\mathbf{y}'\right)_{i+1j}\big|_{y_{i+1}=x_i}&=-a_{x_i}^{-1}\mathcal{C}_{t,T}\left(\mathbf{x},\mathbf{y}'\right)_{ij},\\
\nabla_{y_{i+1}}^+\mathcal{B}_{t,T}\left(\mathbf{y},\mathbf{x}'\right)_{i+1j}\big|_{y_{i+1}=x_i}&=-a_{x_i}^{-1}\mathcal{D}_{t,T}\left(\mathbf{x},\mathbf{x}'\right)_{ij}.
\end{align*}
Finally, the equation 
\begin{equation*}
F\left(t,T,\left(\mathbf{x},\mathbf{y}\right)\right)=\sum_{\delta_i,\epsilon_i \in \{0,1\}}\prod_{i=1}^N \mathsf{T}_{f_{t,t+1}}(x_i,x_i+\delta_i)\prod_{i=1}^{N+1}\mathsf{T}_{f_{t,t+1}}(y_i,y_i+\epsilon_i)F\left(t+1,T,\left(\mathbf{x}+\boldsymbol{\delta},\mathbf{y}+\boldsymbol{\epsilon}\right)\right)
\end{equation*}
follows from multinearity of the determinant and the following set of equations for the individual entries
\begin{align*}
\mathcal{A}_{t,T}\left(\mathbf{y},\mathbf{y}'\right)_{ij}&=\sum_{\epsilon \in \{0,1\}}\mathsf{T}_{f_{t,t+1}}\left(y_i,y_i+\epsilon\right)\mathcal{A}_{t+1,T}\left(\mathbf{y}+\epsilon,\mathbf{y}'\right)_{ij},\\
\mathcal{D}_{t,T}\left(\mathbf{x},\mathbf{x}'\right)_{ij}&=\sum_{\epsilon \in \{0,1\}}\mathsf{T}_{f_{t,t+1}}\left(x_i,x_i+\epsilon\right)\mathcal{D}_{t+1,T}\left(\mathbf{x}+\epsilon,\mathbf{x}'\right)_{ij},\\
\mathcal{B}_{t,T}\left(\mathbf{y},\mathbf{x}'\right)_{ij}&=a_{x_j'}^{-1}\left(\mathsf{T}_{f_{t,T}}\mathbf{1}_{\llbracket0,x_j'\rrbracket}(y_i)-\mathbf{1}_{j\ge i}\right)\\
&= a_{x_j'}^{-1}\left(\sum_{\epsilon \in \{0,1\}}\mathsf{T}_{f_{t,t+1}}(y_i,y_i+\epsilon)\mathsf{T}_{f_{t+1,T}}\mathbf{1}_{\llbracket0,x_j'\rrbracket}(y_i)-\sum_{\epsilon \in \{0,1\}}\mathsf{T}_{f_{t,t+1}}(y_i,y_i+\epsilon)\mathbf{1}_{j\ge i}\right)\\
&=\sum_{\epsilon \in \{0,1\}}\mathsf{T}_{f_{t,t+1}}(y_i,y_i+\epsilon)\mathcal{B}_{t+1,T}\left(\mathbf{y}+\epsilon,\mathbf{x}'\right)_{ij},\\
\mathcal{C}_{t,T}\left(\mathbf{x},\mathbf{y}'\right)_{ij}&=a_{x_i}\nabla_{x_i}^+\nabla_{y_j}^-\mathsf{T}_{f_{t,T}}\mathbf{1}_{\llbracket0,y_j'\rrbracket}(x_i)=-\nabla_{y_j'}^-\left[a_{y_j'}\mathsf{T}_{f_{t,T}}\left(x_i,y_j'\right)\right] \\
&=-\nabla_{y_j'}^-\left[a_{y_j'}\sum_{\epsilon \in \{0,1\}}\mathsf{T}_{f_{t,t+1}}\left(x_i,x_i+\epsilon\right)\mathsf{T}_{f_{t+1,T}}\left(x_i+\epsilon,y_j'\right)\right]\\ 
&=\sum_{\epsilon \in \{0,1\}}\mathsf{T}_{f_{t,t+1}}\left(x_i,x_i+\epsilon\right) \mathcal{C}_{t+1,T}\left(\mathbf{x}+\epsilon,\mathbf{y}'\right)_{ij}.
\end{align*}
This completes the proof.
\end{proof}

\subsection{Warren-Windridge geometric dynamics
coupling}\label{GeometricCouplingSection}

In section we take a different approach to study the transition probabilities of the Warren-Windridge geometric dynamics. We take $f_{0,1}(z)=(1+\beta z)^{-1}$. To ease notation we write $\mathsf{T}(x,y)$ for $\mathsf{T}_{f_{0,1}}(x,y)$ and $\mathsf{P}^{(N)}$ and $\mathfrak{P}^{(N)}$ for $\mathsf{P}_{f_{0,1}}^{(N)}$ and $\mathfrak{P}_{f_{0,1}}^{(N)}$ respectively.

\begin{defn}
We write $\mathcal{G}^{N,N+1}\left[(\mathbf{x},\mathbf{y}),(\mathbf{x}',\mathbf{y}')\right]$ for the transition probabilities of the following single-step discrete-time stochastic process $\left(\mathsf{X}(t),\mathsf{Y}(t);t=0,1\right)$  in $\mathbb{W}_{N,N+1}$. The $\mathsf{X}$-component moves autonomously, with $\mathsf{X}(1)$ given $\mathsf{X}(0)=\mathbf{x}$ distributed as $\mathfrak{P}^{(N)}\left(\mathsf{x},\cdot\right)$. Given the updated component $\mathsf{X}(1)$, we update component $\mathsf{Y}$ as follows:
\begin{itemize}
    \item Coordinate $\mathsf{Y}_i$ is first moved to the intermediate position $\max \{\mathsf{Y}_i(0),\mathsf{X}_{i-1}(1)+1\}$ (push).
    \item $\mathsf{Y}_i$ subsequently attempts to jump to $z\ge \max \{\mathsf{Y}_i(0),\mathsf{X}_{i-1}(1)+1\}$ with the inhomogeneous geometric probability $\mathsf{T}(\max \{\mathsf{Y}_i(0),\mathsf{X}_{i-1}(1)+1\},z)$. 
    \item All jumps to $z\ge \mathsf{X}_i(0)=x_i$ are suppressed and in such case $\mathsf{Y}_i(1)$ is taken to be $x_i$ (block).
\end{itemize}
\end{defn}

\begin{rmk}
Of course, the process $\left(\mathsf{X}(t),\mathsf{Y}(t);t=0,1\right)$ can be extended to all times $t\ge 0$. For the $s$-th step we simply replace $f_{0,1}(z)$ by $f_{s,s+1}(z)=(1+\beta_s z)^{-1}$ and all other quantities are defined analogously using $f_{s,s+1}$. All the results that follow go through with the obvious notational modification. 
\end{rmk}

Observe that, in the definition of the stochastic process $\left(\mathsf{X}(t),\mathsf{Y}(t);t=0,1\right)$ above the interaction of the $\mathsf{Y}$-components with the $\mathsf{X}$-components is exactly given by the interaction between consecutive levels in the Warren-Windridge geometric dynamics of Definition \ref{DefGeometricDynamics}. As mentioned previously we believe\footnote{We have shown this for $N=1$ by tediously checking all possibilities.} that $\mathsf{Q}_{0,1}^{N,N+1}$, with $f_{0,1}(z)=(1+\beta z)^{-1}$, should correspond to Warren-Windridge geometric dynamics step with parameter $\beta$.  In particular, we believe that the question mark on the equality below should be removed
\begin{equation}\label{Equation?Geom}
\mathcal{G}^{N,N+1}\left[(\mathbf{x},\mathbf{y}),(\mathbf{x}',\mathbf{y}')\right]  \overset{?}{=} \mathsf{Q}_{0,1}^{N,N+1}\left[(\mathbf{x},\mathbf{y}),(\mathbf{x}',\mathbf{y}')\right].
\end{equation}
 However, when geometric jumps are involved, a proof along the lines of the one given for Proposition \ref{PropBernoulliDynamics} above turns out to be tricky. So we take a different route inspired by the work of Warren-Windridge \cite{WarrenWindridge} in the homogeneous case. Observe that, we have
\begin{equation*}
\mathcal{G}^{N,N+1}\left[(\mathbf{x},\mathbf{y}),(\mathbf{x}',\mathbf{y}')\right]=\mathfrak{P}^{(N)}(\mathbf{x},\mathbf{x}') \mathcal{R}^{N,N+1}(\mathbf{y}';\mathbf{x}',\mathbf{x},\mathbf{y})
\end{equation*}
where we write
\begin{equation*}
\mathcal{R}^{N,N+1}(\mathbf{y}';\mathbf{x}',\mathbf{x},\mathbf{y})=\mathbb{P}\left(\mathsf{Y}(1)=\mathbf{y}'|\mathsf{X}(1)=\mathbf{x},\mathsf{X}(0)=\mathbf{x},\mathsf{Y}(0)=\mathbf{y}\right).
\end{equation*}

Our starting point is the following explicit expression for $\mathcal{R}^{N,N+1}$.

\begin{prop}\label{RGeometricExplicitTrans}
We have, with $(\mathbf{x},\mathbf{y}),(\mathbf{x}',\mathbf{y}')\in \mathbb{W}_{N,N+1}$,
\begin{align*}
 \mathcal{R}^{N,N++1}(\mathbf{y}';\mathbf{x}',\mathbf{x},\mathbf{y})=\mathsf{Blk}\left(y_1',y_1,x_1\right)\prod_{n=2}^{N}\mathsf{b}\left(x_n,y_n'\right)\mathsf{Psh}\left(y_{n}',y_{n},x_{n-1}'\right) \mathsf{Psh}\left(y_{N+1}',y_{N+1},x_N'\right),
\end{align*}
where the factors are defined by 
\begin{align*}
    \mathsf{Blk}\left(y',y,x\right)&=\mathsf{T}(y,y')\mathsf{b}\left(x,y'\right),\\
\mathsf{Psh}\left(y',y,x'\right)&=\mathsf{T}\left(y,y'\right)\mathbf{1}_{x'<y}+\mathsf{T}\left(x'+1,y'\right)\mathbf{1}_{x'\ge y},\\
    \mathsf{b}\left(x,y'\right)&=\mathbf{1}_{y'<x}+(1+\beta a_{y'})\mathbf{1}_{y'=x}.
\end{align*}
\end{prop}

\begin{proof}
The $\mathsf{Psh}$ factor is a direct translation in symbols of the pushing mechanism. The fact that the $\mathsf{Blk}$ factor corresponds to blocking can be seen by multiplying out $\mathsf{T}(y,y')$ with $\mathsf{b}(x,y')$  
and using the identity, with $y\ge x$,
\begin{equation*}
\sum_{m=x}^\infty\mathsf{T}(y,m)=\mathsf{T}(y,x)(1+\beta a_x).
\end{equation*}
This follows from the trivial to check identity, for $y\le x \le m$,
\begin{equation*}
\frac{\mathsf{T}(y,m)}{\mathsf{T}(y,x)}=(1+\beta a_x)\mathsf{T}(x,m).
\end{equation*}
Similarly, the fact that the factor with subscript $n$ in the product corresponds to the interaction of $\mathsf{Y}_n$ with $\mathsf{X}_{n-1}(1)$ by pushing and $\mathsf{X}_{n}(0)$ by blocking can be seen by multiplying out $\mathsf{Psh}\left(y',y,x'\right)$ with $\mathsf{b}\left(x,y'\right)$
and the relation 
\begin{align*}
\sum_{m=x}^\infty \mathsf{Psh}\left(m,y,x'\right)=\mathsf{Psh}\left(x,y,x'\right)(1+\beta a_x),
\end{align*}
which is again a consequence of the previous identity.
\end{proof}

%\begin{rmk}
%It should be possible to prove the identity (\ref{Equation?Geom}) directly using the explicit formula above and the explicit formula for  but checking it is quite tedius.  
%\end{rmk}

We now go on to prove an analogue of Proposition \ref{Second2LevelInterProp} with $\mathsf{Q}_{0,1}^{N,N+1}$ replaced by $\mathcal{G}^{N,N+1}$ (clearly, if we knew equality (\ref{Equation?Geom}) were true there would be nothing to show). The following identity is the key ingredient in the computation.

\begin{prop}\label{KeyComputationGeomtric}
We have 
\begin{align*}
 \sum_{x_i=y_i'}^{x_i'\wedge (y_{i+1}-1)} \frac{1}{a_{x_i}}\mathsf{T}(x_i,x_i')\mathsf{Blk}(y_i',y_i,x_i)&\mathsf{Psh}(y_{i+1}',y_{i+1},x_{i}')\\&=\frac{1}{a_{x_i}'}\mathsf{T}(y_i,y_i')\mathsf{T}(y_{i+1},y_{i+1}') \mathbf{1}_{y_i\le y_i'<y_{i+1}\le y_{i+1}'}\mathbf{1}_{y_i'\le x_i'<y_{i+1}'}.
\end{align*}
\end{prop}

\begin{proof}
Observe that, the left hand side is zero if we do not have $y_i'<y_{i+1}$ or if we do not have $y_i'\le x_i'<y_{i+1}'$. Thus, we can restrict to this scenario in which case the indicator functions on the right are both $1$. By using $\mathsf{Blk}(y_i',y_i,x_i)=\mathsf{T}(y_i,y_i')\mathsf{b}(x_i,y_i')$ and cancelling out $\mathsf{T}(y_i,y_i')$ it will suffice to establish the following claim
\begin{align*}
\sum_{x_i=y_i'}^{x_i'\wedge (y_{i+1}-1)} \frac{a_{x_i'}}{a_{x_i}}\mathsf{T}(x_i,x_i')\left(\mathbf{1}_{y_i'<x_i}+(1+\beta a_{y_i'})\mathbf{1}_{x_i=y_i'}\right)\left(\mathsf{T}(y_{i+1},y_{i+1}')\mathbf{1}_{x_i'<y_{i+1}}+\mathsf{T}(x_i'+1,y_{i+1}')\mathbf{1}_{x_i'\ge y_{i+1}}\right)\\=\mathsf{T}(y_{i+1},y_{i+1}').
\end{align*}
Let us denote by $\mathcal{E}(y_i')$ the left hand side above as a function of $y_i'$. We will now show that $\mathcal{E}(x_i'\wedge (y_{i+1}-1))=\mathsf{T}(y_{i+1},y_{i+1}')$ and that $\mathcal{E}(m)=\mathcal{E}(m-1)$ from which the claim follows. 

Suppose $x_i'<y_{i+1}$ then $x_i'\wedge (y_{i+1}-1)$. It is immediate then that $\mathcal{E}(x_i'\wedge (y_{i+1}-1))=\mathsf{T}(y_{i+1},y_{i+1}')$. If instead we suppose $x_i'\ge y_{i+1}$ then $x_i=y_i'=y_{i+1}-1$. After some elementary manipulations using the explicit formulae we again obtain $\mathcal{E}(x_i'\wedge (y_{i+1}-1))=\mathsf{T}(y_{i+1},y_{i+1}')$. We now show $\mathcal{E}(m)=\mathcal{E}(m-1)$. We observe that the sums $\sum_{x_i=m+1}^{x_i\wedge (y_{i+1}-1)}$ on both sides of the desired identity $\mathcal{E}(m)=\mathcal{E}(m-1)$ are equal (the summands are exactly equal). We thus need to check that the sum $\sum_{x_i=m}^{m}$ on the left hand side and the sum $\sum_{x_i=m-1}^{m}$ are equal. After some manipultions this boils down to proving the identity
\begin{equation*}
\frac{1}{a_m}\mathsf{T}(m,x_i')(1+\beta a_m)=\frac{1}{a_m}\mathsf{T}(m,x_i')+\frac{1}{a_{m-1}}\mathsf{T}(m-1,x_i')(1+\beta a_{m-1}).
\end{equation*}
Elementary manipulations readily establish this.
\end{proof}

We finally prove the required intertwinings.

 \begin{prop}\label{GeomInterProp} We have the intertwinings
 \begin{align}  \left[\mathcal{G}^{N,N+1}\Pi_N\right]\left((\mathbf{x},\mathbf{y}),\mathbf{x}'\right)&=\left[\Pi_N\mathfrak{P}^{(N)}\right]\left((\mathbf{x},\mathbf{y}),\mathbf{x}'\right), \ &(\mathbf{x},\mathbf{y})\in \mathbb{W}_{N,N+1}, \mathbf{x}'\in \mathbb{W}_{N},\label{Geom2levelInter1}\\
\left[\mathfrak{P}^{(N+1)}\mathfrak{L}_{N+1,N}\right]\left(\mathbf{y},(\mathbf{x}',\mathbf{y}')\right) &=\left[\mathfrak{L}_{N+1,N}\mathcal{G}^{N,N+1}\right]\left(\mathbf{y},(\mathbf{x}',\mathbf{y}')\right), \ &\mathbf{y}\in \mathbb{W}_{N+1}, \left(\mathbf{x}',\mathbf{y}'\right)\in \mathbb{W}_{N,N+1}.\label{Geom2levelInter2}
 \end{align}
\end{prop}

\begin{proof}
 The proof of (\ref{Geom2levelInter1}) is immediate from  the definition of $\mathcal{G}^{N,N+1}$. For identity
 (\ref{Geom2levelInter2}) we make repeated use of Proposition \ref{KeyComputationGeomtric}. The right hand side of (\ref{Geom2levelInter2}) is equal to
\begin{equation*}
\frac{\mathfrak{h}_{N}(\mathbf{x}')}{\mathfrak{h}_{N+1}(\mathbf{y})} \sum_{\mathbf{x}\prec \mathbf{y}} \prod_{i=1}^N \frac{1}{a_{x_{i}}}\prod_{i=1}^N \mathsf{T}(x_i,x_i')\mathbf{1}_{\mathbf{x}\prec \mathbf{x}'}\mathcal{R}^{N,N+1}(\mathbf{y}';\mathbf{x}',\mathbf{x},\mathbf{y}).
\end{equation*}
Now, recall that 
\begin{equation*}
 \mathsf{P}^{(N)}(\mathbf{x},\mathbf{x}')=\det\left(\mathsf{T}(x_i,x_j')\right)_{i=1}^N=\prod_{i=1}^N \mathsf{T}(x_i,x_i')\mathbf{1}_{\mathbf{x}\prec \mathbf{x}'}.
\end{equation*}
Moreover, we can restrict the sum above to a sum over $x_i\in \llbracket y_i',x_i'\wedge (y_{i+1}-1)\rrbracket$ for otherwise the summand is zero and use the form of $\mathcal{R}^{N,N+1}$ from Proposition \ref{RGeometricExplicitTrans} to obtain that the right hand side of (\ref{Geom2levelInter2}) is equal to
\begin{align*}
\frac{\mathfrak{h}_{N}(\mathbf{x}')}{\mathfrak{h}_{N+1}(\mathbf{y})} \sum_{x_i\in \llbracket y_i',x_i'\wedge (y_{i+1}-1) \rrbracket} \prod_{i=1}^N \frac{1}{a_{x_{i}}} \prod_{i=1}^N \mathsf{T}(x_i,x_i')\mathsf{Blk}\left(y_1',y_1,x_1\right)\\
\times\prod_{n=2}^{N}\mathsf{b}\left(x_n,y_n'\right)\mathsf{Psh}\left(y_{n}',y_{n},x_{n-1}'\right) \mathsf{Psh}\left(y_{N+1}',y_{N+1},x_N'\right).
\end{align*}
Consider the factor
\begin{align*}
\frac{1}{a_{x_1}}\mathsf{T}(x_1,x_1')\mathsf{Blk}(y_1',y_1,x_1)\mathsf{Psh}(y_2',y_2,x_1').
\end{align*}
Summing over $x_1 \in \llbracket y_1',x_1'\wedge (y_2-1)\rrbracket$ using Proposition \ref{KeyComputationGeomtric} we get that this is equal to 
\begin{equation*}
\frac{1}{a_{x_1}'}\mathsf{T}(y_1,y_1')\mathsf{T}(y_2,y_2')\mathbf{1}_{y_1\le y_1'<y_2\le y_2'}\mathbf{1}_{y_1'\le x_1'<y_2'}.
\end{equation*}
We can then combine the factor $\mathsf{T}(y_2,y_2')$ with $\mathsf{b}(x_2,y_2')$ to get $\mathsf{Blk}(y_2',y_2,x_2)$ and perform the sum over $x_2 \in \llbracket y_2',x_2'\wedge (y_3-1)\rrbracket$ using Proposition \ref{KeyComputationGeomtric}. If we keep iterating this procedure of using Proposition \ref{KeyComputationGeomtric} we obtain that the right hand side of (\ref{Geom2levelInter2}) is equal to 
\begin{align*}
\frac{\mathfrak{h}_N(\mathbf{x}')}{\mathfrak{h}_{N+1}(\mathbf{y})}\prod_{i=1}^N\frac{1}{a_{x_i'}}\prod_{i=1}^{N+1}\mathsf{T}\left(y_i,y_i'\right)
\mathbf{1}_{\mathbf{y}\prec \mathbf{y}'}\mathbf{1}_{\mathbf{x}'\prec \mathbf{y}'}=\left[\mathfrak{P}^{(N+1)}\mathfrak{L}_{N+1,N}\right]\left(\mathbf{y},(\mathbf{x}',\mathbf{y}')\right),
\end{align*}
which concludes the proof.
\end{proof}

\subsection{Continuous-time pure-birth chain coupling}

We have the following continuous-time analogue for pure-birth chains of our previous discrete-time results. This result was proven in \cite{DeterminantalStructures}.

\begin{prop}\label{CtsTimeInterProp}
Consider the continuous-time stochastic process $\left(\mathsf{X}(t),\mathsf{Y}(t);t\ge 0\right)$  in $\mathbb{W}_{N,N+1}$. The $\mathsf{X}$-component evolves autonomously according to the semigroup $\left(\mathfrak{P}^{(N)}_{\exp(-tz)}\right)_{t\ge 0}$.
The $\mathsf{Y}$-component evolves as $N+1$ independent pure-birth chain with jump rate $a_x$ at $x\in \mathbb{Z}_+$ with the following interactions with the $\mathsf{X}$-component:
\begin{itemize}
    \item If the clock of $\mathsf{Y}_i$ rings for it to jump and $\mathsf{Y}_i=\mathsf{X}_i$ then nothing happens (block).
    \item If $\mathsf{X}_{i-1}=x$ and $\mathsf{Y}_i=x+1$ and $\mathsf{X}_{i-1}$ jumps to $x+1$ then $\mathsf{Y}_i$ instantaneously moves to $x+2$ (push).
\end{itemize}
Then, the time-homogeneous transition probabilities of $\left(\mathsf{X}(t),\mathsf{Y}(t);t\ge 0\right)$ from $(\mathbf{x},\mathbf{y})\in \mathbb{W}_{N,N+1}$ to $(\mathbf{x}',\mathbf{y}')\in \mathbb{W}_{N,N+1}$ in time $t$ are given by:
\begin{equation}
\mathsf{Q}_{\exp(-tz)}^{N,N+1}\left[(\mathbf{x},\mathbf{y}),(\mathbf{x}',\mathbf{y}')\right]=\frac{\mathfrak{h}_N(\mathbf{x}')}{\mathfrak{h}_N(\mathbf{x})}\mathcal{Q}_{\exp(-tz)}^{N,N+1}\left[(\mathbf{x},\mathbf{y}),(\mathbf{x}',\mathbf{y}')\right],
\end{equation}
where $\mathcal{Q}_f^{N,N+1}$ is defined as in Definition \ref{DefinitionQmatrix} and Remark \ref{RemarkTwoLevelfunction}. Moreover, we have the intertwinings:
 \begin{align}  \left[\mathsf{Q}_{\exp(-tz)}^{N,N+1}\Pi_N\right]\left((\mathbf{x},\mathbf{y}),\mathbf{x}'\right)&=\left[\Pi_N\mathfrak{P}_{\exp(-tz)}^{(N)}\right]\left((\mathbf{x},\mathbf{y}),\mathbf{x}'\right), \ &(\mathbf{x},\mathbf{y})\in \mathbb{W}_{N,N+1}, \mathbf{x}'\in \mathbb{W}_{N},\\
\left[\mathfrak{P}_{\exp(-tz)}^{(N+1)}\mathfrak{L}_{N+1,N}\right]\left(\mathbf{y},(\mathbf{x}',\mathbf{y}')\right) &=\left[\mathfrak{L}_{N+1,N}\mathsf{Q}_{\exp(-tz)}^{N,N+1}\right]\left(\mathbf{y},(\mathbf{x}',\mathbf{y}')\right), \ &\mathbf{y}\in \mathbb{W}_{N+1}, \left(\mathbf{x}',\mathbf{y}'\right)\in \mathbb{W}_{N,N+1}.
 \end{align}
\end{prop}

\begin{rmk}
Observe that, in the definition of the stochastic process $\left(\mathsf{X}(t),\mathsf{Y}(t);t\ge 0\right)$ above the interaction of the $\mathsf{Y}$-components with the $\mathsf{X}$-components is exactly given by the interaction between consecutive levels in the pure-birth push-block dynamics of Definition \ref{DefCtsTimeDynamics}. 
\end{rmk}

\begin{rmk}
An analogous result holds for $\mathcal{Q}_{\exp(-tz)}^{N,N+1}$ without the Doob-transform by $\mathfrak{h}_N$, except that now the autonomous $\mathsf{X}$-component evolves as $N$ independent pure-birth chains killed when they collide (the whole process is killed as well).
\end{rmk}

\subsection{Couplings of Borodin-Ferrari-Olshanski}\label{BorodinFerrariCouplings}

Given two Markov semigroups which are intertwined such as $\mathfrak{P}_f^{(N)}$ and $\mathfrak{P}_f^{(N+1)}$ (with an appropriate choice of $f$) there are different interesting ways to couple the corresponding stochastic processes. We now briefly survey, and compare to our setting, couplings developed\footnote{We note that in these constructions the intertwining from Proposition \ref{IntertwiningSingleLevelNorm} between $\mathfrak{P}_f^{(N)}$ and $\mathfrak{P}_f^{(N+1)}$ is taken as input. In the coalescing random walk framework the intertwining can be arrived at without apriori knowledge, see for example the discussion before Proposition \ref{DefNormalisedKernel}.} by Borodin, Ferrari and Olshanski which have been made use of in many works \cite{BorodinFerrari,SchurDynamics,BorodinOlshanski,MacDonaldProcesses}, see also \cite{BorodinPetrov,BufetovPetrovYangBaxter,BufetovMucciconiPetrov} for further developments and generalisations. Although the origin\footnote{These couplings have their origin in an idea of Diaconis and Fill \cite{DiaconisFill} from their study of strong stationary times \cite{DiaconisFill,Fill,Miclo1,Miclo2} for convergence to equilibrium for Markov chains.} of these constructions is very different from the coalescing random walk framework we have developed, it turns out that these couplings in the case of pure-birth and Bernoulli dynamics coincide with ours. On the other hand, in the case of geometric jumps their coupling is different from the Warren-Windridge geometric dynamics, which are in some sense preferable since they have more Markovian projections. 

\subsubsection{Sequential-update}

Consider the Markov kernel, see \cite{BorodinFerrari},
\begin{equation}\label{SequentialKernel}
\mathfrak{P}_f^{(N,N+1),\textnormal{Seq}}\left[(\mathbf{x},\mathbf{y}),(\mathbf{x}',\mathbf{y}')\right]=\begin{cases}
\frac{\mathfrak{P}_f^{(N)}(\mathbf{x},\mathbf{x}')\mathfrak{P}_{f}^{(N+1)}(\mathbf{y},\mathbf{y}')\mathfrak{L}_{N+1,N}(\mathbf{y}',\mathbf{x}')}{\left[\mathfrak{L}_{N+1,N}\mathfrak{P}_f^{(N)}\right](\mathbf{y},\mathbf{x}')}, &\left[\mathfrak{L}_{N+1,N}\mathfrak{P}_f^{(N)}\right](\mathbf{y},\mathbf{x}')>0,\\
0, &\textnormal{otherwise}.
\end{cases}
\end{equation}
on the state space 
\begin{equation*}
\mathbb{S}^{N,N+1}_{\textnormal{Seq}}=\left\{(\mathbf{x},\mathbf{y})\in \mathbb{W}_N\times \mathbb{W}_{N+1}\big|\mathfrak{L}_{N+1,N}(\mathbf{y},\mathbf{x})>0\right\}=\mathbb{W}_{N,N+1}.
\end{equation*}
By virtue of the intertwining from Theorem \ref{IntertwiningSingleLevelNorm} it is immediate that this is correctly normalised. Moreover, the conditional distribution of the projection on $\mathbb{W}_N$ given the projection on $\mathbb{W}_{N+1}$ is given by $\mathfrak{L}_{N+1,N}$, see \cite{BorodinFerrari}. It can be checked that in the case of Bernoulli jumps, namely with $f(z)=1-\alpha z$, the dynamics arising from (\ref{SequentialKernel}) are exactly the sequential-update push-block Bernoulli dynamics we have been considering. In particular, (\ref{SequentialKernel}) with $f(z)=1-\alpha z$ is equal to (\ref{TwoLevelNormalised}) with $s=0,t=1$ and $f_{0,1}(z)=f(z)$. In the case of geometric jumps however, the coupling obtained from (\ref{SequentialKernel}) is different from the Warren-Windridge type coupling we have been considering in this work. It can be checked that its projection on the right edge is still the geometric push-TASEP but the projection on the left edge is not even Markovian (and the evolution of the rest of the array is indeed different).

\subsubsection{Parallel-update}\label{ParallelUpdateSection}

Consider the Markov kernel, see \cite{BorodinFerrari}, 
\begin{equation}\label{ParallelKernel}
\mathfrak{P}_f^{(N,N+1),\textnormal{Par}}\left[(\mathbf{x},\mathbf{y}),(\mathbf{x}',\mathbf{y}')\right]=\frac{\mathfrak{P}_f^{(N)}(\mathbf{x},\mathbf{x}')\mathfrak{P}_{f}^{(N+1)}(\mathbf{y},\mathbf{y}')\mathfrak{L}_{N+1,N}(\mathbf{y}',\mathbf{x})}{\left[\mathfrak{L}_{N+1,N}\mathfrak{P}_f^{(N)}\right](\mathbf{y},\mathbf{x})}
\end{equation}
on the more complicated state space:
\begin{equation*}
\mathbb{S}^{N,N+1}_{\textnormal{Par}}=\left\{(\mathbf{x},\mathbf{y})\in \mathbb{W}_N\times \mathbb{W}_{N+1}\big|\left[\mathfrak{L}_{N+1,N}\mathfrak{P}_f^{(N)}\right](\mathbf{y},\mathbf{x})>0\right\}.
\end{equation*}
Again, by virtue of the intertwining from Theorem \ref{IntertwiningSingleLevelNorm} it is immediate that this correctly normalised. Also, the conditional distribution of the projection on $\mathbb{W}_N$ given the projection on $\mathbb{W}_{N+1}$ is now given by $\mathfrak{L}_{N+1,N}\mathfrak{P}_f^{(N)}$, see \cite{BorodinFerrari}.  In the case of Bernoulli jumps this gives rise, as the projection on the left edge, to parallel-update Bernoulli TASEP. It would be possible to obtain analogous results to the ones in this paper (show the existence of determinantal correlations, compute the correlation kernel and so on) for the parallel-update model but the situation is more involved 
and we will not pursue the details here.

\subsubsection{Continuous-time}

For continuous time and countable state space an analogous framework for coupling intertwined semigroups was developed\footnote{In the earlier paper of Borodin-Ferrari \cite{BorodinFerrari} the continuous-time dynamics were obtained as a scaling limit from discrete-time.} by Borodin and Olshanski in \cite{BorodinOlshanski}. The two-level dynamics are now defined through their explicit infinitesimal jump rates, see \cite{BorodinOlshanski}, instead of their transition kernels which in continuous time are almost never explicit (unlike the explicit kernels (\ref{SequentialKernel}), (\ref{ParallelKernel}) for discrete-time). It can be checked that the dynamics coming from the coupling of \cite{BorodinOlshanski}, associated to the intertwining from Theorem \ref{IntertwiningSingleLevelNorm}, actually match the pure-birth push-block dynamics we have been considering in this paper. So in particular, $\mathsf{Q}^{N,N+1}_{\textnormal{exp}(-tw)}$ is a rare example of an explicit transition kernel for the coupling framework developed in \cite{BorodinOlshanski}.

\section{Dynamics on arrays}\label{SectionDynamicsOnArrays}

\subsection{The space-time inhomogeneous case}

Combining our previous results we prove the following.

\begin{prop}\label{PropMultiLevelSpaceTime}
 Let $N\ge 1$ be arbitrary. Consider an initial configuration in $\mathbb{IA}_N$ distributed according to
 \begin{equation}\label{GibbsDistribution}
\mu_N\left(\mathbf{x}^{(N)}\right)\mathfrak{L}_{N,N-1}\left(\mathbf{x}^{(N)},\mathbf{x}^{(N-1)}\right)\cdots \mathfrak{L}_{2,1}\left(\mathbf{x}^{(2)},\mathbf{x}^{(1)}\right),
 \end{equation}
 for some probability measure $\mu_N$ on $\mathbb{W}_N$. Then, perform $M_1$ steps of sequential-update Bernoulli dynamics with parameters $\alpha_1,\dots,\alpha_{M_1}$, $M_2$ steps of Warren-Windridge geometric dynamics with parameters $\beta_1,\dots,\beta_{M_2}$ and finally continuous-time pure-birth dynamics for time $t$. The parameters $\alpha_i,\beta_i$ are assumed to satisfy $0\le \alpha_i \le (\sup_{x\in \mathbb{Z}_+}a_x)^{-1}$ and $0\le \beta_i <(\sup_{x\in \mathbb{Z}_+}a_x-\inf_{x\in \mathbb{Z}_+}a_x)^{-1}$. Then, the distribution of the resulting configuration in $\mathbb{IA}_N$ is given by
 \begin{equation}\label{EvolvedGibbsDistribution}
\left[\mu_N\mathfrak{P}^{(N)}_{\prod_{i=1}^{M_1}(1-\alpha_iw)\prod_{i=1}^{M_2}(1+\beta_iw)^{-1} \exp(-tw)}\right]\left(\mathbf{x}^{(N)}\right) \mathfrak{L}_{N,N-1}\left(\mathbf{x}^{(N)},\mathbf{x}^{(N-1)}\right)\cdots \mathfrak{L}_{2,1}\left(\mathbf{x}^{(2)},\mathbf{x}^{(1)}\right).
\end{equation}
 Finally, consider the process $\left(\mathsf{X}_i^{(n)}(t);t\ge 0\right)_{1\le i \le n; 1\le n \le N}$ in $\mathbb{IA}_N$ defined in Theorem \ref{ThmCorrelationKernelNI}. Assume it is initialised according to (\ref{GibbsDistribution}). Then, for any $1\le n \le N$, the projection on the $n$-th row evolves as a Markov process with transition probabilities $\mathfrak{P}_{s,t}^{(n)}=\mathfrak{P}_{f_{s,t}}^{(n)}$ from (\ref{SemigroupIntro}).
\end{prop}

\begin{proof}
The proof is by induction, making use of Propositions \ref{Second2LevelInterProp}, \ref{GeomInterProp} and \ref{CtsTimeInterProp}, by virtue of the Markov functions theory of Rogers-Pitman \cite{RogersPitman}. This type of argument, and variations, have been documented in many places \cite{Warren,BorodinFerrari,Toda,InterlacingDiffusions,SurfaceGrowthKarlinMcGregor,InteractingDiffusionsPosDef,arista2023matrix} and we do not repeat all details. The main point is that under such dynamics, the Gibbs-type property of the initial distribution of the array, namely the fact that the law of the first $N-1$ rows of the array given the $N$-th row being equal to $\mathbf{y}\in \mathbb{W}_N$ is given by:
\begin{equation*}
 \mathfrak{L}_{N,N-1}\left(\mathbf{y},\mathbf{x}^{(N-1)}\right) \mathfrak{L}_{N-1,N-2}\left(\mathbf{x}^{(N-1)},\mathbf{x}^{(N-2)}\right)\cdots \mathfrak{L}_{2,1}\left(\mathbf{x}^{(2)},\mathbf{x}^{(1)}\right),   
\end{equation*}
is preserved for all times $t>0$, and the projections on single levels are Markovian with the desired transition probabilities $\mathfrak{P}_{s,t}^{(n)}=\mathfrak{P}_{f_{s,t}}^{(n)}$. Then, the measure $\mu_N$ on level $N$ is evolved accordingly,
\begin{align*}
\left[\mu_N\mathfrak{P}^{(N)}_{(1-\alpha_1 w)}\cdots \mathfrak{P}^{(N)}_{(1-\alpha_{M_1} w)}\mathfrak{P}^{(N)}_{(1+\beta_{1}w)^{-1}}\cdots\mathfrak{P}^{(N)}_{(1+\beta_{M_2}w)^{-1}}\mathfrak{P}^{(N)}_{\exp(-tw)}\right](\mathbf{x})=\\
\left[\mu_N\mathfrak{P}^{(N)}_{\prod_{i=1}^{M_1}(1-\alpha_iw)\prod_{i=1}^{M_2}(1+\beta_iw)^{-1} \exp(-tw)}\right](\mathbf{x}),
\end{align*}
by virtue of Proposition \ref{PropKMcomposition}, which is where we need the (technical) assumption on the $\beta_i$ (recall the assumption on the $\alpha_i$ parameters is there for the transition probabilities to be well-defined). Observe that, the order in which we take the jumps is not important. 
\end{proof}

We note that, the deterministic law of the fully-packed configuration is of the form (\ref{GibbsDistribution}) by taking $\mu_N(\mathbf{x})=\mathbf{1}_{(0,1,\dots,N-1)}(\mathbf{x})$.

\subsection{The space-level inhomogeneous case}\label{SpaceLevelInhomogeneousSection}

We now consider the space and level inhomogeneous setting of Theorem \ref{ThmSpaceLevelInhomogeneous}. Our process evolves with either only continuous-time pure-birth dynamics, or only discrete time Bernoulli or only discrete time geometric dynamics which are all homogeneous in time. The extra new ingredient for the argument are the functions $\mathsf{H}^\bullet_{(\gamma_1,\dots,\gamma_N)}$, which are defined recursively in (\ref{RecursiveDefH_N}) and then shown to have a particular determinant expression in Proposition \ref{PropEigenfunctionSymmetryExpression} which reveals the desired parameter symmetry. Other than that, most of the work has already been done.  

Recall that, we write $\mathsf{T}_t^\bullet$ with $\bullet \in \{\textnormal{pb, B, g}\}$, where the abbreviations stand for pure-birth, Bernoulli and geometric, to denote $\mathsf{T}_f$ with the following choices of function $f(w)$: if $\bullet=\textnormal{pb}$, then $f(w)=\exp(-tw)$, if 
$\bullet=\textnormal{B}$, then $f(w)=(1-w)^t$ and if $\bullet=\textnormal{g}$, then $f(w)=(1+w)^{-t}$.

\begin{lem}\label{LemEigenfunctionDynamics} We have the following eigenfunction relations
\begin{equation*}
\mathsf{T}_t^{\bullet}h_{\gamma}^\bullet(x)=c_{t,\gamma}^\bullet h_{\gamma}^\bullet(x),
\end{equation*}
where the functions $h_\gamma^\bullet$ are given by
\begin{equation*}
 h_{\gamma}^{\textnormal{pb}}(x)=p_x(-\gamma), \ \ 
  h_{\gamma}^{\textnormal{B}}(x)=p_x\left(1-\gamma^{-1}\right),   \ \ 
 h_{\gamma}^{\textnormal{g}}(x)=p_x(\gamma^{-1}-1), 
\end{equation*}
and the constants $c_{t,\gamma}^\bullet$ by,
\begin{equation*}
  c_{t,\gamma}^{\textnormal{pb}}=e^{t\gamma}, \ \ c_{t,\gamma}^{\textnormal{B}}=\gamma^{-t}, \ \ c_{t,\gamma}^{\textnormal{g}}=\gamma^t.
\end{equation*}
For $\bullet=\textnormal{g}$ we require the technical condition $\sup_{k\in \mathbb{Z}_+}a_k-\inf_{k\in \mathbb{Z}_+}a_k<1$ on $\mathbf{a}$. Moreover, if the parameter $\gamma$ satisfies $\gamma \ge 0$ for $\bullet=\textnormal{pb}$, $0<\gamma\le 1$ for $\bullet=\textnormal{B}$, $\gamma \ge 1$ for $\bullet=\textnormal{g}$, then $h_{\gamma}^\bullet$ is strictly positive. 
\end{lem}

\begin{proof}
The eigenfunction relation in the pure-birth and Bernoulli cases follows directly from Lemma \ref{LemmaEigenfunction}. For the geometric case, subject to the condition on $\mathbf{a}$, it also follows from Lemma \ref{LemmaEigenfunction} but only for $\gamma$ close to $1$. Instead we argue as follows. A direct computation using the explicit formula for $\mathsf{T}_1^{\textnormal{g}}$ from Lemma \ref{LemmaComposition} gives for any $\gamma\in \mathbb{C}$,
\begin{equation*}
\mathsf{T}_1^{\textnormal{g}}h_\gamma^{\textnormal{g}}(x)=c_{1,\gamma}^{\textnormal{g}}h_\gamma^{\textnormal{g}}(x).
\end{equation*}
Then, by virtue of Lemma \ref{LemmaComposition}, subject to the condition on $\mathbf{a}$, we obtain the result for any $t \in \mathbb{Z}_+$. The final statement regarding positivity of $h^\bullet_{\gamma}(x)$ is immediate from the fact that the polynomial $p_x$ is strictly positive on $(-\infty,0]$. This gives a range of $\gamma$ which is not optimal, but we restrict to it for simplicity.
\end{proof}

Recall Definition \ref{Definition1DTransKernelwithdrift}. 

\begin{prop}
In the setting of Definition \ref{Definition1DTransKernelwithdrift} we have,
\begin{equation*}
\mathsf{T}_{t,\gamma}^\bullet(x,y)=\frac{1}{c_{t,\gamma}^\bullet} \frac{h_\gamma^\bullet(y)}{h_\gamma^\bullet(x)}\mathsf{T}_t^\bullet(x,y).
\end{equation*}   
\end{prop}
\begin{proof}
A direct computation reveals this Doob transform relation between $\mathsf{T}_{t,\gamma}^\bullet$ and $\mathsf{T}_{t}^\bullet$.
\end{proof}
Clearly, we have the relation
\begin{equation}\label{RelationTransitionKernelsDiffParameters}
\mathsf{T}_{t,{\gamma_1}}^\bullet(x,y)=\frac{c_{t,{\gamma_2}}^\bullet}{c_{t,{\gamma_1}}^\bullet} \frac{h_{\gamma_1}^\bullet(y)h_{\gamma_2}^\bullet(x)}{h_{\gamma_2}^\bullet(y)h_{\gamma_1}^\bullet(x)}\mathsf{T}_{t,\gamma_2}^\bullet(x,y).
\end{equation}
We then define the following kernels from $\mathbb{W}_{N+1}$ to $\mathbb{W}_N$ by,
\begin{align}
\Lambda_{N+1,N}^{\gamma_{N+1},\bullet}\left(\mathbf{y},\mathbf{x}\right)&=\prod_{j=1}^{N}v_{\gamma_{N+1}}^\bullet (x_j)\mathbf{1}_{\mathbf{x}\prec \mathbf{y}},\\
\Lambda_{N+1,N}^{\gamma_{N+1},\gamma_N,\bullet}\left(\mathbf{y},\mathbf{x}\right)&=\prod_{j=1}^{N}v_{\gamma_{N+1}}^\bullet (x_j)\frac{h_{\gamma_N}^{\bullet}(x_j)}{h_{\gamma_{N+1}}^{\bullet}(x_j)}\mathbf{1}_{\mathbf{x}\prec \mathbf{y}},
\end{align}
where $v_{\gamma}^\bullet(x)$ is given by:
\begin{equation*}
 v_{\gamma}^{\textnormal{pb}}(x)=(a_x+\gamma)^{-1}, \ v_{\gamma}^{\textnormal{B}}(x)=(\gamma a_x+1-\gamma)^{-1}, \ v_{\gamma}^{\textnormal{g}}(x)=(\gamma a_x-1+\gamma)^{-1}.  
\end{equation*}

\begin{defn}
Define the strictly positive function $\mathsf{H}^\bullet_{(\gamma_1,\dots,\gamma_N)}$ on $\mathbb{W}_N$ as follows:
\begin{equation}\label{RecursiveDefH_N}
\mathsf{H}^\bullet_{(\gamma_1,\dots,\gamma_N)}\left(\mathbf{x}\right)=\left[\Lambda_{N,N-1}^{\gamma_{N},\gamma_{N-1},\bullet}\Lambda_{N-1,N-2}^{\gamma_{N-1},\gamma_{N-2},\bullet}\cdots \Lambda_{2,1}^{\gamma_{2},\gamma_{1},\bullet}\mathbf{1}\right](\mathbf{x}).
\end{equation}    
\end{defn}

For the rest of this section it will be convenient to use the following notation (note that $\gamma$ here is a scalar)
\begin{equation*}
\mathcal{P}_t^{\gamma,\bullet,N}(\mathbf{x},\mathbf{y})=\det\left(\mathsf{T}^\bullet_{t, \gamma}(x_i,y_j)\right)_{i,j=1}^N.
\end{equation*}
Observe that $\mathcal{P}^{\gamma,\bullet,N}_t(\mathbf{x},\mathbf{y})$ coincides with $\mathsf{P}_f^{(N)}(\mathbf{x},\mathbf{y})$ from Definition \ref{DefKMsemigroup} with the following choices of function $f$ and inhomogeneity $\mathbf{a}$: if $\bullet=\textnormal{pb}$ then $f(w)=\exp(-tw)$ and the inhomogeneity is $\mathbf{a}+\gamma$, for $\bullet=\textnormal{B}$, then $f(w)=(1-w)^t$ and the inhomogeneity is $\gamma \mathbf{a}+1-\gamma$ and for $\bullet=\textnormal{g}$ then $f(w)=(1-w)^{-t}$ and the inhomogeneity is $\gamma \mathbf{a}+\gamma-1$. Here, for a scalar $c$, $\mathbf{a}+c$ is simply the sequence $(a_x+c)_{x\in \mathbb{Z}_+}$. Thus, by virtue of Theorem \ref{IntertwiningSingleLevel} we have the intertwining, with $\mathbf{x}\in \mathbb{W}_N$, $\mathbf{y}\in \mathbb{W}_{N+1}$,
\begin{equation*}
\mathcal{P}_t^{\gamma_{N+1},\bullet,N+1}\Lambda_{N+1,N}^{\gamma_{N+1},\bullet}\left(\mathbf{y},\mathbf{x}\right)=\Lambda_{N+1,N}^{\gamma_{N+1},\bullet}\mathcal{P}_t^{\gamma_{N+1},\bullet,N}\left(\mathbf{y},\mathbf{x}\right).
\end{equation*}
From relation (\ref{RelationTransitionKernelsDiffParameters}) we then obtain yet another intertwining:
\begin{equation}\label{IntertwiningDifferentParameters}
\mathcal{P}_t^{\gamma_{N+1},\bullet,N+1}\Lambda_{N+1,N}^{\gamma_{N+1},\gamma_N,\bullet}\left(\mathbf{y},\mathbf{x}\right)=\left(\frac{c^\bullet_{t,\gamma_N}}{c^\bullet_{t,\gamma_{N+1}}}\right)^N\Lambda_{N+1,N}^{\gamma_{N+1},\gamma_N,\bullet}\mathcal{P}_t^{\gamma_{N},\bullet,N}\left(\mathbf{y},\mathbf{x}\right).
\end{equation}
Hence, by induction we obtain that
$\mathsf{H}^\bullet_{(\gamma_1,\dots,\gamma_N)}$ is a strictly positive eigenfunction of $\mathcal{P}_t^{\gamma_N,\bullet,N}$.
\begin{lem}
We have: 
\begin{equation*}
\mathcal{P}_t^{\gamma_N,\bullet,N}\mathsf{H}^\bullet_{(\gamma_1,\dots,\gamma_N)}=\left(c^\bullet_{t,\gamma_N}\right)^{-(N-1)}\prod_{j=1}^{N-1}c^\bullet_{t,\gamma_j}\mathsf{H}^\bullet_{(\gamma_1,\dots,\gamma_N)}.   
\end{equation*}
\end{lem}
With all these preliminaries in place we arrive to the following definition. 

\begin{defn}\label{LevelInhomogeousVariousDef} For all $N\ge 1$, define the Markov kernel $\Lambda_{N+1,N}^{\boldsymbol{\gamma},\bullet}$ from $\mathbb{W}_{N+1}$ to $\mathbb{W}_N$ by 
\begin{equation}
\Lambda_{N+1,N}^{\boldsymbol{\gamma},\bullet}\left(\mathbf{y},\mathbf{x}\right)=\prod_{j=1}^{N}v_{\gamma_{N+1}}^\bullet (x_j)\frac{h_{\gamma_N}^{\bullet}(x_j)}{h_{\gamma_{N+1}}^{\bullet}(x_j)}\frac{\mathsf{H}^\bullet_{(\gamma_1,\dots,\gamma_N)}(\mathbf{x})}{\mathsf{H}^\bullet_{(\gamma_1,\dots,\gamma_{N+1})}(\mathbf{y})}\mathbf{1}_{\mathbf{x}\prec \mathbf{y}},
\end{equation}
and define abusing notation the Markov kernel $\mathcal{P}_t^{\boldsymbol{\gamma},\bullet, N}$ from $\mathbb{W}_N$ to itself by
\begin{equation}\label{SemigroupIntermediateExpression}
\mathcal{P}_t^{\boldsymbol{\gamma},\bullet, N}\left(\mathbf{x},\mathbf{y}\right)=\left(c_{t,\gamma_N}^\bullet\right)^{N-1} \prod_{j=1}^{N-1}\frac{1}{c_{t,\gamma_j}^\bullet} \frac{\mathsf{H}_{(\gamma_1,\dots,\gamma_N)}^\bullet(\mathbf{y})}{\mathsf{H}_{(\gamma_1,\dots,\gamma_N)}^\bullet(\mathbf{x})}\det \left(\mathsf{T}_{t,\gamma_N}^\bullet (x_i,y_j)\right)_{i,j=1}^N.
\end{equation}
\end{defn}

Here, we have abused notation since in principle we have already defined $\mathcal{P}_t^{\boldsymbol{\gamma},\bullet, N}$ in (\ref{ExplicitKernelDriftsIntro}). From expression (\ref{SemigroupIntermediateExpression}) the symmetry in the parameters $\gamma_1,\gamma_2,\dots,\gamma_N$ is not obvious at all due to the recursive definition of $\mathsf{H}^\bullet_{(\gamma_1,\dots,\gamma_N)}$ from (\ref{RecursiveDefH_N}). Of course, we will prove shortly in Proposition \ref{PropSemigroupParameterSymmetry} below that (\ref{ExplicitKernelDriftsIntro}) and (\ref{SemigroupIntermediateExpression}) are one and the same. The following intertwining is then immediate from (\ref{IntertwiningDifferentParameters}), with $\mathbf{y}\in \mathbb{W}_{N+1}$, $\mathbf{x}\in \mathbb{W}_N$,
\begin{equation}
\mathcal{P}_t^{\boldsymbol{\gamma},\bullet,N+1}\Lambda_{N+1,N}^{\boldsymbol{\gamma},\bullet}\left(\mathbf{y},\mathbf{x}\right)=\Lambda_{N+1,N}^{\boldsymbol{\gamma},\bullet}\mathcal{P}_t^{\boldsymbol{\gamma},\bullet,N}\left(\mathbf{y},\mathbf{x}\right),
\end{equation}
and we also note the following explicit formulae.

\begin{lem}\label{LemmaAlternativeExpressions} We have the expressions
\begin{align*}
 \mathcal{P}_t^{\boldsymbol{\gamma},\bullet, N}\left(\mathbf{x},\mathbf{y}\right)&= \prod_{j=1}^{N}\frac{1}{c_{t,\gamma_j}^\bullet} \frac{h_{\gamma_N}^\bullet(y_j)}{h_{\gamma_N}^\bullet(x_j)}\frac{\mathsf{H}_{(\gamma_1,\dots,\gamma_N)}^\bullet(\mathbf{y})}{\mathsf{H}_{(\gamma_1,\dots,\gamma_N)}^\bullet(\mathbf{x})}\det \left(\mathsf{T}_{t}^\bullet (x_i,y_j)\right)_{i,j=1}^N\\
&=\left(c^\bullet_{t,\gamma_{N+1}}\right)^N\prod_{j=1}^{N}\frac{1}{c_{t,\gamma_j}^\bullet} \frac{h_{\gamma_N}^\bullet(y_j)h_{\gamma_{N+1}}^\bullet(x_j)}{h_{\gamma_N}^\bullet(x_j)h_{\gamma_{N+1}}^\bullet(y_j)}\frac{\mathsf{H}_{(\gamma_1,\dots,\gamma_N)}^\bullet(\mathbf{y})}{\mathsf{H}_{(\gamma_1,\dots,\gamma_N)}^\bullet(\mathbf{x})}\mathcal{P}_t^{\gamma_{N+1},\bullet,N}(\mathbf{x},\mathbf{y}).
\end{align*}
\end{lem}

We now obtain an alternative expression for $\mathsf{H}^\bullet_{(\gamma_1,\dots,\gamma_N)}$ which will give us a final expression for $\mathcal{P}_t^{\boldsymbol{\gamma},\bullet,N}$ from which parameter symmetry will be obvious. The following little observation is the key ingredient.

\begin{lem}\label{LemIdentityEigenfunction}
We have the identity, $x,y\in \mathbb{Z}_+$, $y>x$,
\begin{equation}\label{IdentityEigenfunction}
\sum_{m=x}^{y-1}\frac{v_{\gamma_2}^\bullet (m)h_{\gamma_1}^\bullet(m)}{h_{\gamma_2}^\bullet(m)}=c^\bullet_{\gamma_2;\gamma_1}\left[\frac{h_{\gamma_1}^\bullet(x)}{h_{\gamma_2}^\bullet(x)}-\frac{h_{\gamma_1}^\bullet(y)}{h_{\gamma_2}^\bullet(y)}\right],
\end{equation}
where the constant $c^\bullet_{\gamma_2;\gamma_1}$ is given by 
\begin{equation*}
 c^{\textnormal{pb}}_{\gamma_2;\gamma_1}=\frac{1}{\gamma_2-\gamma_1}, \  c^{\textnormal{B}}_{\gamma_2;\gamma_1}=\frac{\gamma_1}{\gamma_1-\gamma_2}, \  c^{\textnormal{g}}_{\gamma_2;\gamma_1}=\frac{\gamma_1}{\gamma_2-\gamma_1}.
\end{equation*}
\end{lem}

\begin{proof}
This is proven by induction on $y$. For $y=x+1$ an elementary computation gives 
\begin{equation*}
   \frac{h_{\gamma_1}^\bullet(x)}{h_{\gamma_2}^\bullet(x)}-\frac{h_{\gamma_1}^\bullet(x+1)}{h_{\gamma_2}^\bullet(x+1)} =  \frac{1}{c^\bullet_{\gamma_2;\gamma_1}}\frac{v_{\gamma_2}^\bullet (x)h_{\gamma_1}^\bullet(x)}{h_{\gamma_2}^\bullet(x)}.
\end{equation*}
Then, we have a telescoping sum and the conclusion follows.
\end{proof}

\begin{prop}\label{PropEigenfunctionSymmetryExpression}
We have
\begin{align*}
\mathsf{H}_{(\gamma_1,\dots,\gamma_N)}^\bullet \left(x_1,\dots,x_N\right)=\prod_{j=2}^N\prod_{i=1}^{j-1}c_{\gamma_j;\gamma_i}^\bullet \prod_{j=1}^N \frac{1}{h_{\gamma_N}^\bullet(x_j)}\det\left(h_{\gamma_i}^\bullet(x_j)\right)_{i,j=1}^N.
\end{align*}
\end{prop}

\begin{proof}
This is obtained by induction. For the inductive step we can bring the sums inside the determinant by multinearity and then use the identity (\ref{IdentityEigenfunction}) from Lemma \ref{LemIdentityEigenfunction}.
\end{proof}

By combining the above we obtain the final expression for $\mathcal{P}_t^{\boldsymbol{\gamma},\bullet,N}$ matching (\ref{ExplicitKernelDriftsIntro}).

\begin{prop}\label{PropSemigroupParameterSymmetry} We have
\begin{equation}
\mathcal{P}_t^{\boldsymbol{\gamma},\bullet,N}(\mathbf{x},\mathbf{y})=\prod_{j=1}^{N}\frac{1}{c_{t,\gamma_j}^\bullet} \times \frac{\det\left(h_{\gamma_i}^\bullet(y_j)\right)_{i,j=1}^N}{\det\left(h_{\gamma_i}^\bullet(x_j)\right)_{i,j=1}^N}\det \left(\mathsf{T}_{t}^\bullet (x_i,y_j)\right)_{i,j=1}^N.
\end{equation}
In particular, $\mathcal{P}_t^{\boldsymbol{\gamma},\bullet,N}$ is symmetric in the parameters $\gamma_1,\gamma_2,\dots,\gamma_N$.
\end{prop}

\begin{proof}
Combine Lemma \ref{LemmaAlternativeExpressions} and Proposition \ref{PropEigenfunctionSymmetryExpression}.
\end{proof}

We need one more definition.

\begin{defn}
Define the sub-Markov kernel $\mathfrak{W}_t^{N,N+1,\bullet,\gamma_{N+1}}$ from $\mathbb{W}_{N,N+1}$ to itself as follows:
\begin{enumerate}
    \item If $\bullet=\textnormal{pb}$, then with $t\in \mathbb{R}_+$, $\mathfrak{W}_t^{N,N+1,\textnormal{pb},\gamma_{N+1}}=\mathcal{Q}_{\exp(-tw)}^{N,N+1}$, with the underlying inhomogeneity sequence required to define $\mathcal{Q}_{\exp(-tw)}^{N,N+1}$ being $\mathbf{a}+\gamma_{N+1}$.
    \item If $\bullet=\textnormal{B}$, then with $t\in \mathbb{Z}_+$, $\mathfrak{W}_t^{N,N+1,\textnormal{B},\gamma_{N+1}}=\mathcal{Q}_{(1-w)^t}^{N,N+1}$, with the underlying inhomogeneity sequence required to define $\mathcal{Q}_{(1-w)^t}^{N,N+1}$ being $\gamma_{N+1}\mathbf{a}-\gamma_{N+1}+1$.
    \item If $\bullet=\textnormal{g}$, we first define the sub-Markov kernel $\tilde{\mathcal{G}}^{N,N+1}$ by
    \begin{equation*}
\tilde{\mathcal{G}}^{N,N+1}\left[\left(\mathbf{x},\mathbf{y}\right),\left(\mathbf{x}',\mathbf{y}'\right)\right]= \frac{\mathfrak{h}_N(\mathbf{x})}{\mathfrak{h}_N(\mathbf{x}')}\mathcal{G}^{N,N+1}\left[\left(\mathbf{x},\mathbf{y}\right),\left(\mathbf{x}',\mathbf{y}'\right)\right],
    \end{equation*}
    with the underlying inhomogeneity sequence defining $\mathcal{G}^{N,N+1}$ being $\gamma_{N+1}\mathbf{a}+\gamma_{N+1}-1$ and $\beta=1$. Then, with $t\in \mathbb{Z}_+$, $\mathfrak{W}_t^{N,N+1,\textnormal{g},\gamma_{N+1}}=\tilde{\mathcal{G}}^{N,N+1}\cdots \tilde{\mathcal{G}}^{N,N+1}$ convolved with itself $t$ times. .
\end{enumerate}
\end{defn}

\begin{rmk}
The slightly more involved definition for the geometric case is because $\mathcal{G}^{N,N+1}$ is already Markov as it is the analogue of $\mathsf{Q}_f^{N,N+1}$; in fact, as mentioned earlier, they should be equal for $f(z)=(1+\beta z)^{-1}$. So we need to unnormalise it first to obtain $\tilde{\mathcal{G}}^{N,N+1}$, the analogue of $\mathcal{Q}_f^{N,N+1}$, in order to treat all three cases uniformly.
\end{rmk}

The following is simply a re-writing of Propositions \ref{Second2LevelInterProp}, \ref{GeomInterProp} and \ref{CtsTimeInterProp} in the uniform notation of $\mathfrak{W}_t^{N,N+1,\bullet,\gamma_{N+1}}$ (but we note again that the way the parameter $\gamma_{N+1}$ appears in the inhomogeneity sequence is different in each case).

\begin{prop}\label{SpaceLevelPenultimateIntertwining} Let $N\ge 1$ and $t\ge 0$. Then, we have
 \begin{align*}  \left[\mathfrak{W}_{t}^{N,N+1,\bullet,\gamma_{N+1}}\Pi_N\right]\left((\mathbf{x},\mathbf{y}),\mathbf{x}'\right)&=\left[\Pi_N\mathcal{P}_t^{\gamma_{N+1},\bullet,N}\right]\left((\mathbf{x},\mathbf{y}),\mathbf{x}'\right),\\
\left[\mathcal{P}_t^{\gamma_{N+1},\bullet,N+1}\Lambda_{N+1,N}^{\gamma_{N+1},\bullet}\right]\left(\mathbf{y},(\mathbf{x}',\mathbf{y}')\right) &=\left[\Lambda_{N+1,N}^{\gamma_{N+1},\bullet}\mathfrak{W}_{t}^{N,N+1,\bullet,\gamma_{N+1}}\right]\left(\mathbf{y},(\mathbf{x}',\mathbf{y}')\right), 
 \end{align*}
 where we view $\Lambda_{N+1,N}^{\gamma_{N+1},\bullet}$ as a Markov kernel from $\mathbb{W}_{N+1}$ to $\mathbb{W}_{N,N+1}$ as before.
\end{prop}

Observe that, by virtue of Lemma \ref{LemmaAlternativeExpressions} we can Doob transform the semigroup $\mathcal{P}_t^{\gamma_{N+1},\bullet,N}$ to $\mathcal{P}_t^{\boldsymbol{\gamma},\bullet,N}$ using the eigenfunction
\begin{equation*}
\prod_{j=1}^N \frac{h_{\gamma_n}^\bullet(x_j)}{h_{\gamma_{N+1}}^\bullet(x_j)}\mathsf{H}^\bullet_{(\gamma_1,\dots,\gamma_N)}(x_1,\dots,x_N)
\end{equation*}
having eigenvalue $\left(c^\bullet_{t,\gamma_{N+1}}\right)^{-N}\prod_{j=1}^{N}c_{t,\gamma_j}^\bullet $. Hence, we can correctly define the Markov kernel $\mathsf{D}_{t}^{N,N+1,\bullet,\gamma_{N+1}}$ on $\mathbb{W}_{N,N+1}$ by,
\begin{align*}
\mathsf{D}_{t}^{N,N+1,\bullet,\gamma_{N+1}}\left[(\mathbf{x},\mathbf{y}),(\mathbf{x}',\mathbf{y}')\right]&=\left(c^\bullet_{t,\gamma_{N+1}}\right)^N\prod_{j=1}^{n}\frac{1}{c_{t,\gamma_j}^\bullet} \frac{h_{\gamma_N}^\bullet(y_j)h_{\gamma_{N+1}}^\bullet(x_j)}{h_{\gamma_N}^\bullet(x_j)h_{\gamma_{N+1}}^\bullet(y_j)}\times \\
&\times \frac{\mathsf{H}_{(\gamma_1,\dots,\gamma_N)}^\bullet(\mathbf{y})}{\mathsf{H}_{(\gamma_1,\dots,\gamma_N)}^\bullet(\mathbf{x})}\mathfrak{W}_{t}^{N,N+1,\bullet,\gamma_{N+1}}\left[(\mathbf{x},\mathbf{y}),(\mathbf{x}',\mathbf{y}')\right].
\end{align*}
Putting everything together we obtain, by virtue of Proposition \ref{SpaceLevelPenultimateIntertwining}:

\begin{prop}\label{FinalIntertwiningSpaceLevel} Let $N\ge 1$ and $t\ge 0$. Then, we have
 \begin{align}  \left[\mathsf{D}_{t}^{N,N+1,\bullet,\gamma_{N+1}}\Pi_N\right]\left((\mathbf{x},\mathbf{y}),\mathbf{x}'\right)&=\left[\Pi_N\mathcal{P}_t^{\boldsymbol{\gamma},\bullet,N}\right]\left((\mathbf{x},\mathbf{y}),\mathbf{x}'\right), \ (\mathbf{x},\mathbf{y})\in \mathbb{W}_{N,N+1}, \mathbf{x}'\in \mathbb{W}_{N},\\
\left[\mathcal{P}_t^{\boldsymbol{\gamma},\bullet,N+1}\Lambda_{N+1,N}^{\boldsymbol{\gamma},\bullet}\right]\left(\mathbf{y},(\mathbf{x}',\mathbf{y}')\right) &=\left[\Lambda_{N+1,N}^{\boldsymbol{\gamma},\bullet}\mathsf{D}_{t}^{N,N+1,\bullet,\gamma_{N+1}}\right]\left(\mathbf{y},(\mathbf{x}',\mathbf{y}')\right), \ \mathbf{y}\in \mathbb{W}_{N+1}, \left(\mathbf{x}',\mathbf{y}'\right)\in \mathbb{W}_{N,N+1},
 \end{align}
 where we view $\Lambda_{N+1,N}^{\boldsymbol{\gamma},\bullet}$ as a Markov kernel from $\mathbb{W}_{N+1}$ to $\mathbb{W}_{N,N+1}$ as before.
\end{prop}

Theorem \ref{ThmSpaceLevelInhomogeneous} is then a direct consequence of the following proposition with the choice $\mu_N(\mathbf{x})=\mathbf{1}_{(0,1,\dots,N-1)}(\mathbf{x})$.

\begin{prop}\label{PropMultilevelSpaceLevel}
Let $N\ge 1$ be arbitrary. Consider the process $\left(\mathsf{X}_i^{(n),\bullet}(t);t\ge 0\right)_{1\le i \le n; 1 \le n \le N}$ from Definition \ref{DefSpaceLevelIhomogeneous} in $\mathbb{IA}_N$. We assume that the sequence $\boldsymbol{\gamma}$ satisfies (\ref{ParameterRanges}) and if $\bullet=\textnormal{B}$ then $\sup_{x\in \mathbb{Z}_+}a_x\le 1$ or if $\bullet=\textnormal{g}$ then $\sup_{x\in \mathbb{Z}_+}a_x-\inf_{x\in \mathbb{Z}_+}a_x<1$. Suppose the initial condition in $\mathbb{IA}_N$ is of the form,
\begin{equation}
\mu_N\left(\mathbf{x}^{(N)}\right)\Lambda_{N,N-1}^{\boldsymbol{\gamma},\bullet}\left(\mathbf{x}^{(N)},\mathbf{x}^{(N-1)}\right) \cdots \Lambda_{2,1}^{\boldsymbol{\gamma},\bullet}\left(\mathbf{x}^{(2)},\mathbf{x}^{(1)}\right),
\end{equation}
for some probability measure $\mu_N$ on $\mathbb{W}_N$. Then, for any $1\le n \le N$, the projection on the $n$-th row $\left(\mathsf{X}^{(n),\bullet}(t);t\ge 0\right)$
evolves as a Markov process with transition probabilities $\left(\mathcal{P}_t^{\boldsymbol{\gamma},\bullet,n}\right)_{t\ge 0}$ given in (\ref{ExplicitKernelDriftsIntro}) and the distribution of $\left(\mathsf{X}^{(1),\bullet}(\mathfrak{t}),\dots,\mathsf{X}^{(N),\bullet}(\mathfrak{t})\right)$ in $\mathbb{IA}_N$ for fixed time $\mathfrak{t}$ is given by
\begin{equation}
\left[\mu_N\mathcal{P}_\mathfrak{t}^{\boldsymbol{\gamma},\bullet,N}\right]\left(\mathbf{x}^{(N)}\right)\Lambda_{N,N-1}^{\boldsymbol{\gamma},\bullet}\left(\mathbf{x}^{(N)},\mathbf{x}^{(N-1)}\right) \cdots \Lambda_{2,1}^{\boldsymbol{\gamma},\bullet}\left(\mathbf{x}^{(2)},\mathbf{x}^{(1)}\right).
\end{equation}
\end{prop}

\begin{proof}
Similarly to the proof of Proposition \ref{PropMultiLevelSpaceTime} this is obtained by induction using Proposition \ref{FinalIntertwiningSpaceLevel} now instead, see \cite{Warren,BorodinFerrari} for completely analogous arguments. As in Proposition \ref{PropMultiLevelSpaceTime}, the extra condition for $\bullet=\textnormal{B}$ is so that the transition probabilities are positive and for $\bullet=\textnormal{g}$ so that we can use our convolution of kernels result. 
\end{proof}

\section{Computation of the correlation kernels}\label{SectionComputationKernels}

By virtue of display (\ref{EvolvedGibbsDistribution}) in Proposition \ref{PropMultiLevelSpaceTime}, with $\mu_N(\mathbf{x})=\mathbf{1}_{\mathbf{x}=(0,1,\dots,N-1)}$, a simple computation gives that, after running the dynamics described in Theorem \ref{ThmCorrelationKernelArray} we can write out the resulting distribution of the array $\left(\mathsf{X}_i^{(n)}\right)_{1\le i \le n; 1\le n \le N}$, for all $N\ge 1$, explicitly as a product of determinants. We then apply the Eynard-Mehta theorem \cite{EynardMehta,BorodinRains,JohanssonDeterminantal} in the form that can be found in Lemma 3.4 of \cite{BorodinFerrariPrahoferSasamoto} (we do not recall this theorem explicitly as it has been documented many times). This immediately gives the existence of a determinantal point process structure. Finding an explicit expression for the correlation kernel $\mathfrak{K}_f$ is our next task. We need to introduce some notation. 
\begin{defn}
Define the functions $\Psi_{N-i}^{(N)}$, for $i=1,\dots, N$, by 
\begin{equation}
\Psi_{N-i}^{(N)}(x)=-\frac{1}{a_x}\frac{1}{2\pi \textnormal{i}}\oint_{\mathsf{C}_{\mathbf{a},0}}\frac{w^{N-i}f(w)}{p_{x+1}(w)}dw, \ x \in \mathbb{Z}_+,
\end{equation}
with $f(w)=\prod_{i=1}^{M_1}(1-\alpha_iw)\prod_{i=1}^{M_2}(1+\beta_iw)^{-1} \exp(-tw)$ as first defined in (\ref{FunctionfDef}).    
\end{defn}

Up to a multiplicative constant, $\mathfrak{P}_f^{(N)}((0,\dots,N-1),\mathbf{x})$ is equal to $\det(\Psi_{N-j}^{(N)}(x_i))_{i,j=1}^N$. Note that, the choice of the $\mathsf{C}_\mathbf{a}$ contour in the integral also gives the same functions $\Psi$ (however the choice of $\mathsf{C}_{\mathbf{a},0}$ will be important below). Define the kernel $\phi^{(n)}(x_1,x_2)$ as the convolution of $\phi$ from Lemma \ref{LemmaPhiRep} with itself $n$ times.
\begin{defn}
 For $1\le n \le N-1$ and $1\le j \le N$ define the functions $\Psi_{n-j}^{(n)}$ by convolution with $\phi^{(N-n)}$:
\begin{equation*}
\Psi_{n-j}^{(n)}(y)=\left[\phi^{(N-n)}\Psi_{N-j}^{(N)}\right](y), \ y\in \mathbb{Z}_+.
\end{equation*}   
\end{defn}

We will make use of the following intermediate lemmas. 
\begin{lem}
Let $1\le n \le N$ and $1\le j \le N$. Then, we have, with $f$ as in (\ref{FunctionfDef}),
\begin{equation}
\Psi_{n-j}^{(n)}(x)=-\frac{1}{a_x}\frac{1}{2\pi \textnormal{i}}\oint_{\mathsf{C}_{\mathbf{a},0}} \frac{w^{n-j}f(w)}{p_{x+1}(w)}dw.
\end{equation}
\end{lem}

\begin{proof}
We prove this by induction. We need to show
\begin{equation*}
\left[\phi\Psi_{n-j}^{(n)}\right](x)=\Psi_{n-1-j}^{(n-1)}(x). 
\end{equation*}
The left hand side is equal to 
\begin{equation*}
-\frac{1}{a_x}\frac{1}{2\pi \textnormal{i}}\sum_{y>x}-\frac{1}{a_y}\oint_{\mathsf{C}_{\mathbf{a},0}} \frac{w^{n-1-j}f(w)}{p_{y+1}(w)}dw.
\end{equation*}
We deform the contour $\mathsf{C}_{\mathbf{a},0}$ to the contour $\tilde{\mathsf{C}}_{\mathbf{a}}$ from the proof of Lemma \ref{LemmaNormalisation}. Then, we can bring the sum inside the integral, use identity (\ref{SeriesIdentity}) and deform the contour back to  $\mathsf{C}_{\mathbf{a},0}$, without crossing any poles, which gives $\Psi_{n-1-j}^{(n-1)}(x)$ and concludes the proof.
\end{proof}

\begin{lem}
Let $1 \le k \le N$. Then, we have
\begin{equation*}
\phi^{(k)}(y,x)=-\frac{1}{a_{y}}\frac{1}{2\pi \textnormal{i}}\oint_{\mathsf{C}_{\mathbf{a},0}} \frac{p_x(w)}{p_{y+1}(w)w^k}dw.
\end{equation*}
\end{lem}

\begin{proof}
This is Lemma 3.5 in \cite{DeterminantalStructures}. 
\end{proof}

\begin{lem}
Let $1\le k \le N$. Then, we have
\begin{equation*}
\phi^{(k)}(\mathsf{virt},x)=\frac{1}{2\pi \textnormal{i}}\oint_{\mathsf{C}_0}\frac{p_x(w)}{w^k}dw.
\end{equation*}
\end{lem}

\begin{proof}
This is Lemma 3.6 in \cite{DeterminantalStructures}.
\end{proof}

\begin{defn}\label{PhiFunctionDefinition} Let $1\le n \le N$ and $1\le j \le n$. Define the following function, with $f$ as in (\ref{FunctionfDef}),
\begin{equation}
  \Phi_{n-j}^{(n)}(x)=\frac{1}{2\pi \textnormal{i}}\oint_{\mathsf{C}_0}\frac{p_x(w)}{f(w)w^{n-j+1}}dw, \ x\in \mathbb{Z}_+.
\end{equation}
\end{defn}

\begin{prop}\label{PropBiorthogonality}
Let $1\le n \le N$. Then, we have
\begin{align*}
\sum_{x=0}^\infty \Psi_i^{(n)}(x)\Phi_j^{(n)}(x)=\mathbf{1}_{i=j},
\end{align*}
for $0\le i,j \le n-1$.
\end{prop}

\begin{proof}
Observe that since $i$ is non-negative we can use the $\mathsf{C}_\mathbf{a}$ contour in the definition of $\Psi_i^{(n)}$. Then, by deforming, if required, the contour $\mathsf{C}_{0}$ to a contour inside the region $\mathcal{U}$ defined in (\ref{Rectangle}), we can make use of Lemma \ref{LemmaExpansion}, to compute with $g(w)=f(w)w^i\in \mathsf{Hol}\left(\mathbb{H}_{-R}\right)$ where $R>R(\mathbf{a})$,
\begin{align*}
\sum_{x=0}^\infty \Psi_i^{(n)}(x)\Phi_j^{(n)}(x)&=\sum_{x=0}^\infty-\frac{1}{a_x}\frac{1}{2\pi \textnormal{i}}\oint_{\mathsf{C}_{\mathbf{a},0}} \frac{w^{i}f(w)}{p_{x+1}(w)}dw\frac{1}{2\pi \textnormal{i}}\oint_{\mathsf{C}_0}\frac{p_x(u)}{f(u)u^{j+1}}du\\
&=\frac{1}{2\pi \textnormal{i}}\oint_{\mathsf{C}_0}\frac{1}{f(u)u^{j+1}}\sum_{x=0}^\infty p_x(u)\mathsf{T}_{g}(x) du\\
&=\frac{1}{2\pi \textnormal{i}}\oint_{\mathsf{C_0}}\frac{1}{f(u)u^{j+1}} f(u)u^idu=\mathbf{1}_{i=j},
\end{align*}
as desired.
\end{proof}

\begin{lem}\label{LemmaLinearSpan}
The functions $\left\{\Phi_j^{(n)}(\cdot); 0\le j \le n-1\right\}$ span the space 
\begin{equation*}
\textnormal{span}\left\{\phi^{(1)}(\mathsf{virt},\cdot), \phi^{(2)}(\mathsf{virt},\cdot),\dots,\phi^{(n)}(\mathsf{virt},\cdot)\right\}.
\end{equation*}
\end{lem}

\begin{proof}
Using the Cauchy integral formula we see that
\begin{equation*}
   \textnormal{span}\left\{\phi^{(1)}(\mathsf{virt},\cdot), \phi^{(2)}(\mathsf{virt},\cdot),\dots,\phi^{(n)}(\mathsf{virt},\cdot)\right\}=\textnormal{span}\left\{x\mapsto\frac{d^{k-1}}{dw^{k-1}}p_x(w)\bigg|_{w=0}; 1\le k \le n\right\}. 
\end{equation*}
Again, using the Cauchy integral formula, since $f$ has no zeroes in $\mathsf{C}_0$, we have
\begin{equation*}
\Phi_j^{(n)}(x)=\frac{1}{j!}\frac{d^j}{dw^j}\left(\frac{p_x(w)}{f(w)}\right)\bigg|_{w=0}=\frac{1}{j!}\frac{d^j}{dw^j}p_x(w)\bigg|_{w=0}+\sum_{k=0}^{j-1}c_{k,j} \frac{d^k}{dw^k}p_x(w)\bigg|_{w=0}.
\end{equation*}
Thus, we get 
\begin{equation*}
\textnormal{span}\left\{\Phi_j^{(n)}(\cdot); 0\le j \le n-1\right\}=\textnormal{span}\left\{x\mapsto\frac{d^{k-1}}{dw^{k-1}}p_x(w)\bigg|_{w=0}; 1\le k \le n\right\}
\end{equation*}
as required.
\end{proof}

\begin{lem}
 For $1\le n \le N$, we have
 \begin{equation*}
 \phi(\mathsf{virt},\cdot)=\Phi_0^{(n)}(\cdot)\equiv 1.
 \end{equation*}
\end{lem}
\begin{proof}
This is due to the fact that $f(0)=1$ and the fact that $f$ has no zeroes in $\mathsf{C}_0$.
\end{proof}

\begin{proof}[Proof of Theorem \ref{ThmCorrelationKernelArray}]
We apply a variant of the Eynard-Mehta theorem, in exactly the form found in Lemma 3.4 of \cite{BorodinFerrariPrahoferSasamoto}, by virtue of all the preceding results in this section, to obtain the explicit form of the correlation kernel $\mathfrak{K}_f$ as follows:
\begin{equation*}
  \mathfrak{K}_f\left[(n_1,x_1);(n_2,x_2)\right]=-\phi^{(n_2-n_1)}(x_1,x_2)\mathbf{1}_{n_2>n_1}+\sum_{k=1}^{n_2}\Psi_{n_1-k}^{(n_1)}(x_1)\Phi_{n_2-k}^{(n_2)}(x_2).
\end{equation*}
We need to simplify the sum 
\begin{align*}
\sum_{k=1}^{n_2}\Psi_{n_1-k}^{(n_1)}(x_1)\Phi_{n_2-k}^{(n_2)}(x_2)&=-\frac{1}{a_{x_1}}\frac{1}{(2\pi \textnormal{i})^2}\sum_{k=1}^{n_2}\oint_{\mathsf{C}_{\mathbf{a},0}} \frac{f(w)w^{n_1-k}}{p_{x_1+1}(w)}dw\oint_{\mathsf{C}_0}\frac{p_{x_2}(u)}{f(u)u^{n_2-k+1}}du\\
&=-\frac{1}{a_{x_1}}\frac{1}{(2\pi \textnormal{i})^2}\oint_{\mathsf{C}_{\mathbf{a},0}}dw\oint_{\mathsf{C}_0}du\frac{p_{x_2}(u)f(w)}{p_{x_1+1}(w)f(u)}\sum_{k=1}^{\infty}\frac{w^{n_1-k}}{u^{n_2-k+1}}\\
&=-\frac{1}{a_{x_1}}\frac{1}{(2\pi \textnormal{i})^2}\oint_{\mathsf{C}_{\mathbf{a},0}}dw\oint_{\mathsf{C}_0}du\frac{p_{x_2}(u)f(w)}{p_{x_1+1}(w)f(u)}\frac{w^{n_1}}{u^{n_2}}\frac{1}{w-u},
\end{align*}
where we have extended the sum over $k$ from $n_2$ to infinity. This is allowed because there are no additional contributions for $k>n_2$ since there are no poles in $u$ in $\mathsf{C}_0$ and thus the $u$-contour integral vanishes.  Moreover, note that since $|u|<|w|$ the geometric series converges. This concludes the proof.
\end{proof}

\begin{lem}\label{H_NRep}
The functions $\mathfrak{h}_N$ can be written as
\begin{equation}
\mathfrak{h}_N\left(\mathbf{x}\right)=\det\left(\phi^{(N+1-i)}\left(\mathsf{virt},x_j\right)\right)_{i,j=1}^N=\det\left(\frac{1}{(i-1)!}\left(-\frac{d}{dw}\right)^{i-1}p_{x_j}(w)\big|_{w=0}\right)_{i,j=1}^N.
\end{equation}
\end{lem}

\begin{proof}
This easily follows by the formula for $\mathfrak{h}_N(\mathbf{x})$ from Definition \ref{RecursiveDefh_n}.
\end{proof}

\begin{prop}\label{PropMultiTimeDistribution} Consider the process $\left(\mathsf{X}_k^{(n)}(t);t \ge 0\right)_{1\le k \le n; n\ge 1}$ with discrete-time Bernoulli or geometric dynamics determined by the functions $f_{i,i+1}(z)$ as in Theorem \ref{ThmCorrelationKernelNI}. Let $M\ge1$ and $t_1<t_2<\cdots<t_M$ be arbitrary times. The joint distribution of $\left(\mathsf{X}^{(N)}(t_1),\dots,\mathsf{X}^{(N)}(t_M)\right)$ is then given by
\begin{align*}
\frac{1}{Z}\det\left(\tilde{\Phi}^{(t_1)}_{j}\left(x_j^{(1)}\right)\right)_{i,j=1}^N \prod_{r=1}^{M-1}\det\left(\mathsf{T}_{f_{t_{r},t_{r+1}}}\left(x^{(r)}_i,x^{(r+1)}_j\right)\right)_{i,j=1}^N \det\left(\tilde{\Psi}^{(t_M)}_{j}\left(x_j^{(M)}\right)\right)_{i,j=1}^N,
\end{align*}
where the functions $\tilde{\Psi}_j$ and $\tilde{\Phi}_j$ are given by:
\begin{align*}
\tilde{\Phi}^{(t_1)}_{j}(x)&=-\frac{1}{2\pi \textnormal{i}}\frac{1}{a_x}\oint_{\mathsf{C}_{\mathbf{a},0}}\frac{f_{0,t_1}(w)w^{j-1}}{p_{x+1}(w)}dw,\\
 \tilde{\Psi}^{(t_M)}_{j}(x)&=\frac{1}{2\pi \textnormal{i}}\oint_{\mathsf{C}_0}\frac{p_x(w)}{f_{0,t_M}(w)w^{j}}dw,
\end{align*}
and $Z$ is some normalisation constant. In the case of $\left(\mathsf{X}_k^{(n)}(t);t \ge 0\right)_{1\le k \le n; n\ge 1}$ following the continuous-time pure-birth dynamics the conclusion is exactly the same but with $t_i\in \mathbb{R}_+$ and $f_{t_1,t_2}(z)=e^{-(t_2-t_1)z}$ instead. 
\end{prop}

\begin{proof}
By virtue of Proposition \ref{PropMultiLevelSpaceTime}, we also note Lemma \ref{H_NRep}, $\left(\mathsf{X}^{(N)}(t);t\ge 0\right)$ evolves as a Markov process with transition probabilities given by (\ref{SemigroupIntro}). Then, by using Markov property we can write down the following formula for the distribution of $\left(\mathsf{X}^{(N)}(t_1),\dots,\mathsf{X}^{(N)}(t_M)\right)$,
\begin{align*}
\frac{1}{\tilde{Z}}\det\left(\mathsf{T}_{f_{0,t_1}}\left(i-1,x^{(1)}_j\right)\right)_{i,j=1}^N\det\left(\mathsf{T}_{f_{t_1,t_2}}\left(x^{(1)}_i,x^{(2)}_j\right)\right)_{i,j=1}^N \cdots\det\left(\mathsf{T}_{f_{t_{M-1},t_{M}}}\left(x^{(1)}_i,x^{(2)}_j\right)\right)_{i,j=1}^N \mathfrak{h}_N\left(x^{(N)}\right),
\end{align*}
for some normalisation constant $\tilde{Z}$. Then, by virtue of Lemma \ref{LemmaLinearSpan} (in particular, since the $\tilde{\Psi}_i^{(t_M)}(z)$ functions, which are essentially the $\Phi$ functions from Definition \ref{PhiFunctionDefinition}, span the space spanned by  the $\phi^{(i)}(\mathsf{virt},z)$ functions) and simple row operations on the determinants defining the first and last factors we obtain the desired result. Finally, the exact same argument gives the conclusion in the continuous-time case.
\end{proof}

\begin{proof}[Proof of Theorem \ref{ThmCorrelationKernelNI}] It only remains to compute the correlation functions, as the probabilistic statement that projections on single levels are Markovian follows from Proposition \ref{PropMultiLevelSpaceTime}. Define, for $1\le k \le M$, the functions $\tilde{\Phi}_i^{(t_k)}$ and $\tilde{\Psi}_j^{(t_k)}$ by the convolutions
\begin{align*}
\tilde{\Phi}_j^{(t_k)}\left(x\right)&=\left[\tilde{\Phi}_j^{(t_1)}\mathsf{T}_{f_{t_1,t_2}}\cdots \mathsf{T}_{f_{t_{k-1},t_k}}\right]\left(x\right),\\
\tilde{\Psi}_j^{(t_k)}\left(x\right)&=\left[\mathsf{T}_{f_{t_k,t_{k+1}}}\cdots \mathsf{T}_{f_{t_{M-1},t_M}}\tilde{\Psi}_j^{(t_M)}\right]\left(x\right).
\end{align*}
Then, arguing as in the proof of Lemma \ref{LemmaComposition} we obtain the following explicit description for them, for any $1\le k \le M$:
\begin{align*}
\tilde{\Phi}^{(t_k)}_{j}(x)&=-\frac{1}{2\pi \textnormal{i}}\frac{1}{a_x}\oint_{\mathsf{C}_{\mathbf{a},0}}\frac{f_{0,t_k}(w)w^{j-1}}{p_{x+1}(w)}dw,\\
 \tilde{\Psi}^{(t_k)}_{j}(x)&=\frac{1}{2\pi \textnormal{i}}\oint_{\mathsf{C}_0}\frac{p_x(w)}{f_{0,t_k}(w)w^{j}}dw.
\end{align*}
Finally, for any $1\le k \le M$, we have by an obvious adaptation of Proposition \ref{PropBiorthogonality}, 
\begin{equation*}
\sum_{x=0}^\infty \tilde{\Phi}_i^{(t_k)}(x) \tilde{\Psi}_{j}^{(t_k)}(x)=\mathbf{1}_{i=j},
\end{equation*}
for $1\le i,j \le N$. Thus, by virtue of Proposition \ref{PropMultiTimeDistribution} and the Eynard-Mehta theorem in the form found for example in \cite{BorodinDeterminantal}, the correlation functions of the underlying point process are determinantal and the correlation kernel $\mathcal{K}_N$ is given by
\begin{align*}
\mathcal{K}_N\left[(s,x_1);(t,x_2)\right]&=-\mathbf{1}_{t>s}\mathsf{T}_{f_{s,t}}\left(x_1,x_2\right)+\sum_{k=1}^N \tilde{\Psi}^{(s)}_k(x_1)\tilde{\Phi}^{(t)}_k(x_2)\\
&=-\mathbf{1}_{t>s}\mathsf{T}_{f_{s,t}}\left(x_1,x_2\right)-\frac{1}{a_{x_2}}\frac{1}{(2 \pi \textnormal{i})^2}\oint_{\mathsf{C}_{\mathbf{a},0}} dw\oint_{\mathsf{C}_0} du \frac{p_{x_1}(u)f_{0,t}(w)}{p_{x_2+1}(w)f_{0,s}(u)} \sum_{k=1}^N \frac{w^{k-1}}{u^k} \\
&=-\mathbf{1}_{t>s}\mathsf{T}_{f_{s,t}}\left(x_1,x_2\right)-\frac{1}{a_{x_2}}\frac{1}{(2 \pi \textnormal{i})^2}\oint_{\mathsf{C}_{\mathbf{a},0}} dw\oint_{\mathsf{C}_0} du \frac{p_{x_1}(u)f_{0,t}(w)}{p_{x_2+1}(w)f_{0,s}(u)} \frac{w^N}{u^N}\frac{1}{w-u},
\end{align*}
where we have extended the sum over $k$ to a sum from $-\infty$ to $N$ since for $k\le 0$ there are no poles in $u$ in $\mathsf{C}_0$ and thus no additional contributions to the sum. Moreover, since $|u|<|w|$ the resulting geometric series is convergent. The proof for the continuous-time case is exactly the same but with $t_i\in \mathbb{R}_+$ and $f_{s,t}(z)=e^{-(t-s)z}$ instead.
\end{proof}

\section{Walks conditioned to never intersect}\label{SectionConditionedWalks}

In this section we prove Theorem \ref{ThmConditioning}. The reader is advised to recall the corresponding notation and definitions therein. The following proposition is the main result of this section, from which Theorem \ref{ThmConditioning} will easily follow.

\begin{prop}\label{NonCollisionProbability}
Let $N\ge 1$ be fixed and write $\boldsymbol{\gamma}=(\gamma_1,\dots,\gamma_N)$. Consider the stochastic process $\left(\mathsf{X}^{\bullet}_{\boldsymbol{\gamma}}(t);t\ge 0\right)=\left(\mathsf{x}_{\gamma_1}^\bullet(t),\dots,\mathsf{x}_{\gamma_N}^\bullet(t); t\ge 0\right)$, with $\mathsf{X}_{\boldsymbol{\gamma}}(0)=\mathbf{x}\in \mathbb{W}_N$, with coordinates $\left(\mathsf{x}_{\gamma_i}^\bullet(t);t\ge 0\right)$ being independent and having transition probabilities $\left(\mathsf{T}_{t,\gamma_i}^\bullet\right)_{t\ge 0}$. Assume that  $\boldsymbol{\gamma}=(\gamma_1,\dots,\gamma_N)$, satisfy (\ref{ParameterRanges}),
and for $i=1,\dots,N-1$, (\ref{ParameterConditions1}), (\ref{ParameterConditions2}), (\ref{ParameterConditions3}). Recall the first collision time $\tau_{\textnormal{col}}^\bullet$ defined by
\begin{equation*}
\tau_{\textnormal{col}}^{\bullet}=\inf\{t>0:\mathsf{X}^{\bullet}_{\boldsymbol{\gamma}}(t-)\nprec \mathsf{X}^{\bullet}_{\boldsymbol{\gamma}}(t)\},
\end{equation*}
where if $\bullet\in \{\textnormal{B, g}\}$ then $\mathsf{X}^{\bullet}_{\boldsymbol{\gamma}}(t-)=\mathsf{X}^{\bullet}_{\boldsymbol{\gamma}}(t-1)$ while if $\bullet=\textnormal{pb}$ then $\mathsf{X}^{\bullet}_{\boldsymbol{\gamma}}(t-)=\lim_{s\uparrow t}\mathsf{X}^{\bullet}_{\boldsymbol{\gamma}}(s)$.
Then, we have the explicit expression, with $\mathbb{P}_\mathbf{x}$ denoting the law of $\left(\mathsf{X}^{\bullet}_{\boldsymbol{\gamma}}(t);t\ge 0\right)$ starting from $\mathbf{x}$,
\begin{equation}
\mathbb{P}_{\mathbf{x}}\left(\tau_{\textnormal{col}}^\bullet =\infty \right)=\frac{\det \left(h_{\gamma_i}^\bullet(x_j)\right)_{i,j=1}^N}{\prod_{i=1}^N h_{\gamma_i}^\bullet(x_i)}.
\end{equation}
\end{prop}

\begin{proof}
By a Doob-transform \cite{Doob,RevuzYor} of the corresponding LGV/Karlin-McGregor formula \cite{KarlinMcGregor,LGV} we have
\begin{equation*}
\mathbb{P}_\mathbf{x}\left(\mathsf{X}^{\bullet}_{\boldsymbol{\gamma}}(t)=\mathbf{y},\tau_{\textnormal{col}}^{\bullet}>t\right)=\prod_{i=1}^N c_{t,\gamma_i}^\bullet \prod_{i=1}^N\frac{h_{\gamma_i}^\bullet(y_i)}{h_{\gamma_i}^\bullet(x_i)} \det\left(\mathsf{T}_t^\bullet\left(x_i,y_j\right)\right)_{i,j=1}^N.
\end{equation*}
Thus, by summing over $\mathbb{W}_N$ we obtain 
\begin{equation*}
\mathbb{P}_\mathbf{x}\left(\tau_{\textnormal{col}}^{\bullet}>t\right)=\sum_{\mathbf{y}\in \mathbb{W}_N} \prod_{i=1}^N c_{t,\gamma_i}^\bullet \prod_{i=1}^N\frac{h_{\gamma_i}^\bullet(y_i)}{h_{\gamma_i}^\bullet(x_i)} \det\left(\mathsf{T}_t^\bullet\left(x_i,y_j\right)\right)_{i,j=1}^N.
\end{equation*}
By writing out the determinant explicitly in terms of the Leibniz formula and then instead of summing  over $\mathbb{W}_N$ we sum over $\mathbb{Z}_+^N$ and subtract the sum over $\mathbb{Z}_+^N\setminus \mathbb{W}_N$ we get,
\begin{align*}
\frac{\det\left(h_{\gamma_i}^\bullet(x_j)\right)_{i,j=1}^N}{\prod_{i=}^N h_{\gamma_i}^\bullet(x_i)}-\sum_{\sigma \in \mathfrak{S}(N)}\textnormal{sgn}(\sigma) \prod_{i=1}^N\frac{h_{\gamma_i}^\bullet(y_i)}{h_{\gamma_i}^\bullet(x_i)}\mathbb{P}_{\left(x_{\sigma(1)},\dots,x_{\sigma(N)}\right)}\left(\mathsf{X}_{\boldsymbol{\gamma}}^\bullet(t)\notin \mathbb{W}_N\right).
\end{align*}
Here, $\mathfrak{S}(N)$ denotes the symmetric group on $N$ symbols. We now proceed to show that $\mathbb{P}_{\left(x_{\sigma(1)},\dots,x_{\sigma(N)}\right)}\left(\mathsf{X}_{\boldsymbol{\gamma}}^\bullet(t)\notin \mathbb{W}_N\right) \to 0$, as $t \to \infty$, for any permutation $\sigma\in \mathfrak{S}(N)$. Observe that, it suffices to prove the following.
Suppose $\left(\mathsf{x}_{\gamma_1}^\bullet(t);t\ge 0\right)$ and $\left(\mathsf{x}_{\gamma_2}^\bullet(t);t\ge 0\right)$ are independent and follow the dynamics $\left(\mathsf{T}_{t,\gamma_1}^\bullet\right)_{t\ge 0}$ and $\left(\mathsf{T}_{t,\gamma_2}^\bullet\right)_{t\ge 0}$ respectively with initial conditions $\mathsf{x}_{\gamma_i}^\bullet(0)=x_i$ where $x_1,x_2 \in \mathbb{Z}_+$ are not necessarily ordered. Then, as $t \to \infty$ we have,
\begin{equation}\label{AsymptoticOrdering}
\mathbb{P}\left(\mathsf{x}^\bullet_{\gamma_1}(t)<\mathsf{x}^\bullet_{\gamma_2}(t)\big|\left(\mathsf{x}^\bullet_{\gamma_1}(0),\mathsf{x}^\bullet_{\gamma_2}(0)\right)=(x_1,x_2)\right)\to 1,
\end{equation}
which completes the proof. This last claim can be established as follows. Let us consider the pure-birth case. We can then couple $\left(\mathsf{x}_{\gamma_1}^{\textnormal{pb}}(t);t\ge 0\right), \left(\mathsf{x}_{\gamma_2}^{\textnormal{pb}}(t);t\ge 0\right)$ with two independent standard Poisson processes $\left(\mathsf{y}_1(t);t\ge 0\right)$ and $\left(\mathsf{y}_2(t);t\ge 0\right)$ of rates $\sup_{k\in \mathbb{Z}_+} a_k+\gamma_1$ and $\inf_{k\in \mathbb{Z}_+} a_k+\gamma_2$ respectively such that almost surely,
\begin{equation*}
\mathsf{x}_{\gamma_2}^{\textnormal{pb}}(t)\ge \mathsf{y}_2(t), \ \ \mathsf{x}_{\gamma_1}^{\textnormal{pb}}(t)\le \mathsf{y}_1(t), \forall t\ge 0.
\end{equation*}
Note that, since $\mathsf{y}_1,\mathsf{y}_2$ are standard Poisson processes we have almost surely
\begin{equation*}
\frac{\mathsf{y}_1(t)}{t}\overset{t\to \infty}{\longrightarrow} \sup_{k\in \mathbb{Z}_+} a_k+\gamma_1, \ \ \frac{\mathsf{y}_2(t)}{t} \overset{t\to \infty}{\longrightarrow} \inf_{k\in \mathbb{Z}_+} a_k+\gamma_2.
\end{equation*}
Thus, we readily obtain 
\begin{equation*}
\mathbb{P}\left(\mathsf{x}^{\textnormal{pb}}_{\gamma_1}(t)<\mathsf{x}^{\textnormal{pb}}_{\gamma_2}(t)\big|\left(\mathsf{x}^{\textnormal{pb}}_{\gamma_1}(0),\mathsf{x}^{\textnormal{pb}}_{\gamma_2}(0)\right)=(x_1,x_2)\right)\ge \mathbb{P}\left(\mathsf{y}_{1}(t)<\mathsf{y}_{2}(t)|(\mathsf{y}_1(0),\mathsf{y}_2(0))=(x_1,x_2)\right)\overset{t\to \infty}{\longrightarrow} 1,
\end{equation*}
if we have $\inf_{k\in \mathbb{Z}_+} a_k+\gamma_2>\sup_{k\in \mathbb{Z}_+} a_k+\gamma_1$ which proves the claim. Finally, the Bernoulli and geometric cases follow by the exact same argument by comparing to the corresponding homogeneous dynamics. 
\end{proof}

\begin{rmk}
We note that, as long as one has (\ref{AsymptoticOrdering}) then the whole argument goes through. The coupling with the homogeneous dynamics, which is where the conditions (\ref{ParameterConditions1}), (\ref{ParameterConditions2}), (\ref{ParameterConditions3}) are required, gives a much stronger asymptotic statement than the one needed to establish (\ref{AsymptoticOrdering}).
\end{rmk}

From our previous results (but a direct proof is also possible) we easily obtain strict positivity for the probability we just computed. 

\begin{prop}\label{StrictPositivity}
We have, for any $\mathbf{x}\in \mathbb{W}_N$, with $\bullet \in \{\textnormal{pb}, \textnormal{B},\textnormal{g}\}$ and the $\gamma_i$ parameters satisfying the corresponding conditions in Proposition \ref{NonCollisionProbability},
\begin{align*}
\mathbb{P}_\mathbf{x}\left(\tau_{\textnormal{col}}^{\bullet}=\infty\right)>0.
\end{align*}
\end{prop}

\begin{proof}
This is a consequence of the formula given in Proposition \ref{NonCollisionProbability}  above along with the recursive formula for $\mathsf{H}_{(\gamma_1,\dots,\gamma_N)}^\bullet$ found in (\ref{RecursiveDefH_N}) and Proposition \ref{PropEigenfunctionSymmetryExpression}.
\end{proof}

Putting everything together we obtain Theorem \ref{ThmConditioning}.

\begin{proof} [Proof of Theorem \ref{ThmConditioning}]
Observe that, by conditioning on $\tau^\bullet_{\textnormal{col}}>t+s$, using the strong Markov property and taking $s \to \infty$ (by virtue of Proposition \ref{StrictPositivity} there are no issues of dividing by $0$), we have
\begin{align*}
\mathbb{P}\left(\mathsf{X}^{\bullet,\textnormal{n.c.}}_{\boldsymbol{\gamma}}(t)=\mathbf{y}\big|\mathsf{X}^{\bullet,\textnormal{n.c.}}_{\boldsymbol{\gamma}}(0)=\mathbf{x}\right)&=\mathbb{P}_\mathbf{x}\left(\mathsf{X}^{\bullet}_{\boldsymbol{\gamma}}(t)=\mathbf{y}\big|\tau_{\textnormal{col}}^{\bullet}=\infty\right)\\
&=\lim_{s \to \infty} \frac{\mathbb{P}_\mathbf{y}\left(\tau^{\bullet}_{\textnormal{col}}>s\right)}{\mathbb{P}_\mathbf{x}\left(\tau^{\bullet}_{\textnormal{col}}>t+s\right)}\mathbb{P}_\mathbf{x}\left(\mathsf{X}^{\bullet}_{\boldsymbol{\gamma}}(t)=\mathbf{y},\tau_{\textnormal{col}}^{\bullet}>t\right)\\
&=\frac{\mathbb{P}_\mathbf{y}\left(\tau^{\bullet}_{\textnormal{col}}=\infty\right)}{\mathbb{P}_\mathbf{x}\left(\tau^{\bullet}_{\textnormal{col}}=\infty\right)}\mathbb{P}_\mathbf{x}\left(\mathsf{X}^{\bullet}_{\boldsymbol{\gamma}}(t)=\mathbf{y},\tau_{\textnormal{col}}^{\bullet}>t\right).
\end{align*}
The conclusion then follows by plugging in the explicit formula from Proposition \ref{NonCollisionProbability}.
\end{proof}

\section{Transition kernels and determinantal processes for the edge particle systems for general initial condition}\label{SectionEdgeTransitionKernels}

In this section we prove Theorem \ref{EdgeParticleSystemsThmIntro}, by building on our preceding results, as Theorem \ref{EdgeExplicitTrans} and Theorem \ref{EdgeDeterminantal} below.

Define the Markov kernel $\mathfrak{L}^{(N)}$ from $\mathbb{W}_N$ to $\mathbb{IA}_N$ by the formula 
\begin{equation}
\mathfrak{L}^{(N)}\left(\mathbf{x},\left(\mathbf{y}^{(1)},\dots,\mathbf{y}^{(N)}\right)\right)=\mathbf{1}_{\mathbf{y}^{(N)}=\mathbf{x}}\prod_{n=1}^{N-1}\mathfrak{L}_{n+1,n}\left(\mathbf{y}^{(n+1)},\mathbf{y}^{(n)}\right).
\end{equation}
Also, define $\hat{\mathfrak{L}}^{(N)}\left(\mathbf{x},\left(\mathbf{y}^{(1)},\dots,\mathbf{y}^{(N)}\right)\right)=\mathfrak{h}_N(\mathbf{x})\mathfrak{L}^{(N)}\left(\mathbf{x},\left(\mathbf{y}^{(1)},\dots,\mathbf{y}^{(N)}\right)\right)$. Recall also the definition of $\overline{\mathbb{W}}_N$ from (\ref{WeylChamberEqualcoordinates}) which is the state space for $\left(\mathsf{X}_1^{(N)}(t),\mathsf{X}_1^{(N-1)}(t),\dots,\mathsf{X}_1^{(1)}(t);t\ge 0\right)$. We denote by $\mathcal{E}_{\textnormal{r}}^{(N)}$ and $\mathcal{E}_{\textnormal{l}}^{(N)}$ the projections on the right and left edges of $\mathbb{IA}_N$ respectively:
\begin{align*}
\mathcal{E}_{\textnormal{r}}^{(N)}\left[\left(\mathbf{x}^{(1)},\mathbf{x}^{(2)},\dots,\mathbf{x}^{(N)}\right)\right]&=\left(x_1^{(1)},x_2^{(2)},\dots,x_N^{(N)}\right)\in \mathbb{W}_N,\\
\mathcal{E}_{\textnormal{l}}^{(N)}\left[\left(\mathbf{x}^{(1)},\mathbf{x}^{(2)},\dots,\mathbf{x}^{(N)}\right)\right]&=\left(x_1^{(N)},x_1^{(N-1)},\dots,x_1^{(1)}\right)\in \overline{\mathbb{W}}_N.
\end{align*}
We can view $\mathcal{E}_{\textnormal{r}}^{(N)}$ and $\mathcal{E}_{\textnormal{l}}^{(N)}$ as Markov kernels from $\mathbb{IA}_N$ to $\mathbb{W}_N$ and $\overline{\mathbb{W}}_N$ respectively. Finally, define the Markov kernels $\mathsf{E}_{N,\textnormal{r}}$ and $\mathsf{E}_{N,\textnormal{l}}$ from $\mathbb{W}_N$ to $\mathbb{W}_N$ and $\overline{\mathbb{W}}_N$ respectively by the compositions:
\begin{equation*}
\mathsf{E}_{N,\textnormal{r}}=\mathfrak{L}^{(N)}\mathcal{E}^{(N)}_{\textnormal{r}}, \ \ \mathsf{E}_{N,\textnormal{l}}=\mathfrak{L}^{(N)}\mathcal{E}^{(N)}_{\textnormal{l}}.
\end{equation*}
The proposition below is our starting point.

\begin{prop}\label{EdgeIntertwining} In the setting of Theorem \ref{ThmCorrelationKernelNI}, with the notations and assumptions therein, recall that we denote by $\mathfrak{E}_{f_{s,t},\textnormal{r}}^{(N)}$ and $\mathfrak{E}_{f_{s,t},\textnormal{l}}^{(N)}$ the transition kernels from time $s$ to time $t$ of the autonomous systems $\left(\mathsf{X}_1^{(1)}(t),\mathsf{X}_2^{(2)}(t),\dots,\mathsf{X}_N^{(N)}(t);t\ge 0\right)$ and $\left(\mathsf{X}_1^{(N)}(t),\mathsf{X}_1^{(N-1)}(t),\dots,\mathsf{X}_1^{(1)}(t);t\ge 0\right)$ on the right and left edge of the array. Then, we have the intertwinings
\begin{align}
\mathfrak{P}_{f_{s,t}}^{(N)}\mathsf{E}_{N,\textnormal{r}}(\mathbf{x},\mathbf{y})&=\mathsf{E}_{N,\textnormal{r}}\mathfrak{E}_{f_{s,t},\textnormal{r}}^{(N)}(\mathbf{x},\mathbf{y}), \ \ \mathbf{x},\mathbf{y} \in \mathbb{W}_N, \label{RightEdgeTopInter}\\
\mathfrak{P}_{f_{s,t}}^{(N)}\mathsf{E}_{N,\textnormal{l}}(\mathbf{x},\mathbf{y})&=\mathsf{E}_{N,\textnormal{l}}\mathfrak{E}_{f_{s,t},\textnormal{l}}^{(N)}(\mathbf{x},\mathbf{y}), \ \ \mathbf{x}\in \mathbb{W}_N, \mathbf{y}\in \overline{\mathbb{W}}_N.
\end{align}
\end{prop}

\begin{proof}
Let us write $\mathfrak{IA}^{(N)}_{f_{s,t}}$ for the transition probabilities from time $s$ to time $t$ of the Markov chain $\left(\mathsf{X}_i^{(n)}(t);t \ge 0\right)_{1\le i \le n; 1\le n \le N}$ in $\mathbb{IA}_N$ from Theorem \ref{ThmCorrelationKernelNI}. Then, since the evolution of $\left(\mathsf{X}_1^{(1)}(t),\mathsf{X}_2^{(2)}(t),\dots,\mathsf{X}_N^{(N)}(t);t\ge 0\right)$ and $\left(\mathsf{X}_1^{(N)}(t),\mathsf{X}_1^{(N-1)}(t),\dots,\mathsf{X}_1^{(1)}(t);t\ge 0\right)$ is autonomous we have
\begin{align*}
\mathfrak{IA}_{f_{s,t}}^{(N)}\mathcal{E}^{(N)}_{\textnormal{r}}&=\mathcal{E}^{(N)}_{\textnormal{r}}\mathfrak{E}_{f_{s,t},\textnormal{r}}^{(N)},\\
\mathfrak{IA}_{f_{s,t}}^{(N)}\mathcal{E}^{(N)}_{\textnormal{l}}&=\mathcal{E}^{(N)}_{\textnormal{l}}\mathfrak{E}_{f_{s,t},\textnormal{l}}^{(N)}.
\end{align*}
On the other hand, by virtue of Proposition \ref{PropMultiLevelSpaceTime} we have
\begin{equation*}
\mathfrak{P}_{f_{s,t}}^{(N)}\mathfrak{L}^{(N)}=\mathfrak{L}^{(N)}\mathfrak{IA}_{f_{s,t}}^{(N)}.
\end{equation*}
Combining the equations in the above displays we obtain the desired statement.
\end{proof}

We define, for $x,y\in \mathbb{Z}_+$,
\begin{equation*}
 \psi_\textnormal{r}(x,y)=a_y^{-1}\mathbf{1}_{x\le y} \ \textnormal{ and } \ \psi_{\textnormal{l}}(x,y)=a_y^{-1}\mathbf{1}_{y<x}.    
\end{equation*}
Write $\psi_{\textnormal{r}}^{n}(x,y)$ for the convolution of $\psi_{\textnormal{r}}(x,y)$ with itself $n$ times, with $\psi^{0}_\textnormal{r}(x,y)=\mathbf{1}_{x=y}$, and similarly for $\psi_{\textnormal{l}}^{n}(x,y)$.  Define the kernels $\hat{\mathsf{E}}_{N,\textnormal{r}}$, $\hat{\mathsf{E}}_{N,\textnormal{l}}$ by
\begin{align*}
\hat{\mathsf{E}}_{N,\textnormal{r}}(\mathbf{x},\mathbf{y})&=\mathfrak{h}_N(\mathbf{x})\mathsf{E}_{N,\textnormal{r}}(\mathbf{x},\mathbf{y}),  \ \ \mathbf{x},\mathbf{y}\in \mathbb{W}_N,\\
\hat{\mathsf{E}}_{N,\textnormal{l}}(\mathbf{x},\mathbf{y})&=\mathfrak{h}_N(\mathbf{x})\mathsf{E}_{N,\textnormal{l}}(\mathbf{x},\mathbf{y}), \ \   \mathbf{x}\in \mathbb{W}_N, \mathbf{y}\in \overline{\mathbb{W}}_N.
\end{align*}

\begin{prop}\label{EdgeKernelDetExpression}
$\hat{\mathsf{E}}_{N,\textnormal{r}}$ has the following determinant expression, with $\mathbf{x},\mathbf{y}\in \mathbb{W}_N$,
\begin{equation}
\hat{\mathsf{E}}_{N,\textnormal{r}}(\mathbf{x},\mathbf{y})=\det\left(\psi_{\textnormal{r}}^{N-j}(x_i,y_j)\right)_{i,j=1}^N,
\end{equation}
and similarly $\hat{\mathsf{E}}_{N,\textnormal{l}}$, with $\mathbf{x}\in \mathbb{W}_N$, $\mathbf{y}\in \overline{\mathbb{W}}_N$ is given by 
\begin{equation}
\hat{\mathsf{E}}_{N,\textnormal{l}}(\mathbf{x},\mathbf{y})=\det\left(\psi_{\textnormal{l}}^{j-1}(x_i,y_j))\right)_{i,j=1}^N.
\end{equation}
\end{prop}

\begin{proof}
Observe that, we have 
\begin{equation*}
\hat{\mathsf{E}}_{N,\textnormal{r}}(\mathbf{x},\mathbf{y})=\sum_{\substack{\mathbf{x}^{(1)},\mathbf{x}^{(2)},\dots,\mathbf{x}^{(N-1)}\\ \mathbf{x}^{(N)}=\mathbf{x}; \ (x_1^{(1)},x_2^{(2)},\dots,x_N^{(N)})=\mathbf{y}}}\prod_{i=1}^{N-1}\mathbf{1}_{\mathbf{x}^{(i)}\prec \mathbf{x}^{(i+1)}}\prod_{j=1}^{i}\frac{1}{a_{x_{j}^{(i)}}}, \ \ \mathbf{x},\mathbf{y}\in \mathbb{W}_N.
\end{equation*}
We now observe that, each term in the sum above is in correspondence with a set of $N$ non-intersecting paths in a certain LGV graph \cite{LGV}. The LGV graph is given by the square grid $\mathbb{Z}_+\times \mathbb{N}$ with horizontal edges directed to the right and vertical edges directed down. All horizontal edges have weight $1$ and all vertical edges with horizontal co-ordinate $x$ have weight $a_x^{-1}$. We are then looking at $N$ non-intersecting paths from $(x_1,N),\dots, (x_N,N)$ to $(y_1,1),\dots,(y_N,N)$ which end with a vertical edge, see Figure \ref{LGVgraphsProof} for an illustration. Hence, by the LGV formula \cite{LGV} we obtain 
\begin{equation*}
 \hat{\mathsf{E}}_{N,\textnormal{r}}(\mathbf{x},\mathbf{y})=\det\left(\textnormal{Weight of such path } (x_i,N) \to (y_j,j)\right)_{i,j=1}^N, \ \mathbf{x},\mathbf{y}\in \mathbb{W}_N.
\end{equation*}
Finally, it is not hard to see that in this LGV graph
\begin{equation*}
\textnormal{Weight of such path } (x_i,N) \to (y_j,j) =\psi_\textnormal{r}^{N-j}(x_i,y_j)
\end{equation*}
and this completes the proof for $\hat{\mathsf{E}}_{N,\textnormal{r}}$. We now turn our attention to the left edge. Observe that,
\begin{equation*}
\hat{\mathsf{E}}_{N,\textnormal{l}}(\mathbf{x},\mathbf{y})=\sum_{\substack{\mathbf{x}^{(1)},\mathbf{x}^{(2)},\dots,\mathbf{x}^{(N-1)}\\ \mathbf{x}^{(N)}=\mathbf{x}; \ (x_1^{(N)},x_1^{(N-1)},\dots,x_1^{(1)})=\mathbf{y}}}\prod_{i=1}^{N-1}\mathbf{1}_{\mathbf{x}^{(i)}\prec \mathbf{x}^{(i+1)}}\prod_{j=1}^{i}\frac{1}{a_{x_{j}^{(i)}}}, \ \ \mathbf{x}\in \mathbb{W}_N, \mathbf{y} \in \overline{\mathbb{W}}_N.
\end{equation*}
We then observe that, each term in the sum above is in correspondence with a set of $N$ non-intersecting paths in a different LGV graph \cite{LGV}. This graph has vertex set $\mathbb{Z}_+\times \mathbb{N}$. It has horizontal edges from $(x+1,n)$ to $(x,n)$, namely directed to the left, of weight $1$. Moreover, it has diagonal edges directed from $(x+1,n+1)$ to $(x,n)$ of weight $a_x^{-1}$. We are then looking at $N$ non-intersecting paths from $(x_1,N),\dots,(x_N,N)$ to $(y_1,N),\dots,(y_N,1)$ ending with a diagonal edge, see Figure \ref{LGVgraphsProof} for an illustration. Then, by the LGV \cite{LGV} formula we get
\begin{equation*}
 \hat{\mathsf{E}}_{N,\textnormal{l}}(\mathbf{x},\mathbf{y})=\det\left(\textnormal{Weight of such path } (x_i,N) \to (y_j,N-j+1)\right)_{i,j=1}^N, \ \mathbf{x}\in \mathbb{W}_N, \mathbf{y}\in \overline{\mathbb{W}}_N.
\end{equation*}
Finally, we see that in this LGV graph
\begin{equation*}
\textnormal{Weight of such path } (x_i,N) \to (y_j,N-j+1) =\psi_\textnormal{l}^{j-1}(x_i,y_j)
\end{equation*}
and this gives the desired expression for $\hat{\mathsf{E}}_{N,\textnormal{l}}$ and completes the proof.
\end{proof}

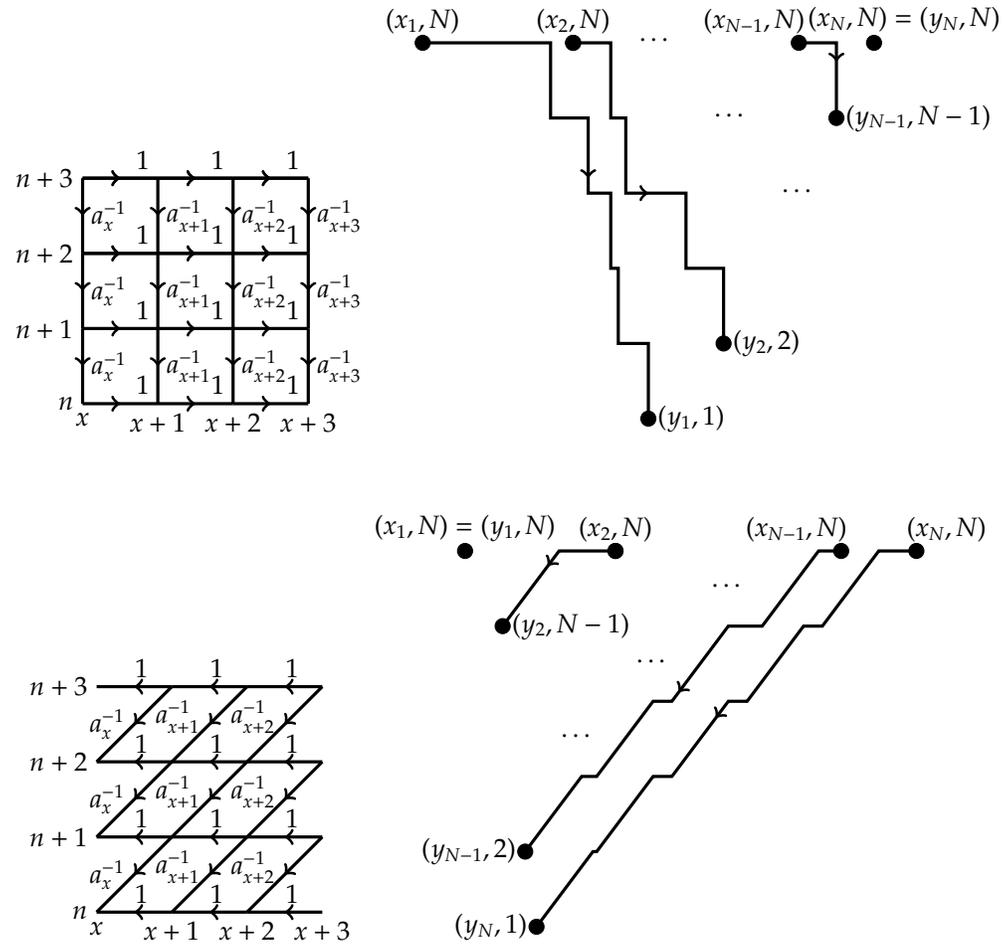
\begin{figure}
\captionsetup{singlelinecheck = false, justification=justified}
\centering
\begin{tikzpicture}
\draw[dotted] (0,0) grid (3,3);

\node[left] at (0,0) {$n$};

\node[left] at (0,1) {$n+1$};

\node[left] at (0,2) {$n+2$};

\node[left] at (0,3) {$n+3$};

\node[below] at (0,0) {$x$};

\node[below] at (1,0) {$x+1$};

\node[below] at (2,0) {$x+2$};

\node[below] at (3,0) {$x+3$};

\draw[very thick,middlearrow={>}] (0,0) -- (1,0);

\draw[very thick,middlearrow={>}] (1,0) -- (2,0);

\draw[very thick,middlearrow={>}] (2,0) -- (3,0);

\draw[very thick,middlearrow={>}] (0,1) -- (1,1);

\draw[very thick,middlearrow={>}] (1,1) -- (2,1);

\draw[very thick,middlearrow={>}] (2,1) -- (3,1);

\draw[very thick,middlearrow={>}] (0,2) -- (1,2);

\draw[very thick,middlearrow={>}] (1,2) -- (2,2);

\draw[very thick,middlearrow={>}] (2,2) -- (3,2);

\draw[very thick,middlearrow={>}] (0,3) -- (1,3);

\draw[very thick,middlearrow={>}] (1,3) -- (2,3);

\draw[very thick,middlearrow={>}] (2,3) -- (3,3);

\draw[very thick,middlearrow={>}] (0,1) -- (0,0);

\draw[very thick,middlearrow={>}] (0,2) -- (0,1);

\draw[very thick,middlearrow={>}] (0,3) -- (0,2);

\draw[very thick,middlearrow={>}] (1,1) -- (1,0);

\draw[very thick,middlearrow={>}] (1,2) -- (1,1);

\draw[very thick,middlearrow={>}] (1,3) -- (1,2);

\draw[very thick,middlearrow={>}] (0,2) -- (1,2);

\draw[very thick,middlearrow={>}] (1,2) -- (2,2);

\draw[very thick,middlearrow={>}] (2,2) -- (3,2);

\draw[very thick,middlearrow={>}] (0,3) -- (1,3);

\draw[very thick,middlearrow={>}] (1,3) -- (2,3);

\draw[very thick,middlearrow={>}] (2,3) -- (3,3);

\draw[very thick,middlearrow={>}] (2,3) -- (2,2);

\draw[very thick,middlearrow={>}] (2,2) -- (2,1);

\draw[very thick,middlearrow={>}] (2,1) -- (2,0);

\draw[very thick,middlearrow={>}] (3,3) -- (3,2);

\draw[very thick,middlearrow={>}] (3,2) -- (3,1);

\draw[very thick,middlearrow={>}] (3,1) -- (3,0);

\node[above] at (0.8,0) {$1$};

\node[above] at (1.8,0) {$1$};

\node[above] at (2.8,0) {$1$};

\node[above] at (0.8,1) {$1$};

\node[above] at (1.8,1) {$1$};

\node[above] at (2.8,1) {$1$};

\node[above] at (0.8,2) {$1$};

\node[above] at (1.8,2) {$1$};

\node[above] at (2.8,2) {$1$};

\node[above] at (0.8,3) {$1$};

\node[above] at (1.8,3) {$1$};

\node[above] at (2.8,3) {$1$};

\node[right] at (0,0.5) {$a_x^{-1}$};

\node[right] at (0,1.5) {$a_x^{-1}$};

\node[right] at (0,2.5) {$a_x^{-1}$};

\node[right] at (1,0.5) {$a_{x+1}^{-1}$};

\node[right] at (1,1.5) {$a_{x+1}^{-1}$};

\node[right] at (1,2.5) {$a_{x+1}^{-1}$};

\node[right] at (2,0.5) {$a_{x+2}^{-1}$};

\node[right] at (2,1.5) {$a_{x+2}^{-1}$};

\node[right] at (2,2.5) {$a_{x+2}^{-1}$};

\node[right] at (3,0.5) {$a_{x+3}^{-1}$};

\node[right] at (3,1.5) {$a_{x+3}^{-1}$};

\node[right] at (3,2.5) {$a_{x+3}^{-1}$};

\end{tikzpicture}
\begin{tikzpicture}
%\draw[dotted] (0,0) grid (6,5);

\draw[fill] (0,5) circle [radius=0.1];

\draw[fill] (2,5) circle [radius=0.1];

\draw[fill] (5,5) circle [radius=0.1];

\draw[fill] (6,5) circle [radius=0.1];

\node[] at (3.1,5) {$\cdots$};

\node[] at (4.1,4) {$\cdots$};

\node[] at (5,3) {$\cdots$};

\node[above] at (6.4,5) {$(x_N,N)=(y_N,N)$};

\node[above left]  at (5.2,5) {$(x_{N-1},N)$};

\node[above] at (2,5) {$(x_{2},N)$};

\node[above] at (0,5) {$(x_{1},N)$};

\draw[fill] (5.5,4) circle [radius=0.1];

\draw[fill] (2,5) circle [radius=0.1];

\draw[fill] (4,1) circle [radius=0.1];

\draw[fill] (3,0) circle [radius=0.1];

\node[right] at (5.5,4) {$(y_{N-1},N-1)$};

\node[right] at (4,1) {$(y_{2},2)$};

\node[right] at (3,0) {$(y_{1},1)$};

\draw[very thick,middlearrow={>}] (5,5) -- (5.5,5) -- (5.5,4);

\draw[very thick,middlearrow={>}] (2,5) -- (2.5,5) -- (2.5,4) -- (2.7,4) -- (2.7,3) -- (3.5, 3) -- (3.5,2) -- (4,2) -- (4,1);

\draw[very thick,middlearrow={>}] (0,5) -- (1.7,5) -- (1.7,4) -- (2.2,4) -- (2.2,3) -- (2.5, 3) -- (2.5,2)  -- (2.6,2) -- (2.6,1) -- (2.6,1) -- (3,1) -- (3,0);

\end{tikzpicture}

\bigskip

\bigskip 

\begin{tikzpicture}
%\draw[dotted] (0,0) grid (3,3);

\node[left] at (0,0) {$n$};

\node[left] at (0,1) {$n+1$};

\node[left] at (0,2) {$n+2$};

\node[left] at (0,3) {$n+3$};

\node[below] at (0,0) {$x$};

\node[below] at (1,0) {$x+1$};

\node[below] at (2,0) {$x+2$};

\node[below] at (3,0) {$x+3$};

\draw[very thick,middlearrow={>}] (1,0) -- (0,0);

\draw[very thick,middlearrow={>}] (2,0) -- (1,0);

\draw[very thick,middlearrow={>}] (3,0) -- (2,0);

\draw[very thick,middlearrow={>}] (1,1) -- (0,1);

\draw[very thick,middlearrow={>}] (2,1) -- (1,1);

\draw[very thick,middlearrow={>}] (3,1) -- (2,1);

\draw[very thick,middlearrow={>}] (1,2) -- (0,2);

\draw[very thick,middlearrow={>}] (2,2) -- (1,2);

\draw[very thick,middlearrow={>}] (3,2) -- (2,2);

\draw[very thick,middlearrow={>}] (1,3) -- (0,3);

\draw[very thick,middlearrow={>}] (2,3) -- (1,3);

\draw[very thick,middlearrow={>}] (3,3) -- (2,3);

\draw[very thick,middlearrow={>}] (1,1) -- (0,0);

\draw[very thick,middlearrow={>}] (2,1) -- (1,0);

\draw[very thick,middlearrow={>}] (3,1) -- (2,0);

\draw[very thick,middlearrow={>}] (1,2) -- (0,1);

\draw[very thick,middlearrow={>}] (2,2) -- (1,1);

\draw[very thick,middlearrow={>}] (3,2) -- (2,1);

\draw[very thick,middlearrow={>}] (1,3) -- (0,2);

\draw[very thick,middlearrow={>}] (2,3) -- (1,2);

\draw[very thick,middlearrow={>}] (3,3) -- (2,2);

\node[above] at (0.6,0) {$1$};

\node[above] at (1.6,0) {$1$};

\node[above] at (2.6,0) {$1$};

\node[above] at (0.6,1) {$1$};

\node[above] at (1.6,1) {$1$};

\node[above] at (2.6,1) {$1$};

\node[above] at (0.6,2) {$1$};

\node[above] at (1.6,2) {$1$};

\node[above] at (2.6,2) {$1$};

\node[above] at (0.6,3) {$1$};

\node[above] at (1.6,3) {$1$};

\node[above] at (2.6,3) {$1$};

\node[left] at (0.5,0.5) {$a_x^{-1}$};

\node[left] at (0.5,1.5) {$a_x^{-1}$};

\node[left] at (0.5,2.5) {$a_x^{-1}$};

\node[left] at (1.5,0.6) {$a_{x+1}^{-1}$};

\node[left] at (1.5,1.6) {$a_{x+1}^{-1}$};

\node[left] at (1.5,2.6) {$a_{x+1}^{-1}$};

\node[left] at (2.5,0.6) {$a_{x+2}^{-1}$};

\node[left] at (2.5,1.6) {$a_{x+2}^{-1}$};

\node[left] at (2.5,2.6) {$a_{x+2}^{-1}$};

\end{tikzpicture}
\begin{tikzpicture}
%\draw[dotted] (0,0) grid (6,5);

\draw[fill] (0,5) circle [radius=0.1];

\draw[fill] (2,5) circle [radius=0.1];

\draw[fill] (5,5) circle [radius=0.1];

\draw[fill] (6,5) circle [radius=0.1];

\node[] at (1.5,2.5) {$\cdots$};

\node[] at (2.5,3.5) {$\cdots$};

\node[] at (3.5,4.5) {$\cdots$};

\node[above] at (6.4,5) {$(x_N,N)$};

\node[above left]  at (5.2,5) {$(x_{N-1},N)$};

\node[above] at (2,5) {$(x_{2},N)$};

\node[above] at (0,5) {$(x_{1},N)=(y_1,N)$};

\draw[fill] (0.5,4) circle [radius=0.1];

\draw[fill] (2,5) circle [radius=0.1];

\draw[fill] (0.8,1) circle [radius=0.1];

\draw[fill] (0.95,0) circle [radius=0.1];

\node[right] at (0.5,4) {$(y_{2},N-1)$};

\node[left] at (0.8,1) {$(y_{N-1},2)$};

\node[left] at (0.95,0) {$(y_{N},1)$};

\draw[very thick,middlearrow={>}] (2,5) -- (1.25,5) -- (0.5,4);

\draw[very thick,middlearrow={>}] (5,5) -- (4.7,5) -- (3.95,4) -- (3.5,4) -- (2.75,3)-- (2.5,3) -- (1.75, 2) -- (1.55,2) -- (0.8, 1);

\draw[very thick,middlearrow={>}] (6,5) -- (5.5,5) -- (4.75,4) -- (4.5,4) -- (3.75,3) -- (3.5,3) -- (2.75,2) -- (2.5,2) -- (1.75,1)-- (1.7,1) -- (0.95,0);

\end{tikzpicture}

\caption{An illustration of the LGV graphs used in the proof of Proposition \ref{EdgeKernelDetExpression}. The top two figures correspond to the computation of $\hat{\mathsf{E}}_{N,\textnormal{r}}$ and the bottom two to the computation of $\hat{\mathsf{E}}_{N,\textnormal{l}}$.}\label{LGVgraphsProof}
\end{figure}

A simple calculation gives the following.

\begin{lem}
We define, for $x,y\in \mathbb{Z}_+$,
\begin{equation*}
\psi_{\textnormal{r}}^{-1}(x,y)=a_x(\mathbf{1}_{x=y}-\mathbf{1}_{x+1=y})\ \ and \  \ \psi_{\textnormal{l}}^{-1}(x,y)=a_x(\mathbf{1}_{x+1=y}-\mathbf{1}_{x=y}).    
\end{equation*}
Then, we have
\begin{align*}
\sum_{m=0}^\infty \psi_{\textnormal{r}}(x,m)\psi_{\textnormal{r}}^{-1}(m,y)&=\sum_{m=0}^\infty\psi_{\textnormal{r}}^{-1}(x,m)\psi_{\textnormal{r}}(m,y)=\mathbf{1}_{x=y},\\
\sum_{m=0}^\infty \psi_{\textnormal{l}}(x,m)\psi_{\textnormal{l}}^{-1}(m,y)&=\sum_{m=0}^\infty\psi_{\textnormal{l}}^{-1}(x,m)\psi_{\textnormal{l}}(m,y)=\mathbf{1}_{x=y}.
\end{align*}
\end{lem}
In particular, by virtue of the lemma above, for any $m,n\in \mathbb{Z}$, $\psi_{\textnormal{r}}^m\psi_{\textnormal{r}}^{n}(x,y)=\psi_{\textnormal{r}}^{m+n}(x,y)$ and similarly $\psi_{\textnormal{l}}^m\psi_{\textnormal{l}}^{n}(x,y)=\psi_{\textnormal{l}}^{m+n}(x,y)$. Note that, by viewing in the obvious way $\psi_{\textnormal{r}}^{-1}$ and $\psi_{\textnormal{l}}^{-1}$ as integral kernels with respect to counting measure on $\mathbb{Z}_+$, we have for $f$ on $\mathbb{Z}_+$:
\begin{equation}
\psi_{\textnormal{r}}^{-1}f(x)=-a_x\nabla^+f(x) \textnormal{ and } \psi_{\textnormal{l}}^{-1}f(x)=a_x\nabla^+f(x).
\end{equation}
We have the following explicit left inverses (we do not need the right inverse) for $\hat{\mathsf{E}}_{N,\textnormal{r}}$ and $\hat{\mathsf{E}}_{N,\textnormal{l}}$.

\begin{prop}\label{PropInverse}
The kernel $\hat{\mathsf{E}}_{N,\textnormal{r}}^{-1}$ defined by $\hat{\mathsf{E}}_{N,\textnormal{r}}^{-1}(\mathbf{x},\mathbf{y})=\det\left(\psi_{\textnormal{r}}^{-(N-i)}\left(x_i,y_j\right)\right)_{i,j=1}^N$ satisfies
\begin{equation*}
\hat{\mathsf{E}}_{N,\textnormal{r}}^{-1}\hat{\mathsf{E}}_{N,\textnormal{r}}\left(\mathbf{x},\mathbf{y}\right)=\mathbf{1}_{\mathbf{x}=\mathbf{y}}, \ \ \mathbf{x},\mathbf{y}\in \mathbb{W}_N.
\end{equation*}
Similarly, the kernel $\hat{\mathsf{E}}_{N,\textnormal{l}}^{-1}$ defined by $\hat{\mathsf{E}}_{N,\textnormal{l}}^{-1}(\mathbf{x},\mathbf{y})=\det\left(\psi_{\textnormal{l}}^{-(i-1)}\left(x_i,y_j\right)\right)_{i,j=1}^N$satisfies 
\begin{equation*}
\hat{\mathsf{E}}_{N,\textnormal{l}}^{-1}\hat{\mathsf{E}}_{N,\textnormal{l}}\left(\mathbf{x},\mathbf{y}\right)=\mathbf{1}_{\mathbf{x}=\mathbf{y}}, \ \ \mathbf{x},\mathbf{y}\in \overline{\mathbb{W}}_N.
\end{equation*}
\end{prop}

\begin{proof}
We consider the right edge first. We compute using the Cauchy-Binet formula\footnote{We note that we can use the Cauchy-Binet formula, by virtue of the form of the determinant formulae for $\mathsf{E}_{N,\textnormal{r}}^{-1}$ and $\mathsf{E}_{N,\textnormal{r}}$ since we are computing $\mathsf{E}_{N,\textnormal{r}}^{-1}\mathsf{E}_{N,\textnormal{r}}$. If we were trying to compute $\mathsf{E}_{N,\textnormal{r}}\mathsf{E}_{N,\textnormal{r}}^{-1}$ instead then the Cauchy-Binet formula actually cannot be applied directly.}
\begin{equation*}
\hat{\mathsf{E}}_{N,\textnormal{r}}^{-1}\hat{\mathsf{E}}_{N,\textnormal{r}}\left(\mathbf{x},\mathbf{y}\right)=\det\left(\psi_{\textnormal{r}}^{i-j}(x_i,y_j)\right)_{i,j=1}^N, \ \mathbf{x},\mathbf{y}\in \mathbb{W}_N.
\end{equation*}
We now show that the determinant on the right hand side boils down to $\mathbf{1}_{\mathbf{x}=\mathbf{y}}$. Clearly the diagonal terms give $\mathbf{1}_{\mathbf{x}=\mathbf{y}}$. We claim all other contributions to the determinant are zero. Take any permutation $\sigma$ of $\{1,\dots,N\}$ different from the identity. Then, there exist $i>j$ (depending on $\sigma$) such that $\sigma(i)<i$, $\sigma(j)>j$ and $\sigma(i)<\sigma(j)$. Also, we observe that $\psi_{\textnormal{r}}^{-1}(x,\cdot)$ is supported on $\{x,x+1\}$ and more generally, for $k\ge 1$,
\begin{equation}\label{ObservationSupport}
\psi_{\textnormal{r}}^{-k}(x,\cdot) \textnormal{ is supported on } \{x,x+1,\dots,x+k\}.
\end{equation}
We now show at least one of the factors in the product
\begin{equation*}
\prod_{k=1}^{N}\psi_{\textnormal{r}}^{k-\sigma(k)}\left(x_k,y_{\sigma(k)}\right)
\end{equation*}
is zero. With $i,j$ as above suppose $\psi_{\textnormal{r}}^{i-\sigma(i)}(x_i,y_{\sigma(i)})>0$, for otherwise we are done (since $i>\sigma(i)$, $\psi_\textnormal{r}^{i-\sigma(i)}(x,y)\ge 0$). Then, we must have $y_{\sigma(i)}\ge x_i$. Since moreover $x_i-x_j \ge i-j$ and $y_{\sigma(j)}-y_{\sigma(i)}\ge \sigma(j)-\sigma(i)$ we get 
\begin{equation*}
y_{\sigma(j)}-x_j\ge i-j+\sigma(j)-\sigma(i) >\sigma(j)-j
\end{equation*}
which implies by observation (\ref{ObservationSupport}) that $\psi_\textnormal{r}^{j-\sigma(j)}(x_j,y_{\sigma(j)})=0$ and completes the proof for the right edge. For the left edge, we again compute using the Cauchy-Binet formula 
\begin{equation*}
\hat{\mathsf{E}}_{N,\textnormal{l}}^{-1}\hat{\mathsf{E}}_{N,\textnormal{l}}\left(\mathbf{x},\mathbf{y}\right)=\det\left(\psi_{\textnormal{l}}^{j-i}(x_i,y_j)\right)_{i,j=1}^N, \ \mathbf{x},\mathbf{y}\in \overline{\mathbb{W}}_N.
\end{equation*}
As before, the diagonal terms give $\mathbf{1}_{\mathbf{x}=\mathbf{y}}$ and we show next that all other contributions to the determinant are zero. Again, take $\sigma$ an arbitrary permutation and (depending on $\sigma$) $i>j$ such that $\sigma(i)<i$, $\sigma(j)>j$ and $\sigma(i)<\sigma(j)$. As before, for $k\ge 1$,
\begin{equation*}
\psi_{\textnormal{l}}^{-k}(x,\cdot) \textnormal{ is supported on } \{x,x+1,\dots,x+k\}.
\end{equation*}
We show that at least one of the factors in the product 
\begin{equation*}
\prod_{k=1}^{N}\psi_{\textnormal{r}}^{\sigma(k)-k}\left(x_k,y_{\sigma(k)}\right)
\end{equation*}
is zero. With $i,j$ as above suppose $\psi_{\textnormal{l}}^{\sigma(j)-j}(x_j,y_{\sigma(j)})>0$, for otherwise we are done. This implies $y_{\sigma(j)}<x_j$. Hence, we get 
\begin{equation*}
y_{\sigma(i)}\le y_{\sigma(j)}<x_j\le x_i   
\end{equation*}
which implies, since $\sigma(i)-i<0$, $\psi_{\textnormal{l}}^{\sigma(i)-i}(x_i,y_{\sigma(i)})=0$ and completes the proof.
\end{proof}

Putting everything together we obtain the following formulae for $\mathfrak{E}_{f_{s,t},\textnormal{r}}^{(N)}$ and $\mathfrak{E}_{f_{s,t},\textnormal{l}}^{(N)}$.

\begin{thm}\label{EdgeExplicitTrans}
Assume the conditions and notation of Proposition \ref{EdgeIntertwining}. Then, the transition kernel $\mathfrak{E}_{f_{s,t},\textnormal{r}}^{(N)}$ of $\left(\mathsf{X}_1^{(1)}(t),\mathsf{X}_2^{(2)}(t),\dots,\mathsf{X}_N^{(N)}(t);t\ge 0\right)$ is given by the explicit formula
\begin{equation}
\mathfrak{E}_{f_{s,t},\textnormal{r}}^{(N)}(\mathbf{x},\mathbf{y})=\det\left(\psi_\textnormal{r}^{-(N-i)}\mathsf{T}_{f_{s,t}}\psi_\textnormal{r}^{N-j}\left(x_i,y_j\right)\right)_{i,j=1}^N, \ \mathbf{x},\mathbf{y}\in \mathbb{W}_N,
\end{equation}
while the transition kernel $\mathfrak{E}_{f_{s,t},\textnormal{l}}^{(N)}$ of $\left(\mathsf{X}_1^{(N)}(t),\mathsf{X}_1^{(N-1)}(t),\dots,\mathsf{X}_1^{(1)}(t);t\ge 0\right)$ is given by the explicit formula
\begin{equation}
\mathfrak{E}_{f_{s,t},\textnormal{l}}^{(N)}(\mathbf{x},\mathbf{y})=\det\left(\psi_\textnormal{l}^{-(i-1)}\mathsf{T}_{f_{s,t}}\psi_\textnormal{l}^{j-1}\left(x_i,y_j\right)\right)_{i,j=1}^N, \ \mathbf{x},\mathbf{y}\in \overline{\mathbb{W}}_N.
\end{equation}
\end{thm}

\begin{proof}
We give the proof for $\mathfrak{E}_{f_{s,t},\textnormal{r}}^{(N)}$ as the proof for $\mathfrak{E}_{f_{s,t},\textnormal{l}}^{(N)}$ is completely analogous. Observe that, the intertwining (\ref{RightEdgeTopInter}) can be written in terms of $\hat{\mathsf{E}}_{N,\textnormal{r}}$ and $\mathsf{P}_{f_{s,t}}^{(N)}$ instead of $\mathsf{E}_{N,\textnormal{r}}$ and $\mathfrak{P}_{f_{s,t}}^{(N)}$,
\begin{equation*}
\hat{\mathsf{E}}_{N,\textnormal{r}}\mathfrak{E}_{f_{s,t},\textnormal{r}}^{(N)}\left(\mathbf{x},\mathbf{y}\right)=\mathsf{P}_{f_{s,t}}^{(N)}\hat{\mathsf{E}}_{N,\textnormal{r}}\left(\mathbf{x},\mathbf{y}\right), \ \mathbf{x},\mathbf{y}\in \mathbb{W}_N.
\end{equation*}
Using Proposition \ref{PropInverse} we then obtain
\begin{equation}\label{SemigroupConjugation}
\mathfrak{E}_{f_{s,t},\textnormal{r}}^{(N)}\left(\mathbf{x},\mathbf{y}\right)=\hat{\mathsf{E}}_{N,\textnormal{r}}^{-1}\mathsf{P}_{f_{s,t}}^{(N)}\hat{\mathsf{E}}_{N,\textnormal{r}}\left(\mathbf{x},\mathbf{y}\right), \ \mathbf{x},\mathbf{y}\in \mathbb{W}_N.
\end{equation}
The final expression then follows by an application of the Cauchy-Binet formula using the explicit expressions found in Propositions \ref{EdgeKernelDetExpression} and \ref{PropInverse}.
\end{proof}

We now prove that $\mathfrak{E}_{f_{s,t},\textnormal{r}}^{(N)}(\mathbf{x},\cdot)$ and $\mathfrak{E}_{f_{s,t},\textnormal{l}}^{(N)}(\mathbf{y},\cdot)$ can be realised as marginals of certain measures on $\mathbb{IA}_N$ with determinantal correlation functions.

\begin{thm}\label{EdgeDeterminantal}
Assume the conditions and notation of Proposition \ref{EdgeIntertwining}.  Let $\mathbf{x}\in \mathbb{W}_N$ and $\mathbf{y}\in \overline{\mathbb{W}}_N$.  Then, the (signed) measure on $\mathbb{IA}_N$,
\begin{align}\label{RightEdgeDeterminantal}
&\left[\hat{\mathsf{E}}_{N,\textnormal{r}}^{-1}\mathsf{P}_{f_{s,t}}^{(N)}\hat{\mathfrak{L}}^{(N)}\right]\left(\mathbf{x},\left(\mathbf{z}^{(1)},\dots,\mathbf{z}^{(N)}\right)\right)\nonumber\\
 &=\det\left(\left(-a_{x_i}\nabla_{x_i}^+\right)^{N-i}\mathsf{T}_{f_{s,t}}\left(x_i,z_j^{(N)}\right)\right)_{i,j=1}^N \prod_{n=1}^{N-1}\det\left(\phi\left(z_i^{(n)},z_j^{(n+1)}\right)\right)_{i,j=1}^{n+1},
\end{align}
has $\mathfrak{E}_{f_{s,t},\textnormal{r}}^{(N)}(\mathbf{x},\cdot)$ as its right edge marginal on coordinates $(z_1^{(1)},z_2^{(2)},\dots,z_N^{(N)})$.
Moreover, the (signed) measure on $\mathbb{IA}_N$,
\begin{align}\label{LeftEdgeDeterminantal}
&\left[\hat{\mathsf{E}}_{N,\textnormal{l}}^{-1}\mathsf{P}_{f_{s,t}}^{(N)}\hat{\mathfrak{L}}^{(N)}\right]\left(\mathbf{y},\left(\mathbf{z}^{(1)},\dots,\mathbf{z}^{(N)}\right)\right)\nonumber\\
 &=\det\left(\left(a_{y_i}\nabla_{y_i}^+\right)^{i-1}\mathsf{T}_{f_{s,t}}\left(y_i,z_j^{(N)}\right)\right)_{i,j=1}^N \prod_{n=1}^{N-1}\det\left(\phi\left(z_i^{(n)},z_j^{(n+1)}\right)\right)_{i,j=1}^{n+1},\end{align}
has $\mathfrak{E}_{f_{s,t},\textnormal{l}}^{(N)}(\mathbf{y},\cdot)$ as its left edge marginal on coordinates $(z_1^{(N)},z_1^{(N-1)},\dots,z_1^{(1)})$. In particular, both $\mathfrak{E}_{f_{s,t},\textnormal{r}}^{(N)}(\mathbf{x},\cdot)$ and $\mathfrak{E}_{f_{s,t},\textnormal{l}}^{(N)}(\mathbf{y},\cdot)$ are marginals of (signed) measures with determinantal correlation functions.
\end{thm}

\begin{proof}
By virtue of (\ref{SemigroupConjugation}) and the definition of  $\mathcal{E}^{(N)}_{\textnormal{r}}$ we get that the measure $\hat{\mathsf{E}}_{N,\textnormal{r}}^{-1}\mathsf{P}_{f_{s,t}}^{(N)}\hat{\mathfrak{L}}^{(N)}\left(\mathbf{x},\cdot\right)$ on $\mathbb{IA}_N$ has $\mathfrak{E}_{f_{s,t},\textnormal{r}}^{(N)}(\mathbf{x},\cdot)$ as its right edge marginal. The expression (\ref{RightEdgeDeterminantal}) then simply follows by writing out explicitly all of the involved quantities. The argument for the left edge is completely analogous. Finally, the fact that the measures (\ref{RightEdgeDeterminantal}) and (\ref{LeftEdgeDeterminantal}) give rise to determinantal correlation functions is again a consequence of the Eynard-Mehta theorem, see \cite{BorodinDeterminantal} (which makes sense even for signed measures).
\end{proof}

\begin{rmk}
When the initial condition for the right edge system is $\mathbf{x}=(0,1,\dots,N-1)$ and similarly for the left edge system is $\mathbf{y}=(0,0,\dots,0)$, after a simple computation by noting that $ \mathfrak{h}_N((0,1,\dots,N-1);\mathbf{a})=a_0^{-(N-1)}a_1^{-(N-2)}\cdots a^{-1}_{N-2}$, (\ref{RightEdgeDeterminantal}) and (\ref{LeftEdgeDeterminantal}) are both seen to be equal to,
\begin{equation*}
\frac{\det\left(\mathsf{T}_{f_{s,t}}\left(i-1,z_j^{(N)}\right)\right)_{i,j=1}^N}{\mathfrak{h}_N((0,1,\dots,N-1);\mathbf{a})}\prod_{n=1}^{N-1}\det\left(\phi\left(z_i^{(n)},z_j^{(n+1)}\right)\right)_{i,j=1}^{n+1},
\end{equation*}
 which is nothing else than the distribution at $t$ of the corresponding push-block dynamics in $\mathbb{IA}_N$ if started from the fully-packed configuration at time $s$.
\end{rmk}

It would be interesting to solve the corresponding biorthogonalisation problem and obtain an explicit expression for the correlation kernel of the determinantal measures from Theorem \ref{EdgeDeterminantal}. This would be a substantial task, for example in the level/particle inhomogeneous setting this is the whole point of the papers \cite{MatetskiRemenik1,MatetskiRemenik2}, see also \cite{NikosTASEP}, and we leave this for future work.

Finally, consider the setting of Definition \ref{DefSpaceLevelIhomogeneous}, see also Section \ref{SpaceLevelInhomogeneousSection}. Recall that particles in this setup follow either only geometric or only Bernoulli or only pure-birth dynamics and their evolution is time-homogeneous. Also recall the notation $\bullet\in \{\textnormal{pb},\textnormal{B},\textnormal{g}\}$ corresponding to pure-birth, Bernoulli and geometric respectively. Let $\mathfrak{E}_{t,\textnormal{r}}^{\boldsymbol{\gamma},\bullet,N}$ and $\mathfrak{E}_{t,\textnormal{l}}^{\boldsymbol{\gamma},\bullet,N}$ denote the transition kernels of the right and left edge particles $\left(\mathsf{X}_1^{(1),\bullet}(t),\mathsf{X}_2^{(2),\bullet}(t),\dots,\mathsf{X}_N^{(N),\bullet}(t);t \ge 0\right)$ and $\left(\mathsf{X}_1^{(N),\bullet}(t),\mathsf{X}_1^{(N-1),\bullet}(t),\dots,\mathsf{X}_1^{(1),\bullet}(t);t \ge 0\right)$ respectively  (we use a single time variable $t$ since time is homogeneous). Define the Markov kernel $\mathfrak{L}^{\boldsymbol{\gamma},\bullet,N}$ from $\mathbb{W}_N$ to $\mathbb{IA}_N$
\begin{equation*}
\mathfrak{L}^{\boldsymbol{\gamma},\bullet,N}=\left(\mathbf{x},\left(\mathbf{y}^{(1)},\dots,\mathbf{y}^{(N)}\right)\right)=\mathbf{1}_{\mathbf{y}^{(N)}=\mathbf{x}}\prod_{n=1}^{N-1}\Lambda_{n+1,n}^{\boldsymbol{\gamma},\bullet}\left(\mathbf{y}^{(n+1)},\mathbf{y}^{(n)}\right),
\end{equation*}
where $\Lambda_{n+1,n}^{\boldsymbol{\gamma},\bullet}$ is given in Definition \ref{LevelInhomogeousVariousDef}. Under, the conditions of Proposition \ref{PropMultilevelSpaceLevel} we have the following result.
\begin{prop}
Let $t\ge 0$. Then, we have the intertwinings
\begin{align*}
\mathcal{P}_t^{\boldsymbol{\gamma},\bullet,N}\mathfrak{L}^{\boldsymbol{\gamma},\bullet,N}\mathcal{E}^{(N)}_{\textnormal{r}}(\mathbf{x},\mathbf{y})&=\mathfrak{L}^{\boldsymbol{\gamma},\bullet,N}\mathcal{E}^{(N)}_{\textnormal{r}}\mathfrak{E}_{t,\textnormal{r}}^{\boldsymbol{\gamma},\bullet,N}(\mathbf{x},\mathbf{y}), \ \ \mathbf{x},\mathbf{y} \in \mathbb{W}_N, \\
\mathcal{P}_t^{\boldsymbol{\gamma},\bullet,N}\mathfrak{L}^{\boldsymbol{\gamma},\bullet,N}\mathcal{E}^{(N)}_{\textnormal{l}}(\mathbf{x},\mathbf{y})&=\mathfrak{L}^{\boldsymbol{\gamma},\bullet,N}\mathcal{E}^{(N)}_{\textnormal{l}}\mathfrak{E}_{t,\textnormal{l}}^{\boldsymbol{\gamma},\bullet,N}(\mathbf{x},\mathbf{y}), \ \ \mathbf{x}\in \mathbb{W}_N, \mathbf{y}\in \overline{\mathbb{W}}_N,
\end{align*}
where recall the transition kernel $\mathcal{P}_t^{\boldsymbol{\gamma},\bullet,N}$ was given in Definition \ref{LevelInhomogeousVariousDef}.
\end{prop}

\begin{proof}
Exact same proof as Proposition \ref{EdgeIntertwining} making use of Proposition \ref{PropMultilevelSpaceLevel}  now instead.   
\end{proof}

Similar but more involved arguments to the ones presented previously may be used to invert the kernels $\mathfrak{L}^{\boldsymbol{\gamma},\bullet,N}\mathcal{E}^{(N)}_{\textnormal{r}}$ and $\mathfrak{L}^{\boldsymbol{\gamma},\bullet,N}\mathcal{E}^{(N)}_{\textnormal{l}}$. Then, analogues of Theorems \ref{EdgeExplicitTrans} and \ref{EdgeDeterminantal} in this setting can be obtained.

\section{Extremal measures for the inhomogeneous Gelfand-Tsetlin graph}\label{SectionGraph}

We begin by connecting the quantities of interest, namely the measures $\mathcal{M}_N^{\boldsymbol{\omega}}$ and Markov kernels $\Lambda_{N+1,N}^{\mathbf{GT}_+(\mathbf{a})}$ to familiar objects we have seen in the previous sections.

\begin{prop}\label{PropIdentificationGraph} For all $N \ge 1$, we have, with $\mathbf{x}\in \mathbb{W}_N$ and $\mathbf{y}\in \mathbb{W}_{N+1}$,
\begin{align*}
\mathcal{M}_N^{{\boldsymbol{\omega}}}(\mathbf{x};\mathbf{a})&\equiv \mathfrak{P}_{f_{\boldsymbol{\omega}}}^{(N)}\left((0,1,\dots,N-1),\mathbf{x}\right), \\
\Lambda_{N+1,N}^{\mathbf{GT}_+(\mathbf{a})}(\mathbf{y},\mathbf{x})&\equiv \mathfrak{L}_{N+1,N}(\mathbf{y},\mathbf{x}).
\end{align*}
\end{prop}
\begin{proof}
Direct comparison of the formulae for the left and right hand side of each equality by noting that
\begin{equation}\label{h_nEvaluationPacked}
 \mathfrak{h}_N((0,1,\dots,N-1);\mathbf{a})=a_0^{-(N-1)}a_1^{-(N-2)}\cdots a^{-1}_{N-2},
\end{equation}
which is obtained by an elementary computation.
\end{proof}

\begin{prop}
Let ${\boldsymbol{\omega}}$ be defined as in (\ref{Defomega}). Then, $\left(\mathcal{M}_N^{\boldsymbol{\omega}}\left(\cdot;\mathbf{a}\right)\right)_{N=1}^\infty$ form a coherent sequence of probability measures.
\end{prop}
\begin{proof}
This follows by virtue of Proposition \ref{PropIdentificationGraph} above. The conditions (\ref{Defomega}) on the sequence ${\boldsymbol{\omega}}$ are so that $f_{\boldsymbol{\omega}} \in \mathsf{Hol}\left(\mathbb{H}_{-R}\right)$ with $R>R(\mathbf{a})$ and the corresponding kernel $\mathfrak{P}_{f_{{\boldsymbol{\omega}}}}^{(N)}$ by virtue of its probabilistic description is non-negative. Finally, the intertwining in Proposition \ref{IntertwiningSingleLevelNorm} gives the coherency property. 
\end{proof}

\begin{rmk}\label{RmkProbabilisticTermsGraph}
Observe that, by virtue of Proposition \ref{PropMultiLevelSpaceTime}, $\mathfrak{P}_{f_{\boldsymbol{\omega}}}^{(N)}\left((0,1,\dots,N-1),\mathbf{x}\right)\equiv\mathcal{M}_N^{{\boldsymbol{\omega}}}(\mathbf{x};\mathbf{a})$, has the following probabilistic interpretation. Starting from the fully-packed configuration, we run (a possibly infinite number of) sequential-update Bernoulli steps with parameters $(\alpha_{i})_{i=1}^\infty$, Warren-Windridge geometric steps with parameters $(\beta_i)_{i=1}^\infty$ and continuous-time pure-birth dynamics for time $t$. Then, the resulting distribution of the $N$-th row of the array, for any $N\ge 1$, is given by $\mathfrak{P}_{f_{\boldsymbol{\omega}}}^{(N)}\left((0,1,\dots,N-1),\cdot\right)$.
\end{rmk}

It remains to prove extremality. The argument goes via a family of symmetric functions\footnote{These are closely related to the functions $\mathsf{H}_{(\gamma_1,\dots,\gamma_N)}^\bullet$ from Section \ref{SpaceLevelInhomogeneousSection}, which are also very closely related to the factorial Schur polynomials \cite{MacDonaldSchurvariations} as we shall see in the proof. In the homogeneous case $a_x \equiv 1$, they essentially boil down to (a normalised version of) the standard Schur polynomials \cite{MacdonaldBook}.} $\mathcal{F}_\mathbf{x}(\mathbf{w};\mathbf{a})$ defined in (\ref{NormalisedFunctions}) below, a multivariate extension of the polynomials $p_x(w)$, to an application of De-Finetti's theorem \cite{AldousDeFinetti}. The main idea is to consider the generating function\footnote{For other uses of such generating functions, in the homogeneous case $a_x\equiv 1$ where they go by the name Schur generating functions, and in particular for applications to questions of global asymptotics,  see \cite{SchurGenerating1,SchurGenerating2,SchurGenerating3}.} (\ref{GeneratingFunction}) of a coherent sequence of measures $(\mu_N)_{N=1}^\infty$ with respect to $\mathcal{F}_\mathbf{x}$. Remarkably, in the case of the measures $\left(\mathcal{M}_N^{{\boldsymbol{\omega}}}\left(\cdot;\mathbf{a}\right)\right)_{N=1}^\infty$ this generating function factorises, see (\ref{FactorisationProperty}), and along with a certain positivity property of $\mathcal{F}_\mathbf{x}$ this allows us to make use of De-Finetti's theorem \cite{AldousDeFinetti}. The positivity property is where we require the condition $\inf_k a_k \ge 1$ on $\mathbf{a}$.

\begin{proof}[Completion of proof of Theorem \ref{TheoremGraph}]
For $\mathbf{x}\in \mathbb{W}_N$ define the following function $\mathsf{F}_\mathbf{x}$, a multivariate polynomial,  by the explicit formula:
\begin{equation*}
\mathsf{F}_{\mathbf{x}}\left(w_1,\dots,w_N;\mathbf{a}\right)=\frac{\det\left(p_{x_j}(w_i)\right)_{i,j=1}^N}{\det\left((-w_i)^{j-1}\right)_{i,j=1}^N}.
\end{equation*}
By taking the limit $w_1,\dots,w_N \to 0$ we get, after recalling the representation of $\mathfrak{h}_N$ from Lemma \ref{H_NRep},
\begin{equation*}
\mathsf{F}_\mathbf{x}\left(0,\dots,0;\mathbf{a}\right)=\mathfrak{h}_N\left(\mathbf{x}\right)>0.
\end{equation*}
We then define for $\mathbf{x}\in \mathbb{W}_N$, the function $\mathcal{F}_\mathbf{x}$, the normalized version of $\mathsf{F}_\mathbf{x}$ at $\mathbf{w}=(0,\dots,0)$, so that $\mathcal{F}_{\mathbf{x}}(0,\dots,0;\mathbf{a})\equiv 1$ by,
\begin{equation}\label{NormalisedFunctions}
\mathcal{F}_{\mathbf{x}}(w_1,\dots,w_N;\mathbf{a})=\frac{\mathsf{F}_\mathbf{x}(w_1,\dots,w_N;\mathbf{a})}{\mathsf{F}_\mathbf{x}(0,\dots,0;\mathbf{a})}.
\end{equation}
A simple computation gives, for any $\mathbf{y} \in \mathbb{W}_{N+1}$, the identity
\begin{equation}\label{MarkovKernelFunction}
\mathcal{F}_{\mathbf{y}}\left(w_1,\dots,w_{N-1},0;\mathbf{a}\right)=\sum_{\mathbf{x}\prec \mathbf{y}}\mathfrak{L}_{N+1,N}(\mathbf{y},\mathbf{x})\mathcal{F}_{\mathbf{x}}\left(w_1,\dots,w_{N-1};\mathbf{a}\right).
\end{equation}
We now observe that we have the following representation of $\mathsf{F}_\mathbf{x}$ in terms of factorial Schur polynomials $s_{\boldsymbol{\lambda}}(\cdot|\cdot)$, see  \cite{MacDonaldSchurvariations},
\begin{align*}
\prod_{j=1}^N \prod_{l=0}^{x_{j}-1}a_{x_l}\mathsf{F}_\mathbf{x}(1-z_1,\dots,1-z_N;\mathbf{a})&=\frac{\det\left(\prod_{l=0}^{x_j-1}(z_i+a_l-1)\right)_{i,j=1}^N}{\det\left(z_i^{j-1}\right)_{i,j=1}^N}\\
&=\frac{\det\left(\prod_{l=1}^{\lambda_j+N-j}(z_i+\tilde{a}_l)\right)_{i,j=1}^N}{\det\left(z_i^{N-j}\right)_{i,j=1}^N}=s_{\boldsymbol{\lambda}}(\mathbf{z}|\tilde{\mathbf{a}}),
\end{align*}
where the partition $\boldsymbol{\lambda}=(\lambda_1,\lambda_2,\lambda_3,\dots)$ and sequence $\tilde{\mathbf{a}}$ (indexed by $\mathbb{N}$ instead of $\mathbb{Z}_+$) are given in terms of $\mathbf{x}$ and $\mathbf{a}$ by $\lambda_j=x_{N-j+1}-N+j$ and $\tilde{a}_{l}=a_{l-1}-1 \ge 0$ and $s_{\boldsymbol{\lambda}}(\mathbf{z}|\tilde{\mathbf{a}})$ is the factorial Schur polynomial in variables $\mathbf{z}$ indexed by the partition $\boldsymbol{\lambda}$ and with parameter sequence $\tilde{\mathbf{a}}$, see \cite{MacDonaldSchurvariations}. Here, we note that the partition $\boldsymbol{\lambda}$ has length, denoted by $l\left(\boldsymbol{\lambda}\right)$, at most $N$, see \cite{MacdonaldBook,MacDonaldSchurvariations}. From the combinatorial formula \cite{MacDonaldSchurvariations} for the factorial Schur polynomial we can write 
\begin{equation*}
s_{\boldsymbol{\lambda}}\left(z_1,\dots,z_N|\tilde{\mathbf{a}}\right)=\sum_{\boldsymbol{\mu}\in \mathbb{Z}_+^N}c(\boldsymbol{\mu}|\boldsymbol{\lambda};\tilde{\mathbf{a}})z_1^{\mu_1}z_2^{\mu_2}\cdots z_N^{\mu_N},   
\end{equation*}
with $c(\boldsymbol{\mu}|\boldsymbol{\lambda};\tilde{\mathbf{a}})\ge 0$, since $\tilde{a}_l\ge 0$ by our assumption that $\inf_{k\in \mathbb{Z}_+}a_k\ge 1$. Abusing notation we will also write $c(\boldsymbol{\mu}|\mathbf{x};\mathbf{a})$ for $c(\boldsymbol{\mu}|\boldsymbol{\lambda};\tilde{\mathbf{a}})$ under the correspondence between $\mathbf{x}$ and $\boldsymbol{\lambda}$ and $\mathbf{a}$ and $\tilde{\mathbf{a}}$ above. Moreover, since $s_{\boldsymbol{\lambda}}\left(\mathbf{z}|\tilde{\mathbf{a}}\right)$ is a symmetric polynomial in the variables $z_i$ we must have, for any permutation $\sigma$ of $\{1,\dots,N\}$,
\begin{equation*}
c\left(\mu_1,\dots, \mu_N|\boldsymbol{\lambda};\tilde{\mathbf{a}}\right)=c\left(\mu_{\sigma(1)},\dots, \mu_{\sigma(N)}|\boldsymbol{\lambda};\tilde{\mathbf{a}}\right).
\end{equation*}
In particular, we also have 
\begin{equation}\label{MonomialExpansion}
s_{\boldsymbol{\lambda}}\left(z_1,\dots,z_N|\tilde{\mathbf{a}}\right)=\sum_{\boldsymbol{\mu} \  \textnormal{partition}, \ l(\boldsymbol{\mu})\le N}c(\boldsymbol{\mu}|\boldsymbol{\lambda};\tilde{\mathbf{a}})m_{\boldsymbol{\mu}}\left(z_1,\dots,z_N\right),
\end{equation}
where $m_{\boldsymbol{\mu}}$ is the monomial symmetric polynomial associated to a partition $\boldsymbol{\mu}$, see \cite{MacdonaldBook}. By making use of the connection between $\mathsf{F}_\mathbf{x}$ and the factorial Schur polynomial we obtain
\begin{equation}\label{MultivariateFnExpansion}
\mathcal{F}_\mathbf{x}\left(1-z_1,\dots,1-z_N;\mathbf{a}\right)=\sum_{\boldsymbol{\mu}\in \mathbb{Z}_+^N} \xi\left(\boldsymbol{\mu}|\mathbf{x};\mathbf{a}\right)z_1^{\mu_1}z_2^{\mu_2}\cdots z_N^{\mu_N},
\end{equation}
where $\xi\left(\cdot|\mathbf{x};\mathbf{a}\right)$ is a probability measure on $\mathbb{Z}_+^N$ (since $\mathcal{F}_\mathbf{x}(\mathbf{0})=1$) satisfying for any permutation $\sigma$ of $\{1,\dots,N\}$, $\xi\left(\mu_1,\dots, \mu_N|\mathbf{x},\mathbf{a}\right)=\xi\left(\mu_{\sigma(1)},\dots, \mu_{\sigma(N)}|\mathbf{x},\mathbf{a}\right)$.

We now define a map from coherent sequences of measures to certain sequences of analytic functions on the unit polydisk given by,
\begin{align}
(\mu_N)_{N=1}^\infty&\mapsto (\mathcal{T}_N\mu_N)_
{N=1}^\infty\nonumber,\\
\left[\mathcal{T}_N\mu_N\right](z_1,\dots,z_N)&=\sum_{\mathbf{x}\in \mathbb{W}_N}\mu_N(\mathbf{x})\mathcal{F}_{\mathbf{x}}\left(1-z_1,\dots,1-z_N;\mathbf{a}\right). \label{GeneratingFunction}
\end{align}
The fact that this map is well-defined can be seen as follows. Since $\mu_N$ is a probability measure on $\mathbb{W}_N$ and by virtue of the expansion (\ref{MultivariateFnExpansion}) the function $
\left[\mathcal{T}_N\mu_N\right](z_1,\dots,z_N)$
converges uniformly on $\mathbb{D}^N$, where $\mathbb{D}=\{z\in \mathbb{C}:|z|\le 1\}$ is the closed unit disk, and it is analytic in its interior $\left(\mathbb{D}^\circ\right)^N$. It is moreover, the characteristic function of an exchangeable measure on $\mathbb{Z}_+^N$ (the convolution of $\mu_N$ and $\xi$). Since $(\mu_N)_{N=1}^{\infty}$ is coherent an easy computation using (\ref{MarkovKernelFunction}) reveals that, for all $N \ge 1$,
\begin{equation*}
\left[\mathcal{T}_N\mu_N\right]\left(z_1,\dots,z_{N-1},0\right)=\left[\mathcal{T}_{N-1}\mu_{N-1}\right]\left(z_1,\dots,z_{N-1}\right).
\end{equation*}
In particular, by virtue of this consistency, the sequence $(\mathcal{T}_N\mu_N)_{N=1}^{\infty}$  can be viewed (via a projective limit) as a function on $\mathbb{D}^{\infty}$ and it is the characteristic function of an exchangeable measure on $\mathbb{Z}_+^\infty$. The key property of the transform $\mathcal{T}_N$ is that it factorizes in the case of $\mathcal{M}_N^{{\boldsymbol{\omega}}}$. Namely, a computation using the Cauchy-Binet formula and Lemma \ref{LemmaComposition} gives 
\begin{equation}\label{FactorisationProperty}
\left[\mathcal{T}_N\mathcal{M}_N^{{\boldsymbol{\omega}}}\right](z_1,\dots,z_N)=\prod_{i=1}^N f_{{\boldsymbol{\omega}}}(1-z_i),
\end{equation}
 valid in a small (this restriction comes from Lemma \ref{LemmaComposition}) neighourhoud of $\mathbf{z}=(1,\dots,1)$, where we have used (\ref{h_nEvaluationPacked}) to see  that $\mathcal{F}_{(0,\dots,N-1)}(w_1,\dots,w_N;\mathbf{a})\equiv 1$. Observe that, since $f_{{\boldsymbol{\omega}}}$ is holomorphic in $\mathbb{H}_{-\beta_1^{-1}}$ then the function $(z_1,\dots,z_N)\mapsto \prod_{i=1}^N f_{{\boldsymbol{\omega}}}(1-z_i)$ is analytic in $\left(\mathbb{D}^\circ\right)^N$. Hence, by the identity theorem for analytic functions we can extend equality (\ref{FactorisationProperty}) first to $\mathbf{z} \in \left(\mathbb{D}^\circ\right)^N$ and by continuity to $\mathbb{D}^N$.

Hence, by De-Finetti's theorem \cite{AldousDeFinetti}, for any $f_{{\boldsymbol{\omega}}}$, the corresponding function $\left(\mathcal{T}_N\mathcal{M}_N^{{\boldsymbol{\omega}}}\right)_{N=1}^{\infty}$ is an extreme point in the convex set of characteristic functions of exchangeable measures on $\mathbb{Z}_+^{\infty}$ by virtue of the factorization property (\ref{FactorisationProperty}). Thus, since the map that takes $(\mu_N)_{N=1}^{\infty}\mapsto(\mathcal{T}_N\mu_N)_{N=1}^\infty$ is affine, and as we show next also injective, $(\mathcal{M}_N^{{\boldsymbol{\omega}}})_{N=1}^\infty$ is an extreme point in the original convex set of coherent sequences of probability measures. 

It remains to show that the map given by $(\mu_N)_{N=1}^{\infty}\mapsto(\mathcal{T}_N\mu_N)_{N=1}^\infty$ is injective. Let $N\ge 1$ be fixed but arbitrary. We show that given the function $\mathcal{T}_N\mu_N(\mathbf{z})$ we can recover $\mu_N$ uniquely. We give a combinatorial argument although a complex analytic approach using a suitable orthogonality property of the functions $\mathcal{F}_{\mathbf{x}}$ is  also possible. Observe that, if we know the function $\mathcal{T}_N\mu_N(\mathbf{z})$ we then know its  collection of Taylor coefficients at $(0,\dots,0)$, namely the coefficients of the terms $z_1^{\mu_1}z_2^{\mu_2}\cdots z_N^{\mu_N}$, for $\boldsymbol{\mu}$ a partition, $l(\boldsymbol{\mu})\le N$, that we denote by $\left(\mathfrak{v}_{\boldsymbol{\mu}}\right)_{\boldsymbol{\mu} \ \textnormal{partition}, \ l(\boldsymbol{\mu})\le N}$. In particular, by looking at $\mathfrak{v}_{\boldsymbol{\mu}}$ we have the equality
\begin{equation}\label{Recovermu_N}
\sum_{\mathbf{x}\in \mathbb{W}_N}\mu_N(\mathbf{x})g_N(\mathbf{x})c(\boldsymbol{\mu}|\mathbf{x};\mathbf{a})=\mathfrak{v}_{\boldsymbol{\mu}},
\end{equation}
where $g_N(\mathbf{x})$ is an explicit non-zero function involving $\mathfrak{h}_N(\mathbf{x})$ and the $a_i$'s. Having knowledge of these quantities we want to solve for $\mu_N(\mathbf{x})$. Towards this end, let us view the variable $\mathbf{x}$ as a partition $\boldsymbol{\lambda}$ and the sequence $\mathbf{a}$ as the sequence $\tilde{\mathbf{a}}$ under the correspondence discussed earlier, $\lambda_j=x_{N-j+1}-N+j$ and $\tilde{a}_{l}=a_{l-1}-1$. In particular, we can rewrite (\ref{Recovermu_N}) in this notation, where $\mu_N(\boldsymbol{\lambda}),g_N(\boldsymbol{\lambda})$ denote the values of $\mu_N(\mathbf{x}),g_N(\mathbf{x})$ under the above correspondence between $\mathbf{x}$ and $\boldsymbol{\lambda}$ and $\mathbf{a}$ and $\tilde{\mathbf{a}}$, as follows:
\begin{equation}\label{MoreRecovery}
\sum_{\boldsymbol{\lambda} \ \textnormal{partition}, \ l(\boldsymbol{\lambda})\le N}\mu_N(\boldsymbol{\lambda})g_N(\boldsymbol{\lambda})c(\boldsymbol{\mu}|\boldsymbol{\lambda};\tilde{\mathbf{a}})=\mathfrak{v}_{\boldsymbol{\mu}}.
\end{equation}
We need one final well-known fact. Both $\left\{s_{\boldsymbol{\lambda}}\left(\cdot|\tilde{\mathbf{a}}\right)\right\}_{\boldsymbol{\lambda} \ \textnormal{partition}, \ l(\boldsymbol{\lambda})\le N}$ and $\left\{m_{\boldsymbol{\mu}}(\cdot)\right\}_{\boldsymbol{\mu} \ \textnormal{partition}, \ l(\boldsymbol{\mu})\le N}$ form bases for the ring of symmetric polynomials in $N$ variables, see \cite{MacdonaldBook,MacDonaldSchurvariations}. Then, by virtue of (\ref{MonomialExpansion}), the matrix $\left[\mathbf{A}_{\boldsymbol{\mu}\boldsymbol{\lambda}}\right]_{\boldsymbol{\mu} \boldsymbol{\lambda}}=\left[c(\boldsymbol{\mu}|\boldsymbol{\lambda};\tilde{\mathbf{a}})\right]_{\boldsymbol{\mu} \boldsymbol{\lambda}}$ is the change of basis matrix which is thus invertible. Hence, from (\ref{MoreRecovery}) we can solve for $\mu_N$:
\begin{equation}
\mu_N\left(\boldsymbol{\lambda}\right)=\frac{1}{g_N\left(\boldsymbol{\lambda}\right)}\left[\mathbf{A}^{-1}\mathfrak{v}\right]_{\boldsymbol{\lambda}}
\end{equation}
and the desired conclusion follows.
\end{proof}

Finally, extremal coherent sequences $(\mathcal{M}_N^{\boldsymbol{\omega}})_{N=1}^\infty$ indexed by points ${\boldsymbol{\omega}}$ are distinct for distinct ${\boldsymbol{\omega}}$. 

\begin{prop}
Suppose ${\boldsymbol{\omega}} \neq \tilde{{\boldsymbol{\omega}}}$, then the sequences $(\mathcal{M}_N^{{\boldsymbol{\omega}}})_{N=1}^\infty$ and $(\mathcal{M}_N^{\tilde{\boldsymbol{\omega}}})_{N=1}^\infty$ are distinct.
\end{prop}

\begin{proof}
This is a direct consequence of Lemma \ref{LemmaExpansion} since $f_{{\boldsymbol{\omega}}} \neq f_{\tilde{{\boldsymbol{\omega}}}}$ (and also implicitly follows from the proof above).
\end{proof}

\section{Duality via a height function}\label{SectionDuality}

We now prove the results stated in Section \ref{IntroDualitySection}. Recall our labelling convention from the second paragraph of Section \ref{IntroDualitySection}. To ease notation we will frequently drop dependence on the time variable $t$ if there is no risk of confusion. Recall the terminology that for a configuration $\left(x_k^{(n)}\right)_{0\le k \le n; n \ge 0}$ in $\mathbb{IA}_\infty$,  $\left(x_k^{(n)}\right)_{n\ge k}$ is called the $k$-th column of the configuration.

\begin{proof}[Proof of Proposition \ref{PropInvolution}]
 Observe that, we can write the map $\mathsf{Hgt}$ as follows
 \begin{equation*}
  \left(\left(x_0^{(j)}\right)_{j\ge 0}, \left(x_1^{(j)}\right)_{j\ge 1}, \left(x_2^{(j)}\right)_{j\ge 2},\dots\right)  \mapsto \left(\left(\mathsf{h}_0(j)\right)_{j\ge 0}, \left(\mathsf{h}_1(j)+1\right)_{j\ge 1}, \left(\mathsf{h}_2(j)+2\right)_{j\ge 2},\dots\right)
 \end{equation*}
with the map from  
 the $i$-th column of $\left(x_i^{(j)}\right)_{0\le i \le j;j\ge 0}$ to the $i$-th column of $\mathsf{Hgt}\left(\left(x_i^{(j)}\right)_{0\le i \le j;j\ge 0}\right)$ given by,
 \begin{equation*}
\left(x_i^{(j)}\right)_{j\ge i} \mapsto \left(\mathsf{h}_i(j)+i\right)_{j\ge i}.
 \end{equation*}
 Here, we note that $\mathsf{h}_i(\cdot)$ is only dependent on the $i$-th column $\left(x_i^{(j)}\right)_{j\ge i}$ of the configuration $\left(x_i^{(j)}\right)_{0\le i \le j;j\ge 0}$. Then, observe that for each $i\in \mathbb{Z}_+$, 
\begin{equation*}
\left(x_i^{(j)}-i\right)_{j\ge i}\mapsto \left(\mathsf{h}_i(j)\right)_{j\ge i}
\end{equation*}
 is simply the map that takes a partition to its conjugate partition. This map is an involution. The conclusion follows.
\end{proof}

\begin{proof}[Proof of Proposition \ref{PropWellDefinedDynamics}]
We first prove the statement for the continuous time pure-birth dynamics. Observe that, if there exists, with positive probability, a finite $\tau_*>0$ such that $\mathsf{X}(\tau_*)\notin \mathbb{IA}_\infty^*$, then on this event of positive probability, for all $t\ge \tau_*$ we have $\mathsf{X}(t)\notin \mathbb{IA}_\infty^*$. Hence, it suffices to show that for any fixed $t\ge 0$, almost surely $\mathsf{X}(t) \in \mathbb{IA}_\infty^*$. Suppose not. So with positive probability there exists $i\in \mathbb{Z}_+$ such that $\mathsf{X}_i^{(j)}(t)>i$ for all $j\ge i$. Now, let $n_i$ be the first $k$ such that $\mathsf{X}_i^{(k)}(0)=x_i^{(k)}=i$ (by the interlacing, for all $k\ge n_i$ we must have $\mathsf{X}_i^{(k)}(0)=x_i^{(k)}=i$). If $\mathsf{X}_i^{(j)}(t)>i$ for all $j\ge i$ then for all $j\ge n_i$ at least one particle from each ``diagonal" of particles (we drop time dependence from the notation):
\begin{equation*}
\mathsf{X}_0^{(j-i)}, \mathsf{X}_1^{(j-i+1)},\dots,\mathsf{X}_i^{(j)},
\end{equation*}
has moved of its own volition by time $t$, see Figure \ref{NonExplosionFigure} for an illustration. Note that this is because  pushing of particles only occurs along diagonals from lower to higher levels. For each diagonal $j\ge n_i$ consider the exponential clock of the particle that tries to move first. These clocks are independent for different diagonals. Moreover, by our assumption on the environment $\boldsymbol{\theta}$ all these clocks have uniformly bounded rates. Thus, the event that they all ring by time $t$ has probability zero, which gives a contradiction. 

In discrete time it suffices to show that the process is in $\mathbb{IA}_\infty^*$ after a single time step. We use the same argument, but instead of looking at exponential random variables with bounded rates we are dealing with infinitely many independent $0$-$1$ random variables (corresponding to whether at least one particle from each diagonal moves) with success probabilities uniformly strictly between $0$ and $1$. The event that all of them are $1$ has probability zero.
\end{proof}

\begin{figure}
\captionsetup{singlelinecheck = false, justification=justified}
\centering
\begin{tikzpicture}

 \draw[fill] (0,0) circle [radius=0.075];
\draw[fill] (1,0.5) circle [radius=0.075];

\draw[fill] (3,1.5) circle [radius=0.075];

 \draw[fill] (0,1) circle [radius=0.075];
\draw[fill] (1,1.5) circle [radius=0.075];

\draw[fill] (3,2.5) circle [radius=0.075];

 \draw[fill] (0,2) circle [radius=0.075];
\draw[fill] (1,2.5) circle [radius=0.075];

\draw[fill] (3,3.5) circle [radius=0.075];

\node[above] at (1.5, 3.5) {$\vdots$};

\node[left] at (-0.5,0) {$\mathsf{X}_0^{(n_i-i)}$};

\node[left] at (-0.5,1) {$\mathsf{X}_0^{(n_i+1-i)}$};

\node[left] at (-0.5,2) {$\mathsf{X}_0^{(n_i+2-i)}$};

\node[right] at (3.5,1.5) {$\mathsf{X}_i^{(n_i)}$};

\node[right] at (3.5,2.5) {$\mathsf{X}_0^{(n_i+1)}$};

\node[right] at (3.5,3.5) {$\mathsf{X}_0^{(n_i+2)}$};

\node[below right] at (1,0.4) {$\mathsf{X}_1^{(n_i+1-i)}$};

\node[above ] at (1,2.8) {$\mathsf{X}_1^{(n_i+3-i)}$};

\node[above right] at (1,1.35) {$\mathsf{X}_1^{(n_i+2-i)}$};

\draw[rotate around={115:(1.5,0.75)},dotted, thick] (1.5,0.75) ellipse (10pt and 60pt);

\draw[rotate around={115:(1.5,1.75)},dotted, thick] (1.5,1.75) ellipse (10pt and 60pt);

\draw[rotate around={115:(1.5,2.75)},dotted, thick] (1.5,2.75) ellipse (10pt and 60pt);

\node[above left] at (2.5,0.75) {$\dots$};

\node[above left] at (2.75,2) {$\dots$};

\node[above left] at (2.25,2.75) {$\dots$};

\end{tikzpicture}

\caption{A depiction of the particles that need to move by time $t$ in order for the process to leave $\mathbb{IA}_\infty^*$. At least one exponential clock from every such ``diagonal" of particles needs to ring by time $t$. This is an event of probability zero. }\label{NonExplosionFigure}
\end{figure}
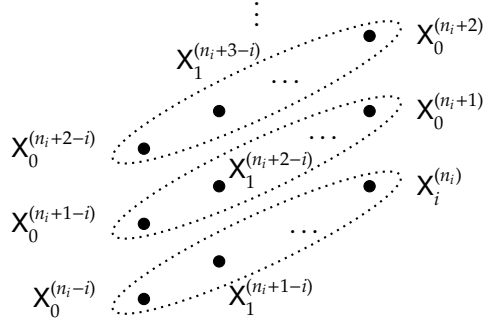

\begin{proof}[Proof of Theorem \ref{DualityThmIntro}]
Observe that, the evolution of the first $k$ columns of $\left(\mathsf{X}(t);t\ge 0\right)$, namely $\left(\mathsf{X}_i^{(j)}(t);t\ge 0\right)_{0\le i \le k-1; j \ge i}$ in all three types of dynamics is autonomous. Moreover, under $\mathsf{Hgt}$ it is mapped to the first $k$ columns of $\left(\mathsf{Hgt}(\mathsf{X}(t));t\ge 0\right)$. Hence, it suffices to prove the result for the restriction of the dynamics on any fixed number of consecutive columns starting from the first one. Moreover, we restrict our attention to the dynamics of the first two columns as these involve both types of possible interactions (pushing and blocking) between particles. We first consider the continuous-time pure-birth case. We also drop dependence on time from the notation from now on.

Suppose the exponential clock of particle $\mathsf{X}_0^{(j)}$ which is at spatial location $x$ rings. This happens at rate $\boldsymbol{\theta}(x,j)$. Suppose the particle is not blocked. Then, it moves to location $x+1$. We note that since $\mathsf{X}_0^{(j)}$ was not blocked before attempting to move we must have had $\mathsf{h}_0(x)=j$ just before the move. After the move we get $\mathsf{h}_0(x)=j+1$. In particular, we get that $\mathsf{Hgt}(\mathsf{X})_0^{(x)}$ moved from $j$ to $j+1$ and this happened with rate $\boldsymbol{\theta}(x,j)=\hat{\boldsymbol{\theta}}(j,x)$.  If moreover $\mathsf{X}_1^{(j+1)}$ was at location $x+1$ just before the clock rang, then it is pushed to location $x+2$. Furthermore, just before this pushing move we must have had $\mathsf{h}_1(x+1)=j$ which after the pushing becomes $j+1$. In particular, $\mathsf{Hgt}(\mathsf{X})_1^{(x+1)}=\mathsf{h}_1(x+1)+1$ is pushed from $j+1$ to $j+2$. The pushing of the $\mathsf{X}$ particles and thus of the $\mathsf{Hgt}(\mathsf{X})$ particles is propagated to higher levels in this fashion. See Figure \ref{CtsTimeHeightFunctionDynamics} for an illustration. Now, if $\mathsf{X}_0^{(j)}$ was blocked before the clock rang, namely we had $\mathsf{X}_0^{(j)}=\mathsf{X}_0^{(j-1)}=x$ then nothing happens for either the $\mathsf{X}$ process and thus for the $\mathsf{Hgt}(\mathsf{X})$ process as well. In the reverse direction, if $\mathsf{Hgt}(\mathsf{X})_0^{(x)}=\mathsf{Hgt}(\mathsf{X})_0^{(x-1)}$ then this implies that there is no $j$ such that $\mathsf{X}_0^{(j)}=x$ and hence $\mathsf{Hgt}(\mathsf{X})_0^{(x)}=\mathsf{h}_0(x)$ cannot change/move. Thus, we observe that the dynamics described above for $\mathsf{Hgt}(\mathsf{X})$ are exactly the continuous-time pure-birth push-block dynamics in environment $\hat{\boldsymbol{\theta}}$ and this completes the proof of the continuous-time case.

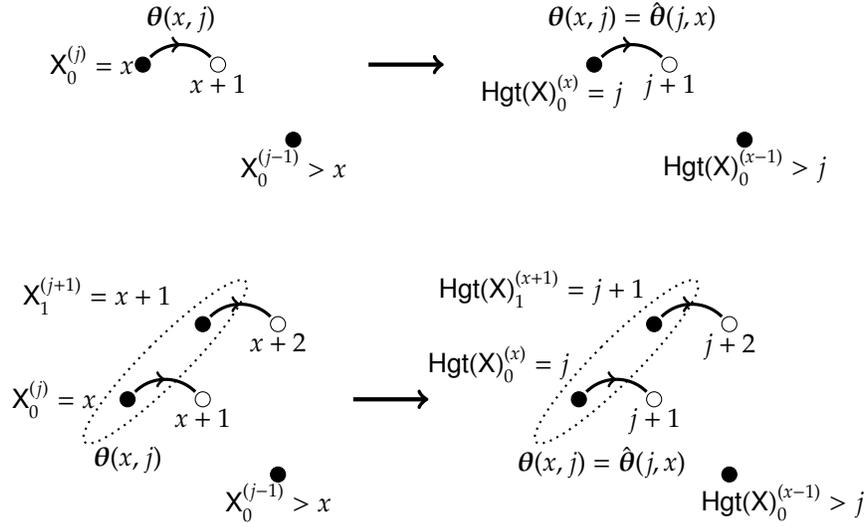
\begin{figure}
\captionsetup{singlelinecheck = false, justification=justified}
\centering
\begin{tikzpicture}

 \draw[fill] (0,0) circle [radius=0.1];

  \node[left] at (0,0) {$\mathsf{X}_0^{(j)}=x$};

\draw[] (1,0) circle [radius=0.1];

 \node[below] at (1,0) {$x+1$};

\draw[middlearrow={>}, very thick] (0.1,0.1) to [out=45, in=135] (0.9,0.1);

\node[above] at (0.5,0.3) {$\boldsymbol{\theta}(x,j)$};

 \draw[fill] (2,-1) circle [radius=0.1];

 \node[below] at (2,-1) {$\mathsf{X}_0^{(j-1)}>x$};

 \draw[ultra thick,->] (3,0) to (4,0);

  \draw[fill] (6,0) circle [radius=0.1];

  \node[below left,align=left] at (6.5,0) {$\mathsf{Hgt}(\mathsf{X})_0^{(x)}=j$};

\draw[] (7,0) circle [radius=0.1];

 \node[below] at (7,0) {$j+1$};

\draw[middlearrow={>}, very thick] (6.1,0.1) to [out=45, in=135] (6.9,0.1);

\node[above] at (6.5,0.3) {$\boldsymbol{\theta}(x,j)=\hat{\boldsymbol{\theta}}(j,x)$};

 \draw[fill] (8,-1) circle [radius=0.1];

 \node[below] at (8,-1) {$\mathsf{Hgt}(\mathsf{X})_0^{(x-1)}>j$};

\end{tikzpicture}

\bigskip

\bigskip

\begin{tikzpicture}

 \draw[fill] (0,0) circle [radius=0.1];

  \node[left] at (-0.3,0) {$\mathsf{X}_0^{(j)}=x$};

\draw[] (1,0) circle [radius=0.1];

 \node[below] at (1,0) {$x+1$};

\draw[middlearrow={>}, very thick] (0.1,0.1) to [out=45, in=135] (0.9,0.1);

 \draw[fill] (1,1) circle [radius=0.1];

  \node[above left] at (0.7,1) {$\mathsf{X}_1^{(j+1)}=x+1$};

\draw[] (2,1) circle [radius=0.1];

 \node[below] at (2,1) {$x+2$};

\draw[middlearrow={>}, very thick] (1.1,1.1) to [out=45, in=135] (1.9,1.1);

\draw[rotate around={135:(0.5,0.5)},dotted,thick] (0.5,0.5) ellipse (8pt and 43pt);

\node[below] at (0,-0.5) {$\boldsymbol{\theta}(x,j)$};

 \draw[fill] (2,-1) circle [radius=0.1];

 \node[below] at (2,-1) {$\mathsf{X}_0^{(j-1)}>x$};

 \draw[ultra thick,->] (3,0) to (4,0);

  \draw[fill] (6,0) circle [radius=0.1];

  \node[above left] at (6,0.1) {$\mathsf{Hgt}(\mathsf{X})_0^{(x)}=j$};

\draw[] (7,0) circle [radius=0.1];

 \node[below] at (7,0) {$j+1$};

  \draw[fill] (7,1) circle [radius=0.1];

 \draw[] (8,1) circle [radius=0.1];

  \node[below] at (8,1) {$j+2$};

    \node[above left] at (7,1.1) {$\mathsf{Hgt}(\mathsf{X})_1^{(x+1)}=j+1$};

\draw[middlearrow={>}, very thick] (6.1,0.1) to [out=45, in=135] (6.9,0.1);

\draw[middlearrow={>}, very thick] (7.1,1.1) to [out=45, in=135] (7.9,1.1);

\draw[rotate around={135:(6.5,0.5)},dotted,thick] (6.5,0.5) ellipse (8pt and 43pt);

\node[below] at (6.3,-0.5) {$\boldsymbol{\theta}(x,j)=\hat{\boldsymbol{\theta}}(j,x)$};

 \draw[fill] (8,-1) circle [radius=0.1];

 \node[below right] at (7.5,-1) {$\mathsf{Hgt}(\mathsf{X})_0^{(x-1)}>j$};

\end{tikzpicture}

\caption{On the top figure particle $\mathsf{X}_0^{(j)}$ at location $x$ jumps to $x+1$ at rate $\boldsymbol{\theta}(x,j)$ (we assume $\mathsf{X}_0^{(j-1)}>x$ so that it is not blocked). This induces on the side of the $\mathsf{Hgt}(\mathsf{X})$ process a jump of $\mathsf{Hgt}(\mathsf{X})_0^{(x)}=\mathsf{h}_0(x)=j$ to $j+1$ which happens at rate $\boldsymbol{\theta}(x,j)=\hat{\boldsymbol{\theta}}(j,x)$ (note that necessarily $\mathsf{Hgt}(\mathsf{X})_0^{(x-1)}=\mathsf{h}_0(x-1)>j$ since there is at least one particle, namely $\mathsf{X}_0^{(j)}$, at $x$ so the move is not blocked). In the bottom figure we moreover assume $\mathsf{X}_0^{(j+1)}=x+1$ so that in particular $\mathsf{X}_0^{(j+1)}$ is pushed to $x+1$ when $\mathsf{X}_0^{(j)}$ moves. Then, in terms of the $\mathsf{Hgt}(\mathsf{X})$ process, this induces a move of $\mathsf{Hgt}(\mathsf{X})_1^{(x+1)}=\mathsf{h}_1(x+1)+1=j+1$ to $j+2$.}\label{CtsTimeHeightFunctionDynamics}
\end{figure}

We now turn to the discrete-time case. It suffices to prove that $\mathsf{Hgt}$ maps the Warren-Windridge geometric dynamics in environment $\boldsymbol{\theta}$ to sequential-update Bernoulli dynamics in environment $\hat{\boldsymbol{\theta}}$. Since $\mathsf{Hgt}$ is an involution the reverse statement also follows. 

We will require the following observation. First, recall that in the sequential-update Bernoulli dynamics we update each level of the array sequentially. However, we can also obtain the same configuration at the end of the time-step if we follow a slightly different update rule. Let us call a collection of particles from the same column a stack of particles if they have the same spatial location. Moreover, we call two stacks of particles from consecutive columns adjacent if their spatial location differs by $1$. Consider the following updating rules (with push-block interactions between particles being as for sequential-update Bernoulli dynamics). For each individual stack of particles lower-level particles are updated first. For a sequence of adjacent stacks of particles  we update stacks of lower-indexed columns first. Beyond these rules we can update particles in any order. Then, we can observe that at the end of this time-step (completion of the updating process) we obtain the same configuration as for the sequential-update Bernoulli dynamics, assuming of course we use the same Bernoulli random variables to decide whether each particle moves. See Figure \ref{FigureSequentialUpdate} for an illustration.

\begin{figure}
\captionsetup{singlelinecheck = false, justification=justified}
\centering
\begin{tikzpicture}

\node[below] at (0,-0.3) {0};

\node[below] at (1,-0.3) {1};

\node[below] at (2,-0.3) {2};

\node[below] at (3,-0.3) {3};

\node[below] at (4,-0.3) {4};

 \draw[fill] (0,2) circle [radius=0.075];
 
\draw[fill] (1,0) circle [radius=0.075];

\draw[fill] (1,1) circle [radius=0.075];

\draw[fill] (2,2) circle [radius=0.075];

\draw[fill] (4,2) circle [radius=0.075];

 \draw[fill] (2,1) circle [radius=0.075];

\node[below right] at (1,0) {$\mathsf{X}_0^{(0)}$};

\node[above ] at (1,1.2) {$\mathsf{X}_0^{(1)}$};

\node[below right] at (2,1) {$\mathsf{X}_1^{(1)}$};

\node[above left] at (3.8,2) {$\mathsf{X}_2^{(2)}$};

\node[above right] at (2,2) {$\mathsf{X}_1^{(2)}$};

\node[above left] at (-0.2,2) {$\mathsf{X}_0^{(2)}$};

\draw[rotate around={0:(0,2)},dotted, thick] (0,2) ellipse (8pt and 12pt);

\draw[rotate around={0:(4,2)},dotted, thick] (4,2) ellipse (8pt and 12pt);

\draw[rotate around={0:(1,0.5)},dotted, thick] (1,0.5) ellipse (10pt and 20pt);

\draw[rotate around={0:(2,1.5)},dotted, thick] (2,1.5) ellipse (10pt and 20pt);

\end{tikzpicture}

\caption{In this figure we have four stacks of particles $(\mathsf{X}_0^{(2)})$, $(\mathsf{X}_0^{(0)},\mathsf{X}_0^{(1)})$, $(\mathsf{X}_1^{(1)},\mathsf{X}_1^{(2)})$ and $(\mathsf{X}_2^{(2)})$ which are depicted encircled. Two of them are adjacent, namely $(\mathsf{X}_0^{(0)},\mathsf{X}_0^{(1)})$, $(\mathsf{X}_1^{(1)},\mathsf{X}_1^{(2)})$. We can choose which stack of particles to update next in any order, except that adjacent stacks must be updated together. Particles of the same stack are updated from bottom to top. If a particle decides to stay put then the rest of the particles in that stack stay put as well (they are blocked). For adjacent stacks we need to update particles from the stack belonging to the lowest indexed column first. Then, once a particle stays put we move on to the next stack, starting from the first particle which has not been pushed. For example, in the figure above, for the adjacent stacks $(\mathsf{X}_0^{(0)},\mathsf{X}_0^{(1)})$ and $(\mathsf{X}_1^{(1)},\mathsf{X}_1^{(2)})$ we update $\mathsf{X}_0^{(0)}$ first. If it stays put we then move to the next stack and update $\mathsf{X}_1^{(1)}$. If instead it moves, it pushes $\mathsf{X}_1^{(1)}$. We then go on to update $\mathsf{X}_0^{(1)}$. If $\mathsf{X}_0^{(1)}$ stays put we then update $\mathsf{X}_1^{(2)}$. If instead $\mathsf{X}_0^{(1)}$ moves it pushes $\mathsf{X}_1^{(2)}$ in which case the update of the two adjacent stacks is complete. It is easy to see that after all stacks have been updated the configuration is the same as the one obtained from the sequential-update Bernoulli dynamics (if we use the same Bernoulli random variables to decide whether a particle moves or not).
}\label{FigureSequentialUpdate}
\end{figure}
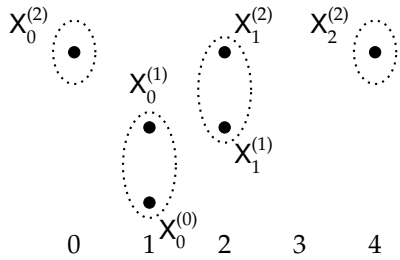

Assume that the $\mathsf{X}$ process follows the Warren-Windridge geometric dynamics. Suppose we are about to update particle $\mathsf{X}_0^{(j)}$ which is at location $x$ and assume that particle $\mathsf{X}_0^{(j-1)}$ was at location $m$ at the end of the last time-step. Then, $\mathsf{X}_0^{(j)}$ moves to location $y$, with $x \le y<m$, with probability
\begin{equation*}
\left(1-\boldsymbol{\theta}(y,j)\right)\prod_{k=x}^{y-1} \boldsymbol{\theta}(k,j)
\end{equation*}
and to $y=m$ with probability $\prod_{k=x}^{m} \boldsymbol{\theta}(k,j)$. We can view this as performing sequentially independent Bernoulli trials with success probabilities $\boldsymbol{\theta}(k,j)$ of whether to move from $k$ to $k+1$. Thus, from the point of view of $\mathsf{Hgt}(\mathsf{X})$, sequentially each $\mathsf{Hgt}(\mathsf{X})_0^{(k)}=\mathsf{h}_0(k)=j$, for $x\le k <m$ decides whether to jump from $j$ to $j+1$ with probability $\boldsymbol{\theta}(k,j)=\hat{\boldsymbol{\theta}}(j,k)$, and if one particle in this stack decides to stay put then all particles on higher levels stay put. See Figure \ref{DiscreteTimeHeighFunctionFigure} for an illustration. Now, regarding the pushing interaction, suppose that $\mathsf{X}_1^{(j+1)}$ was at a spatial location $\tilde{x}\le y$. Then, $\mathsf{X}_1^{(j+1)}$  gets pushed to location $y+1$. This implies that $\mathsf{Hgt}(\mathsf{X})_1^{(k)}=\mathsf{h}_1(k)+1=j+1$ for $\tilde{x}\le k \le y+1$ all get pushed from $j+1$ to $j+2$. Moreover, such a pushed particle $\mathsf{Hgt}(\mathsf{X})_1^{(k)}$ cannot move again in this time-step since in the Warren-Windridge geometric dynamics the particles of the $\mathsf{X}$ process are blocked by the positions of the lower-level particles at the end of the previous time-step (in particular, no more particles from the column with index $1$ can get past $\tilde{x}$). When $\mathsf{X}$ particles are blocked, nothing happens as in continuous time. We also note that the usual update from lower to higher levels of Warren-Windridge geometric dynamics for $\mathsf{X}$ gives rise to update rules as in the previous paragraph for $\mathsf{Hgt}(\mathsf{X})$, which as explained are equivalent to sequential-update Bernoulli dynamics. Thus, we observe that the dynamics described above for $\mathsf{Hgt}(\mathsf{X})$ are exactly the sequential update Bernoulli push-block dynamics in environment $\hat{\boldsymbol{\theta}}$ and this completes the proof of the theorem.

\begin{figure}
\captionsetup{singlelinecheck = false, justification=justified}
\centering
\begin{tikzpicture}

 \draw[fill] (0,0) circle [radius=0.1];

 \draw[very thick,->] (0.1,0) to (5.9,0);

 \draw[] (6,0) circle [radius=0.1];

  \draw[] (1,0) circle [radius=0.1];

 \draw[] (2,0) circle [radius=0.1];

 \draw[] (3,0) circle [radius=0.1];

 \draw[] (5,0) circle [radius=0.1];

 \node[below left] at (0,0) {$\mathsf{X}_0^{(j)}=x$};

 \node[below] at (1,0) {$x+1$};

 \node[below] at (2,0) {$x+2$};

 \node[below] at (3,0) {$x+3$};

  \node[below] at (5,0) {$y-1$};

  \node[below] at (6,0) {$y$};

  \node[above] at (4,0) {$\cdots$};

  \node[above left] at (0.5,0.3) {\footnotesize $\boldsymbol{\theta}(x,j)$};

    \node[above] at (1.3,0.3) {\footnotesize $\boldsymbol{\theta}(x+1,j)$};

   \node[above] at (2.5,0.3) {\footnotesize
 $\boldsymbol{\theta}(x+2,j)$};

     %\node[above] at (3.7,0.3) {\footnotesize$\lambda(x+3,j)$};

     \node[above] at (5,0.3) {\footnotesize
 $\boldsymbol{\theta}(y-1,j)$};

     \node[above] at (6,0.8) {\footnotesize
 $1-\boldsymbol{\theta}(y,j)$};

\node[below] at (3,-0.5) {$\left(1-\boldsymbol{\theta}(y,j)\right)\prod_{k=x}^{y-1} \boldsymbol{\theta}(k,j)$};

\draw[middlearrow={>}, very thick] (0.1,0.1) to [out=45, in=135] (0.9,0.1);

\draw[middlearrow={>}, very thick] (1.1,0.1) to [out=45, in=135] (1.9,0.1);

\draw[middlearrow={>}, very thick] (2.1,0.1) to [out=45, in=135] (2.9,0.1);

\draw[middlearrow={>}, very thick] (3.1,0.1) to [out=45, in=180] (3.5,0.3);

\draw[middlearrow={>}, very thick] (4.5,0.3) to [out=0, in=135] (4.9,0.1);

\draw[middlearrow={>}, very thick] (5.1,0.1) to [out=45, in=135] (5.9,0.1);

\draw[middlearrow={>}, very thick]  (6,0.1) to [out=135, in=45, distance=13mm] (6,0.1);

\end{tikzpicture}\ \ \ 
\begin{tikzpicture}

 \draw[fill] (0,0) circle [radius=0.1];

\draw[] (1,0) circle [radius=0.1];

  \draw[fill] (0,1) circle [radius=0.1];

  \draw[] (1,1) circle [radius=0.1];

    \draw[fill] (0,2) circle [radius=0.1];

  \draw[] (1,2) circle [radius=0.1];

    \draw[fill] (0,4) circle [radius=0.1];

  \draw[] (1,4) circle [radius=0.1];

  \node[] at (0.5,3.3) {$\vdots$}; 

      \draw[fill] (0,5) circle [radius=0.1];

\draw[middlearrow={>}, very thick]  (0,5.1) to [out=135, in=45, distance=13mm] (0,5.1);

 \node[below] at (0,-0.2) {$j$};

 \node[below] at (1,-0.2) {$j+1$};

  \node[above] at (0.5,0.3) {\footnotesize $\boldsymbol{\theta}(x,j)=\hat{\boldsymbol{\theta}}(j,x)$};

    \node[above] at (0.5,1.3) {\footnotesize $\boldsymbol{\theta}(x+1,j)=\hat{\boldsymbol{\theta}}(j,x+1)$};

    \node[above] at (0.5,2.3) {\footnotesize $\boldsymbol{\theta}(x+2,j)=\hat{\boldsymbol{\theta}}(j,x+2)$};

    \node[above] at (0.5,4.3) {\footnotesize $\boldsymbol{\theta}(y-1,j)=\hat{\boldsymbol{\theta}}(j,y-1)$};

     \node[above] at (0,5.9) {\footnotesize
 $1-\boldsymbol{\theta}(y,j)=1-\hat{\boldsymbol{\theta}}(j,y)$};

 \node[left] at (0,0) {$\mathsf{Hgt}\left(\mathsf{X}\right)_0^{(x)}$};

 \node[left] at (0,1) {$\mathsf{Hgt}\left(\mathsf{X}\right)_0^{(x+1)}$};

 \node[left] at (0,2) {$\mathsf{Hgt}\left(\mathsf{X}\right)_0^{(x+2)}$};

 \node[left] at (0,4) {$\mathsf{Hgt}\left(\mathsf{X}\right)_0^{(y-1)}$};

 \node[left] at (0,5) {$\mathsf{Hgt}\left(\mathsf{X}\right)_0^{(y)}$};

\draw[middlearrow={>}, very thick] (0.1,0.1) to [out=45, in=135] (0.9,0.1);

\draw[middlearrow={>}, very thick] (0.1,1.1) to [out=45, in=135] (0.9,1.1);

\draw[middlearrow={>}, very thick] (0.1,2.1) to [out=45, in=135] (0.9,2.1);

\draw[middlearrow={>}, very thick] (0.1,4.1) to [out=45, in=135] (0.9,4.1);

\end{tikzpicture}

\caption{Particle $\mathsf{X}_0^{(j)}$ jumps from $x$ to $y$ with inhomogeneous geometric probability $\left(1-\boldsymbol{\theta}(y,j)\right)\prod_{k=x}^{y-1} \boldsymbol{\theta}(k,j)$. This is equivalent to a particle performing sequential Bernoulli trials of whether to move to the next site with success probability $\boldsymbol{\theta}(k,j)$ if at site $k$. On the side of the $\mathsf{Hgt}$ process, this corresponds to, sequentially in $k$, $\mathsf{Hgt}(\mathsf{X})_0^{(k)}$ jumping from $j$ to $j+1$ with probability $\boldsymbol{\theta}(k,j)=\hat{\boldsymbol{\theta}}(j,k)$.}\label{DiscreteTimeHeighFunctionFigure}
\end{figure}
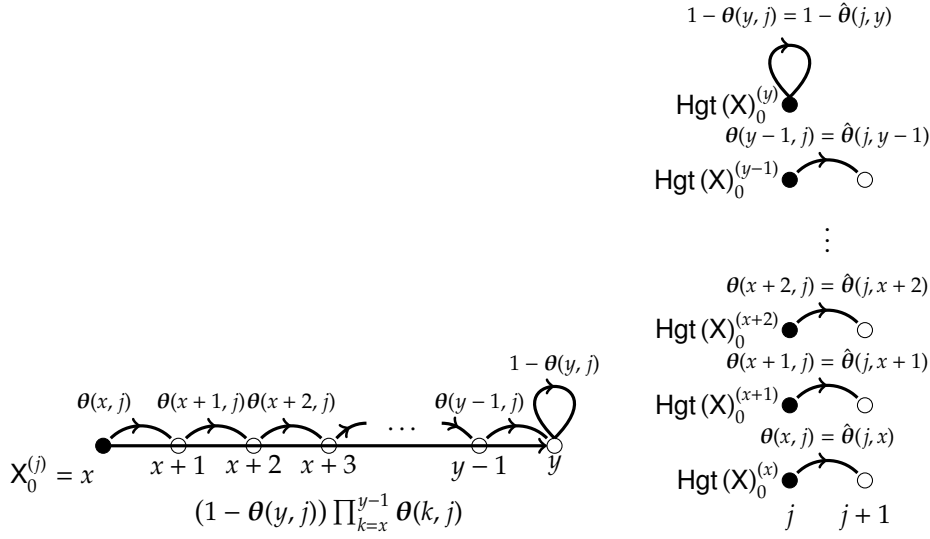

\end{proof}

\begin{rmk}
Observe that, the restriction of $\mathsf{Hgt}$ to the first column, given by the function $j\mapsto \mathsf{h}_0(j)$, is simply the standard height function associated to $\left(\mathsf{X}_0^{(j)}\right)_{j\ge 0}$ viewed as a corner growth model. Then, it is easy to see that the evolution of the height function, viewed as a particle system, follows the corresponding dynamics from Theorem \ref{DualityThmIntro} in environment $\hat{\boldsymbol{\theta}}$ (we are swapping $x$ and $y$), see Figure \ref{CornerGrowthHeightfunction} for an illustration. 
\end{rmk}

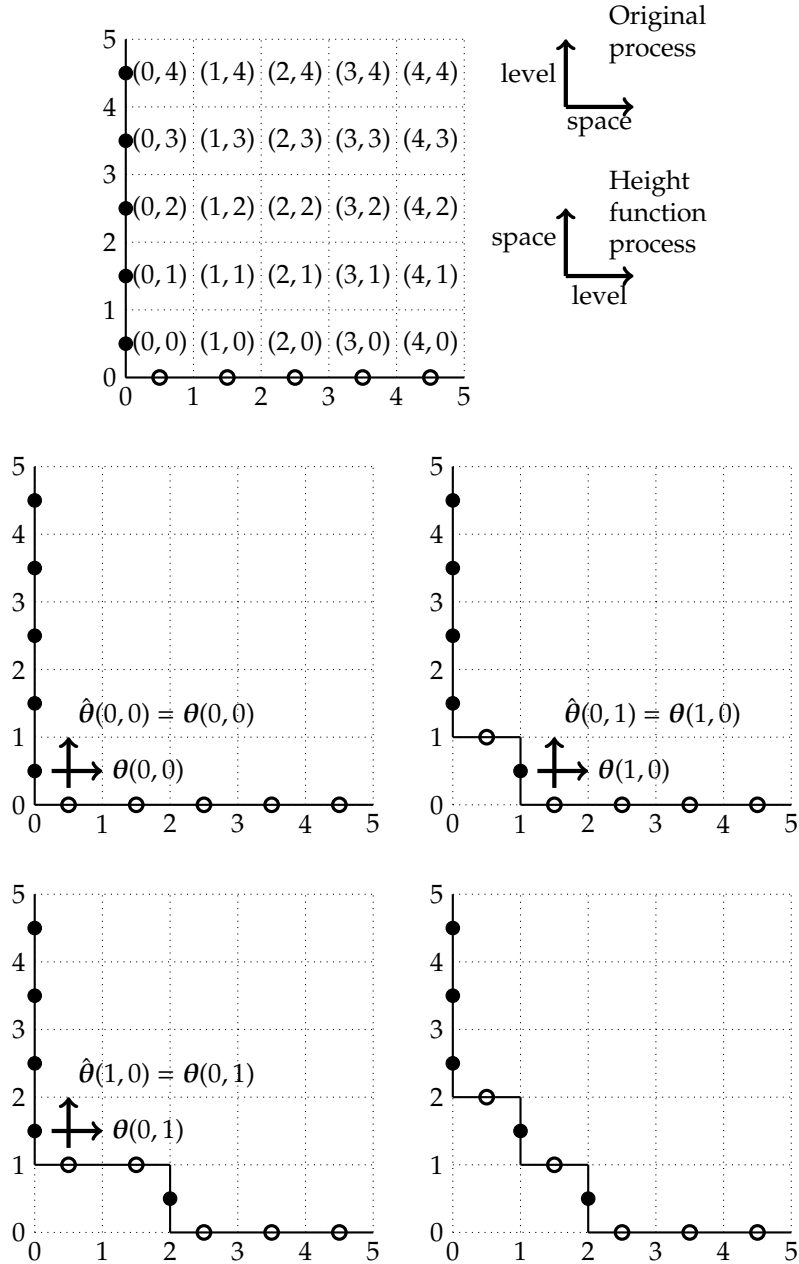
\begin{figure}
\captionsetup{singlelinecheck = false, justification=justified}
\centering
\begin{tikzpicture}[scale=0.90]

\draw[dotted] (0,0) grid (5,5);

\draw[thick] (0,0) to (0,5);

\draw[thick] (0,0) to (5,0);

\draw[fill] (0,0.5) circle [radius=0.1];

 \draw[fill] (0,1.5) circle [radius=0.1];
 
\draw[fill] (0,2.5) circle [radius=0.1];

\draw[fill] (0,3.5) circle [radius=0.1];

\draw[fill] (0,4.5) circle [radius=0.1];

\draw[very thick] (0.5,0) circle [radius=0.1];

 \draw[very thick] (1.5,0) circle [radius=0.1];
 
\draw[very thick] (2.5,0) circle [radius=0.1];

\draw[very thick] (3.5,0) circle [radius=0.1];

\draw[very thick] (4.5,0) circle [radius=0.1];

\node[below] at (0,0) {$0$}; 

\node[below] at (1,0) {$1$}; 

\node[below] at (2,0) {$2$}; 

\node[below] at (3,0) {$3$}; 

\node[below] at (4,0) {$4$}; 

\node[below] at (5,0) {$5$}; 

\node[left] at (0,0) {$0$}; 

\node[left] at (0,1) {$1$}; 

\node[left] at (0,2) {$2$}; 

\node[left] at (0,3) {$3$}; 

\node[left] at (0,4) {$4$}; 

\node[left] at (0,5) {$5$}; 

\node[] at (0.5,0.5) {$(0,0)$};

\node[] at (1.5,0.5) {$(1,0)$};

\node[] at (2.5,0.5) {$(2,0)$};

\node[] at (3.5,0.5) {$(3,0)$};

\node[] at (4.5,0.5) {$(4,0)$};

\node[] at (0.5,1.5) {$(0,1)$};

\node[] at (1.5,1.5) {$(1,1)$};

\node[] at (2.5,1.5) {$(2,1)$};

\node[] at (3.5,1.5) {$(3,1)$};

\node[] at (4.5,1.5) {$(4,1)$};

\node[] at (0.5,2.5) {$(0,2)$};

\node[] at (1.5,2.5) {$(1,2)$};

\node[] at (2.5,2.5) {$(2,2)$};

\node[] at (3.5,2.5) {$(3,2)$};

\node[] at (4.5,2.5) {$(4,2)$};

\node[] at (0.5,3.5) {$(0,3)$};

\node[] at (1.5,3.5) {$(1,3)$};

\node[] at (2.5,3.5) {$(2,3)$};

\node[] at (3.5,3.5) {$(3,3)$};

\node[] at (4.5,3.5) {$(4,3)$};

\node[] at (0.5,4.5) {$(0,4)$};

\node[] at (1.5,4.5) {$(1,4)$};

\node[] at (2.5,4.5) {$(2,4)$};

\node[] at (3.5,4.5) {$(3,4)$};

\node[] at (4.5,4.5) {$(4,4)$};

\draw[ultra thick,->] (6.5,4) -- (6.5,5);

\draw[ultra thick,->] (6.5,4) -- (7.5,4);

\node[left] at (6.5,4.5) {level};

\node[below] at (7,4) {space};

\node[align=left,above right] at (7,4.5) {Original\\ process};

\draw[ultra thick,->] (6.5,1.5) -- (6.5,2.5);

\draw[ultra thick,->] (6.5,1.5) -- (7.5,1.5);

\node[left] at (6.5,2) {space};

\node[below] at (7,1.5) {level};

\node[align=left,above right] at (7,1.6) {Height\\ function \\ process};

\end{tikzpicture}

\bigskip 

\begin{tikzpicture}[scale=0.90]

\draw[dotted] (0,0) grid (5,5);

\draw[thick] (0,0) to (0,5);

\draw[thick] (0,0) to (5,0);

\draw[fill] (0,0.5) circle [radius=0.1];

 \draw[fill] (0,1.5) circle [radius=0.1];
 
\draw[fill] (0,2.5) circle [radius=0.1];

\draw[fill] (0,3.5) circle [radius=0.1];

\draw[fill] (0,4.5) circle [radius=0.1];

\draw[very thick] (0.5,0) circle [radius=0.1];

 \draw[very thick] (1.5,0) circle [radius=0.1];
 
\draw[very thick] (2.5,0) circle [radius=0.1];

\draw[very thick] (3.5,0) circle [radius=0.1];

\draw[very thick] (4.5,0) circle [radius=0.1];

\node[below] at (0,0) {$0$}; 

\node[below] at (1,0) {$1$}; 

\node[below] at (2,0) {$2$}; 

\node[below] at (3,0) {$3$}; 

\node[below] at (4,0) {$4$}; 

\node[below] at (5,0) {$5$}; 

\node[left] at (0,0) {$0$}; 

\node[left] at (0,1) {$1$}; 

\node[left] at (0,2) {$2$}; 

\node[left] at (0,3) {$3$}; 

\node[left] at (0,4) {$4$}; 

\node[left] at (0,5) {$5$};

\draw[ultra thick,->] (0.25,0.5) to (1,0.5);

\draw[ultra thick,->] (0.5,0.25) to (0.5,1);

\node[right] at (1,0.5) {$\boldsymbol{\theta}(0,0)$};

\node[above right] at (0.5,1) {$\hat{\boldsymbol{\theta}}(0,0)=\boldsymbol{\theta}(0,0)$};

\end{tikzpicture}\ \ \ \ \ 
\begin{tikzpicture}[scale=0.90]

\draw[dotted] (0,0) grid (5,5);

\draw[thick] (0,1) to (0,5);

\draw[thick] (1,0) to (5,0);

\draw[thick] (0,1) to (1,1);

\draw[thick] (1,1) to (1,0);

\draw[fill] (1,0.5) circle [radius=0.1];

 \draw[fill] (0,1.5) circle [radius=0.1];
 
\draw[fill] (0,2.5) circle [radius=0.1];

\draw[fill] (0,3.5) circle [radius=0.1];

\draw[fill] (0,4.5) circle [radius=0.1];

\draw[very thick] (0.5,1) circle [radius=0.1];

 \draw[very thick] (1.5,0) circle [radius=0.1];
 
\draw[very thick] (2.5,0) circle [radius=0.1];

\draw[very thick] (3.5,0) circle [radius=0.1];

\draw[very thick] (4.5,0) circle [radius=0.1];

\node[below] at (0,0) {$0$}; 

\node[below] at (1,0) {$1$}; 

\node[below] at (2,0) {$2$}; 

\node[below] at (3,0) {$3$}; 

\node[below] at (4,0) {$4$}; 

\node[below] at (5,0) {$5$}; 

\node[left] at (0,0) {$0$}; 

\node[left] at (0,1) {$1$}; 

\node[left] at (0,2) {$2$}; 

\node[left] at (0,3) {$3$}; 

\node[left] at (0,4) {$4$}; 

\node[left] at (0,5) {$5$};

\draw[ultra thick,->] (1.25,0.5) to (2,0.5);

\draw[ultra thick,->] (1.5,0.25) to (1.5,1);

\node[right] at (2,0.5) {$\boldsymbol{\theta}(1,0)$};

\node[above right] at (1.5,1) {$\hat{\boldsymbol{\theta}}(0,1)=\boldsymbol{\theta}(1,0)$};

\end{tikzpicture}

\bigskip

\begin{tikzpicture}[scale=0.90]

\draw[dotted] (0,0) grid (5,5);

\draw[thick] (0,1) to (0,5);

\draw[thick] (2,0) to (5,0);

\draw[thick] (0,1) to (2,1);

\draw[thick] (2,1) to (2,0);

\draw[fill] (2,0.5) circle [radius=0.1];

 \draw[fill] (0,1.5) circle [radius=0.1];
 
\draw[fill] (0,2.5) circle [radius=0.1];

\draw[fill] (0,3.5) circle [radius=0.1];

\draw[fill] (0,4.5) circle [radius=0.1];

\draw[very thick] (0.5,1) circle [radius=0.1];

 \draw[very thick] (1.5,1) circle [radius=0.1];
 
\draw[very thick] (2.5,0) circle [radius=0.1];

\draw[very thick] (3.5,0) circle [radius=0.1];

\draw[very thick] (4.5,0) circle [radius=0.1];

\node[below] at (0,0) {$0$}; 

\node[below] at (1,0) {$1$}; 

\node[below] at (2,0) {$2$}; 

\node[below] at (3,0) {$3$}; 

\node[below] at (4,0) {$4$}; 

\node[below] at (5,0) {$5$}; 

\node[left] at (0,0) {$0$}; 

\node[left] at (0,1) {$1$}; 

\node[left] at (0,2) {$2$}; 

\node[left] at (0,3) {$3$}; 

\node[left] at (0,4) {$4$}; 

\node[left] at (0,5) {$5$};

\draw[ultra thick,->] (0.25,1.5) to (1,1.5);

\draw[ultra thick,->] (0.5,1.25) to (0.5,2);

\node[right] at (1,1.5) {$\boldsymbol{\theta}(0,1)$};

\node[above right] at (0.5,2) {$\hat{\boldsymbol{\theta}}(1,0)=\boldsymbol{\theta}(0,1)$};

\end{tikzpicture}\ \ \ \ \ 
\begin{tikzpicture}[scale=0.90]

\draw[dotted] (0,0) grid (5,5);

\draw[thick] (0,2) to (0,5);

\draw[thick] (2,0) to (5,0);

\draw[thick] (0,2) to (1,2);

\draw[thick] (1,2) to (1,1);

\draw[thick] (1,1) to (2,1);

\draw[thick] (2,1) to (2,0);

\draw[fill] (2,0.5) circle [radius=0.1];

 \draw[fill] (1,1.5) circle [radius=0.1];
 
\draw[fill] (0,2.5) circle [radius=0.1];

\draw[fill] (0,3.5) circle [radius=0.1];

\draw[fill] (0,4.5) circle [radius=0.1];

\draw[very thick] (0.5,2) circle [radius=0.1];

 \draw[very thick] (1.5,1) circle [radius=0.1];
 
\draw[very thick] (2.5,0) circle [radius=0.1];

\draw[very thick] (3.5,0) circle [radius=0.1];

\draw[very thick] (4.5,0) circle [radius=0.1];

\node[below] at (0,0) {$0$}; 

\node[below] at (1,0) {$1$}; 

\node[below] at (2,0) {$2$}; 

\node[below] at (3,0) {$3$}; 

\node[below] at (4,0) {$4$}; 

\node[below] at (5,0) {$5$}; 

\node[left] at (0,0) {$0$}; 

\node[left] at (0,1) {$1$}; 

\node[left] at (0,2) {$2$}; 

\node[left] at (0,3) {$3$}; 

\node[left] at (0,4) {$4$}; 

\node[left] at (0,5) {$5$};

\end{tikzpicture}

\caption{The autonomous left-edge particle system $(\mathsf{X}_0^{(j)})_{j\ge 0}$ is depicted as filled circles. The corresponding height functions $(\mathsf{Hgt}(\mathsf{X})_0^{(j)})_{j\ge 0}=(\mathsf{h}_0(j))_{j\ge 0}$, viewed as particles, are depicted as non-filled circles. In the first figure the coordinates of the squares correspond to (space,level). We consider continuous-time pure-birth dynamics. Particle $\mathsf{X}_0^{(j)}$ at location $x$ moves to the right by one at rate $\boldsymbol{\theta}(x,j)$, subject to the constraint that it does not overtake $\mathsf{X}_0^{(j-1)}$. As shown in the figure, this gives rise to a move of $\mathsf{Hgt}(\mathsf{X})_0^{(x)}$ from $j$ to $j+1$ at rate $\boldsymbol{\theta}(x,j)=\hat{\boldsymbol{\theta}}(j,x)$. When viewing the above in terms of the growth of the height function shown in the figure the result is intuitively clear.}\label{CornerGrowthHeightfunction}
\end{figure}

\section{The domino tiling shuffling algorithm connection}\label{SectionShuffling}

\subsection{The Aztec diamond graph, probability measures on dimers, square probabilities, gauge equivalence, particles}

\begin{figure}
\captionsetup{singlelinecheck = false, justification=justified}
\centering
\begin{tikzpicture}

\draw[ultra thick,blue] (2,2) to (2,3);
\draw[ultra thick,orange] (3,2) to (4,2);

 \draw[fill] (2,0) circle [radius=0.05];
\draw[fill] (3,0) circle [radius=0.05];

 \draw[fill] (1,1) circle [radius=0.05];
\draw[fill] (2,1) circle [radius=0.05];
 \draw[fill] (3,1) circle [radius=0.05];
  \draw[fill] (4,1) circle [radius=0.05];

\draw[fill] (0,2) circle [radius=0.05];
 \draw[fill] (1,2) circle [radius=0.05];
\draw[fill] (2,2) circle [radius=0.05];
 \draw[fill] (3,2) circle [radius=0.05];
\draw[fill] (4,2) circle [radius=0.05];
 \draw[fill] (5,2) circle [radius=0.05];

\draw[fill] (0,3) circle [radius=0.05];
 \draw[fill] (1,3) circle [radius=0.05];
\draw[fill] (2,3) circle [radius=0.05];
 \draw[fill] (3,3) circle [radius=0.05];
\draw[fill] (4,3) circle [radius=0.05];
 \draw[fill] (5,3) circle [radius=0.05];

  \draw[fill] (1,4) circle [radius=0.05];
\draw[fill] (2,4) circle [radius=0.05];
 \draw[fill] (3,4) circle [radius=0.05];
  \draw[fill] (4,4) circle [radius=0.05];

 \draw[fill] (2,5) circle [radius=0.05];
\draw[fill] (3,5) circle [radius=0.05];

\draw[very thick] (2,0) to (3,0);
\draw[very thick] (2,0) to (2,1);
\draw[very thick] (3,0) to (3,1);
\draw[very thick] (2,1) to (3,1);

\draw[very thick] (1,1) to (2,1);
\draw[very thick] (3,1) to (4,1);

\draw[very thick] (1,1) to (1,2);
\draw[very thick] (2,1) to (2,2);
\draw[very thick] (3,1) to (3,2);
\draw[very thick] (4,1) to (4,2);

\draw[very thick] (0,2) to (1,2);
\draw[very thick] (1,2) to (2,2);
\draw[very thick] (2,2) to (3,2);
\draw[very thick] (4,2) to (5,2);

\draw[very thick] (0,2) to (0,3);
\draw[very thick] (1,2) to (1,3);
\draw[very thick] (3,2) to (3,3);
\draw[very thick] (4,2) to (4,3);
\draw[very thick] (5,2) to (5,3);

\draw[very thick] (0,3) to (1,3);
\draw[very thick] (1,3) to (2,3);
\draw[very thick] (2,3) to (3,3);
\draw[very thick] (3,3) to (4,3);
\draw[very thick] (4,3) to (5,3);

\draw[very thick] (1,4) to (1,3);
\draw[very thick] (2,4) to (2,3);
\draw[very thick] (3,4) to (3,3);
\draw[very thick] (4,4) to (4,3);

\draw[very thick] (1,4) to (2,4);
\draw[very thick] (2,4) to (3,4);
\draw[very thick] (3,4) to (4,4);

\draw[very thick] (2,4) to (2,5);
\draw[very thick] (3,4) to (3,5);

\draw[very thick] (2,5) to (3,5);

\node[] at (0.5,2.5) {$(0,1)$};

\node[] at (1.5,1.5) {$(0,2)$};

\node[] at (2.5,0.5) {$(0,3)$};

\node[] at (1.5,3.5) {$(1,1)$};

\node[] at (2.5,2.5) {$(1,2)$};

\node[] at (3.5,1.5) {$(1,3)$};

\node[] at (2.5,4.5) {$(2,1)$};

\node[] at (3.5,3.5) {$(2,2)$};

\node[] at (4.5,2.5) {$(2,3)$};

\draw[ultra thick,->] (-0.5,3) to (0.5,4);

\node[above] at (0,4) {space $x=0,\dots, N-1$};

\draw[ultra thick,->] (-0.5,2) to (0.5,1);

\node[below] at (0,1) {level $n=1,\dots, N$};

\end{tikzpicture}
\begin{tikzpicture}
 \draw[fill] (0,0) circle [radius=0.05];
\draw[fill] (0,1) circle [radius=0.05];
 \draw[fill] (1,0) circle [radius=0.05];
\draw[fill] (1,1) circle [radius=0.05];

\draw[very thick] (0,0) to (0,1);
\draw[very thick] (0,1) to (1,1);
\draw[very thick] (0,0) to (1,0);
\draw[very thick] (1,1) to (1,0);

\node[] at (0.5,0.5) {$(x,n)$};

\node[left] at (0,0.5) {w};
\node[right] at (1,0.5) {e};
\node[below] at (0.5,0) {s};
\node[above] at (0.5,1) {n};

\end{tikzpicture}

\caption{The Aztec diamond graph $\mathsf{AG}_N$ for $N=3$, along  with the corresponding coordinate system for the $N^2=9$ squares it consists of. The coordinates of the blue vertical edge are $(\textnormal{w},(1,3))$, while the coordinates of the horizontal orange edge are $(\textnormal{n},(1,3))$.}\label{AztecDiamondGraph}
\end{figure}
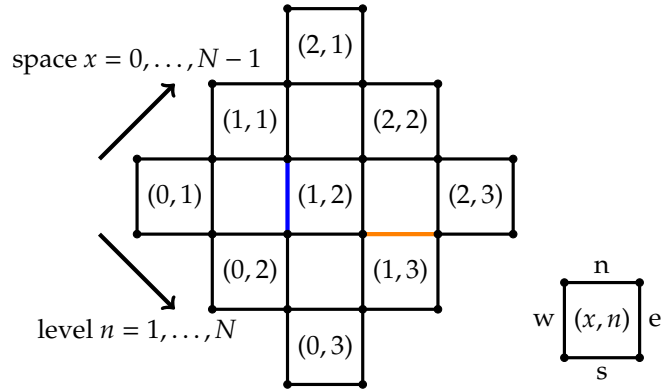

Instead of considering domino tilings of the Aztec diamond it is equivalent, and more convenient, to consider dimer coverings of the associated Aztec diamond graph. The Aztec diamond graph $\mathsf{AG}_N$ of size $N$, consisting of $N^2$ squares whose vertices and edges constitute the vertex and edge set, is defined as in Figure \ref{AztecDiamondGraph}. We associate to it a coordinate system as in Figure \ref{AztecDiamondGraph}, with space-level coordinate $(x,n)$ being the centre of the corresponding square. We identify a square with its space-level coordinate so that $\mathsf{AG}_N$ consists of squares $\{(x,n):0\le x \le N-1, 1\le n \le N\}$.
Each square consists of $4$ edges and we call these north, south, west and east in the obvious way, see Figure \ref{AztecDiamondGraph}. We denote an edge $\mathsf{e}$ of $\mathsf{AG}_N$ by $\mathsf{e}=(\bullet,(x,n))$ where $(x,n)$ gives the square the edge belongs to and $\bullet\in \{\textnormal{n},\textnormal{s},\textnormal{w},\textnormal{e}\}$ identifies which of the four edges of that square it is. 

A dimer covering of $\mathsf{AG}_N$ is a collection of edges of $\mathsf{AG}_N$ such that every vertex is covered and no two edges are incident on the same vertex. We denote the set of all dimer coverings of $\mathsf{AG}_N$ by $\mathsf{DC}_N$. It is easy to see that a dimer covering 
of $\mathsf{AG}_N$ is equivalent to a tiling of the Aztec diamond of size $N$ by dominos, see for example Figure \ref{EquivalenceTilingDimer} for an illustration. We denote the set of all dimer coverings of $\mathsf{AG}_N$ by $\mathsf{DC}_N$. It is well-known, see \cite{ProppShuffle}, that $|\mathsf{DC}_N|=2^{\frac{N(N+1)}{2}}$.

\begin{figure}
\captionsetup{singlelinecheck = false, justification=justified}
\centering
\begin{tikzpicture}

\fill[lightgray] (0,0) rectangle (0.5,0.5);
\fill[lightgray] (0.5,0.5) rectangle (1,1);
\draw[ultra thick] (0,0) rectangle (0.5,1);
\draw[ultra thick] (0.5,0) rectangle (1,1);

\draw[very thick, <->]     (1.3,0.5)   -- (1.8,0.5);

 \draw[fill] (2.25,0.25) circle [radius=0.05];
\draw[fill] (2.25,0.75) circle [radius=0.05];

\draw[fill] (2.75,0.25) circle [radius=0.05];
\draw[fill] (2.75,0.75) circle [radius=0.05];

\draw[very thick] (2.25,0.25) to (2.25,0.75);
\draw[very thick] (2.75,0.25) to (2.75,0.75);

\end{tikzpicture}\ \ \ \ \ \
\begin{tikzpicture}

\fill[lightgray] (0,0) rectangle (0.5,0.5);
\fill[lightgray] (0.5,0.5) rectangle (1,1);
\draw[ultra thick] (0,0) rectangle (1,0.5);
\draw[ultra thick] (0,0.5) rectangle (1,1);

\draw[very thick, <->]     (1.3,0.5)   -- (1.8,0.5);

 \draw[fill] (2.25,0.25) circle [radius=0.05];
\draw[fill] (2.25,0.75) circle [radius=0.05];

\draw[fill] (2.75,0.25) circle [radius=0.05];
\draw[fill] (2.75,0.75) circle [radius=0.05];

\draw[very thick] (2.25,0.25) to (2.75,0.25);
\draw[very thick] (2.25,0.75) to (2.75,0.75);

\end{tikzpicture}

\bigskip

\bigskip

\begin{tikzpicture}

\fill[lightgray] (0,0.5) rectangle (0.5,1);
\fill[lightgray] (0.5,1) rectangle (1,1.5);
\fill[lightgray] (1,1.5) rectangle (1.5,2);

\fill[lightgray] (0.5,0) rectangle (1,0.5);
\fill[lightgray] (1,0.5) rectangle (1.5,1);
\fill[lightgray] (1.5,1) rectangle (2,1.5);

\draw[ultra thick] (0.5,0) rectangle (1.5,0.5);
\draw[ultra thick] (0,0.5) rectangle (0.5,1.5);
\draw[ultra thick] (0.5,0.5) rectangle (1,1.5);
\draw[ultra thick] (1,0.5) rectangle (1.5,1.5);
\draw[ultra thick] (1.5,0.5) rectangle (2,1.5);
\draw[ultra thick] (0.5,1.5) rectangle (1.5,2);

\draw[very thick, <->]     (2.3,1)   -- (2.8,1);

 \draw[fill] (3.25,0.75) circle [radius=0.05];
\draw[fill] (3.25,1.25) circle [radius=0.05];

 \draw[fill] (3.75,0.75) circle [radius=0.05];
\draw[fill] (3.75,1.25) circle [radius=0.05];

 \draw[fill] (4.25,0.75) circle [radius=0.05];
\draw[fill] (4.25,1.25) circle [radius=0.05];

 \draw[fill] (4.75,0.75) circle [radius=0.05];
\draw[fill] (4.75,1.25) circle [radius=0.05];

 \draw[fill] (3.75,0.25) circle [radius=0.05];
\draw[fill] (4.25,0.25) circle [radius=0.05];

 \draw[fill] (3.75,1.75) circle [radius=0.05];
\draw[fill] (4.25,1.75) circle [radius=0.05];

\draw[very thick] (3.25,0.75) to (3.25,1.25);

\draw[very thick] (3.75,0.75) to (3.75,1.25);

\draw[very thick] (4.25,0.75) to (4.25,1.25);

\draw[very thick] (4.75,0.75) to (4.75,1.25);

\draw[very thick] (3.75,0.25) to (4.25,0.25);

\draw[very thick] (3.75,1.75) to (4.25,1.75);

\end{tikzpicture}\ \ \ \ \ 
\begin{tikzpicture}

\fill[lightgray] (0,0.5) rectangle (0.5,1);
\fill[lightgray] (0.5,1) rectangle (1,1.5);
\fill[lightgray] (1,1.5) rectangle (1.5,2);

\fill[lightgray] (0.5,0) rectangle (1,0.5);
\fill[lightgray] (1,0.5) rectangle (1.5,1);
\fill[lightgray] (1.5,1) rectangle (2,1.5);

\draw[ultra thick] (0.5,0) rectangle (1.5,0.5);
\draw[ultra thick] (0,0.5) rectangle (1,1);
\draw[ultra thick] (0,1) rectangle (1,1.5);

\draw[ultra thick] (1,0.5) rectangle (1.5,1.5);
\draw[ultra thick] (1.5,0.5) rectangle (2,1.5);
\draw[ultra thick] (0.5,1.5) rectangle (1.5,2);

\draw[very thick, <->]     (2.3,1)   -- (2.8,1);

 \draw[fill] (3.25,0.75) circle [radius=0.05];
\draw[fill] (3.25,1.25) circle [radius=0.05];

 \draw[fill] (3.75,0.75) circle [radius=0.05];
\draw[fill] (3.75,1.25) circle [radius=0.05];

 \draw[fill] (4.25,0.75) circle [radius=0.05];
\draw[fill] (4.25,1.25) circle [radius=0.05];

 \draw[fill] (4.75,0.75) circle [radius=0.05];
\draw[fill] (4.75,1.25) circle [radius=0.05];

 \draw[fill] (3.75,0.25) circle [radius=0.05];
\draw[fill] (4.25,0.25) circle [radius=0.05];

 \draw[fill] (3.75,1.75) circle [radius=0.05];
\draw[fill] (4.25,1.75) circle [radius=0.05];

\draw[very thick] (3.25,0.75) to (3.75,0.75);

\draw[very thick] (3.25,1.25) to (3.75,1.25);

\draw[very thick] (4.25,0.75) to (4.25,1.25);

\draw[very thick] (4.75,0.75) to (4.75,1.25);

\draw[very thick] (3.75,0.25) to (4.25,0.25);

\draw[very thick] (3.75,1.75) to (4.25,1.75);

\end{tikzpicture}

\caption{An illustration of the correspondence between domino tilings of the Aztec diamond and dimer coverings of the corresponding graph.}\label{EquivalenceTilingDimer}
\end{figure}
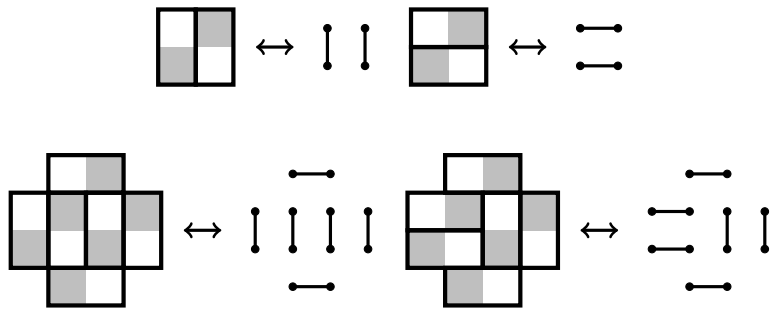

\begin{defn}
 A weighting $\mathcal{W}$ of $\mathsf{AG}_N$ is a function from the edge set of $\mathsf{AG}_N$ to $(0,\infty)$. We write $\mathcal{W}_{\mathsf{e}}$ for its value at the edge $\mathsf{e}$. Given a weighting $\mathcal{W}$ of $\mathsf{AG}_N$ we define the probability measure $\mathbb{P}_{\mathcal{W}}^{(N)}$ on dimer coverings $\mathsf{DC}_N$ as follows
\begin{equation*}
\mathbb{P}_{\mathcal{W}}^{(N)}\left(\mathfrak{d}\right)=\frac{1}{Z_\mathcal{W}} \prod_{\mathsf{e}\in \mathfrak{d}} \mathcal{W}_{\mathsf{e}}, \ \ \mathfrak{d}\in \mathsf{DC}_N,
\end{equation*}
where $Z_{\mathcal{W}}=\sum_{\mathfrak{d}\in \mathsf{DC}_N} \prod_{\mathsf{e}\in \mathfrak{d}}\mathcal{W}_{\mathsf{e}}$ is the normalisation constant/partition function.   
\end{defn}

 We observe that, by virtue of strict positivity of $\mathcal{W}$, the measure $\mathbb{P}^{(N)}_{\mathcal{W}}$ is supported on the whole of $\mathsf{DC}_N$.

\begin{defn}
Given a weighting $\mathcal{W}$ of $\mathsf{AG}_N$ as above, we associate certain probabilities $\rho_{\mathcal{W}}$ to each square $(x,n)$ of $\mathsf{AG}_N$ as follows:
\begin{align*}
\rho_\mathcal{W}(x,n)=\frac{\mathcal{W}_{\textnormal{w},(x,n)}\mathcal{W}_{\textnormal{e},(x,n)}}{\mathcal{W}_{\textnormal{w},(x,n)}\mathcal{W}_{\textnormal{e},(x,n)}+\mathcal{W}_{\textnormal{n},(x,n)}\mathcal{W}_{\textnormal{s},(x,n)}}.
\end{align*}

\end{defn}

Observe that, $\rho_\mathcal{W}(x,n)$ is strictly between $0$ and $1$.

An elementary gauge transformation of a weighting $\mathcal{W}$ of $\mathsf{AG}_N$ is a multiplication of the weights of edges incident to a single vertex in $\mathsf{AG}_N$ by the same number in $(0,\infty)$. See Figure \ref{GaugeEquivalentFigure}. A gauge transformation is the application of a finite number of elementary gauge transformations. We say that two weightings $\mathcal{W}$ and $\tilde{\mathcal{W}}$ of $\mathsf{AG}_N$ are gauge-equivalent if one can be obtained from the other by a gauge transformation. The following lemma is easy to see.

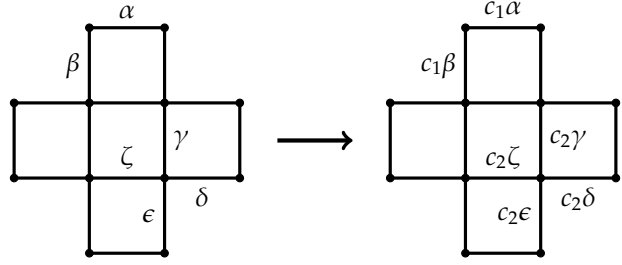
\begin{figure}
\captionsetup{singlelinecheck = false, justification=justified}
\centering
\begin{tikzpicture}

 \draw[fill] (1,0) circle [radius=0.05];
\draw[fill] (2,0) circle [radius=0.05];

 \draw[fill] (0,1) circle [radius=0.05];
\draw[fill] (0,2) circle [radius=0.05];

\draw[fill] (1,1) circle [radius=0.05];
\draw[fill] (2,1) circle [radius=0.05];

\draw[fill] (1,3) circle [radius=0.05];
\draw[fill] (1,2) circle [radius=0.05];

\draw[fill] (3,1) circle [radius=0.05];
\draw[fill] (3,2) circle [radius=0.05];

\draw[fill] (2,2) circle [radius=0.05];
\draw[fill] (2,3) circle [radius=0.05];

\draw[very thick] (1,0) to (2,0);
\draw[very thick] (0,1) to (1,1);

\draw[very thick] (1,0) to (1,1);
\draw[very thick] (1,1) to (2,1);
\draw[very thick] (2,0) to (2,1);

\draw[very thick] (2,1) to (3,1);
\draw[very thick] (3,1) to (3,2);

\draw[very thick] (3,2) to (2,2);
\draw[very thick] (2,2) to (2,1);

\draw[very thick] (0,1) to (0,2);
\draw[very thick] (0,2) to (1,2);

\draw[very thick] (1,2) to (1,1);
\draw[very thick] (1,2) to (1,3);

\draw[very thick] (1,3) to (2,3);
\draw[very thick] (2,3) to (2,2);

\draw[very thick] (2,2) to (1,2);

\node[above] at (1.5,3) {$\alpha$};

\node[left] at (1,2.5) {$\beta$};

\node[right] at (2,1.5) {$\gamma$};

\node[below] at (2.5,1) {$\delta$};

\node[left] at (2,0.5) {$\epsilon$};

\node[above] at (1.5,1) {$\zeta$};

\draw[ultra thick,->] (3.5,1.5) to (4.5,1.5) ;

 \draw[fill] (6,0) circle [radius=0.05];
\draw[fill] (7,0) circle [radius=0.05];

 \draw[fill] (5,1) circle [radius=0.05];
\draw[fill] (5,2) circle [radius=0.05];

\draw[fill] (6,1) circle [radius=0.05];
\draw[fill] (7,1) circle [radius=0.05];

\draw[fill] (6,3) circle [radius=0.05];
\draw[fill] (6,2) circle [radius=0.05];

\draw[fill] (8,1) circle [radius=0.05];
\draw[fill] (8,2) circle [radius=0.05];

\draw[fill] (7,2) circle [radius=0.05];
\draw[fill] (7,3) circle [radius=0.05];

\draw[very thick] (6,0) to (7,0);
\draw[very thick] (5,1) to (6,1);

\draw[very thick] (6,0) to (6,1);
\draw[very thick] (6,1) to (7,1);
\draw[very thick] (7,0) to (7,1);

\draw[very thick] (7,1) to (8,1);
\draw[very thick] (8,1) to (8,2);

\draw[very thick] (8,2) to (7,2);
\draw[very thick] (7,2) to (7,1);

\draw[very thick] (5,1) to (5,2);
\draw[very thick] (5,2) to (6,2);

\draw[very thick] (6,2) to (6,1);
\draw[very thick] (6,2) to (6,3);

\draw[very thick] (6,3) to (7,3);
\draw[very thick] (7,3) to (7,2);

\draw[very thick] (7,2) to (6,2);

\node[above] at (6.5,3) {$c_1 \alpha$};

\node[left] at (6,2.5) {$c_1 \beta$};

\node[right] at (7,1.5) {$c_2 \gamma$};

\node[below] at (7.5,1) {$c_2 \delta$};

\node[left] at (7,0.5) {$c_2 \epsilon$};

\node[above] at (6.5,1) {$c_2 \zeta$};

\end{tikzpicture}

\caption{Two elementary gauge transformations. The corresponding weights are multiplied by $c_1$ and $c_2$ respectively.}\label{GaugeEquivalentFigure}
\end{figure}

\begin{lem}\label{LemGaugeEquivalence}
Suppose $\mathcal{W}$ and $\tilde{\mathcal{W}}$ are gauge-equivalent weightings of $\mathsf{AG}_N$. Then, we have
\begin{equation}
\mathbb{P}_{\mathcal{W}}^{(N)}=\mathbb{P}^{(N)}_{\tilde{\mathcal{W}}} \ \textnormal{ and } \rho_{\mathcal{W}}(x,n)=\rho_{\tilde{\mathcal{W}}}(x,n) \textnormal{ for all squares } (x,n) \textnormal{ in } \mathsf{AG}_N.
\end{equation}
\end{lem}

Given a dimer cover $\mathfrak{d} \in \mathsf{DC}_N$ of $\mathsf{AG}_N$ we can associate to it a particle configuration by putting a particle in each square with a south or east dominos, see Figure \ref{DimerParticles}. The particle inherits the coordinates of the square. It is a combinatorial fact that there exist exactly $n$ particles at level $n$. We order them in terms of their horizontal $x$ coordinate so that the particle configuration at level $n$ is in $\mathbb{W}_n$. 

We note that the particle configuration is not one-to-one with the dimer covering since some information is lost. For example squares corresponding to particles can be covered by either a north-south or west-east dimer pair. This information can be recovered by keeping track of two sets of particles and extending the state space, see \cite{DuitsKuijlaars}, but we will not do it here.

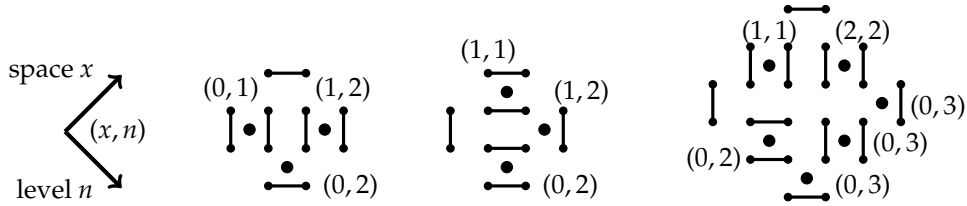
\begin{figure}
\captionsetup{singlelinecheck = false, justification=justified}
\centering
\begin{tikzpicture}
  \draw[ultra thick,->] (0,0) -- (0.75,0.75);

\draw[ultra thick,->] (0,0) -- (0.75,-0.75);

\node[above left] at (0.5,0.5) {space $x$};

\node[below left] at (0.5,-0.5) {level $n$};  

\node[right] at (0.2,0) {$(x,n)$};
\end{tikzpicture}\ \ \ \ \ \
\begin{tikzpicture}

 \draw[fill] (0.25,0.75) circle [radius=0.05];
\draw[fill] (0.25,1.25) circle [radius=0.05];

 \draw[fill] (0.75,0.75) circle [radius=0.05];
\draw[fill] (0.75,1.25) circle [radius=0.05];

 \draw[fill] (1.25,0.75) circle [radius=0.05];
\draw[fill] (1.25,1.25) circle [radius=0.05];

 \draw[fill] (1.75,0.75) circle [radius=0.05];
\draw[fill] (1.75,1.25) circle [radius=0.05];

 \draw[fill] (0.75,0.25) circle [radius=0.05];
\draw[fill] (1.25,0.25) circle [radius=0.05];

 \draw[fill] (0.75,1.75) circle [radius=0.05];
\draw[fill] (1.25,1.75) circle [radius=0.05];

\draw[very thick] (0.25,0.75) to (0.25,1.25);

\draw[very thick] (0.75,0.75) to (0.75,1.25);

\draw[very thick] (1.25,0.75) to (1.25,1.25);

\draw[very thick] (1.75,0.75) to (1.75,1.25);

\draw[very thick] (0.75,0.25) to (1.25,0.25);

\draw[very thick] (0.75,1.75) to (1.25,1.75);

\draw[fill] (0.5,1) circle [radius=0.075];

\draw[fill] (1.5,1) circle [radius=0.075];

\draw[fill] (1,0.5) circle [radius=0.075];

\node[above] at (0.25,1.25) {$(0,1)$};

\node[above right] at (1.25,1.25) {$(1,2)$};

\node[ right] at (1.35,0.25) {$(0,2)$};

\end{tikzpicture}\ \ \ \ \ \ \ \ \
\begin{tikzpicture}

 \draw[fill] (0.25,0.75) circle [radius=0.05];
\draw[fill] (0.25,1.25) circle [radius=0.05];

 \draw[fill] (0.75,0.75) circle [radius=0.05];
\draw[fill] (0.75,1.25) circle [radius=0.05];

 \draw[fill] (1.25,0.75) circle [radius=0.05];
\draw[fill] (1.25,1.25) circle [radius=0.05];

 \draw[fill] (1.75,0.75) circle [radius=0.05];
\draw[fill] (1.75,1.25) circle [radius=0.05];

 \draw[fill] (0.75,0.25) circle [radius=0.05];
\draw[fill] (1.25,0.25) circle [radius=0.05];

 \draw[fill] (0.75,1.75) circle [radius=0.05];
\draw[fill] (1.25,1.75) circle [radius=0.05];

\draw[very thick] (0.25,0.75) to (0.25,1.25);

\draw[very thick] (0.75,0.75) to (1.25,0.75);

\draw[very thick] (0.75,1.25) to (1.25,1.25);

\draw[very thick] (1.75,0.75) to (1.75,1.25);

\draw[very thick] (0.75,0.25) to (1.25,0.25);

\draw[very thick] (0.75,1.75) to (1.25,1.75);

\draw[fill] (1,1.5) circle [radius=0.075];

\draw[fill] (1.5,1) circle [radius=0.075];

\draw[fill] (1,0.5) circle [radius=0.075];

\node[above] at (0.75,1.75) {$(1,1)$};

\node[ right] at (1.5,1.5) {$(1,2)$};

\node[ right] at (1.35,0.25) {$(0,2)$};

\end{tikzpicture}\ \ \ \ \ \ \ \ \ 
\begin{tikzpicture}
 \draw[fill] (0,1.5) circle [radius=0.05];
\draw[fill] (0,1) circle [radius=0.05];
 \draw[fill] (0.5,1.5) circle [radius=0.05];
\draw[fill] (0.5,1) circle [radius=0.05];
 \draw[fill] (1,1.5) circle [radius=0.05];
\draw[fill] (1,1) circle [radius=0.05];
 \draw[fill] (1.5,1.5) circle [radius=0.05];
\draw[fill] (1.5,1) circle [radius=0.05];
 \draw[fill] (2,1.5) circle [radius=0.05];
\draw[fill] (2,1) circle [radius=0.05];
 \draw[fill] (2.5,1.5) circle [radius=0.05];
\draw[fill] (2.5,1) circle [radius=0.05];

 \draw[fill] (0.5,2) circle [radius=0.05];
 \draw[fill] (1,2) circle [radius=0.05];
 \draw[fill] (1.5,2) circle [radius=0.05];
 \draw[fill] (2,2) circle [radius=0.05];

 \draw[fill] (0.5,0.5) circle [radius=0.05];
 \draw[fill] (1,0.5) circle [radius=0.05];
 \draw[fill] (1.5,0.5) circle [radius=0.05];
 \draw[fill] (2,0.5) circle [radius=0.05];

 \draw[fill] (1,0) circle [radius=0.05];
 \draw[fill] (1.5,0) circle [radius=0.05];

 \draw[fill] (1,2.5) circle [radius=0.05];
 \draw[fill] (1.5,2.5) circle [radius=0.05];

 \draw[very thick] (0,1.5) to (0,1);

 \draw[very thick] (0.5,1) to (1,1);

 \draw[very thick] (0.5,0.5) to (1,0.5);

 \draw[very thick] (0.5,1.5) to (0.5,2);

 \draw[very thick] (1,1.5) to (1,2);

 \draw[very thick] (1,2.5) to (1.5,2.5);

 \draw[very thick] (1.5,1.5) to (1.5,2);

 \draw[very thick] (2,1.5) to (2,2);

  \draw[very thick] (2.5,1.5) to (2.5,1);

 \draw[very thick] (1.5,0.5) to (1.5,1);

  \draw[very thick] (2,0.5) to (2,1);

  \draw[very thick] (1,0) to (1.5,0);

  \draw[fill] (0.75,1.75) circle [radius=0.075];

\draw[fill] (1.75,1.75) circle [radius=0.075];

  \draw[fill] (1.25,0.25) circle [radius=0.075];

\draw[fill] (0.75,0.75) circle [radius=0.075];

\draw[fill] (1.75,0.75) circle [radius=0.075];

\draw[fill] (2.25,1.25) circle [radius=0.075];

\node[above] at (0.75,1.9) {$(1,1)$};

\node[above right] at (1.5,1.9) {$(2,2)$};

\node[ left] at (0.5,0.5) {$(0,2)$};

\node[ right] at (1.5,0.15) {$(0,3)$};

\node[ right] at (2,0.7) {$(0,3)$};

\node[ right] at (2.5,1.2) {$(0,3)$};

\end{tikzpicture}

\caption{An illustration of the particle configurations associated to a dimer covering. In 
the first figure we have particles at (space,level)=$(x,n)$ locations $(0,1)$, $(0,2)$, $(1,2)$, in the second figure at locations $(1,1)$, $(0,2)$, $(1,2)$ and in the third figure at locations $(1,1)$, $(0,2)$, $(2,2)$, $(0,3)$, $(1,3)$ and $(2,3)$.}\label{DimerParticles}
\end{figure}

\subsection{Urban renewal}

Let $k\ge 1$ be arbitrary. We now define a map $\mathcal{UR}_k^{k+1}$ that takes a weighting $\mathcal{W}$ of $\mathsf{AG}_{k+1}$ to a weighting $\mathcal{UR}_k^{k+1}\left(\mathcal{W}\right)$ of $\mathsf{AG}_k$. This map is called the urban renewal in the literature \cite{ProppShuffle}. It is also known as the spider move. It is of central importance in the shuffling algorithm described shortly as it gives the sampling probabilities at the various iterations. It is defined as follows.
\begin{defn}
Given a weighting $\mathcal{W}$ of $\mathsf{AG}_{k+1}$, the weighting $\mathcal{UR}_k^{k+1}\left(\mathcal{W}\right)$ of $\mathsf{AG}_k$ is defined as follows, with $0\le x \le k-1$ and $1\le n \le k$,
\begin{align*}
\mathcal{UR}_k^{k+1}\left(\mathcal{W}\right)_{\textnormal{e},(x,n)}&=\frac{\mathcal{W}_{\textnormal{e},(x,n)}}{\mathcal{W}_{\textnormal{e},(x,n)}\mathcal{W}_{\textnormal{w},(x,n)}+{\mathcal{W}_{\textnormal{s},(x,n)}\mathcal{W}_{\textnormal{n},(x,n)}}} ,  \\
\mathcal{UR}_k^{k+1}\left(\mathcal{W}\right)_{\textnormal{w},(x,n)}&= \frac{\mathcal{W}_{\textnormal{w},(x+1,n+1)}}{\mathcal{W}_{\textnormal{e},(x+1,n+1)}\mathcal{W}_{\textnormal{w},(x+1,n+1)}+{\mathcal{W}_{\textnormal{s},(x+1,n+1)}\mathcal{W}_{\textnormal{n},(x+1,n+1)}}}  ,  \\
\mathcal{UR}_k^{k+1}\left(\mathcal{W}\right)_{\textnormal{n},(x,n)}&= \frac{\mathcal{W}_{\textnormal{n},(x+1,n)}}{\mathcal{W}_{\textnormal{e},(x+1,n)}\mathcal{W}_{\textnormal{w},(x+1,n)}+{\mathcal{W}_{\textnormal{s},(x+1,n)}\mathcal{W}_{\textnormal{n},(x+1,n)}}}   , \\
\mathcal{UR}_k^{k+1}\left(\mathcal{W}\right)_{\textnormal{s},(x,n)}&= \frac{\mathcal{W}_{\textnormal{s},(x,n+1)}}{\mathcal{W}_{\textnormal{e},(x,n+1)}\mathcal{W}_{\textnormal{w},(x,n+1)}+{\mathcal{W}_{\textnormal{s},(x,n+1)}\mathcal{W}_{\textnormal{n},(x,n+1)}}} .
\end{align*}

\end{defn}
See Figure \ref{UrbanRenewalFigure} for an illustration. By construction $\mathcal{UR}_k^{k+1}\left(\mathcal{W}\right)$ inherits the strict positivity of $\mathcal{W}$ and thus the corresponding probability measure $\mathbb{P}_{\mathcal{UR}_k^{k+1}\left(\mathcal{W}\right)}^{(k)}$ on $\mathsf{DC}_k$ is supported on the whole of $\mathsf{DC}_k$. We finally use the notation, with $k\le n$,
\begin{equation*}
\mathcal{UR}_k^{n}=\mathcal{UR}_{n-1}^{n}\mathcal{UR}^{n-1}_{n-2}\cdots\mathcal{UR}_k^{k+1}
\end{equation*}
with the convention that $\mathcal{UR}_n^{n}$ is the identity map.

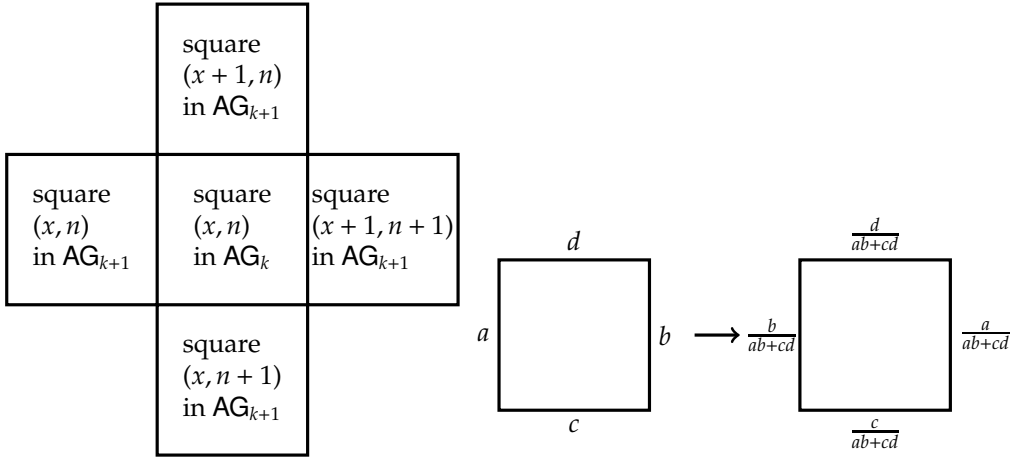
\begin{figure}
\captionsetup{singlelinecheck = false, justification=justified}
\centering
\begin{tikzpicture}
%\draw[dotted] (0,0) grid (6,6);
\draw[very thick] (2,0) rectangle (4,2);
\draw[very thick] (0,2) rectangle (2,4);
\draw[very thick] (4,2) rectangle (6,4);
\draw[very thick] (2,4) rectangle (4,6);

\node[ align=left] at (3,1) {square\\ $(x,n+1)$ \\in $\mathsf{AG}_{k+1}$};

\node[ align=left] at (1,3) {square\\ $(x,n)$ \\in $\mathsf{AG}_{k+1}$};

\node[ align=left] at (5,3) {square\\ $(x+1,n+1)$ \\in $\mathsf{AG}_{k+1}$};

\node[ align=left] at (3,5) {square\\ $(x+1,n)$ \\in $\mathsf{AG}_{k+1}$};

\node[ align=left] at (3,3) {square\\ $(x,n)$ \\in $\mathsf{AG}_{k}$};

\end{tikzpicture}
\begin{tikzpicture}

\draw[very thick] (0,0) rectangle (2,2);

\draw[very thick,->] (2.6,1) to (3.2,1) ;

\draw[very thick] (4,0) rectangle (6,2);

\node[left] at (0,1) {$a$}; 
\node[right] at (2,1) {$b$}; 
\node[below] at (1,0) {$c$}; 
\node[above] at (1,2) {$d$}; 

\node[left] at (4.1,1) {$\frac{b}{ab+cd}$}; 
\node[right] at (6,1) {$\frac{a}{ab+cd}$}; 
\node[below] at (5,0) {$\frac{c}{ab+cd}$}; 
\node[above] at (5,2) {$\frac{d}{ab+cd}$};

\end{tikzpicture}

\caption{An illustration of the urban renewal move. We can embed $\mathsf{AG}_k$ in $\mathsf{AG}_{k+1}$ so that for each square $(x,n)$ in $\mathsf{AG}_k$ its east, west, north and south edges are given by the west, east, south and north edges of the squares $(x,n)$, $(x+1,n+1)$, $(x+1,n)$ and $(x,n+1)$ respectively in $\mathsf{AG}_{k+1}$. We then replace the weights of each of these four squares in $\mathsf{AG}_{k+1}$ by the weights coming from the operation depicted above. This gives the weights for square $(x,n)$ of $\mathsf{AG}_k$ as defined in the text.}\label{UrbanRenewalFigure}
\end{figure}

\subsection{The shuffling algorithm}
Let $N\ge 1$ be fixed and consider a weighting $\mathcal{W}$ of $\mathsf{AG}_N$. We now describe an algorithm, called the shuffling algorithm \cite{ProppShuffle}, which generates a random dimer cover $\mathfrak{d}$ of $\mathsf{AG}_N$ distributed according to $\mathbb{P}_{\mathcal{W}}^{(N)}$. 

We proceed in an inductive fashion. Beginning with $\mathsf{AG}_1$ we cover the single square $(0,1)$ with a west-east pair of dimers with probability $\rho_{\mathcal{UR}^N_1\left(\mathcal{W}\right)}(0,1)$ or with a north-south pair of dimers with probability $1-\rho_{\mathcal{UR}^N_1\left(\mathcal{W}\right)}(0,1)$.
Then, after $k$ steps of the algorithm we have generated a random dimer covering $\mathfrak{d}^{(k)}\in \mathsf{DC}_{k}$ and we do the following:

\begin{enumerate}
    \item We embed $\mathsf{AG}_k$ into $\mathsf{AG}_{k+1}$ so that square $(x,n)$ in $\mathsf{AG}_k$ consists of the west, north, east and south edges of squares $(x,n)$, $(x,n+1)$, $(x+1,n+1)$ and $(x+1,n)$ in $\mathsf{AG}_{k+1}$ respectively as depicted in the Figure \ref{ShufflingAlgorithmFigure}. This embeds $\mathfrak{d}^{(k)}$ into a subcollection of edges of $\mathsf{AG}_{k+1}$.
    \item If in this embedding two dimers of $\mathfrak{d}^{(k)}$ belong to the same square of $\mathsf{AG}_{k+1}$ (in this embedding) we remove them. See Figure \ref{ShufflingAlgorithmFigure} for an illustration.
    \item We then move all dimers by one edge in the opposite direction of their names (viewed as dimers of $\mathsf{AG}_{k+1}$). Namely, a north dimer moves down by one, a south dimer moves up by one, a west dimer moves right by one and an east dimer moves left by one. See Figure \ref{ShufflingAlgorithmFigure} for an illustration.
    \item This leaves a number of squares not covered by any dimers which are filled in the following fashion. If square $(x,n)$ is empty it is covered with a west-east dimer pair with probability $\rho_{\mathcal{UR}^{N}_{k+1}\left(\mathcal{W}\right)}(x,n)$ and covered with a north-south dimer pair with probability $1-\rho_{\mathcal{UR}^{N}_{k+1}\left(\mathcal{W}\right)}(x,n)$. This gives a random element $\mathfrak{d}^{(k+1)}$ of $\mathsf{DC}_{k+1}$.
\end{enumerate}

We record a couple of observations about the algorithm. First, items (1), (2), (3) in the description of the algorithm do not depend on the weighting $\mathcal{W}$ in any way. Second, randomness only comes in item (4) of the description of the algorithm and moreover, at step $k+1$, it only depends on the weighting $\mathcal{UR}_{k+1}^{N}\left(\mathcal{W}\right)$ through the square probabilities $\rho_{\mathcal{UR}_{k+1}^{N}\left(\mathcal{W}\right)}$. 

\begin{figure}
\captionsetup{singlelinecheck = false, justification=justified}

\centering
\begin{tikzpicture}
 \draw[fill] (0,0) circle [radius=0.05];
 \draw[fill] (1,0) circle [radius=0.05];
 \draw[fill] (0,1) circle [radius=0.05];
 \draw[fill] (1,1) circle [radius=0.05];

\draw[dotted] (0,0) -- (1,0) -- (1,1) -- (0,1)-- (0,0);

\end{tikzpicture}\ \ \ \ \ \ \ \ \ \ \ \ \ 
\begin{tikzpicture}
 \draw[fill] (0,0) circle [radius=0.05];
 \draw[fill] (1,0) circle [radius=0.05];
 \draw[fill] (0,1) circle [radius=0.05];
 \draw[fill] (1,1) circle [radius=0.05];

\draw[very thick] (0,0) -- (0,1);
\draw[very thick] (1,0) -- (1,1);

\node[right] at (1,0.5)  {$\rho_{\mathcal{UR}_1^{4}(\mathcal{W})}(0,1)$};

 \draw[fill] (0.5,0.5) circle [radius=0.1];

\end{tikzpicture}

\bigskip

\bigskip 

\begin{tikzpicture}
 \draw[fill] (0,1) circle [radius=0.05];
 \draw[fill] (0,2) circle [radius=0.05];
 \draw[fill] (1,0) circle [radius=0.05];
 \draw[fill] (2,0) circle [radius=0.05];
 \draw[fill] (2,1) circle [radius=0.05];
 \draw[fill] (1,1) circle [radius=0.05];
 \draw[fill] (1,2) circle [radius=0.05];
 \draw[fill] (2,2) circle [radius=0.05];
 \draw[fill] (3,2) circle [radius=0.05];
 \draw[fill] (3,1) circle [radius=0.05];
 \draw[fill] (1,3) circle [radius=0.05];
 \draw[fill] (2,3) circle [radius=0.05];

 \draw[very thick] (1,1) -- (1,2);
  \draw[very thick] (2,1) -- (2,2);

\end{tikzpicture}\ \ \ \ \ \ \ \ \ \ \ \ \ \
\begin{tikzpicture}
 \draw[fill] (0,1) circle [radius=0.05];
 \draw[fill] (0,2) circle [radius=0.05];
 \draw[fill] (1,0) circle [radius=0.05];
 \draw[fill] (2,0) circle [radius=0.05];
 \draw[fill] (2,1) circle [radius=0.05];
 \draw[fill] (1,1) circle [radius=0.05];
 \draw[fill] (1,2) circle [radius=0.05];
 \draw[fill] (2,2) circle [radius=0.05];
 \draw[fill] (3,2) circle [radius=0.05];
 \draw[fill] (3,1) circle [radius=0.05];
 \draw[fill] (1,3) circle [radius=0.05];
 \draw[fill] (2,3) circle [radius=0.05];

  \draw[very thick] (0,1) -- (0,2);
  \draw[very thick] (3,1) -- (3,2);

\draw[dotted] (1,0) -- (2,0) -- (2,1) -- (1,1) -- (1,0);
\draw[dotted] (1,2) -- (2,2) -- (2,3) -- (1,3) -- (1,2);

\end{tikzpicture}\ \ \ \ \ \ \ \ \ \ \ \ \ \
\begin{tikzpicture}
 \draw[fill] (0,1) circle [radius=0.05];
 \draw[fill] (0,2) circle [radius=0.05];
 \draw[fill] (1,0) circle [radius=0.05];
 \draw[fill] (2,0) circle [radius=0.05];
 \draw[fill] (2,1) circle [radius=0.05];
 \draw[fill] (1,1) circle [radius=0.05];
 \draw[fill] (1,2) circle [radius=0.05];
 \draw[fill] (2,2) circle [radius=0.05];
 \draw[fill] (3,2) circle [radius=0.05];
 \draw[fill] (3,1) circle [radius=0.05];
 \draw[fill] (1,3) circle [radius=0.05];
 \draw[fill] (2,3) circle [radius=0.05];

   \draw[very thick] (0,1) -- (0,2);
  \draw[very thick] (3,1) -- (3,2);

\draw[very thick] (1,0) -- (2,0);
\draw[very thick] (1,1) -- (2,1);
\draw[very thick] (1,2) -- (1,3);
\draw[very thick] (2,2) -- (2,3);

 \draw[fill] (1.5,2.5) circle [radius=0.1];

 \draw[fill] (1.5,0.5) circle [radius=0.1];

 \draw[fill] (2.5,1.5) circle [radius=0.1];

\node[right] at (2,2.5)  {$\rho_{\mathcal{UR}_2^{4}(\mathcal{W})}(1,1)$};

\node[right] at (2,0.5)  {$\left(1-\rho_{\mathcal{UR}_2^{4}(\mathcal{W})}(0,2)\right)$};

\end{tikzpicture}

\bigskip

\bigskip

\begin{tikzpicture}

 \draw[fill] (2,0) circle [radius=0.05];
\draw[fill] (3,0) circle [radius=0.05];

 \draw[fill] (1,1) circle [radius=0.05];
\draw[fill] (2,1) circle [radius=0.05];
 \draw[fill] (3,1) circle [radius=0.05];
  \draw[fill] (4,1) circle [radius=0.05];

\draw[fill] (0,2) circle [radius=0.05];
 \draw[fill] (1,2) circle [radius=0.05];
\draw[fill] (2,2) circle [radius=0.05];
 \draw[fill] (3,2) circle [radius=0.05];
\draw[fill] (4,2) circle [radius=0.05];
 \draw[fill] (5,2) circle [radius=0.05];

\draw[fill] (0,3) circle [radius=0.05];
 \draw[fill] (1,3) circle [radius=0.05];
\draw[fill] (2,3) circle [radius=0.05];
 \draw[fill] (3,3) circle [radius=0.05];
\draw[fill] (4,3) circle [radius=0.05];
 \draw[fill] (5,3) circle [radius=0.05];

  \draw[fill] (1,4) circle [radius=0.05];
\draw[fill] (2,4) circle [radius=0.05];
 \draw[fill] (3,4) circle [radius=0.05];
  \draw[fill] (4,4) circle [radius=0.05];

 \draw[fill] (2,5) circle [radius=0.05];
\draw[fill] (3,5) circle [radius=0.05];

\draw[very thick] (1,2) -- (1,3);

\draw[very thick] (2,2) -- (3,2);

\draw[very thick] (2,1) -- (3,1);

\draw[very thick] (4,2) -- (4,3);

\draw[very thick] (3,3) -- (3,4);

\draw[very thick] (2,3) -- (2,4);
\end{tikzpicture}\ \ \ \ \ \ \ \ \ \ \ \ \
\begin{tikzpicture}

 \draw[fill] (2,0) circle [radius=0.05];
\draw[fill] (3,0) circle [radius=0.05];

 \draw[fill] (1,1) circle [radius=0.05];
\draw[fill] (2,1) circle [radius=0.05];
 \draw[fill] (3,1) circle [radius=0.05];
  \draw[fill] (4,1) circle [radius=0.05];

\draw[fill] (0,2) circle [radius=0.05];
 \draw[fill] (1,2) circle [radius=0.05];
\draw[fill] (2,2) circle [radius=0.05];
 \draw[fill] (3,2) circle [radius=0.05];
\draw[fill] (4,2) circle [radius=0.05];
 \draw[fill] (5,2) circle [radius=0.05];

\draw[fill] (0,3) circle [radius=0.05];
 \draw[fill] (1,3) circle [radius=0.05];
\draw[fill] (2,3) circle [radius=0.05];
 \draw[fill] (3,3) circle [radius=0.05];
\draw[fill] (4,3) circle [radius=0.05];
 \draw[fill] (5,3) circle [radius=0.05];

  \draw[fill] (1,4) circle [radius=0.05];
\draw[fill] (2,4) circle [radius=0.05];
 \draw[fill] (3,4) circle [radius=0.05];
  \draw[fill] (4,4) circle [radius=0.05];

 \draw[fill] (2,5) circle [radius=0.05];
\draw[fill] (3,5) circle [radius=0.05];

\draw[very thick] (0,2) -- (0,3);

\draw[very thick] (2,3) -- (3,3);

\draw[very thick] (2,0) -- (3,0);

\draw[very thick] (5,2) -- (5,3);

\draw[very thick] (1,3) -- (1,4);

\draw[very thick] (4,3) -- (4,4);

\draw[dotted] (1,1) -- (2,1) -- (2,2) -- (1,2) -- (1,1);

\draw[dotted] (3,1) -- (4,1) -- (4,2) -- (3,2) -- (3,1);

\draw[dotted] (2,4) -- (3,4) -- (3,5) -- (2,5) -- (2,4);

\end{tikzpicture}

\bigskip

\bigskip

\begin{tikzpicture}

 \draw[fill] (2,0) circle [radius=0.05];
\draw[fill] (3,0) circle [radius=0.05];

 \draw[fill] (1,1) circle [radius=0.05];
\draw[fill] (2,1) circle [radius=0.05];
 \draw[fill] (3,1) circle [radius=0.05];
  \draw[fill] (4,1) circle [radius=0.05];

\draw[fill] (0,2) circle [radius=0.05];
 \draw[fill] (1,2) circle [radius=0.05];
\draw[fill] (2,2) circle [radius=0.05];
 \draw[fill] (3,2) circle [radius=0.05];
\draw[fill] (4,2) circle [radius=0.05];
 \draw[fill] (5,2) circle [radius=0.05];

\draw[fill] (0,3) circle [radius=0.05];
 \draw[fill] (1,3) circle [radius=0.05];
\draw[fill] (2,3) circle [radius=0.05];
 \draw[fill] (3,3) circle [radius=0.05];
\draw[fill] (4,3) circle [radius=0.05];
 \draw[fill] (5,3) circle [radius=0.05];

  \draw[fill] (1,4) circle [radius=0.05];
\draw[fill] (2,4) circle [radius=0.05];
 \draw[fill] (3,4) circle [radius=0.05];
  \draw[fill] (4,4) circle [radius=0.05];

 \draw[fill] (2,5) circle [radius=0.05];
\draw[fill] (3,5) circle [radius=0.05];

\draw[very thick] (0,2) -- (0,3);

\draw[very thick] (2,3) -- (3,3);

\draw[very thick] (2,0) -- (3,0);

\draw[very thick] (5,2) -- (5,3);

\draw[very thick] (1,3) -- (1,4);

\draw[very thick] (4,3) -- (4,4);

\draw[very thick] (1,1) -- (1,2);

\draw[very thick] (2,1) -- (2,2);

\draw[very thick] (3,1) -- (4,1);

\draw[very thick] (3,2) -- (4,2);

\draw[very thick] (2,4) -- (3,4);

\draw[very thick] (2,5) -- (3,5);

 \draw[fill] (2.5,4.5) circle [radius=0.1];

 \draw[fill] (1.5,1.5) circle [radius=0.1];

 \draw[fill] (2.5,0.5) circle [radius=0.1];

 \draw[fill] (3.5,1.5) circle [radius=0.1];

 \draw[fill] (4.5,2.5) circle [radius=0.1];

 \draw[fill] (3.5,3.5) circle [radius=0.1];

\node[below left] at (1,1)  {$\rho_{\mathcal{UR}_3^{4}(\mathcal{W})}(0,2)$};

\node[below right] at (4,1)  {$\left(1-\rho_{\mathcal{UR}_3^{4}(\mathcal{W})}(1,3)\right)$};

\node[right] at (3,4.5)  {$\left(1-\rho_{\mathcal{UR}_3^{4}(\mathcal{W})}(2,1)\right)$};

\end{tikzpicture}

\caption{A possible iteration of the shuffling algorithm for a weight $\mathcal{W}$ on $\mathsf{AG}_4$. The pictures should be read from left to right and top to bottom. The text that appears in three of the figures corresponds to the probabilities of whether that square is covered with a pair of west-east dimers or north-south dimers respectively. We recall that these probabilities are obtained via running the urban renewal process on $\mathcal{W}$ to obtain weightings of $\mathsf{AG}_1$, $\mathsf{AG}_2$, $\mathsf{AG}_3$ and thus square probabilities. On this page we see only the first three steps of the algorithm. The final step will be depicted on the next pages. We note that item (2) in the description of the algorithm, corresponding to removing dimers, does not make its appearance yet. The particle configuration after one step is $(0,1)$, after two steps $(1,1)$, $(0,2)$, $(1,2)$ and after three steps $(2,1)$, $(0,2)$, $(2,2)$, $(0,3)$, $(1,3)$, $(2,3)$.}\label{ShufflingAlgorithmFigure}
\end{figure}
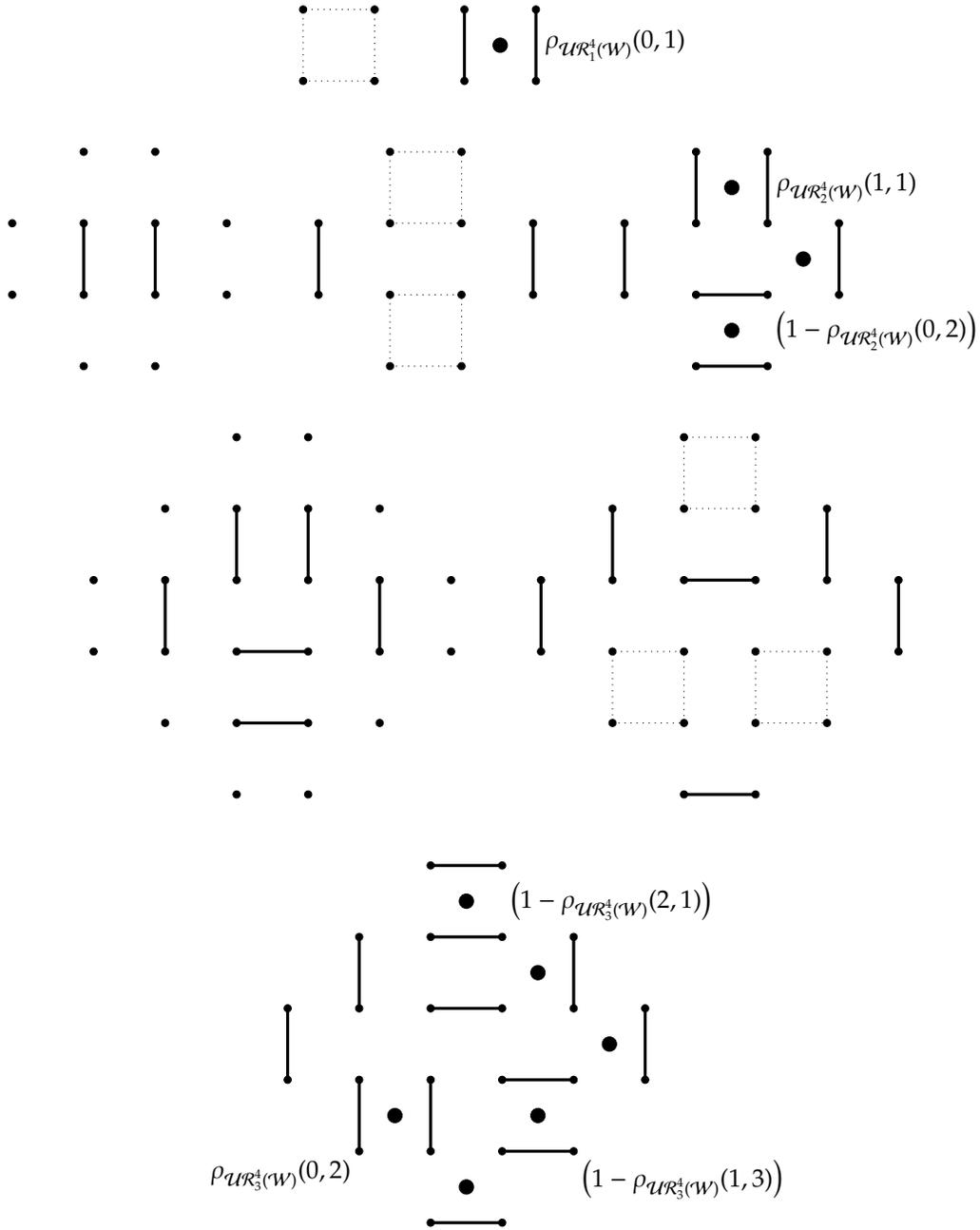

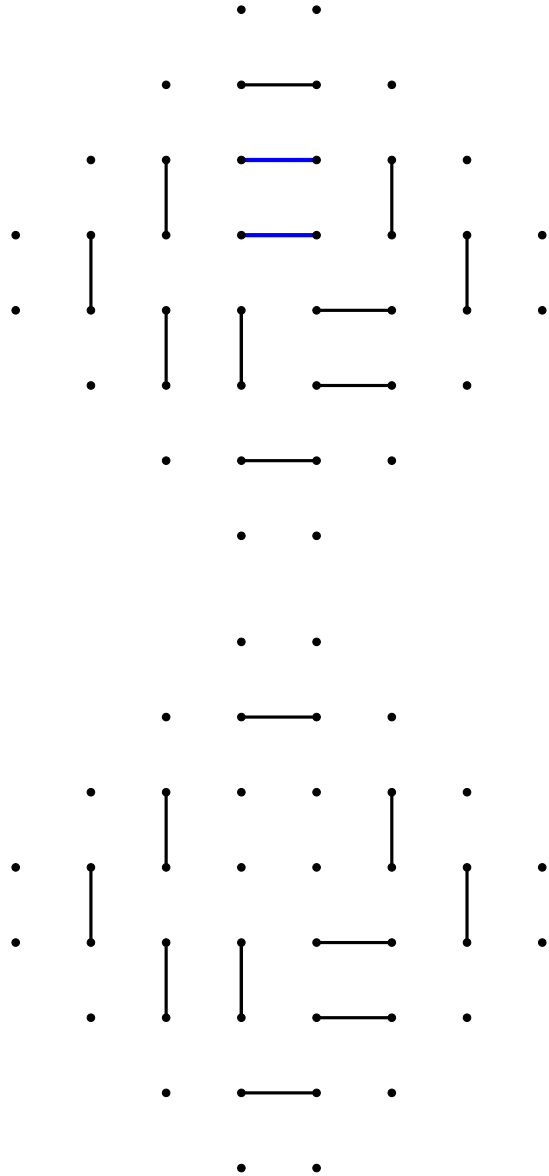
\begin{figure}\ContinuedFloat
\captionsetup{singlelinecheck = false, justification=justified}

\centering

\begin{tikzpicture}

  \draw[ultra thick,blue] (3,4) -- (4,4);

  \draw[ultra thick, blue] (3,5) -- (4,5);

 \draw[fill] (3,0) circle [radius=0.05];
 \draw[fill] (4,0) circle [radius=0.05];

 \draw[fill] (2,1) circle [radius=0.05];
 \draw[fill] (3,1) circle [radius=0.05];
 \draw[fill] (4,1) circle [radius=0.05];
 \draw[fill] (5,1) circle [radius=0.05];

 \draw[fill] (1,2) circle [radius=0.05];
 \draw[fill] (2,2) circle [radius=0.05];
 \draw[fill] (3,2) circle [radius=0.05];
 \draw[fill] (4,2) circle [radius=0.05];
 \draw[fill] (5,2) circle [radius=0.05];
 \draw[fill] (6,2) circle [radius=0.05];

 \draw[fill] (0,3) circle [radius=0.05];
 \draw[fill] (1,3) circle [radius=0.05];
 \draw[fill] (2,3) circle [radius=0.05];
 \draw[fill] (3,3) circle [radius=0.05];
 \draw[fill] (4,3) circle [radius=0.05];
 \draw[fill] (5,3) circle [radius=0.05];
 \draw[fill] (6,3) circle [radius=0.05];
 \draw[fill] (7,3) circle [radius=0.05];

 \draw[fill] (0,4) circle [radius=0.05];
 \draw[fill] (1,4) circle [radius=0.05];
 \draw[fill] (2,4) circle [radius=0.05];
 \draw[fill] (3,4) circle [radius=0.05];
 \draw[fill] (4,4) circle [radius=0.05];
 \draw[fill] (5,4) circle [radius=0.05];
 \draw[fill] (6,4) circle [radius=0.05];
 \draw[fill] (7,4) circle [radius=0.05];

 \draw[fill] (1,5) circle [radius=0.05];
 \draw[fill] (2,5) circle [radius=0.05];
 \draw[fill] (3,5) circle [radius=0.05];
 \draw[fill] (4,5) circle [radius=0.05];
 \draw[fill] (5,5) circle [radius=0.05];
 \draw[fill] (6,5) circle [radius=0.05];

 \draw[fill] (2,6) circle [radius=0.05];
 \draw[fill] (3,6) circle [radius=0.05];
 \draw[fill] (4,6) circle [radius=0.05];
 \draw[fill] (5,6) circle [radius=0.05];

 \draw[fill] (3,7) circle [radius=0.05];
 \draw[fill] (4,7) circle [radius=0.05];

\draw[very thick] (1,3) -- (1,4);

 \draw[very thick] (2,2) -- (2,3);

 \draw[very thick] (3,2) -- (3,3);

 \draw[very thick] (3,2) -- (3,3);

 \draw[very thick] (3,1) -- (4,1);

 \draw[very thick] (4,2) -- (5,2);

 \draw[very thick] (4,3) -- (5,3);

  \draw[very thick] (2,4) -- (2,5);

  \draw[very thick] (6,3) -- (6,4);

  \draw[very thick] (5,4) -- (5,5);

  \draw[very thick] (3,6) -- (4,6);

\end{tikzpicture}

\bigskip

\bigskip

\bigskip

\begin{tikzpicture}

 \draw[fill] (3,0) circle [radius=0.05];
 \draw[fill] (4,0) circle [radius=0.05];

 \draw[fill] (2,1) circle [radius=0.05];
 \draw[fill] (3,1) circle [radius=0.05];
 \draw[fill] (4,1) circle [radius=0.05];
 \draw[fill] (5,1) circle [radius=0.05];

 \draw[fill] (1,2) circle [radius=0.05];
 \draw[fill] (2,2) circle [radius=0.05];
 \draw[fill] (3,2) circle [radius=0.05];
 \draw[fill] (4,2) circle [radius=0.05];
 \draw[fill] (5,2) circle [radius=0.05];
 \draw[fill] (6,2) circle [radius=0.05];

 \draw[fill] (0,3) circle [radius=0.05];
 \draw[fill] (1,3) circle [radius=0.05];
 \draw[fill] (2,3) circle [radius=0.05];
 \draw[fill] (3,3) circle [radius=0.05];
 \draw[fill] (4,3) circle [radius=0.05];
 \draw[fill] (5,3) circle [radius=0.05];
 \draw[fill] (6,3) circle [radius=0.05];
 \draw[fill] (7,3) circle [radius=0.05];

 \draw[fill] (0,4) circle [radius=0.05];
 \draw[fill] (1,4) circle [radius=0.05];
 \draw[fill] (2,4) circle [radius=0.05];
 \draw[fill] (3,4) circle [radius=0.05];
 \draw[fill] (4,4) circle [radius=0.05];
 \draw[fill] (5,4) circle [radius=0.05];
 \draw[fill] (6,4) circle [radius=0.05];
 \draw[fill] (7,4) circle [radius=0.05];

 \draw[fill] (1,5) circle [radius=0.05];
 \draw[fill] (2,5) circle [radius=0.05];
 \draw[fill] (3,5) circle [radius=0.05];
 \draw[fill] (4,5) circle [radius=0.05];
 \draw[fill] (5,5) circle [radius=0.05];
 \draw[fill] (6,5) circle [radius=0.05];

 \draw[fill] (2,6) circle [radius=0.05];
 \draw[fill] (3,6) circle [radius=0.05];
 \draw[fill] (4,6) circle [radius=0.05];
 \draw[fill] (5,6) circle [radius=0.05];

 \draw[fill] (3,7) circle [radius=0.05];
 \draw[fill] (4,7) circle [radius=0.05];

\draw[very thick] (1,3) -- (1,4);

 \draw[very thick] (2,2) -- (2,3);

 \draw[very thick] (3,2) -- (3,3);

 \draw[very thick] (3,2) -- (3,3);

 \draw[very thick] (3,1) -- (4,1);

 \draw[very thick] (4,2) -- (5,2);

 \draw[very thick] (4,3) -- (5,3);

  \draw[very thick] (2,4) -- (2,5);

  \draw[very thick] (6,3) -- (6,4);

  \draw[very thick] (5,4) -- (5,5);

  \draw[very thick] (3,6) -- (4,6);

\end{tikzpicture}

\caption{The final step of the algorithm. The pictures on this page correspond to the first two items in the description of the algorithm. Observe that the blue edges/dimers belong to the same square having space-level coordinate $(2,2)$ in $\mathsf{AG}_4$ and so they are removed.}
\end{figure}

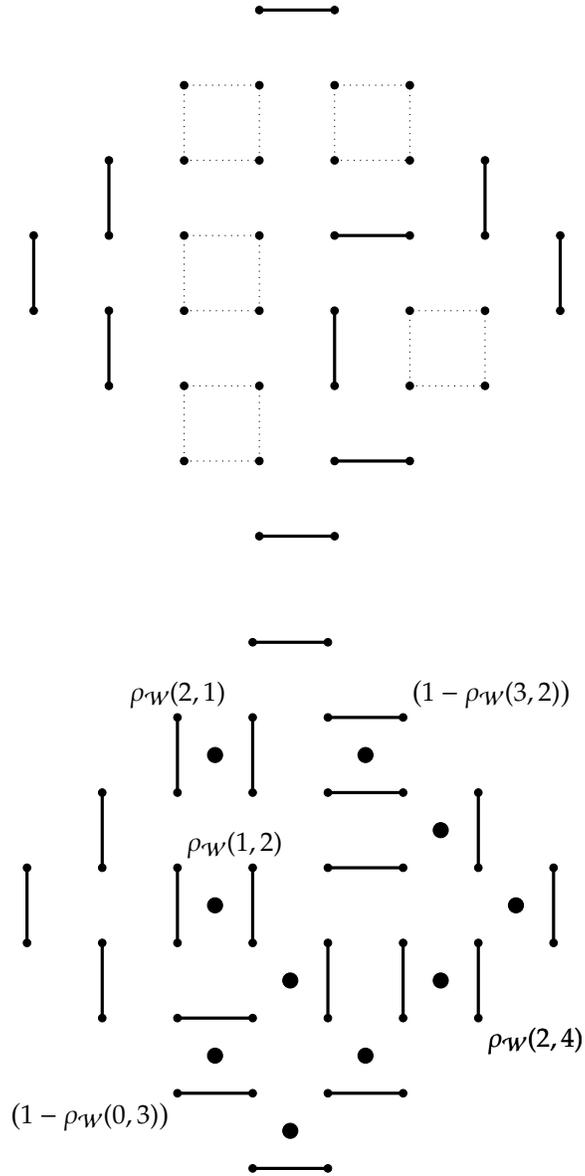
\begin{figure}\ContinuedFloat
\captionsetup{singlelinecheck = false, justification=justified}

\centering

\begin{tikzpicture}

 \draw[fill] (3,0) circle [radius=0.05];
 \draw[fill] (4,0) circle [radius=0.05];

 \draw[fill] (2,1) circle [radius=0.05];
 \draw[fill] (3,1) circle [radius=0.05];
 \draw[fill] (4,1) circle [radius=0.05];
 \draw[fill] (5,1) circle [radius=0.05];

 \draw[fill] (1,2) circle [radius=0.05];
 \draw[fill] (2,2) circle [radius=0.05];
 \draw[fill] (3,2) circle [radius=0.05];
 \draw[fill] (4,2) circle [radius=0.05];
 \draw[fill] (5,2) circle [radius=0.05];
 \draw[fill] (6,2) circle [radius=0.05];

 \draw[fill] (0,3) circle [radius=0.05];
 \draw[fill] (1,3) circle [radius=0.05];
 \draw[fill] (2,3) circle [radius=0.05];
 \draw[fill] (3,3) circle [radius=0.05];
 \draw[fill] (4,3) circle [radius=0.05];
 \draw[fill] (5,3) circle [radius=0.05];
 \draw[fill] (6,3) circle [radius=0.05];
 \draw[fill] (7,3) circle [radius=0.05];

 \draw[fill] (0,4) circle [radius=0.05];
 \draw[fill] (1,4) circle [radius=0.05];
 \draw[fill] (2,4) circle [radius=0.05];
 \draw[fill] (3,4) circle [radius=0.05];
 \draw[fill] (4,4) circle [radius=0.05];
 \draw[fill] (5,4) circle [radius=0.05];
 \draw[fill] (6,4) circle [radius=0.05];
 \draw[fill] (7,4) circle [radius=0.05];

 \draw[fill] (1,5) circle [radius=0.05];
 \draw[fill] (2,5) circle [radius=0.05];
 \draw[fill] (3,5) circle [radius=0.05];
 \draw[fill] (4,5) circle [radius=0.05];
 \draw[fill] (5,5) circle [radius=0.05];
 \draw[fill] (6,5) circle [radius=0.05];

 \draw[fill] (2,6) circle [radius=0.05];
 \draw[fill] (3,6) circle [radius=0.05];
 \draw[fill] (4,6) circle [radius=0.05];
 \draw[fill] (5,6) circle [radius=0.05];

 \draw[fill] (3,7) circle [radius=0.05];
 \draw[fill] (4,7) circle [radius=0.05];

\draw[very thick] (0,3) -- (0,4);

 \draw[very thick] (1,2) -- (1,3);

 \draw[very thick] (4,2) -- (4,3);

 \draw[very thick] (3,0) -- (4,0);

 \draw[very thick] (4,1) -- (5,1);

 \draw[very thick] (4,4) -- (5,4);

  \draw[very thick] (1,4) -- (1,5);

  \draw[very thick] (7,3) -- (7,4);

  \draw[very thick] (6,4) -- (6,5);

  \draw[very thick] (3,7) -- (4,7);

\draw[dotted] (2,1) -- (3,1) -- (3,2) -- (2,2) -- (2,1);

\draw[dotted] (2,3) -- (3,3) -- (3,4) -- (2,4) -- (2,3);

\draw[dotted] (2,5) -- (3,5) -- (3,6) -- (2,6) -- (2,5);

\draw[dotted] (4,5) -- (5,5) -- (5,6) -- (4,6) -- (4,5);

\draw[dotted] (5,2) -- (6,2) -- (6,3) -- (5,3) -- (5,2);

\end{tikzpicture}

\bigskip

\bigskip

\bigskip

\begin{tikzpicture}

 \draw[fill] (3,0) circle [radius=0.05];
 \draw[fill] (4,0) circle [radius=0.05];

 \draw[fill] (2,1) circle [radius=0.05];
 \draw[fill] (3,1) circle [radius=0.05];
 \draw[fill] (4,1) circle [radius=0.05];
 \draw[fill] (5,1) circle [radius=0.05];

 \draw[fill] (1,2) circle [radius=0.05];
 \draw[fill] (2,2) circle [radius=0.05];
 \draw[fill] (3,2) circle [radius=0.05];
 \draw[fill] (4,2) circle [radius=0.05];
 \draw[fill] (5,2) circle [radius=0.05];
 \draw[fill] (6,2) circle [radius=0.05];

 \draw[fill] (0,3) circle [radius=0.05];
 \draw[fill] (1,3) circle [radius=0.05];
 \draw[fill] (2,3) circle [radius=0.05];
 \draw[fill] (3,3) circle [radius=0.05];
 \draw[fill] (4,3) circle [radius=0.05];
 \draw[fill] (5,3) circle [radius=0.05];
 \draw[fill] (6,3) circle [radius=0.05];
 \draw[fill] (7,3) circle [radius=0.05];

 \draw[fill] (0,4) circle [radius=0.05];
 \draw[fill] (1,4) circle [radius=0.05];
 \draw[fill] (2,4) circle [radius=0.05];
 \draw[fill] (3,4) circle [radius=0.05];
 \draw[fill] (4,4) circle [radius=0.05];
 \draw[fill] (5,4) circle [radius=0.05];
 \draw[fill] (6,4) circle [radius=0.05];
 \draw[fill] (7,4) circle [radius=0.05];

 \draw[fill] (1,5) circle [radius=0.05];
 \draw[fill] (2,5) circle [radius=0.05];
 \draw[fill] (3,5) circle [radius=0.05];
 \draw[fill] (4,5) circle [radius=0.05];
 \draw[fill] (5,5) circle [radius=0.05];
 \draw[fill] (6,5) circle [radius=0.05];

 \draw[fill] (2,6) circle [radius=0.05];
 \draw[fill] (3,6) circle [radius=0.05];
 \draw[fill] (4,6) circle [radius=0.05];
 \draw[fill] (5,6) circle [radius=0.05];

 \draw[fill] (3,7) circle [radius=0.05];
 \draw[fill] (4,7) circle [radius=0.05];

\draw[very thick] (0,3) -- (0,4);

 \draw[very thick] (1,2) -- (1,3);

 \draw[very thick] (4,2) -- (4,3);

 \draw[very thick] (3,0) -- (4,0);

 \draw[very thick] (4,1) -- (5,1);

 \draw[very thick] (4,4) -- (5,4);

  \draw[very thick] (1,4) -- (1,5);

  \draw[very thick] (7,3) -- (7,4);

  \draw[very thick] (6,4) -- (6,5);

  \draw[very thick] (3,7) -- (4,7);

    \draw[very thick] (2,1) -- (3,1);

    \draw[very thick] (2,2) -- (3,2);

    \draw[very thick] (2,3) -- (2,4);

    \draw[very thick] (3,3) -- (3,4);

        \draw[very thick] (2,5) -- (2,6);

    \draw[very thick] (3,5) -- (3,6);

    \draw[very thick] (5,2) -- (5,3);

    \draw[very thick] (6,2) -- (6,3);

    \draw[very thick] (4,5) -- (5,5);

    \draw[very thick] (4,6) -- (5,6);

 \draw[fill] (3.5,0.5) circle [radius=0.1];

 \draw[fill] (2.5,1.5) circle [radius=0.1];

 \draw[fill] (4.5,1.5) circle [radius=0.1];

 \draw[fill] (3.5,2.5) circle [radius=0.1];

 \draw[fill] (5.5,2.5) circle [radius=0.1];

  \draw[fill] (2.5,3.5) circle [radius=0.1];

 \draw[fill] (6.5,3.5) circle [radius=0.1];

  \draw[fill] (5.5,4.5) circle [radius=0.1];

    \draw[fill] (2.5,5.5) circle [radius=0.1];

 \draw[fill] (4.5,5.5) circle [radius=0.1];

\node[below left] at (2,1)  {$\left(1-\rho_{\mathcal{W}}(0,3)\right)$};

\node[above right] at (2,4)  {$\rho_{\mathcal{W}}(1,2)$};

\node[above] at (2,6)  {$\rho_{\mathcal{W}}(2,1)$};

\node[below right] at (6,2)  {$\rho_{\mathcal{W}}(2,4)$};

\node[below right] at (6,2)  {$\rho_{\mathcal{W}}(2,4)$};

\node[above right] at (5,6)  {$\left(1-\rho_{\mathcal{W}}(3,2)\right)$};

\end{tikzpicture}

\caption{The final output of the algorithm. The resulting dimer covering is distributed according to $\mathbb{P}_{\mathcal{W}}^{(4)}$. The corresponding particle configuration has coordinates $(2,1)$, $(1,2)$, $(3,2)$, $(0,3)$, $(1,3)$, $(3,3)$, $(0,4)$, $(1,4)$, $(2,4)$, $(3,4)$.}
\end{figure}

We have the following remarkable theorem due to Propp \cite{ProppShuffle}, see the earlier papers \cite{AlternatingSignDominoTilings1,AlternatingSignDominoTilings2} for the uniform case.

\begin{thm}[Propp \cite{ProppShuffle}]
 The random dimer cover $\mathfrak{d}^{(k)}$ of $\mathsf{AG}_k$ obtained after $k$ steps of the shuffle is distributed according to $\mathbb{P}_{\mathcal{UR}^{N}_k\left(\mathcal{W}\right)}^{(k)}$. In particular, after $N$ steps we obtain a random dimer covering $\mathfrak{d}=\mathfrak{d}^{(N)}$ of $\mathsf{AG}_N$ distributed according to $\mathbb{P}_{\mathcal{W}}^{(N)}$, our target distribution.   
\end{thm}

\subsection{Connection to dynamics on interlacing arrays}

Observe that, by virtue of the map from dimer configurations to particles, the shuffling algorithm gives rise to a sequence of random particle configurations. We will be interested in the evolution of particles under this algorithm, thinking of the number of iterations of the shuffle as discrete time. 
\begin{defn}
 Define for $j\le t \le N$, $1\le i \le j$,
\begin{equation}\label{ShuffleParticlesLocation}
\mathsf{x}_i^{(j),\textnormal{sh}}(t)=\textnormal{position of \textit{i}-th particle of level \textit{j} after \textit{t} steps of the shuffle}.
\end{equation}   
\end{defn}
Recall that particles are ordered so that $\left(\mathsf{x}_1^{(j),\textnormal{sh}}(t),\dots,\mathsf{x}_j^{(j),\textnormal{sh}}(t)\right)\in \mathbb{W}_j$ for any $j \le t \le N$. Also, observe that for $t<j$ the particles at level $j$ do not come into play yet as the Aztec diamond graphs involved are of size less than $j$. Finally, observe that the initial condition, for level $j$ at time $j$, is given by
\begin{equation*}
\mathsf{x}_i^{(j),\textnormal{sh}}(j)=i-1, \ 1\le i \le j.
\end{equation*}

\begin{defn}\label{DefShufflePushBlock}
 Consider the sequential-update push-block Bernoulli dynamics in $\mathbb{IA}_N$ from Definition \ref{DefBernoulliDynamics}, except that a particle at space location $x$, at level $n$, at time $t$ has the jump probability $\rho_{\mathcal{UR}_{t+n}^{N}\left(\mathcal{W}\right)}(x,n)$ instead. Starting with the last time the jump probabilities for 
 a certain level are well-defined, namely for level $n$ at time $t=N-n$, we stop/freeze the particles on that level for all subsequent times. Then, this defines a stochastic process $\left(\mathsf{Y}_i^{(j)}(t);0\le t \le N-j\right)_{1\le i \le j, 1 \le j \le N}$ such that for any $1\le n\le N$ and all $0\le t \le N-n$, $\left(\mathsf{Y}_i^{(j)}(t)\right)_{1\le i \le j, 1 \le j \le n}$ is in $\mathbb{IA}_n$.
\end{defn}
 We then prove the following proposition by combining the results of Nordenstam \cite{Nordenstam} and Propp \cite{ProppShuffle}.

\begin{prop}\label{ShufflingParticleFiniteN}
Let $N\ge 1$. Let $\mathcal{W}$ be a weighting of $\mathsf{AG}_N$. Let $\mathsf{Y}_i^{(j)}(t)$ and $\mathsf{x}_i^{(j),\textnormal{sh}}(t)$ be as above. Then, we have the following equality in distribution, jointly in all involved indices,
\begin{align*}
\left(\mathsf{Y}_i^{(j)}(t-j);1\le j \le N, 1 \le i \le j, j\le t \le N\right) \overset{\textnormal{d}}{=}\left(\mathsf{x}_i^{(j),\textnormal{sh}}(t);1\le j \le N, 1 \le i \le j, j\le t \le N\right) .
\end{align*}
\end{prop}

\begin{proof}
The fact that the shuffling algorithm induces the sequential-update push-block dynamics under a different time-shift for each level is explained in Section 3 of \cite{Nordenstam}. There this exact proposition is proven in the case of the uniform weighting. However, the interactions between particles, which correspond to items (1), (2), (3) in the description of the algorithm, are exactly the same for all weightings $\mathcal{W}$. The only thing that changes are the probabilities, corresponding to item (4) in the description of the algorithm, of covering an empty square by a west-east or north-south dimer pair. These correspond to a particle jumping by one to the right or staying put respectively. For $\mathsf{x}_i^{(j),\textnormal{sh}}(t)=x$ these probabilities are given by $\rho_{\mathcal{UR}_t^{N}\left(\mathcal{W}\right)}(x,j)$ and $1-\rho_{\mathcal{UR}_t^{N}\left(\mathcal{W}\right)}(x,j)$ respectively which by virtue of the time-shift give the jump probabilities for $\mathsf{Y}_i^{(j)}(t)$.
\end{proof}

\begin{rmk}
 The equality can be made an almost sure equality by taking the Bernoulli random variables driving the dynamics of $\mathsf{Y}_i^{(j)}(t-j)$ and $\mathsf{x}_i^{(j),\textnormal{sh}}(t)$ to be the same.   
\end{rmk}

\subsection{The shuffle as a Markov chain and consistent weightings}\label{SectionShuffleMarkovConsistent}

Observe that, given $N\ge 1$ and a weighting $\mathcal{W}$ of $\mathsf{AG}_N$, we can view the shuffling algorithm as a certain Markov chain with discrete time $0\le n \le N$ and varying state spaces $\mathsf{DC}_n$ with target marginal distribution at time $N$ given by $\mathbb{P}_\mathcal{W}^{(N)}$. In particular, we have a sequence of Markov transition kernels $\mathsf{Sh}_{n-1,n}^{N,\mathcal{W}}$ from $\mathsf{DC}_{n-1}$ to $\mathsf{DC}_n$ such that,
\begin{equation*}
\mathbb{P}_{\mathcal{UR}_{n-1}^N\left(\mathcal{W}\right)}^{(n-1)}\mathsf{Sh}_{n-1,n}^{N,\mathcal{W}}=\mathbb{P}_{\mathcal{UR}_n^{N}\left(\mathcal{W}\right)}^{(n)} \ , \textnormal{ for } 1\le n \le N.
\end{equation*}
We note moreover that $\mathsf{Sh}_{n-1,n}^{N,\mathcal{W}}$ only depends on $\mathcal{UR}_n^{N}\left(\mathcal{W}\right)$ through the square probabilities $\rho_{\mathcal{UR}_n^{N}\left(\mathcal{W}\right)}\left(x,k\right)$ for $0\le x \le n-1$ and $1\le k \le n$. We would now like to extend the above to $N=\infty$. The following definition gives the key notion.

\begin{defn}\label{DefConsistentWeightings}
We say that a sequence of weightings $\left(\mathcal{W}^{(k)}\right)_{k\ge 1}$ on $\left(\mathsf{AG}_k\right)_{k\ge 1}$ is consistent if for all $k\ge 1$, 
\begin{equation*}
\mathcal{W}^{(k)} \textnormal{ is gauge-equivalent to } \mathcal{UR}_{k}^{k+1}\left(\mathcal{W}^{(k+1)}\right).
\end{equation*}
\end{defn}

\begin{rmk}
As far as we can tell, a classification of consistent weightings of Aztec diamonds has not been written down explicitly in the literature. It boils down to the study of the urban renewal (or spider move) transformations $\left(\mathcal{UR}_{n}^{n+1}\right)_{n\ge 1}$ viewed as a dynamical system and thus it should be related to the works \cite{KenyonGoncharov,ChhitaDuits}.
\end{rmk}

It is easy to check that if $\mathcal{W}$ and $\tilde{\mathcal{W}}$ are gauge-equivalent weightings of $\mathsf{AG}_n$ then 
\begin{equation*}
\mathcal{UR}^{n}_{n-1}\left(\mathcal{W}\right) \textnormal{ and } \mathcal{UR}^n_{n-1}\left(\tilde{\mathcal{W}}\right) \textnormal{ are gauge-equivalent}.
\end{equation*}
In particular, we obtain that for a consistent sequence $\left(\mathcal{W}^{(k)}\right)_{k\ge 1}$, for $N,M\ge n$, the weightings $\mathcal{UR}_{n}^{N}\left(\mathcal{W}^{(N)}\right)$ and $\mathcal{UR}_{n}^{M}\left(\mathcal{W}^{(M)}\right)$,  are both gauge-equivalent to $\mathcal{W}^{(n)}$. Thus, by virtue of Lemma \ref{LemGaugeEquivalence}, both the corresponding measures on $\mathsf{DC}_{n}$ and square probabilities are equal: 
$\mathbb{P}^{(n)}_{\mathcal{UR}_{n}^{N}\left(\mathcal{W}^{(N)}\right)}=\mathbb{P}^{(n)}_{\mathcal{UR}_{n}^{M}\left(\mathcal{W}^{(M)}\right)}=\mathbb{P}^{(n)}_{\mathcal{W}^{(n)}}$ and $\rho_{\mathcal{UR}_{n}^{N}\left(\mathcal{W}^{(N)}\right)}=\rho_{\mathcal{UR}_{n}^{M}\left(\mathcal{W}^{(M)}\right)}=\rho_{\mathcal{W}^{(n)}}$. In particular, we have that for any $N,M \ge n$, $\mathsf{Sh}_{n-1,n}^{N,\mathcal{W}^{(N)}}=\mathsf{Sh}_{n-1,n}^{M,\mathcal{W}^{(M)}}$ and so we can define, for all $n\ge 1$, transition kernels $\mathsf{Sh}_{n-1,n}^{\left(\mathcal{W}^{(k)}\right)_{k\ge 1}}$ from $\mathsf{DC}_{n-1}$ to $\mathsf{DC}_n$
such that 
\begin{equation*}
\mathsf{Sh}_{n-1,n}^{\left(\mathcal{W}^{(k)}\right)_{k\ge 1}}=\mathsf{Sh}_{n-1,n}^{N,\mathcal{W}^{(N)}},  \textnormal{ for } N \ge n, \ \textnormal{ and } \mathbb{P}_{\mathcal{W}^{(n-1)}}^{(n-1)}\mathsf{Sh}_{n-1,n}^{\left(\mathcal{W}^{(k)}\right)_{k\ge 1}}=\mathbb{P}_{\mathcal{W}^{(n)}}^{(n)}, \textnormal{ for all } n \ge 1.
\end{equation*}
Thus, we can couple all the $\mathbb{P}_{\mathcal{W}^{(n)}}^{(n)}$, for $n\ge 1$, in a natural way as marginals of the trajectory of the shuffle Markov chain. To give a formal statement, let us write $\prod_{N=1}^\infty \mathsf{DC}_{N}$ for the path space and for any $m\ge 1$ and indices $n_1<n_2<\cdots<n_m$, write $\pi_{n_1,\dots,n_m}:\prod_{N=1}^\infty \mathsf{DC}_{N}\to \mathsf{DC}_{n_1}\times \cdots \times \mathsf{DC}_{n_m}$ for the obvious projection map. Kolmogorov's extension theorem then gives the following.

\begin{prop}\label{ShufflingCoupling}
Let  $\left(\mathcal{W}^{(k)}\right)_{k\ge 1}$ be a consistent sequence of weightings on $\left(\mathsf{AG}_k\right)_{k\ge 1}$. Then, there exists a unique probability measure $\mathsf{PM}_{(\mathcal{W}^{(k)})_{k\ge 1}}$ on $\prod_{N=1}^\infty \mathsf{DC}_{N}$ such that for any $m\ge 1$ and indices $n_1\le \dots \le n_m$,
\begin{align*}
&\left(\pi_{n_1,\dots,n_m}\right)_*\mathsf{PM}_{(\mathcal{W}^{(k)})_{k\ge 1}}\left(\mathfrak{d}^{(n_1)},\mathfrak{d}^{(n_2)},\dots,\mathfrak{d}^{(n_m)}\right)\\&=\mathbb{P}_{\mathcal{W}^{(n_1)}}^{(n_1)}\left(\mathfrak{d}^{(n_1)}\right)\mathsf{Sh}_{n_1,n_2}^{\left(\mathcal{W}^{(k)}\right)_{k\ge 1}}\left(\mathfrak{d}^{(n_1)},\mathfrak{d}^{(n_2)}\right)\cdots \mathsf{Sh}_{n_{m-1},n_{m}}^{\left(\mathcal{W}^{(k)}\right)_{k\ge 1}}\left(\mathfrak{d}^{(n_{m-1})},\mathfrak{d}^{(n_m)}\right).
\end{align*}
\end{prop}

Moreover, since for consistent weightings $\left(\mathcal{W}^{(k)}\right)_{k\ge 1}$, we have $\rho_{\mathcal{UR}_{t+n}^N\left(\mathcal{W}^{(N)}\right)}=\rho_{\mathcal{W}^{(t+n)}}$, we can extend, for all $N\ge 1$, the sequential-update push-block Bernoulli dynamics process  $\left(\mathsf{Y}_i^{(j)}(t);0\le t \le N-j\right)_{1\le i \le j, 1 \le j \le N}$ from Definition \ref{DefShufflePushBlock} to all times $t\ge 0$ to obtain a process $\left(\mathsf{Y}_i^{(j)}(t);t\ge 0\right)_{1\le i \le j, j\ge 1}$ in $\mathbb{IA}_\infty$. Observe that, this process has jump probabilities at space location $x$, at level $n$, at time $t$ given by $\rho_{\mathcal{W}^{(t+n)}}(x,n)$. Hence, we obtain the following extension of Proposition \ref{ShufflingParticleFiniteN} in the special case of consistent weights. 

\begin{prop}\label{ShufflingParticleInfinite} Let  $\left(\mathcal{W}^{(k)}\right)_{k\ge 1}$ be a consistent sequence of weightings on $\left(\mathsf{AG}_k\right)_{k\ge 1}$.
Consider the coupling of the $\mathbb{P}_{\mathcal{W}^{(k)}}^{(k)}$ obtained by the shuffling algorithm from Proposition \ref{ShufflingCoupling}. Let $\mathsf{x}_i^{(j),\textnormal{sh}}(t)$ be the particle locations as in (\ref{ShuffleParticlesLocation}) and $\mathsf{Y}_i^{(j)}(t)$ as in the above paragraph. Then, we have the equality in distribution
\begin{align*}
\left(\mathsf{Y}_i^{(j)}(t-j);j \in \mathbb{N}, 1 \le i \le j, t\ge j\right) \overset{\textnormal{d}}{=}\left(\mathsf{x}_i^{(j),\textnormal{sh}}(t);j\in \mathbb{N}, 1 \le i \le j, t \ge j\right) .
\end{align*}
\end{prop}

We now give an application of the above framework, in combination with our results from previous sections, to prove a reformulation of Theorem \ref{ShufflingThmIntro}.

\begin{defn}
 Given two sequences $\mathbf{z}^{(1)}=\left(z^{(1)}_x\right)_{x\in \mathbb{Z}_+}, \mathbf{z}^{(2)}=\left(z^{(2)}_x\right)_{x\in \mathbb{Z}_+}\in (0,\infty)^{\mathbb{Z}_+}$ we define for each $k\ge 1$ a weighting $\mathcal{W}^{(k),\mathbf{z}^{(1)},\mathbf{z}^{(2)}}$ on $\mathsf{AG}_k$ as follows, see Figure \ref{AztecWeight} for an illustration,
\begin{align*}
\mathcal{W}_{\textnormal{e},(x,n)}^{(k),\mathbf{z}^{(1)},\mathbf{z}^{(2)}}&=\mathcal{W}_{\textnormal{n},(x,n)}^{(k),\mathbf{z}^{(1)},\mathbf{z}^{(2)}}=1,\\
\mathcal{W}_{\textnormal{w},(x,n)}^{(k),\mathbf{z}^{(1)},\mathbf{z}^{(2)}}&=z^{(1)}_x,\\
\mathcal{W}_{\textnormal{s},(x,n)}^{(k),\mathbf{z}^{(1)},\mathbf{z}^{(2)}}&=z^{(2)}_x.
\end{align*}   
\end{defn}

\begin{figure}
\captionsetup{singlelinecheck = false, justification=justified}
\centering
\begin{tikzpicture}

\draw[very thick] (2,2) to (2,3);
\draw[very thick] (3,2) to (4,2);

 \draw[fill] (2,0) circle [radius=0.05];
\draw[fill] (3,0) circle [radius=0.05];

 \draw[fill] (1,1) circle [radius=0.05];
\draw[fill] (2,1) circle [radius=0.05];
 \draw[fill] (3,1) circle [radius=0.05];
  \draw[fill] (4,1) circle [radius=0.05];

\draw[fill] (0,2) circle [radius=0.05];
 \draw[fill] (1,2) circle [radius=0.05];
\draw[fill] (2,2) circle [radius=0.05];
 \draw[fill] (3,2) circle [radius=0.05];
\draw[fill] (4,2) circle [radius=0.05];
 \draw[fill] (5,2) circle [radius=0.05];

\draw[fill] (0,3) circle [radius=0.05];
 \draw[fill] (1,3) circle [radius=0.05];
\draw[fill] (2,3) circle [radius=0.05];
 \draw[fill] (3,3) circle [radius=0.05];
\draw[fill] (4,3) circle [radius=0.05];
 \draw[fill] (5,3) circle [radius=0.05];

  \draw[fill] (1,4) circle [radius=0.05];
\draw[fill] (2,4) circle [radius=0.05];
 \draw[fill] (3,4) circle [radius=0.05];
  \draw[fill] (4,4) circle [radius=0.05];

 \draw[fill] (2,5) circle [radius=0.05];
\draw[fill] (3,5) circle [radius=0.05];

\draw[very thick] (2,0) to (3,0);
\draw[very thick] (2,0) to (2,1);
\draw[very thick] (3,0) to (3,1);
\draw[very thick] (2,1) to (3,1);

\draw[very thick] (1,1) to (2,1);
\draw[very thick] (3,1) to (4,1);

\draw[very thick] (1,1) to (1,2);
\draw[very thick] (2,1) to (2,2);
\draw[very thick] (3,1) to (3,2);
\draw[very thick] (4,1) to (4,2);

\draw[very thick] (0,2) to (1,2);
\draw[very thick] (1,2) to (2,2);
\draw[very thick] (2,2) to (3,2);
\draw[very thick] (4,2) to (5,2);

\draw[very thick] (0,2) to (0,3);
\draw[very thick] (1,2) to (1,3);
\draw[very thick] (3,2) to (3,3);
\draw[very thick] (4,2) to (4,3);
\draw[very thick] (5,2) to (5,3);

\draw[very thick] (0,3) to (1,3);
\draw[very thick] (1,3) to (2,3);
\draw[very thick] (2,3) to (3,3);
\draw[very thick] (3,3) to (4,3);
\draw[very thick] (4,3) to (5,3);

\draw[very thick] (1,4) to (1,3);
\draw[very thick] (2,4) to (2,3);
\draw[very thick] (3,4) to (3,3);
\draw[very thick] (4,4) to (4,3);

\draw[very thick] (1,4) to (2,4);
\draw[very thick] (2,4) to (3,4);
\draw[very thick] (3,4) to (4,4);

\draw[very thick] (2,4) to (2,5);
\draw[very thick] (3,4) to (3,5);

\draw[very thick] (2,5) to (3,5);

\node[left] at (0,2.5) {$z_0^{(1)}$};

\node[left] at (1,1.5) {$z_0^{(1)}$};

\node[left] at (2,0.5) {$z_0^{(1)}$};

\node[left] at (1,3.5) {$z_1^{(1)}$};

\node[left] at (2,2.5) {$z_1^{(1)}$};

\node[left] at (3,1.5) {$z_1^{(1)}$};

\node[left] at (2,4.5) {$z_2^{(1)}$};

\node[left] at (3,3.5) {$z_2^{(1)}$};

\node[left] at (4,2.5) {$z_2^{(1)}$};

\node[above] at (0.5,2) {$z_0^{(2)}$};

\node[above] at (1.5,1) {$z_0^{(2)}$};

\node[above] at (2.5,0) {$z_0^{(2)}$};

\node[above] at (1.5,3) {$z_1^{(2)}$};

\node[above] at (2.5,2) {$z_1^{(2)}$};

\node[above] at (3.5,1) {$z_1^{(2)}$};

\node[above] at (2.5,4) {$z_2^{(2)}$};

\node[above] at (3.5,3) {$z_2^{(2)}$};

\node[above] at (4.5,2) {$z_2^{(2)}$};

\end{tikzpicture}

\caption{The weighting $\mathcal{W}^{(k),\mathbf{z}^{(1)},\mathbf{z}^{(2)}}$  of $\mathsf{AG}_k$ from the text with $k=3$. The vertical/east dimers get weight $z_x^{(1)}$ if their space location is $x$ and similarly horizontal/south dimers get weight $z_x^{(2)}$. East and north dimers have weight $1$ which is not depicted in the figure. }\label{AztecWeight}
\end{figure}
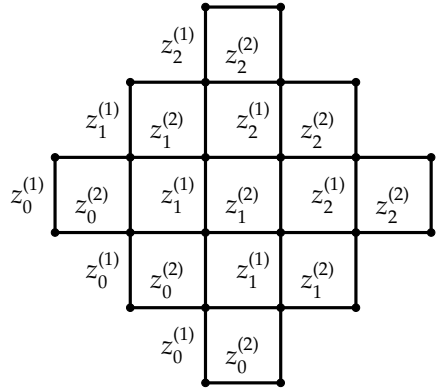

Recall that this is the weighting from the introduction viewed in terms of the dimers instead of dominoes. The following lemma is easy to show. 

\begin{lem}\label{GaugeEquivalenceLemma}
Assume that the sequences $\left(\mathbf{z}^{(1)},\mathbf{z}^{(2)}\right), \left(\tilde{\mathbf{z}}^{(1)},\tilde{\mathbf{z}}^{(2)}\right) \in (0,\infty)^{\mathbb{Z}_+}\times (0,\infty)^{\mathbb{Z}_+}$ satisfy
\begin{equation}\label{ParameterRatio}
\frac{z_x^{(1)}}{z_x^{(1)}+z_x^{(2)}}=\frac{\tilde{z}_x^{(1)}}{\tilde{z}_x^{(1)}+\tilde{z}_x^{(2)}}, \textnormal{ for all } x \in \mathbb{Z}_+.  
\end{equation}
Then, for all $k\ge 1$, the weightings $\mathcal{W}^{(k),\mathbf{z}^{(1)},\mathbf{z}^{(2)}}$ and $\mathcal{W}^{(k),\tilde{\mathbf{z}}^{(1)},\tilde{\mathbf{z}}^{(2)}}$  are gauge-equivalent.
\end{lem}

\begin{proof}
Fix $n\ge 1$. Observe that, for all $x\in \mathbb{Z}_+$, $\frac{z_x^{(1)}}{\tilde{z_x}^{(1)}}=\frac{z_x^{(2)}}{\tilde{z}_x^{(2)}}=r_x$, for some $r_x\in (0,\infty)$. Then, start with $\mathcal{W}^{(k),\mathbf{z}^{(1)},\mathbf{z}^{(2)}}$ and first apply a gauge-transformation of multiplication by $r_{n-1}^{-1}$ at each vertex joining the squares $(n-2,k)$ and $(n-1,k)$, for $k=1,\dots, n$. Observe that, all these vertices lie on a diagonal of slope $135$ degrees. Then, apply a gauge transformation of multiplication by $r_{n-1}$ along the vertices of the next diagonal with the same slope (the diagonal connecting vertices of the squares $(n-2,k)$ for $k=1,\dots,n$) and so on. In the end, we obtain $\mathcal{W}^{(k),\tilde{\mathbf{z}}^{(1)},\tilde{\mathbf{z}}^{(2)}}$.
\end{proof}

We can then define $\mathcal{W}^{(k),\mathbf{a}}$ to be $\mathcal{W}^{(k),\mathbf{z}^{(1)},\mathbf{z}^{(2)}}$ for any fixed choice of $\mathbf{z}^{(1)}$ and $\mathbf{z}^{(2)}$ such that the ratio in (\ref{ParameterRatio}) is equal to $a_x$, for all $x\in \mathbb{Z}_+$, since from a probabilistic standpoint they give rise to the same models, namely for any such choice of $\mathbf{z}^{(1)},\mathbf{z}^{(2)}$ and all $k\ge 1$,
\begin{equation*}
\mathbb{P}^{(k)}_{\mathcal{W}^{(k),\mathbf{a}}}=\mathbb{P}^{(k)}_{\mathcal{W}^{(k),\mathbf{z}^{(1)},\mathbf{z}^{(2)}}}.
\end{equation*}
For simplicity we could take $\mathbf{z}^{(1)}=\mathbf{a}$ and $\mathbf{z}^{(2)}=1^{\mathbb{Z}_+}$. We then have the following theorem, a reformulation of Theorem \ref{ShufflingThmIntro} from the introduction.

\begin{thm}
Let $\mathbf{a}$ be such that $\inf_{k\in \mathbb{Z}_+}a_k>0$ and $\sup_{k\in \mathbb{Z}_+}a_k<1$. Consider the  probability measures $\mathbb{P}^{(k)}_{\mathcal{W}^{(k),\mathbf{a}}}$ on $\mathsf{DC}_k$ associated to the weighting $\mathcal{W}^{(k),\mathbf{a}}$ defined above. Then, there exists a coupling of the $\mathbb{P}^{(k)}_{\mathcal{W}^{(k),\mathbf{a}}}$ such that the following happens. If we denote by $\mathsf{x}_i^{(j)}(m)$, for $m\ge j$, the location of the $i$-th south or east dimer  (equivalently particle) on level $j$ of the random element of $\mathsf{DC}_m$ distributed according to $\mathbb{P}^{(m)}_{\mathcal{W}^{(m),\mathbf{a}}}$ in this coupling, then jointly for all $N\ge 1$, the discrete-time stochastic process 
\begin{equation*}
\left(\mathsf{x}_1^{(N)}(t+N),\mathsf{x}_2^{(N)}(t+N),\dots,\mathsf{x}_N^{(N)}(t+N);t \ge 0\right)
\end{equation*}
evolves as a Markov process in $\mathbb{W}_N$, starting from $(0,1,\dots,N-1)$, with transition probabilities from time $t_1$ to time $t_2$ given by $\mathfrak{P}^{(N)}_{(1-z)^{t_2-t_1}}$. In particular, for any $n\ge 1$ and pairwise distinct time-space points $(t_1,x_1),\dots,(t_n,x_n)$ in $\mathbb{Z}_+\times \mathbb{Z}_+$,
\begin{align*}
\mathbb{P}\left(\exists \ j_1,\dots,j_n \textnormal{ such that } \mathsf{x}_{j_i}^{(N)}(t_i+N)=x_i \textnormal{ for } 1 \le i \le n\right)= \det \left(\mathcal{K}_N\left[(t_i,x_i);(t_j,x_j)\right]\right)_{i,j=1}^n
\end{align*}
where $f_{s,t}(z)=(1-z)^{t-s}$ in the definition of $\mathcal{K}_N$ from (\ref{CorrKernelNonColliding}).
\end{thm}

\begin{proof}
It is easy to check, by an analogous argument to the proof of Lemma \ref{GaugeEquivalenceLemma}, that for all $n\ge 1$, $\mathcal{UR}_{n}^{n+1}\left(\mathcal{W}^{(n+1),\mathbf{a}}\right)$ and $\mathcal{W}^{(n),\mathbf{a}}$ are gauge-equivalent. Hence, the sequence of weightings $\left(\mathcal{W}^{(k),\mathbf{a}}\right)_{k\ge 1}$ is consistent and consider the shuffling algorithm coupling from Proposition \ref{ShufflingParticleInfinite}. Moreover, observe that $\rho_{\mathcal{W}^{(t+n),\mathbf{a}}}\left(x,n\right)=a_x$. Thus, the process $\left(\mathsf{Y}_i^{(j)}(t);t  \ge 0\right)_{1\le i \le j, 1 \le j \le N}$ follows the inhomogeneous in space (but homogeneous in time and levels) dynamics in $\mathbb{IA}_N$ we studied earlier in the paper. Hence, by virtue of Proposition \ref{ShufflingParticleInfinite} and Theorem \ref{ThmCorrelationKernelNI} we obtain the desired statement.
\end{proof}

\section{Line ensembles with fixed starting and final positions in inhomogeneous space}\label{SectionLineEnsembles}

In this section we consider inhomogeneous Toeplitz-like matrices with matrix-valued symbols $\mathbf{f}$ and extend some of the results of \cite{DuitsKuijlaars,DuitsBeggren}. In particular, this allows us to study random walks with fixed starting and final positions, see Section \ref{LineEnsemblSetionIntro} for motivation. 

%This part of the paper is arguably the most technical and uses somewhat different techniques from the rest in matrix-valued orthogonal polynomials and analysis of Riemann-Hilbert problems. 
Beyond the definition of an inhomogeneous Toeoplitz-like matrix and the basic composition property which naturally extends to the matrix symbol setting, see Lemma \ref{LemmaMatrixComposition}, this section is largely independent of the theory developed previously. What would be very interesting would be to extend some of our more probabilistic results, such as the intertwined semigroups and couplings, to the matrix symbol $\mathbf{f}$ setting. We believe that analogues should exist (see also the discussion at the end of Section \ref{ShufflingIntro}) but the naive guess of simply plugging in a matrix-valued $\mathbf{f}$ in the relevant formulae, as far as we can tell, does not work. We leave this investigation for future work. 

We begin with the main definition.

\begin{defn}\label{InhomogeneousBlockToeplitzDef}
Let $\mathfrak{p}\ge 1$. Let $\mathbf{f}$ be a $\mathfrak{p}\times \mathfrak{p}$ matrix-valued function such that all its entries belong to $\mathsf{Hol}\left(\mathbb{H}_{-\epsilon}\right)$ for $\epsilon>0$. Define the inhomogeneous Toeplitz-like matrix $\left[\mathsf{T}_{\mathbf{f}}(x,y)\right]_{x,y\in \mathbb{Z}_+}$ associated to the matrix symbol $\mathbf{f}$ by, with $x,y\in \mathbb{Z}_+$:
\begin{align*}
\mathsf{T}_{\mathbf{f}}\left(x,y\right)&=-\frac{1}{2\pi \textnormal{i}}\frac{1}{a_m}\oint_{\mathsf{C}_{\mathbf{a}}} \frac{p_k(w)}{p_{m+1}(w)}\mathbf{f}(w)_{ij}dw,\\ &\textnormal{ if } x=k\mathfrak{p}+i, y=m\mathfrak{p}+j, \ \textnormal{ for } i,j=0,\dots,\mathfrak{p}-1 \textnormal{ and } m,k\in \mathbb{Z}_+.
\end{align*}
\end{defn}
Observe that, for $\mathfrak{p}=1$ this boils down to the inhomogeneous Toeplitz-like matrices from Section \ref{Section1D}, while for general $\mathfrak{p}$, but in the homogeneous case $a_x\equiv 1$, this is the (block) Toeplitz matrix associated to the matrix symbol $\mathbf{f}(1-z)$. As in the scalar case, more general functions $\mathbf{f}$ could be considered but we restrict to the above class for simplicity.

Now, suppose that we are given $L\ge 2$, $\mathfrak{p}\times\mathfrak{p}$ matrix-valued functions $\mathbf{f}_0(z),\dots,\mathbf{f}_{L-1}(z)$. We assume that all their entries belong to $\mathsf{Hol}\left(\mathbb{H}_{-R}\right)$ for $R>R(\mathbf{a})$. 
As mentioned in the introductory part, for this section it will be more convenient to index matrix entries and co-ordinates of configurations in $\mathbb{W}_{\mathfrak{p}N}$ starting from $0$ instead of 1. Let $M\in \mathbb{Z}_+$ be fixed. We consider a probability measure on $(L-1)$ copies of $\mathbb{W}_{\mathfrak{p}N}$ of the form
\begin{align}\label{ProbabilityMeasure}
\frac{1}{Z} \det\left(\mathsf{T}_{\mathbf{f}_0}\left(j,x_k^{(1)}\right)\right)_{j,k=0}^{\mathfrak{p}N-1}
\prod_{r=1}^{L-2}\det\left(\mathsf{T}_{\mathbf{f}_{r}}\left(x_j^{(r)},x_k^{(r+1)}\right)\right)_{j,k=0}^{\mathfrak{p}N-1}\det\left(\mathsf{T}_{\mathbf{f}_{L-1}}\left(x_j^{(L-1)},\mathfrak{p}M+k\right)\right)_{j,k=0}^{\mathfrak{p}N-1},
\end{align}
where $Z$ is a certain strictly positive normalization constant so that the above measure is a probability measure. We stress that positivity of (\ref{ProbabilityMeasure}) is part of our assumption.

\begin{defn}
 We associate to the probability measure (\ref{ProbabilityMeasure}) a stochastic process:
\begin{equation}
\left(\mathsf{X}_0^{N,\mathfrak{p},L,M}(t), \mathsf{X}_1^{N,\mathfrak{p},L,M}(t),\dots,\mathsf{X}_{\mathfrak{p}N-1}^{N,\mathfrak{p},L,M}(t); t=1,\dots,L-1\right),
\end{equation}
which records the particle positions, such that $(\mathsf{X}_0^{N,\mathfrak{p},L,M}(t),\mathsf{X}_1^{N,\mathfrak{p},L,M}(t),\dots,\mathsf{X}_{\mathfrak{p}N-1}^{N,\mathfrak{p},L,M}(t))\in \mathbb{W}_{\mathfrak{p}N}$, for $t=1,\dots,L-1$.   
\end{defn}
 Moreover, we can extend the definition of this stochastic process to times $t=0$ and $t=L$ by taking:
\begin{align*}
 (\mathsf{X}_0^{N,\mathfrak{p},L,M}(0),\mathsf{X}_1^{N,\mathfrak{p},L,M}(0),\dots,\mathsf{X}_{\mathfrak{p}N-1}^{N,\mathfrak{p},L,M}(0))&=(0,1,\dots,\mathfrak{p}N-1),\\
  (\mathsf{X}_0^{N,\mathfrak{p},L,M}(L),\mathsf{X}_1^{N,\mathfrak{p},L,M}(L),\dots,\mathsf{X}_{\mathfrak{p}N-1}^{N,\mathfrak{p},L,M}(L))&=(\mathfrak{p}M,\mathfrak{p}M+1,\dots,\mathfrak{p}(N+M)-1).
\end{align*}
Thus, by connecting the points $\mathsf{X}_i^{N,\mathfrak{p},L,M}(0), \mathsf{X}_i^{N,\mathfrak{p},L,M}(1),\dots,\mathsf{X}_i^{N,\mathfrak{p},L,M}(L)$ by straight lines, for each $i=0,\dots, \mathfrak{p}N-1$, we can think of the measure (\ref{ProbabilityMeasure}) as giving rise to a line ensemble with $\mathfrak{p}N$ lines with fixed starting and final positions at times $t=0$ and $t=L$ respectively, see Figure \ref{FigureFixedStartEndLine} for an illustration. 

\begin{figure}
\centering
\captionsetup{singlelinecheck = false, justification=justified}
\begin{tikzpicture}
%\draw[dotted] (0,0) grid (10,10);

\draw[fill] (2,0) circle [radius=0.1];
\node[above ] at (2,0.2) {$\mathsf{X}_0^{N,\mathfrak{p},L,M}(1)$};

\draw[fill] (2,2) circle [radius=0.1];
\node[above] at (2,2) {$\mathsf{X}_1^{N,\mathfrak{p},L,M}(1)$};

\draw[fill] (2,3) circle [radius=0.1];
\node[above] at (2,3) {$\mathsf{X}_2^{N,\mathfrak{p},L,M}(1)$};

\draw[fill] (2,5) circle [radius=0.1];
\node[above ] at (2,5.1) {$\mathsf{X}_{\mathfrak{p}N-2}^{N,\mathfrak{p},L,M}(1)$};

\draw[fill] (2,6) circle [radius=0.1];
\node[above left] at (2,6) {$\mathsf{X}_{\mathfrak{p}N-1}^{N,\mathfrak{p},L,M}(1)$};

\draw[fill] (4,1) circle [radius=0.1];
\node[above left] at (4,1) {$\mathsf{X}_0^{N,\mathfrak{p},L,M}(2)$};

\draw[fill] (4,2) circle [radius=0.1];
\node[above] at (3.9,2.15) {$\mathsf{X}_1^{N,\mathfrak{p},L,M}(2)$};

\draw[fill] (4,3) circle [radius=0.1];
\node[above ] at (3.9,3.15) {$\mathsf{X}_2^{N,\mathfrak{p},L,M}(2)$};

\draw[fill] (4,6) circle [radius=0.1];
\node[above] at (4,6) {$\mathsf{X}_{\mathfrak{p}N-2}^{N,\mathfrak{p},L,M}(2)$};

\draw[fill] (4,8) circle [radius=0.1];
\node[above] at (4,8) {$\mathsf{X}_{\mathfrak{p}N-1}^{N,\mathfrak{p},L,M}(2)$};

\draw[fill] (8,4) circle [radius=0.1];
\node[below right] at (8,4) {$\mathsf{X}_0^{N,\mathfrak{p},L,M}(L-1)$};

\draw[fill] (8,5) circle [radius=0.1];
\node[above left] at (8.5,5) {$\mathsf{X}_1^{N,\mathfrak{p},L,M}(L-1)$};

\draw[fill] (8,6) circle [radius=0.1];
\node[above left] at (8.5,6) {$\mathsf{X}_2^{N,\mathfrak{p},L,M}(L-1)$};

\draw[fill] (8,9) circle [radius=0.1];
\node[above] at (8,9) {$\mathsf{X}_{\mathfrak{p}N-2}^{N,\mathfrak{p},L,M}(L-1)$};

\draw[fill] (8,10) circle [radius=0.1];
\node[above] at (8,10) {$\mathsf{X}_{\mathfrak{p}N-1}^{N,\mathfrak{p},L,M}(L-1)$};

\draw[very thick] (0,0) -- (2,0) -- (4,1);

\draw[very thick] (0,1) -- (2,2) -- (4,2);

\draw[very thick] (0,2) -- (2,3) -- (4,3);

\draw[very thick] (0,4) -- (2,5) -- (4,6);

\draw[very thick] (0,5) -- (2,6) -- (4,8);

\draw[very thick]  (8,4) -- (10,5);

\draw[very thick]  (8,5) -- (10,6);

\draw[very thick]  (8,6) -- (10,7);

\draw[very thick]  (8,9) -- (10,9);

\draw[very thick]  (8,10) -- (10,10);

\draw[very thick] (4,1) -- (5,1.5);

\draw[very thick] (4,2) -- (5,2.5);

\draw[very thick] (4,3) -- (5,3.5);

\draw[very thick] (4,6) -- (5,6);

\draw[very thick] (4,8) -- (5,8);

\draw[very thick] (7,3) -- (8,4);

\draw[very thick] (7,4.5) -- (8,5);

\draw[very thick] (7,6) -- (8,6);

\draw[very thick] (7,8) -- (8,9);

\draw[very thick] (7,10) -- (8,10);

\draw[thick,dotted] (5,1.5) -- (7,3) ;

\draw[thick, dotted] (5,2.5) -- (7,4.5) ;

\draw[thick,dotted] (5,3.5) -- (7,6) ;

\draw[thick,dotted] (5,6) -- (7,8) ;

\draw[thick,dotted] (5,8) -- (7,10) ;

\node[above] at (0,3) {$\vdots$};

\node[above] at (2,4) {$\vdots$};

\node[above] at (4,4.5) {$\vdots$};

\node[above] at (6,5.5) {$\vdots$};

\node[above] at (8,7) {$\vdots$};

\node[above] at (10,8) {$\vdots$};

\draw[fill,blue] (0,0) circle [radius=0.1];
\node[above left] at (0,0) {$0$};

\draw[fill,blue] (0,1) circle [radius=0.1];
\node[above left] at (0,1) {$1$};

\draw[fill,blue] (0,2) circle [radius=0.1];
\node[above left] at (0,2) {$2$};

\draw[fill,blue] (0,4) circle [radius=0.1];
\node[above] at (0,4.2) {$\mathfrak{p}N-2$};

\draw[fill,blue] (0,5) circle [radius=0.1];
\node[above ] at (0,5.2) {$\mathfrak{p}N-1$};

\draw[fill,blue] (10,5) circle [radius=0.1];
\node[above] at (10,5) {$\mathfrak{p}M$};

\draw[fill,blue] (10,6) circle [radius=0.1];
\node[above] at (10,6) {$\mathfrak{p}M+1$};

\draw[fill,blue] (10,7) circle [radius=0.1];
\node[above] at (10,7) {$\mathfrak{p}M+2$};

\draw[fill,blue] (10,9) circle [radius=0.1];
\node[above right] at (10,9) {$\mathfrak{p}(N+M)-2$};

\draw[fill,blue] (10,10) circle [radius=0.1];
\node[above right] at (10,10) {$\mathfrak{p}(N+M)-1$};

\end{tikzpicture}
\caption{An illustration of the stochastic process 
\begin{equation*}
\left(\mathsf{X}_0^{N,\mathfrak{p},L,M}(t), \mathsf{X}_1^{N,\mathfrak{p},L,M}(t),\dots,\mathsf{X}_{\mathfrak{p}N-1}^{N,\mathfrak{p},L,M}(t); t=1,\dots,L-1\right),
\end{equation*}
which is basically the random point configuration, depicted as filled circles, corresponding to the probability measure (\ref{ProbabilityMeasure}). We can then join the points with the same subscript by straight lines as explained in the text to obtain a non-intersecting line ensemble. We can extend this to include fixed starting and final positions $(0,1,\dots,\mathfrak{p}N-2,\mathfrak{p}N-1)$ and $(\mathfrak{p}M,\mathfrak{p}M+1,\dots,\mathfrak{p}(N+M)-2,\mathfrak{p}(N+M)-1)$ at times $0$ and $L$ which are depicted as blue filled circles. We note that this line ensemble should not be confused with the non-intersecting paths in LGV graphs from previous sections. This structure, if it exists, will depend on the specifics of the functions $\mathbf{f}_i$, compare for example with Figure \ref{PathsFixedStartingFinal}. Theorem \ref{CorrKernelFixedStartEndpoint} states that the finite dimensional distributions of this stochastic process can be written explicitly in terms of determinants of a correlation kernel which involves the solution of an explicit Riemann-Hilbert problem. Theorem \ref{ThmLineEnsembleTopBottomLimit} below states that, under certain assumptions, as $N \to \infty$, the bottom $m$ lines, for any $m\ge 1$, converge. }\label{FigureFixedStartEndLine}
\end{figure}

\begin{defn}
 Define the function $\mathbf{f}_{r,r'}$ for $r<r'$ by $\mathbf{f}_{r,r'}(z)=\left[\mathbf{f}_r\mathbf{f}_{r+1}\cdots \mathbf{f}_{r'-1}\right](z)$ and the complex matrix-valued measure $\mathbf{M}(z)=\mathbf{M}(z;\mathbf{a})$, which also depends on $L$ and $M+N$ but we suppress this in the notation, by 
\begin{equation*}
\mathbf{M}(z;\mathbf{a})=\frac{\left[\mathbf{f}_0\mathbf{f}_1\cdots \mathbf{f}_{L-1}\right](z)}{p_{M+N}(z;\mathbf{a})}.
\end{equation*}   
\end{defn}

Let $\mathbf{A}^{\textnormal{t}}$ denote the transpose of a matrix $\mathbf{A}$ and let $\mathbf{0}_{\mathfrak{p}}$ and $\mathbf{I}_{\mathfrak{p}}$ be the zero and identity $\mathfrak{p}\times \mathfrak{p}$ matrices respectively. We have the following theorem which generalises the results of Section 4 of \cite{DuitsKuijlaars} to $a_x \nequiv 1$.

\begin{thm}\label{CorrKernelFixedStartEndpoint}
Let $M\ge 0$,$\mathfrak{p}\ge 1$, $L\ge 2$. Suppose that the matrix-valued functions $\mathbf{f}_0,\dots,\mathbf{f}_{L-1}$ have entries in $\mathsf{Hol}\left(\mathbb{H}_{-R}\right)$ for $R>R(\mathbf{a})$. Assume that the expression (\ref{ProbabilityMeasure}) defines a probability measure on $(L-1)$ copies of $\mathbb{W}_{\mathfrak{p}N}$. Denote by the corresponding stochastic processes of non-intersecting paths. Then, this is described by a determinantal point process, namely for any $n\ge 1$ and pairwise distinct time-space points $(t_1,x_1),\dots,(t_n,x_n)$ in $\llbracket 1,L-1\rrbracket \times \mathbb{Z}_+$, we have
\begin{equation}
\mathbb{P}\left(\exists \   j_1,\dots,j_n \textnormal{ such that } \mathsf{X}_{j_i}^{N,\mathfrak{p},L,M}(t_i)=x_i \textnormal{ for } i=1,\dots,n\right)=\det\left(\mathsf{K}_N\left[(t_i,x_i);(t_j,x_j)\right]\right)_{i,j=1}^n,
\end{equation}
where the correlation kernel $\mathsf{K}_N$ given by
\begin{align}
\left[\mathsf{K}_N\left[(r,m\mathfrak{p}+j);(r',k\mathfrak{p}+i)\right]\right]_{i,j=0}^{\mathfrak{p}-1}=\mathbf{1}_{r>r'}\frac{1}{2\pi \textnormal{i}}\frac{1}{a_m}\oint_{\mathsf{C}_{\mathbf{a}}}\mathbf{f}_{r',r}(z)\frac{p_k(z)}{p_{m+1}(z)}dz\nonumber\\
-\frac{1}{(2\pi \textnormal{i})^2}\frac{1}{a_k}\oint_{\mathsf{C}_{\mathbf{a}}}\oint_{\mathsf{C}_{\mathbf{a}}} \mathbf{f}_{r',L}(w)\mathbf{R}_N(w,z)\mathbf{f}_{0,r}(z)\frac{p_k(w)}{p_{M+N}(w)p_{m+1}(z)}dz dw.\label{BlockMatrixFormKernel}
\end{align}
Here, the $\mathfrak{p}\times \mathfrak{p}$ matrix-valued function $\mathbf{R}_N(w,z)$ is given by 
\begin{equation}\label{ReproducingKernelCD}
\mathbf{R}_N(z,w)=\frac{1}{z-w} \begin{pmatrix}
\mathbf{0}_{\mathfrak{p}} \ \mathbf{I}_{\mathfrak{p}}
\end{pmatrix}\mathbf{Y}^{-1}(w)\mathbf{Y}(z)\begin{pmatrix}
\mathbf{I}_{\mathfrak{p}}\\
\mathbf{0}_{\mathfrak{p}}
\end{pmatrix}
\end{equation}
with the $2\mathfrak{p}\times 2 \mathfrak{p}$ matrix-valued function $\mathbf{Y}:\mathbb{C}\setminus\mathsf{C}_{\mathbf{a}}\to \mathbb{C}^{2\mathfrak{p}\times2\mathfrak{p}}$ being the unique solution to the following Riemann-Hilbert problem (RHP) of size $2\mathfrak{p} \times 2\mathfrak{p}$:
\begin{itemize}
    \item $\mathbf{Y}$ is analytic,
    \item $\mathbf{Y}_+(z)=\mathbf{Y}_-(z)\begin{pmatrix}
\mathbf{I}_{\mathfrak{p}} &\mathbf{M}(z;\mathbf{a})\\
\mathbf{0}_{\mathfrak{p}} &\mathbf{I}_{\mathfrak{p}}
\end{pmatrix}$ on $\mathsf{C}_\mathbf{a}$, where $\mathbf{Y}_+$ is the limit of $\mathbf{Y}(z)$ from inside $\mathsf{C}_{\mathbf{a}}$ and $\mathbf{Y}_-$ the limit of $\mathbf{Y}(z)$ from outside $\mathsf{C}_{\mathbf{a}}$ respectively,
    \item $\mathbf{Y}(z)=\left(\mathbf{I}_{2\mathfrak{p}}+\mathcal{O}(z^{-1})\right)\begin{pmatrix}
z^N \mathbf{I}_{\mathfrak{p}} &\mathbf{0}_{\mathfrak{p}}\\
\mathbf{0}_{\mathfrak{p}} &z^{-N}\mathbf{I}_{\mathfrak{p}}
\end{pmatrix}$ as $z\to \infty$.
\end{itemize}
\end{thm}

\begin{rmk}
We observe that for $a_x\equiv 1$, and after the change of variables $z\mapsto 1-z$, $w\mapsto 1-w$ Theorem \ref{CorrKernelFixedStartEndpoint} specialises to Theorem 4.7 of \cite{DuitsKuijlaars} with the identifications, in the notations of \cite{DuitsKuijlaars}, $\mathbf{f}_{r',r}(1-z)=A_{r',r}(z)$, $1-\mathsf{C}_\mathbf{a}=\gamma$ and $\mathbf{R}_N(1-w,1-z)=-R_N(w,z)$.
\end{rmk}

\begin{rmk}
Positivity of (\ref{ProbabilityMeasure}) is in some sense not strictly necessary as long as the normalisation constant $Z\neq 0$. One then has a signed measure of total mass $1$ on point configurations (but there is no associated stochastic process of course). We can still define correlation functions for such signed measures, see \cite{BorodinDeterminantal}. The analysis that follows essentially goes through and the conclusions of Theorem \ref{CorrKernelFixedStartEndpoint} on explicit determinantal correlations remain valid.
\end{rmk}

Using the above theorem and some Riemann-Hilbert problem asymptotic analysis, under a factorisation assumption for $\mathbf{f}_{0,L}(z)$ and some assumptions on the sequence $\mathbf{a}$, we obtain the following limit theorem for the bottom paths as $N\to \infty$. This generalises Theorem 3.1 of \cite{DuitsBeggren} to $a_x\neq 1$. Let us define $\mathfrak{A}_\eta$, with $\eta \in (0,1)$, to be the annulus $\mathfrak{A}_\eta=\left\{z\in \mathbb{C}:\eta < |z|<\eta^{-1}\right\}$ about $0$. 

\begin{thm}\label{ThmLineEnsembleTopBottomLimit} Assume that the parameter sequence $\mathbf{a}$ satisfies
\begin{equation}
\inf_{x\in \mathbb{Z}_+}a_x\ge1-\mathfrak{c} \textnormal{ and } \sup_{x\in \mathbb{Z}_+}a_x\le1+\mathfrak{c},
\end{equation}
for some $0\le \mathfrak{c} <\frac{1}{3}$. Suppose that the matrix-valued functions $\mathbf{f}_0(z),\dots,\mathbf{f}_{L-1}(z)$ have entries in $\mathsf{Hol}\left(\mathbb{H}_{-R}\right)$ for $R>R(\mathbf{a})$ and moreover that the functions $\mathbf{f}_0(1-z)^{\pm 1},\dots, \mathbf{f}_{L-1}(1-z)^{\pm 1}$ are analytic in an annulus $\mathfrak{A}_\eta$ where $\eta <1-2\mathfrak{c}$. Finally, assume we have the factorisations $\mathbf{f}_{0,L}(1-z)=\mathbf{S}_+(z)\mathbf{S}_-(z)=\tilde{\mathbf{S}}_-(z)\tilde{\mathbf{S}}_+(z)$, where $\mathbf{S}_\pm$, $\tilde{\mathbf{S}}_\pm$ are $\mathfrak{p}\times \mathfrak{p}$ matrix-valued functions satisfying 
\begin{itemize}
    \item $\mathbf{S}_+^{\pm 1}(z), \tilde{\mathbf{S}}_{+}^{\pm 1}(z)$ are analytic for $|z|<1$ and continuous for $|z|\le 1$,
    \item $\mathbf{S}_-^{\pm 1}(z),\tilde{\mathbf{S}}_-^{\pm 1}(z)$ are analytic for $|z|>1$ and continuous for $|z|\ge 1$,
    \item $\mathbf{S}_-(z)\sim z^M\mathbf{I}_{\mathfrak{p}}$ and $\tilde{\mathbf{S}}_-(z)\sim z^M\mathbf{I}_{\mathfrak{p}}$ as $z\to \infty$.
\end{itemize}
Consider the line ensemble $\left(\mathsf{X}_0^{N,\mathfrak{p},L,M}(t), \mathsf{X}_1^{N,\mathfrak{p},L,M}(t),\dots,\mathsf{X}_{\mathfrak{p}N-1}^{N,\mathfrak{p},L,M}(t); t=1,\dots,L-1\right)$
 associated to (\ref{ProbabilityMeasure}). Then, for any $m\ge 1$, we have the following convergence in distribution for the bottom $m$ lines in the line ensemble as $N \to \infty$, with $m\le \mathfrak{p}N$,
\begin{equation*}
\left(\mathsf{X}_0^{N,\mathfrak{p},L,M}(t), \dots,\mathsf{X}_{m-1}^{N,\mathfrak{p},L,M}(t);1\le t \le L-1\right)\overset{\textnormal{d}}{\longrightarrow}\left(\mathsf{X}_0^{\infty,\mathfrak{p},L,M}(t),\dots,\mathsf{X}_{m-1}^{\infty,\mathfrak{p},L,M}(t); 1\le t \le L-1\right),
\end{equation*}
where the limiting line ensemble $\left(\left(\mathsf{X}_i^{\infty,\mathfrak{p},L,M}(t)\right)_{i=0}^\infty; t=1,\dots,L-1\right)$ is again determined through its determinantal correlation functions: for any $n\ge 1$ and pairwise distinct time-space points $(t_1,x_1),\dots,(t_n,x_n)$ in $\llbracket 1, L-1 \rrbracket \times \mathbb{Z}_+$ we have 
\begin{equation}
\mathbb{P}\left(\exists \   j_1,\dots,j_n \textnormal{ such that } \mathsf{X}_{j_i}^{\infty,\mathfrak{p},L,M}(t_i)=x_i \textnormal{ for } i=1,\dots,n\right)=\det\left(\mathsf{K}_\infty\left[(t_i,x_i);(t_j,x_j)\right]\right)_{i,j=1}^n,
\end{equation}    
where the kernel $\mathsf{K}_\infty$ is given by:
\begin{align*}
&\left[\mathsf{K}_\infty\left[(r,m\mathfrak{p}+j);(r',k\mathfrak{p}+i)\right]\right]_{i,j=0}^{\mathfrak{p}-1}=-\mathbf{1}_{r>r'}\frac{1}{2\pi \textnormal{i}}\frac{1}{a_m}\oint_{|z|=1}\mathbf{f}_{r',r}(1-z)\frac{p_k(1-z)}{p_{m+1}(1-z)}dz\nonumber\\
&-\frac{1}{(2\pi \textnormal{i})^2}\frac{1}{a_k}\oint_{|z|=1^{-}}\oint_{|w|=1^+} \mathbf{f}_{r',L}(1-w)\mathbf{S}_-^{-1}(w)\mathbf{S}_+^{-1}(z)\mathbf{f}_{0,r}(1-z)\frac{p_k(1-w)}{p_{m+1}(1-z)}\frac{dz dw}{z-w}.
\end{align*}
\end{thm}

\begin{rmk}
Although the functions $\tilde{\mathbf{S}}_-$, $\tilde{\mathbf{S}}_+$ (for $\mathfrak{p}=1$ these can simply be picked the same as $\mathbf{S}_-$, $\mathbf{S}_+$) do not appear in the limiting kernel $\mathsf{K}_\infty$, they are required for the proof. The proof unsurprisingly boils down to computing the asymptotics of $\mathbf{R}_N(w,z)$ as $N\to \infty$. Note that, along with the polynomial $p_{M+N}(\cdot)$ this is the only $N$-dependent quantity in $\mathsf{K}_N$.
\end{rmk}

\begin{rmk}
In \cite{DuitsBeggren} the authors also consider a limit theorem for the top paths in the line ensemble (in this case $\tilde{\mathbf{S}}_-$, $\tilde{\mathbf{S}}_+$ appear in the limiting kernel instead). In our case, if we consider the same scaling, we would need to enforce the inhomogeneity parameters $\mathbf{a}$ to be asymptotically constant to get a limit (which essentially puts us back into the setting of \cite{DuitsBeggren}). It is plausible that interesting limit theorems exist for the top paths for general $\mathbf{a}$ but under a different scaling. We leave this for future work.
\end{rmk}

We first prove Theorem \ref{CorrKernelFixedStartEndpoint}. The following lemma is essential for computations. 

\begin{lem}\label{LemmaMatrixComposition}
  Let $\mathbf{f},\mathbf{g}$ be $\mathfrak{p}\times \mathfrak{p}$ matrix-valued functions such that all their entries belong to $\mathsf{Hol}\left(\mathbb{H}_{-R}\right)$, for $R>R(\mathbf{a})$. Then, we have $\mathsf{T}_\mathbf{f}\mathsf{T}_\mathbf{g}=\mathsf{T}_{\mathbf{f}\mathbf{g}}$.
  
\end{lem}

\begin{proof}
The proof is a straightforward extension of the proof of Lemma \ref{LemmaComposition}.
\end{proof}

\begin{rmk}
One may hope that Lemmas \ref{LemmaNormalisation} and \ref{LemmaDuality} have analogous straightforward extensions for matrix-valued functions $\mathbf{f}$ but as far as we can tell this is not the case.
\end{rmk}

By virtue of Lemma \ref{LemmaMatrixComposition} we obtain that convolutions of the $\mathsf{T}_{\mathbf{f}_r}$ operators satisfy
\begin{align*}
\left[\mathsf{T}_{\mathbf{f}_{r}}\mathsf{T}_{\mathbf{f}_{r+1}}\cdots\mathsf{T}_{\mathbf{f}_{r'-1}}\left(k\mathfrak{p}+i,m\mathfrak{p}+j\right)\right]_{i,j=0}^{\mathfrak{p}-1}=\left[\mathsf{T}_{\mathbf{f}_{r,r'}}\left(k\mathfrak{p}+i,m\mathfrak{p}+j\right)\right]_{i,j=0}^{\mathfrak{p}-1}.
\end{align*}

\begin{defn}
Let $\mathbf{G}=(\mathbf{G}_{\kappa \mu})_{\kappa, \mu=0}^{\mathfrak{p}N-1}$ be the so-called Gram matrix associated to this problem given by, by virtue of Lemma \ref{LemmaMatrixComposition}, 
\begin{equation}\label{GramMatrixOriginalDef}
\mathbf{G}_{\kappa \mu}=\mathsf{T}_{\mathbf{f}_0}\mathsf{T}_{\mathbf{f}_1}\dots
\mathsf{T}_{\mathbf{f}_{L-1}}(\kappa,\mathfrak{p}M+\mu)=\mathsf{T}_{\mathbf{f}_{0,L}}\left(\kappa,\mathfrak{p}M+\mu\right).
\end{equation}
    
\end{defn}
Then, note that $\det(\mathbf{G})=Z>0$, by our assumption that (\ref{ProbabilityMeasure}) is a well-defined probability measure, and thus $\mathbf{G}$ is invertible. Moreover, observe that we can write
\begin{align}\label{GramMatrix}
\mathbf{G}_{\kappa \mu}&=-\frac{1}{2\pi \textnormal{i}}\frac{1}{a_{M+m}}\oint_{\mathsf{C}_\mathbf{a}} p_k(z)\mathbf{M}(z)_{ij}\frac{p_{M+N}(z)}{p_{M+m+1}(z)}dz \\
&\textnormal{ if } \kappa=k\mathfrak{p}+i, \mu=m\mathfrak{p}+j, \ \textnormal{ for } i,j=0,\dots,\mathfrak{p}-1 \textnormal{ and } k,m=0,\dots,N-1.\nonumber
\end{align}
Now, given invertible $\mathfrak{p}N\times \mathfrak{p}N$ matrices $\mathbf{P},\mathbf{Q}$ we define (slightly abusing notation) $\mathfrak{p}\times \mathfrak{p}$ matrix-valued orthogonal polynomials $\mathbf{P}_j,\mathbf{Q}_j$, for $j=0,\dots,N-1$, of degree at most $N-1$ (namely their entries have degrees at most $N-1$) as follows:
\begin{align*}
\begin{pmatrix}
\mathbf{P}_0(z)\\
\mathbf{P}_1(z)\\
\vdots\\
\mathbf{P}_{N-1}(z) \\
\end{pmatrix}=\mathbf{P}\begin{pmatrix}
\mathbf{I}_\mathfrak{p}\\
p_1(z)\mathbf{I}_\mathfrak{p}\\
\vdots\\
p_{N-1}(z)\mathbf{I}_\mathfrak{p}\\
\end{pmatrix}, \ \ \begin{pmatrix}
\mathbf{Q}_0(z)\\
\mathbf{Q}_1(z)\\
\vdots\\
\mathbf{Q}_{N-1}(z) \\
\end{pmatrix}=\mathbf{Q}\begin{pmatrix}
-\frac{1}{a_M}\frac{p_{M+N}(z)}{p_{M+1}(z)}\mathbf{I}_\mathfrak{p}\\
-\frac{1}{a_{M+1}}\frac{p_{M+N}(z)}{p_{M+2}(z)}\mathbf{I}_\mathfrak{p}\\
\vdots\\
-\frac{1}{a_{M+N-1}}\mathbf{I}_\mathfrak{p}\\
\end{pmatrix}.
\end{align*}

The following result relates inverting the matrix $\mathbf{G}$ to finding biorthogonal matrix-valued polynomials.

\begin{prop}
Suppose $\mathbf{P},\mathbf{Q}$ are invertible $\mathfrak{p}N\times \mathfrak{p}N$ matrices and define $\mathbf{P}_j,\mathbf{Q}_j$ matrix-valued polynomials as above. Let $\mathbf{G}$ be the Gram matrix defined in (\ref{GramMatrixOriginalDef}). Then, the following are equivalent 
\begin{enumerate}
    \item $\mathbf{G}^{-1}=\mathbf{Q}^\textnormal{t}\mathbf{P}$
    \item For $j,k=0,\dots,N-1$, 
    \begin{equation*}
    \frac{1}{2\pi \textnormal{i}}\oint_{\mathsf{C}_\mathbf{a}}\mathbf{P}_j(z)\mathbf{M}(z)\mathbf{Q}_k^{\textnormal{t}}(z)dz=\mathbf{1}_{j=k}\mathbf{I}_{\mathfrak{p}}.
    \end{equation*}
\end{enumerate}
\end{prop}
\begin{proof}
We adapt the proof of Proposition 4.5 in \cite{DuitsKuijlaars}. The key is to consider the matrix
\begin{equation*}
\mathbf{X}=\frac{1}{2\pi \textnormal{i}}\oint_{\mathsf{C}_{\mathbf{a}}} \begin{pmatrix}
\mathbf{P}_0(z)\\
\mathbf{P}_1(z)\\
\vdots\\
\mathbf{P}_{N-1}(z) \\
\end{pmatrix}\mathbf{M}(z) \begin{pmatrix}
\mathbf{Q}^{\textnormal{t}}_0(z) \ \mathbf{Q}^{\textnormal{t}}_1(z) \cdots \mathbf{Q}^{\textnormal{t}}_{N-1}(z) 
\end{pmatrix}dz
\end{equation*}
and observe that $\mathbf{P}^{-1}\mathbf{X}\mathbf{Q}^{-\textnormal{t}}=\mathbf{G}$ from which the conclusion follows.
\end{proof}

For any factorization  $\mathbf{G}^{-1}=\mathbf{Q}^\textnormal{t}\mathbf{P}$ we define the reproducing kernel $\mathbf{R}_N(w,z)$, which we will show in the sequel, in the proof of Theorem \ref{CorrKernelFixedStartEndpoint}, that it can be written as (\ref{ReproducingKernelCD}), by
\begin{equation*}
\mathbf{R}_N(w,z)=\sum_{j=0}^{N-1}\mathbf{Q}_j^{\textnormal{t}}(w)\mathbf{P}_j(z).
\end{equation*}
Observe that, $\mathbf{R}_N(w,z)$ is independent of the choice of factorization $\mathbf{Q}^\textnormal{t}\mathbf{P}$ since 
\begin{equation}\label{ReproducingKernelFirstFormula}
\mathbf{R}_N(w,z)=\begin{pmatrix}
-\frac{1}{a_M}\frac{p_{M+N}(w)}{p_{M+1}(w)}\mathbf{I}_\mathfrak{p}\
-\frac{1}{a_{M+1}}\frac{p_{M+N}(w)}{p_{M+2}(w)}\mathbf{I}_\mathfrak{p}
\cdots
-\frac{1}{a_{M+N-1}}\mathbf{I}_\mathfrak{p}\\
\end{pmatrix} \mathbf{G}^{-1} \begin{pmatrix}
\mathbf{I}_\mathfrak{p}\\
p_1(z)\mathbf{I}_\mathfrak{p}\\
\vdots\\
p_{N-1}(z)\mathbf{I}_\mathfrak{p}\\
\end{pmatrix}.
\end{equation}
The following proposition gives the  reproducing property of $\mathbf{R}_N(w,z)$ which also characterises it.

\begin{prop}\label{ReproducingProperty} For any matrix-valued polynomial $\mathbf{S}(z)$ of degree at most $N-1$ we have
\begin{align}
 \frac{1}{2\pi \textnormal{i}}\oint_{\mathsf{C}_\mathbf{a}}\mathbf{S}(w)\mathbf{M}(w)\mathbf{R}_N(w,z)dw&=\mathbf{S}(z),\\
 \frac{1}{2\pi \textnormal{i}}\oint_{\mathsf{C}_\mathbf{a}}\mathbf{R}_N(w,z)\mathbf{M}(z)\mathbf{S}^{\textnormal{t}}(z)dz&=\mathbf{S}^{\textnormal{t}}(w).
\end{align}
Moreover, any bivariate polynomial of degree at most $N-1$ satisfying either equality above for all matrix-valued polynomials $\mathbf{S}(z)$ of degree at most $N-1$ must be equal to $\mathbf{R}_N(w,z)$.
\end{prop}

\begin{proof}
The proof is a word for word adaptation of the proof of Lemma 4.6 in \cite{DuitsKuijlaars}.
\end{proof}

We finally need the following. 

\begin{prop}\label{MonicMatrixValuedPoly}
There is a unique monic matrix-valued polynomial $\mathbf{P}_N$:
\begin{equation*}
\mathbf{P}_N(z)=z^N\mathbf{I}_{\mathfrak{p}}+\cdots
\end{equation*}
of degree $N$ such that 
\begin{equation*}
\frac{1}{2\pi \textnormal{i}} \oint_{\mathsf{C}_\mathbf{a}}\mathbf{P}_N(z)\mathbf{M}(z)z^kdz=\mathbf{0}_{\mathfrak{p}}, \ \ k=0,\dots,N-1.
\end{equation*}
\end{prop}

\begin{proof}
We expand (note that we can write $z^j$ as a linear combination of $p_0(z),\dots,p_j(z)$, see also the proof of Proposition \ref{PropSimilarity})
\begin{equation*}
\mathbf{P}_N(z)=\mathbf{I}_{\mathfrak{p}}z^N+\sum_{j=0}^{N-1}\mathbf{C}_j p_j(z)
\end{equation*}
for some unknown $\mathfrak{p}\times \mathfrak{p}$ matrices $\mathbf{C}_j$ that we would like to solve for. Showing that this is possible proves the proposition. Plugging in the above expansion in the orthogonality relation we obtain the following equations for $\mathbf{C}_j$,
\begin{equation*}
\sum_{j=0}^{N-1}\frac{1}{2\pi \textnormal{i}} \mathbf{C}_j \oint_{\mathsf{C}_{\mathbf{a}}}p_j(z)\mathbf{M}(z)z^k dz=-\frac{1}{2\pi \textnormal{i}} \oint_{\mathsf{C}_{\mathbf{a}}}z^N\mathbf{M}(z)z^k dz, \ \ k=0,\dots,N-1.
\end{equation*}
Let us write $\mathbf{U}$ for the $\mathfrak{p}N \times \mathfrak{p} N$ matrix having $j$-th, $k$-th block, for $j,k=0,\dots,N-1$, given by the $\mathfrak{p}\times \mathfrak{p}$ matrix
\begin{equation*}
\frac{1}{2\pi \textnormal{i}} \oint_{\mathsf{C}_{\mathbf{a}}}p_j(z)\mathbf{M}(z)z^k dz.
\end{equation*}
If we show that $\mathbf{U}$ is invertible then we are done since by inverting it we can solve for $\mathbf{C}_j$. Recall that $\mathbf{G}$ is a $\mathfrak{p}N \times \mathfrak{p}N$ matrix having $j$-th, $k$-th block, for $j,k=0,\dots,N-1$, given by the $\mathfrak{p}\times \mathfrak{p}$ matrix
\begin{equation*}
   \frac{1}{2\pi \textnormal{i}} \oint_{\mathsf{C}_{\mathbf{a}}}p_j(z)\mathbf{M}(z) \left(-\frac{1}{a_{M+k+1}}\right)\frac{p_{M+N}(z)}{p_{M+k+1}(z)}dz
\end{equation*}
and it is invertible, by our assumption that (\ref{ProbabilityMeasure}) is a well-defined probability measure. Let us write $\mathbf{V}$ for the $\mathfrak{p}N \times \mathfrak{p}N$ change of basis matrix given by
\begin{equation*}
  \begin{pmatrix}
\mathbf{I}_\mathfrak{p}\
z\mathbf{I}_\mathfrak{p}
\cdots
z^{N-1}\mathbf{I}_\mathfrak{p}\\
\end{pmatrix}  =\begin{pmatrix}
-\frac{1}{a_M}\frac{p_{M+N}(z)}{p_{M+1}(z)}\mathbf{I}_\mathfrak{p}\
-\frac{1}{a_{M+1}}\frac{p_{M+N}(z)}{p_{M+2}(z)}\mathbf{I}_\mathfrak{p}
\cdots
-\frac{1}{a_{M+N-1}}\mathbf{I}_\mathfrak{p}\\
\end{pmatrix} \mathbf{V}.
\end{equation*}
Clearly $\mathbf{V}$ is invertible. Thus, since we can write $\mathbf{U}=\mathbf{G}\mathbf{V}$ the conclusion follows.
\end{proof}

\begin{proof}[Proof of Theorem \ref{CorrKernelFixedStartEndpoint}] 
By the Eynard-Mehta theorem, for example in the form presented in \cite{BorodinDeterminantal}, we have that the induced point process from (\ref{ProbabilityMeasure}) is determinantal and the correlation kernel is given by
\begin{align}
\mathsf{K}_N\left[(r,x);(r',y)\right]=-\mathbf{1}_{r>r'}\mathsf{T}_{\mathbf{f}_{r',r}}(y,x)+\sum_{\kappa,\mu=0}^{\mathfrak{p}N-1}\mathsf{T}_{\mathbf{f}_{0,r}}(\kappa,x)\mathbf{G}^{-\textnormal{t}}_{\kappa \mu} \mathsf{T}_{\mathbf{f}_{r',L}}(y,\mathfrak{p}M+\mu).
\end{align}

Then, the first term (more precisely the $i,j$-th coordinate) is easily seen to be given by the first term in the expression in the statement of the theorem. So let us focus on the second term involving the sum over $\kappa,\mu$ which we denote by $\tilde{\mathsf{K}}_N$. Let us write $\kappa=\mathfrak{p}\nu_1+\delta_1$ and $\mu=\mathfrak{p}\nu_2+\delta_2$. Then, this term becomes
\begin{align*}
&\tilde{\mathsf{K}}_N[(r,m\mathfrak{p}+j);(r',k\mathfrak{p}+i)]\\&=\sum_{\nu_1,\nu_2=0}^{N-1} \sum_{\delta_1,\delta_2=0}^{\mathfrak{p}-1} \mathsf{T}_{\mathbf{f}_{r',L}}(k\mathfrak{p}+i,\mathfrak{p}M+\mathfrak{p}\nu_2+\delta_2)\mathbf{G}^{-1}_{\mathfrak{p}\nu_2+\delta_2,\mathfrak{p}\nu_1+\delta_1} \mathsf{T}_{\mathbf{f}_{0,r}}(\mathfrak{p}\nu_1+\delta_1,m\mathfrak{p}+j).
\end{align*}
We can write this in block matrix form as follows 
\begin{align*}
&\left[\tilde{\mathsf{K}}_N[(r,m\mathfrak{p}+j);(r',k\mathfrak{p}+i)]\right]_{i,j=0}^{\mathfrak{p}-1}\\
&=\left[-\frac{1}{2\pi \textnormal{i}}\frac{1}{a_{M+\nu_2}}\oint_{\mathsf{C}_\mathbf{a}}\mathbf{f}_{r',L}(w)\frac{p_k(w)}{p_{M+\nu_2+1}(w)}dw\right]_{\nu_2=0}^{N-1}\mathbf{G}^{-1} \left[-\frac{1}{2\pi \textnormal{i}}\frac{1}{a_{m}}\oint_{\mathsf{C}_\mathbf{a}}\mathbf{f}_{0,r}(z)\frac{p_{\nu_1}(z)}{p_{m+1}(z)}dz\right]_{\nu_1=0}^{N-1},
\end{align*}
where the first factor above is a block row vector having length $N$ with $\mathfrak{p}\times \mathfrak{p}$ blocks and the second factor a block column vector of the corresponding size. Combining the integrals and recalling the expression for $\mathbf{R}_N(w,z)$ from (\ref{ReproducingKernelFirstFormula}) gives the expression of the correlation kernel in the statement of the theorem.

It remains to show that $\mathbf{R}_N(w,z)$ can also be written as (\ref{ReproducingKernelCD}). Towards this end, following \cite{DuitsKuijlaars,Delvaux}, we note that the size $2\mathfrak{p} \times 2\mathfrak{p}$ Riemann-Hilbert problem in the statement of the theorem for $\mathbf{Y}(z)$ has a unique solution given by: 
\begin{equation*}
\mathbf{Y}(z)= \begin{pmatrix}
\mathbf{P}_N(z)
&\frac{1}{2\pi \textnormal{i}}\oint_{\mathsf{C}_\mathbf{a}}\frac{\mathbf{P}_{N}(s)\mathbf{M}(s)}{s-z}ds\\
\mathbf{Q}_{N-1}(z) & \frac{1}{2\pi \textnormal{i}}\oint_{\mathsf{C}_\mathbf{a}}\frac{\mathbf{Q}_{N-1}(s)\mathbf{M}(s)}{s-z}ds
\end{pmatrix}, \ \ z \in \mathbb{C}\setminus \mathsf{C}_{\mathbf{a}},
\end{equation*}
where $\mathbf{P}_N(z)$ is the monic matrix-valued polynomial from Proposition \ref{MonicMatrixValuedPoly} and $\mathbf{Q}_{N-1}$ is the matrix-valued polynomial of degree at most $N-1$ satisfying
\begin{equation*}
\frac{1}{2\pi \textnormal{i}}\oint_{\mathsf{C}_{\mathbf{a}}} \mathbf{Q}_{N-1}(z)\mathbf{M}(z)z^k dz= \begin{cases}
\mathbf{0}_\mathfrak{p}, &k=0,\dots,N-2,\\
-\mathbf{I}_\mathfrak{p}, &k=N-1.
\end{cases}
\end{equation*}
The fact that $\mathbf{Q}_{N-1}$ exists and is unique is proven by following the same scheme of proof as for Proposition \ref{MonicMatrixValuedPoly}, by solving for the corresponding coefficient matrices.
Finally, the fact that $\mathbf{R}_N(w,z)$ has a representation as in (\ref{ReproducingKernelCD}) is proven in exactly the same way as in Proposition 4.9 of \cite{DuitsKuijlaars}: by checking that this expression satisfies the characterising reproducing property from Proposition \ref{ReproducingProperty}. This completes the proof.
\end{proof}

Moving on to the proof of Theorem \ref{ThmLineEnsembleTopBottomLimit} the following proposition is the key technical ingredient. Analogous results should hold for a more general class of parameters $\mathbf{a}$. 

\begin{prop}\label{RHPasymptoticAnalysisProp}
Assume that the parameter sequence $\mathbf{a}$ satisfies
\begin{equation}
\inf_{x\in \mathbb{Z}_+}a_x\ge 1-\mathfrak{c} \textnormal{ and } \sup_{x\in \mathbb{Z}_+}a_x\le1+\mathfrak{c},
\end{equation}
for some $0\le \mathfrak{c} <\frac{1}{3}$. For $r\in (\mathfrak{c},1-2\mathfrak{c})$ define the constant $\mathfrak{c}_r$,
\begin{equation} 
\mathfrak{c}_r=\frac{\mathfrak{c}+r}{1-\mathfrak{c}}<1.
\end{equation}
Moreover, assume that the matrix-valued functions $\mathbf{f}_0(1-z)^{\pm 1},\dots, \mathbf{f}_{L-1}(1-z)^{\pm 1}$ are analytic in an annulus $\mathfrak{A}_\eta$ where $\eta <1-2\mathfrak{c}$. Finally, assume we have the factorisations $\mathbf{f}_{0,L}(1-z)=\mathbf{S}_+(z)\mathbf{S}_-(z)=\tilde{\mathbf{S}}_-(z)\tilde{\mathbf{S}}_+(z)$, where $\mathbf{S}_\pm$, $\tilde{\mathbf{S}}_\pm$ are $\mathfrak{p}\times \mathfrak{p}$ matrix-valued functions satisfying 
\begin{itemize}
    \item $\mathbf{S}_+^{\pm 1}(z), \tilde{\mathbf{S}}_{+}^{\pm 1}(z)$ are analytic for $|z|<1$ and continuous for $|z|\le 1$,
    \item $\mathbf{S}_-^{\pm 1}(z),\tilde{\mathbf{S}}_-^{\pm 1}(z)$ are analytic for $|z|>1$ and continuous for $|z|\ge 1$,
    \item $\mathbf{S}_-(z)\sim z^M\mathbf{I}_{\mathfrak{p}}$ and $\tilde{\mathbf{S}}_-(z)\sim z^M\mathbf{I}_{\mathfrak{p}}$ as $z\to \infty$.
\end{itemize}
Consider the $2\mathfrak{p}\times 2 \mathfrak{p}$ matrix-valued function $\tilde{\mathbf{Y}}:\mathbb{C}\setminus\{|z|=1\}\to \mathbb{C}^{2\mathfrak{p}\times2\mathfrak{p}}$ given as the unique solution to the $2\mathfrak{p}\times 2 \mathfrak{p}$ Riemann-Hilbert problem:
\begin{itemize}
    \item $\tilde{\mathbf{Y}}$ is analytic,
    \item $\tilde{\mathbf{Y}}_+(z)=\tilde{\mathbf{Y}}_-(z)\begin{pmatrix}
\mathbf{I}_{\mathfrak{p}} &\mathbf{f}_{0,L}(1-z)p_{M+N}(1-z)^{-1}\\
\mathbf{0}_{\mathfrak{p}} &\mathbf{I}_{\mathfrak{p}}
\end{pmatrix}$ on $|z|=1$,
    \item $\tilde{\mathbf{Y}}(z)=\left(\mathbf{I}_{2\mathfrak{p}}+\mathcal{O}(z^{-1})\right)\begin{pmatrix}
z^N \mathbf{I}_{\mathfrak{p}} &\mathbf{0}_{\mathfrak{p}}\\
\mathbf{0}_{\mathfrak{p}} &z^{-N}\mathbf{I}_{\mathfrak{p}}
\end{pmatrix}$ as $z\to \infty$.
\end{itemize}
Then, as $N\to \infty$ and for $|z|<1-2\mathfrak{c}$, we have 
\begin{equation*}
\tilde{\mathbf{Y}}(z)=\left(\mathbf{I}_{2\mathfrak{p}}+\mathcal{O}\left(\mathfrak{c}_r^N\right)\right)\begin{pmatrix}
\mathbf{0}_{\mathfrak{p}} &\left(\prod_{k=0}^{M+N-1}a_k\right)\tilde{\mathbf{S}}_+(z)\\
-\left(\prod_{k=0}^{M+N-1}a_k^{-1}\right)\mathbf{S}_+^{-1}(z) &\mathbf{0}_{\mathfrak{p}}
\end{pmatrix},
\end{equation*}
for any $r\in (\mathfrak{c},1-2\mathfrak{c})$ such that $\max\{|z|,\eta\}<r$. Similarly, as $N\to \infty$ and for $|z|>(1-2\mathfrak{c})^{-1}$, 
\begin{equation*}
\tilde{\mathbf{Y}}(z)=\left(\mathbf{I}_{2\mathfrak{p}}+\mathcal{O}\left(\mathfrak{c}_r^N\right)\right)\begin{pmatrix}
\left(\prod_{k=0}^{M+N-1}a_k\right)\tilde{\mathbf{S}}_-^{-1}(z)p_{M+N}(1-z) &\mathbf{0}_{\mathfrak{p}}\\
\mathbf{0}_\mathfrak{p} &\left(\prod_{k=0}^{M+N-1}a_k^{-1}\right)\mathbf{S}_-(z)p_{M+N}(1-z)^{-1}
\end{pmatrix},
\end{equation*}
for any $r\in(\mathfrak{c},1-2\mathfrak{c})$ such that $r^{-1}<\min\{|z|,\eta^{-1}\}$.
\end{prop}

\begin{proof} We transform, in several steps, the Riemann-Hilbert problem for $\tilde{\mathbf{Y}}(z)$ to a problem which can be solved explicitly, up to asymptotically negligible terms, by virtue of the factorisation for $\mathbf{f}_{0,L}(1-z)$. The strategy itself is standard and was followed in \cite{DuitsBeggren} but there are a couple of interesting complications due to the non-constant inhomogeneity sequence $\mathbf{a}$.

\textbf{Step 1} We make the following transformation to normalise (in a slightly less standard way) the Riemann-Hilbert problem for $\tilde{\mathbf{Y}}(z)$ at infinity. Namely, define $\mathbf{X}(z)$ by
\begin{align*}
 \mathbf{X}(z)=\begin{cases}
  \tilde{\mathbf{Y}}(z) \begin{pmatrix}
p_{M}(1-z)p_{M+N}(1-z)^{-1}\mathbf{I}_\mathfrak{p} &\mathbf{0}_{\mathfrak{p}}\\
\mathbf{0}_\mathfrak{p} &p_{M+N}(1-z)p_{M}(1-z)^{-1}\mathbf{I}_{\mathfrak{p}}
\end{pmatrix}, &|z|>1,\\
\tilde{\mathbf{Y}}(z), &|z|<1.   
 \end{cases}
\end{align*}
Then, a simple computation shows that $\mathbf{X}(z)$ solves the Riemann-Hilbert problem with conditions, since  $p_m(1-z)\sim \left(\prod_{k=0}^{m-1}a_k^{-1}\right) z^m$ as $z \to \infty$, 
\begin{align*}
  \mathbf{X}_+(z)&=\mathbf{X}_-(z) \begin{pmatrix}
p_{M+N}(1-z)p_{M}(1-z)^{-1}\mathbf{I}_\mathfrak{p} &\mathbf{f}_{0,L}(1-z)p_M(1-z)^{-1}\\
\mathbf{0}_\mathfrak{p} &p_M(1-z)p_{M+N}(1-z)^{-1}\mathbf{I}_\mathfrak{p}
\end{pmatrix}, &|z|=1,\\
\mathbf{X}(z)&=\left(\mathbf{I}_{2\mathfrak{p}}+\mathcal{O}(z^{-1})\right) \begin{pmatrix}
\prod_{k=M}^{M+N-1}a_k\mathbf{I}_{\mathfrak{p}} &\mathbf{0}_{\mathfrak{p}}\\
\mathbf{0}_{\mathfrak{p}} &\prod_{k=M}^{M+N-1}a_k^{-1}\mathbf{I}_{\mathfrak{p}}
\end{pmatrix}, &\textnormal{ as } z\to \infty.
\end{align*}

\textbf{Step 2} We make another transformation. Consider $\mathbf{V}(z)$ given by, where $r\in (\mathfrak{c},1-2\mathfrak{c})$ is picked such that the circles $|z|=r$ and $|z|=r^{-1}$ are within the annulus of analyticity  $\mathfrak{A}_{\eta}$ from the statement, see Figure \ref{ContoursRHP} for an illustration,
\begin{align*}
 \mathbf{V}(z)=\begin{cases}
  \mathbf{X}(z) \begin{pmatrix}
\mathbf{I}_\mathfrak{p} &\mathbf{0}_{\mathfrak{p}}\\
p_M^{2}(1-z)p_{M+N}(1-z)^{-1}\mathbf{f}_{0,L}(1-z)^{-1} &\mathbf{I}_{\mathfrak{p}}
\end{pmatrix}, &1<|z|<r^{-1},\\
\mathbf{X}(z) \begin{pmatrix}
\mathbf{I}_\mathfrak{p} &\mathbf{0}_{\mathfrak{p}}\\
-p_{M+N}(1-z)\mathbf{f}_{0,L}(1-z)^{-1} &\mathbf{I}_{\mathfrak{p}}
\end{pmatrix} & r<|z|<1,\\
\mathbf{X}(z), &|z|<r \textnormal{ or }  |z|>r^{-1}.   
 \end{cases}
\end{align*}
Then, a computation shows that $\mathbf{V}(z)$ solves the following RHP
\begin{align*}
  \mathbf{V}_+(z)&=\mathbf{V}_-(z) \begin{pmatrix}
\mathbf{I}_{\mathfrak{p}} &\mathbf{0}_{\mathfrak{p}}\\
-\mathbf{f}_{0,L}(1-z)^{-1}p_{M+N}(1-z) &\mathbf{I}_{\mathfrak{p}}
\end{pmatrix}, &|z|=r,\\
  \mathbf{V}_+(z)&=\mathbf{V}_-(z) \begin{pmatrix}
\mathbf{0}_{\mathfrak{p}} &\mathbf{f}_{0,L}(1-z)p_M(1-z)^{-1}\\
-\mathbf{f}_{0,L}(1-z)^{-1}p_M(1-z) &\mathbf{0}_{\mathfrak{p}}
\end{pmatrix}, &|z|=1,\\
  \mathbf{V}_+(z)&=\mathbf{V}_-(z) \begin{pmatrix}
\mathbf{I}_{\mathfrak{p}} &\mathbf{0}_{\mathfrak{p}}\\
p_M^{2}(1-z)p_{M+N}(1-z)^{-1}\mathbf{f}_{0,L}(1-z)^{-1} &\mathbf{I}_{\mathfrak{p}}
\end{pmatrix}, &|z|=r^{-1},\\
\mathbf{V}(z)&=\left(\mathbf{I}_{2\mathfrak{p}}+\mathcal{O}(z^{-1})\right) \begin{pmatrix}
\prod_{k=M}^{M+N-1}a_k\mathbf{I}_{\mathfrak{p}} &\mathbf{0}_{\mathfrak{p}}\\
\mathbf{0}_{\mathfrak{p}} &\prod_{k=M}^{M+N-1}a_k^{-1}\mathbf{I}_{\mathfrak{p}}
\end{pmatrix}, &\textnormal{ as } z\to \infty.
\end{align*}

\begin{figure}
\centering
\captionsetup{singlelinecheck = false, justification=justified}
\begin{tikzpicture}

\draw[->] (0,0)  to (3.5,0);

\draw[thick] (0,0) circle [radius=2];

\draw[thick,blue] (0,0) circle [radius=1.6];

\draw[thick,blue] (0,0) circle [radius=2.5];

\draw[thick,dotted] (0,0) circle [radius=2.8];

\draw[thick,dotted] (0,0) circle [radius=1.3];

\node[below] at (0,0) {\small $0$};

\node[below left] at (1.3,0) {\small $\eta$};

\node[below left] at (1.6,0) {\small $r$};

\node[below left] at (2,0) {\small $1$};

\node[below left] at (2.6,0) {\small $r^{-1}$};

\node[below right] at (2.8,0) {\small $\eta^{-1}$};

\draw[fill] (0,0) circle [radius=0.025];

\draw[fill] (1.3,0) circle [radius=0.025];

\draw[fill] (1.6,0) circle [radius=0.025];

\draw[fill] (2,0) circle [radius=0.025];

\draw[fill] (2.5,0) circle [radius=0.025];

\draw[fill] (2.8,0) circle [radius=0.025];

\end{tikzpicture}
\caption{The contours for the Riemann-Hilbert problem for $\mathbf{V}$. The area enclosed in the dotted circles is the annulus $\mathfrak{A}_\eta$.}\label{ContoursRHP}
\end{figure}
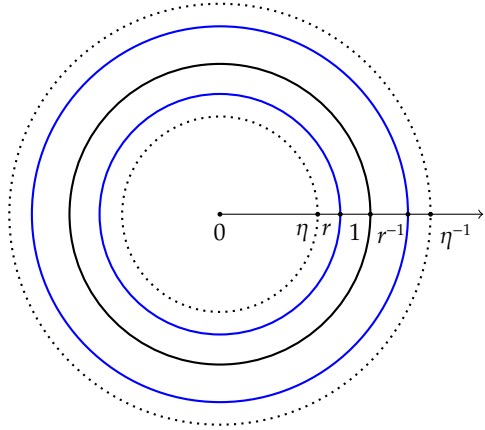

\textbf{Step 3} Now, consider the following RHP problem for $\tilde{\mathbf{V}}(z)$, which disregards the jumps on $|z|=r$ and $|z|=r^{-1}$ for $\mathbf{V}(z)$, jumps that we will show next are exponentially small as $N \to \infty$,
\begin{align*}
  \tilde{\mathbf{V}}_+(z)&=\tilde{\mathbf{V}}_-(z) \begin{pmatrix}
\mathbf{0}_{\mathfrak{p}} &\mathbf{f}_{0,L}(1-z)p_M(1-z)^{-1}\\
-\mathbf{f}_{0,L}(1-z)^{-1}p_M(1-z) &\mathbf{0}_{\mathfrak{p}}
\end{pmatrix}, &|z|=1,\\
\tilde{\mathbf{V}}(z)&=\left(\mathbf{I}_{2\mathfrak{p}}+\mathcal{O}(z^{-1})\right) \begin{pmatrix}
\prod_{k=M}^{M+N-1}a_k\mathbf{I}_{\mathfrak{p}} &\mathbf{0}_{\mathfrak{p}}\\
\mathbf{0}_{\mathfrak{p}} &\prod_{k=M}^{M+N-1}a_k^{-1}\mathbf{I}_{\mathfrak{p}}
\end{pmatrix}, &\textnormal{ as } z\to \infty.
\end{align*}
Then, by our factorisation assumption on $\mathbf{f}_{0,L}(1-z)$, we can construct the explicit solution to this RHP for $\tilde{\mathbf{V}}(z)$ by, since recall $p_M(1-z)\sim \left(\prod_{k=0}^{M-1}a_k^{-1}\right) z^M$ as $z \to \infty$,
\begin{equation}
\tilde{\mathbf{V}}(z)=\begin{cases}
\begin{pmatrix}
\left(\prod_{k=0}^{M+N-1}a_k\right)\tilde{\mathbf{S}}_{-}^{-1}(z)p_{M}(1-z) & \mathbf{0}_{\mathfrak{p}}\\
\mathbf{0}_{\mathfrak{p}} &\left(\prod_{k=0}^{M+N-1}a_k^{-1}\right)\mathbf{S}_{-}(z)p_{M}(1-z)^{-1} 
\end{pmatrix} , \ &|z|>1,\\
 \begin{pmatrix}
\mathbf{0}_{\mathfrak{p}} &\left(\prod_{k=0}^{M+N-1}a_k\right)\tilde{\mathbf{S}}_+(z)\\
-\left(\prod_{k=0}^{M+N-1}a_k^{-1}\right)\mathbf{S}_+^{-1}(z) &\mathbf{0}_{\mathfrak{p}}
\end{pmatrix}, \ &|z|<1.
\end{cases}
\end{equation}
\textbf{Step 4} We now show that the jumps on the circles $|z|=r$ and $|z|=r^{-1}$ for $\mathbf{V}(z)$ are exponentially small. Let $|| \cdot ||$ be any norm on $\mathfrak{p}\times \mathfrak{p}$ matrices. Then,  for $r\in (\mathfrak{c},1-2\mathfrak{c})$ picked such that the circles $|z|=r$ and $|z|=r^{-1}$ are within the annulus of analyticity  $\mathfrak{A}_{\eta}$ of $\mathbf{f}_{0,L}(1-z)^{-1}$, we have, with the implicit constant being independent of $N$ and $r$,
\begin{align*}
\sup_{|z|=r}\left|\left|\mathbf{f}_{0,L}(1-z)^{-1}p_{M+N}(1-z)\right|\right|&=\mathcal{O}\left(\frac{1}{\left(\inf_{k\in \mathbb{Z}_+}a_k\right)^N}\left(\sup_{|z|=r} \sup_{k\in \mathbb{Z}_+} |z+a_k-1|\right)^N\right)\\
&=\mathcal{O}\left(\frac{1}{(1-\mathfrak{c})^N}\left(\sup_{|z|=r} \sup_{x\in [-\mathfrak{c},\mathfrak{c}]} |z+x|\right)^N\right)\\
&=\mathcal{O}\left(\left(\frac{\mathfrak{c}+r}{1-\mathfrak{c}}\right)^N\right),
\end{align*}
and similarly,
\begin{align*}
\sup_{|z|=r^{-1}}\left|\left|\mathbf{f}_{0,L}(1-z)^{-1}p_{M}^2(1-z)p_{M+N}(1-z)^{-1}\right|\right|&=\mathcal{O}\left(\left(\sup_{k\in \mathbb{Z}_+}a_k\right)^N\left(\sup_{|z|=r^{-1}} \sup_{k\in \mathbb{Z}_+} \frac{1}{|z+a_k-1|}\right)^N\right)\\
&=\mathcal{O}\left((1+\mathfrak{c})^N\left(\frac{1}{\inf_{|z|=r^{-1}} \inf_{x\in [-\mathfrak{c},\mathfrak{c}]} |z+x|}\right)^N\right)\\
&=\mathcal{O}\left(\left(\frac{1+\mathfrak{c}}{r^{-1}-\mathfrak{c}}\right)^N\right).
\end{align*}
Observe that, for $r\in (\mathfrak{c},1-2\mathfrak{c})$ with $0\le \mathfrak{c}<\frac{1}{3}$, we have
\begin{equation*}
\mathfrak{c}_r=\frac{\mathfrak{c}+r}{1-\mathfrak{c}}=\max\left\{\frac{\mathfrak{c}+r}{1-\mathfrak{c}},\frac{1+\mathfrak{c}}{r^{-1}-\mathfrak{c}}\right\}<1.
\end{equation*}
Now, consider the function $\mathbf{J}(z)$ given by
\begin{equation*}
\mathbf{J}(z)=\mathbf{V}(z)\tilde{\mathbf{V}}(z)^{-1}.
\end{equation*}
Then, observe that $\mathbf{J}(z)$, which only has jumps on the circles $|z|=r$ and $|z|=r^{-1}$ which are exponentially small, can be solved in terms of a Neumann series, see for example \cite{DeiftBook}, to give
\begin{equation*}
\mathbf{J}(z)=\mathbf{I}_{2\mathfrak{p}}+\mathcal{O}\left(\mathfrak{c}_r^N\right),  \textnormal{ as } N \to \infty,
\end{equation*}
uniformly on compact sets of $\mathbb{C}\setminus\left(\{|z|=r\}\cup \{|z|=r^{-1}\}\right)$. This completes this step.

\textbf{Step 5} We finally undo the transformations to obtain the statement of the proposition. For $|z|<1-2\mathfrak{c}$ we can choose $r\in (\mathfrak{c},1-2\mathfrak{c})$ such that $\max\{|z|,\eta\}<r$, to obtain
\begin{align*}
\tilde{\mathbf{Y}}(z)=\mathbf{X}(z)=\mathbf{V}(z)&=\left(\mathbf{I}_{2\mathfrak{p}}+\mathcal{O}\left(\mathfrak{c}_r^N\right)\right)\tilde{\mathbf{V}}(z)\\
&=\left(\mathbf{I}_{2\mathfrak{p}}+\mathcal{O}\left(\mathfrak{c}_r^N\right)\right)\begin{pmatrix}
\mathbf{0}_{\mathfrak{p}} &\left(\prod_{k=0}^{M+N-1}a_k\right)\tilde{\mathbf{S}}_+(z)\\
-\left(\prod_{k=0}^{M+N-1}a_k^{-1}\right)\mathbf{S}_+^{-1}(z) &\mathbf{0}_{\mathfrak{p}}
\end{pmatrix}.
\end{align*}
For $|z|>(1-2\mathfrak{c})^{-1}$ we can choose $r\in (\mathfrak{c},1-2\mathfrak{c})$ such that $r^{-1}<\min\{|z|,\eta^{-1}\}$, to obtain
\begin{align*}
\tilde{\mathbf{Y}}(z)&=\mathbf{X}(z) \begin{pmatrix}
p_{M}(1-z)p_{M+N}(1-z)^{-1}\mathbf{I}_\mathfrak{p} &\mathbf{0}_{\mathfrak{p}}\\
\mathbf{0}_\mathfrak{p} &p_{M+N}(1-z)p_{M}(1-z)^{-1}\mathbf{I}_{\mathfrak{p}}
\end{pmatrix}\\
&=\mathbf{V}(z)\begin{pmatrix}
p_{M}(1-z)p_{M+N}(1-z)^{-1}\mathbf{I}_\mathfrak{p} &\mathbf{0}_{\mathfrak{p}}\\
\mathbf{0}_\mathfrak{p} &p_{M+N}(1-z)p_{M}(1-z)^{-1}\mathbf{I}_{\mathfrak{p}}
\end{pmatrix}\\
&=\left(\mathbf{I}_{2\mathfrak{p}}+\mathcal{O}\left(\mathfrak{c}_r^N\right)\right)\tilde{\mathbf{V}}(z)\begin{pmatrix}
p_{M}(1-z)p_{M+N}(1-z)^{-1}\mathbf{I}_\mathfrak{p} &\mathbf{0}_{\mathfrak{p}}\\
\mathbf{0}_\mathfrak{p} &p_{M+N}(1-z)p_{M}(1-z)^{-1}\mathbf{I}_{\mathfrak{p}}
\end{pmatrix}\\
&=\left(\mathbf{I}_{2\mathfrak{p}}+\mathcal{O}\left(\mathfrak{c}_r^N\right)\right)\begin{pmatrix}
\left(\prod_{k=0}^{M+N-1}a_k\right)\tilde{\mathbf{S}}_-^{-1}(z)p_{M+N}(1-z) &\mathbf{0}_{\mathfrak{p}}\\
\mathbf{0}_\mathfrak{p} &\left(\prod_{k=0}^{M+N-1}a_k^{-1}\right)\mathbf{S}_-(z)p_{M+N}(1-z)^{-1}
\end{pmatrix}.
\end{align*}
This completes the proof.
\end{proof}

We can finally prove  Theorem \ref{ThmLineEnsembleTopBottomLimit}.

\begin{proof}[Proof of Theorem \ref{ThmLineEnsembleTopBottomLimit}] Note that, by our assumption on $\mathbf{a}$ we can pick the $\mathsf{C}_{\mathbf{a}}$ contour to be the circle $|z-1|=1$. After the change of variables $z\mapsto 1-z$, $w\mapsto 1-w$, we can then write the kernel $\mathsf{K}_N$ presented in block matrix form in (\ref{BlockMatrixFormKernel}) as follows:
\begin{align*}
&\left[\mathsf{K}_N\left[(r,m\mathfrak{p}+j);(r',k\mathfrak{p}+i)\right]\right]_{i,j=0}^{\mathfrak{p}-1}=-\mathbf{1}_{r>r'}\frac{1}{2\pi \textnormal{i}}\frac{1}{a_m}\oint_{|z|=1}\mathbf{f}_{r',r}(1-z)\frac{p_k(1-z)}{p_{m+1}(1-z)}dz\nonumber\\
&+\frac{1}{(2\pi \textnormal{i})^2}\frac{1}{a_k}\oint_{|z|=1}\oint_{|w|=1} \mathbf{f}_{r',L}(1-w)\frac{1}{z-w} \begin{pmatrix}
\mathbf{0}_{\mathfrak{p}} \ \mathbf{I}_{\mathfrak{p}}
\end{pmatrix}\tilde{\mathbf{Y}}^{-1}(w)\tilde{\mathbf{Y}}(z)\begin{pmatrix}
\mathbf{I}_{\mathfrak{p}}\\
\mathbf{0}_{\mathfrak{p}}
\end{pmatrix}\mathbf{f}_{0,r}(1-z)\times \\
&\times\frac{p_k(1-w)}{p_{M+N}(1-w)p_{m+1}(1-z)}dz dw,
\end{align*}
where $\tilde{\mathbf{Y}}$ is the Riemann-Hilbert problem given in Proposition \ref{RHPasymptoticAnalysisProp}. We then deform the contour for $w$ to a circle of radius at least $(1-2\mathfrak{c})^{-1}$ but still within the annulus of analyticity $\mathfrak{A}_\eta$ and similarly deform the contour for $z$ to a circle of radius at most $(1-2\mathfrak{c})$ but also contained in $\mathfrak{A}_\eta$. Note that, by our assumption this is possible. Hence, from Proposition \ref{RHPasymptoticAnalysisProp} we  have for any $\delta\in(\mathfrak{c},1-2\mathfrak{c})$ such that $\delta^{-1}<\min\{|w|,\eta^{-1}\}$
\begin{align*}
&\frac{1}{p_{M+N}(1-w)}\begin{pmatrix}
\mathbf{0}_{\mathfrak{p}} \ \mathbf{I}_{\mathfrak{p}}
\end{pmatrix}\tilde{\mathbf{Y}}^{-1}(w)=\frac{1}{p_{M+N}(1-w)}\begin{pmatrix}
\mathbf{0}_{\mathfrak{p}} \ \mathbf{I}_{\mathfrak{p}}
\end{pmatrix}\left(\mathbf{I}_{2\mathfrak{p}}+\mathcal{O}\left(\mathfrak{c}_\delta^N\right)\right)\times\\
&\times\begin{pmatrix}
\left(\prod_{k=0}^{M+N-1}a_k^{-1}\right)\tilde{\mathbf{S}}_-(w)p_{M+N}(1-w)^{-1} &\mathbf{0}_{\mathfrak{p}}\\
\mathbf{0}_\mathfrak{p} &\left(\prod_{k=0}^{M+N-1}a_k\right)\mathbf{S}_-^{-1}(w)p_{M+N}(1-w)
\end{pmatrix}\\
&=\left(\prod_{k=0}^{M+N-1}a_k\right)\begin{pmatrix}
\mathbf{0}_{\mathfrak{p}} \ \mathbf{S}_-^{-1}(w)
\end{pmatrix}\left(\mathbf{I}_{2\mathfrak{p}}+\mathcal{O}\left(\mathfrak{c}_\delta^N\right)\right),
\end{align*}
and similarly for any $\delta\in (\mathfrak{c},1-2\mathfrak{c})$ such that $\max\{|z|,\eta\}<\delta$,
\begin{align*}
\tilde{\mathbf{Y}}(z)\begin{pmatrix}
\mathbf{I}_{\mathfrak{p}}\\
\mathbf{0}_{\mathfrak{p}}
\end{pmatrix}&=\left(\mathbf{I}_{2\mathfrak{p}}+\mathcal{O}\left(\mathfrak{c}_\delta^N\right)\right) \begin{pmatrix}
\mathbf{0}_{\mathfrak{p}} &\left(\prod_{k=0}^{M+N-1}a_k\right)\tilde{\mathbf{S}}_+(z)\\
-\left(\prod_{k=0}^{M+N-1}a_k^{-1}\right)\mathbf{S}_+^{-1}(z) &\mathbf{0}_{\mathfrak{p}}
\end{pmatrix}\begin{pmatrix}
\mathbf{I}_{\mathfrak{p}}\\
\mathbf{0}_{\mathfrak{p}}
\end{pmatrix}\\
&=-\left(\prod_{k=0}^{M+N-1}a_k^{-1}\right)\left(\mathbf{I}_{2\mathfrak{p}}+\mathcal{O}\left(\mathfrak{c}_\delta^N\right)\right)\begin{pmatrix}
\mathbf{0}_{\mathfrak{p}}\\
\mathbf{S}_+^{-1}(z)
\end{pmatrix}.
\end{align*}
Plugging these expressions in the correlation kernel $\mathsf{K}_N$, noting that $\prod_{k=0}^{M+N-1}a_k^{-1}$ and $\prod_{k=0}^{M+N-1}a_k$  cancel out exactly, and taking the $N \to \infty$ limit first (recall $\mathfrak{c}_\delta<1$), and then deforming the contours to the circles $|z|=1^-$ and $|w|=1^+$ we get
\begin{equation*}
  \lim_{N\to \infty}\left[\mathsf{K}_N\left[(r,m\mathfrak{p}+j);(r',k\mathfrak{p}+i)\right]\right]_{i,j=0}^{\mathfrak{p}-1}=\left[\mathsf{K}_\infty\left[(r,m\mathfrak{p}+j);(r',k\mathfrak{p}+i)\right]\right]_{i,j=0}^{\mathfrak{p}-1}.
\end{equation*}
This implies the convergence of all the correlation functions and since the state space is discrete it implies convergence in distribution of the corresponding point processes, and thus of the minimal particles, and gives the statement of the theorem.
\end{proof}

We finally give an explicit example where all the conditions of Theorem \ref{ThmLineEnsembleTopBottomLimit} are satisfied. We focus on the scalar case $\mathfrak{p}=1$ with functions corresponding to the three types of non-intersecting processes (pure-birth, Bernoulli, geometric) we have been studying in this paper. Recall that for $\mathfrak{p}=1$ we can just take $\tilde{\mathbf{S}}_{\pm}(z)=\mathbf{S}_{\pm}(z)$. For a detailed study of the factorisation in the matrix case $\mathfrak{p}\ge 2$, see \cite{DuitsBeggren}, also \cite{BorodinDuits,ChhitaDuits}.

\begin{prop}\label{PropFactorisation}
Let $\mathfrak{p}=1$. Assume that the parameter sequence $\mathbf{a}$ satisfies
\begin{equation}
\inf_{x\in \mathbb{Z}_+}a_x\ge 1-\mathfrak{c} \textnormal{ and } \sup_{x\in \mathbb{Z}_+}a_x\le 1+\mathfrak{c},
\end{equation}
for some $0\le \mathfrak{c} <\frac{1}{3}$. Suppose the functions $\mathbf{f}_0(z),\dots,\mathbf{f}_{L-1}(z)$ are a product of factors of the form $(1-\alpha z), \ (1+\beta z)^{-1}, \ e^{-tz}$ so that in particular the function $\mathbf{f}_{0,L}(1-z)$ is given by
\begin{equation*}
\mathbf{f}_{0,L}(1-z)=\prod_{r=1}^{L_1}(1-\alpha_r+\alpha_r z)\prod_{r=1}^{L_2}\frac{1}{1+\beta_r-\beta_r z}\prod_{r=1}^{L_3}e^{t_r z-t_r},
\end{equation*}
for some $L_1,L_2,L_3\in \mathbb{Z}_+$ and constants $\alpha_r,\beta_r, t_r \in \mathbb{R}_+$. Moreover, suppose there exist exactly $M$ indices $i_1,\dots,i_M$ such that
\begin{equation*}
 (2-2\mathfrak{c})^{-1}<\alpha_{i_j}<(1+\mathfrak{c})^{-1}, \ \ j=1,\dots,M,
\end{equation*}
and for $r\neq i_j$ we have $\alpha_r<\frac{1-2\mathfrak{c}}{2-2\mathfrak{c}}$. Finally, assume that $\beta_r<\frac{1}{2\mathfrak{c}}-1$. Then, with the factorisation $\mathbf{f}_{0,L}(1-z)=\mathbf{S}_+(z)\mathbf{S}_-(z)$ chosen as follows
\begin{align*}
\mathbf{S}_-(z)&=\prod_{j=1}^M\left(z+\frac{1-\alpha_{i_j}}{\alpha_{i_j}}\right),\\
\mathbf{S}_+(z)&=\prod_{j=1}^{M}\alpha_{i_j} \prod_{r=1, r\neq i_j}^{L_1} (1-\alpha_r+\alpha_r z) \prod_{r=1}^{L_2}\frac{1}{1+\beta_r-\beta_r z} \prod_{r=1}^{L_3}e^{t_r z-t_r},
\end{align*}
all the conditions of Theorem \ref{ThmLineEnsembleTopBottomLimit} are satisfied.
\end{prop}

\begin{proof} Observe that, $\mathbf{f}_i(1-z)^{\pm 1}$ can have poles only at points $(1-\alpha_r^{-1})$ or $(1+\beta_r^{-1})$ for some $r$. Then, elementary computations show that the conditions above on the $\alpha_r,\beta_r$ ensure that $\mathbf{f}_{0,L}(z)\in \mathsf{Hol}\left(\mathbb{H}_{-R}\right)$ for some $R>R(\mathbf{a})=\sup_{x\in \mathbb{Z}_+}a_x-\inf_{x\in \mathbb{Z}_+}a_x$ and that we can find $\eta<1-2\mathfrak{c}$ such that none of the points $(1-\alpha_r^{-1})$ or $(1+\beta_r^{-1})$ are in the intervals $(-\eta^{-1},-\eta)\cup (\eta,\eta^{-1})$.
Finally, it is easy to see by the assumptions on the $\alpha_r,\beta_r$'s the functions $\mathbf{S}^{\pm 1}_-(z)$, $\mathbf{S}^{\pm 1}_+(z)$ are analytic for $|z|>1$ and $|z|<1$ respectively (and continuous up to the boundary) and we have the correct growth condition as $z \to \infty$ for $\mathbf{S}_-(z)$.
\end{proof}

\begin{proof}[Proof of Theorem \ref{InhomogeneousLineEnsemblesIntro}]
Note that, the induced measure on $\mathbb{W}_N^{L-1}$ by the non-intersecting random walks is given by (\ref{ProbabilityMeasure}) with $\mathfrak{p}=1$ and $\mathbf{f}_r=f_{r,r+1}$ as in the statement of Theorem \ref{InhomogeneousLineEnsemblesIntro}. Then, the result follows by combining Theorem \ref{ThmLineEnsembleTopBottomLimit} and Proposition \ref{PropFactorisation} with $\mathsf{K}_\infty^{L,M}=\mathsf{K}_\infty$.
\end{proof}

\section{Convergence to the discrete Bessel point process}\label{SectionBesselConvergence}

\begin{proof}[Proof of Theorem \ref{BesselTheorem}] 
Let us denote by $\tilde{K}_t^{(N)}(y_1,y_2)$ the correlation kernel of the point process with the origin shifted to $N$, namely $\left\{\mathsf{X}_i^{(N)}(t)-N\right\}_i$, which is then given by from Theorem \ref{ThmCorrelationKernelNI}, where we make a couple of changes of variables in the contour integrals,
\begin{align*}
\tilde{K}_t^{(N)}(y_1,y_2)&=-\frac{1}{a_{y_1+N}}\frac{1}{(2 \pi \textnormal{i})^2}\oint_{\mathsf{C}_{\mathbf{a},0}} dw\oint_{\mathsf{C}_0} du \frac{p_{y_2+N}(u)e^{-tw}}{p_{y_1+N+1}(w)e^{-tu}} \frac{w^N}{u^N}\frac{1}{w-u}\\
&=\frac{1}{a_{y_1+N}}\frac{1}{(2 \pi \textnormal{i})^2}\oint_{1-\mathsf{C}_{\mathbf{a},0}} dw\oint_{1-\mathsf{C}_0} du \frac{p_{y_2+N}(1-u)e^{tw}}{p_{y_1+N+1}(1-w)e^{tu}} \frac{(1-w)^N}{(1-u)^N}\frac{1}{w-u}\\
&=\frac{1}{a_{y_1+N}}\frac{1}{(2 \pi \textnormal{i})^2}\oint_{|w|=\epsilon^{-1}} dw\oint_{|u|=(\epsilon')^{-1}} du \frac{p_{y_2+N}(1-u)e^{tw}}{p_{y_1+N+1}(1-w)e^{tu}} \frac{(1-w)^N}{(1-u)^N}\frac{1}{w-u}\\
&=\frac{1}{a_{{y_1+N}}}\frac{1}{\left(2\pi \textnormal{i}\right)^2}\oint_{|w|=\epsilon}dw\oint_{|u|=\epsilon'}du \frac{p_{y_2+N}\left(1-\frac{1}{u}\right)}{p_{y_1+N+1}\left(1-\frac{1}{w}\right)} \frac{e^{tw^{-1}}}{e^{tu^{-1}}} \frac{u^{N-1}}{w^{N+1}}\frac{\left(1-w\right)^N}{\left(1-u\right)^N}\frac{1}{u-w} ,
\end{align*}
with the radii $0<\epsilon<\epsilon'$ chosen appropriately small so that the circles $|w|=\epsilon^{-1}$ and $|z|=(\epsilon')^{-1}$ contain all the relevant poles. Let us pick, for concreteness, for $N$ large enough, $\epsilon=\frac{\zeta}{N}$ and $\epsilon'=\frac{2\zeta}{N}$ and make the change of variables $v=\frac{N}{\zeta}u$ and $z=\frac{N}{\zeta}w$. We then obtain that the kernel is given by, for $N$ large enough,
\begin{align*}
 \tilde{K}^{(N)}_{t}\left(y_1,y_2\right)=\frac{1}{a_{{y_1+N}}}\frac{1}{\left(2\pi \textnormal{i}\right)^2}\oint_{|z|=1}dz\oint_{|v|=2}dv \frac{p_{y_2+N}\left(1-\frac{N}{\zeta v}\right)}{p_{y_1+N+1}\left(1-\frac{N}{\zeta z}\right)} \frac{e^{\frac{tN}{\zeta}z^{-1}}}{e^{\frac{tN}{\zeta}v^{-1}}} \frac{v^{N-1}}{z^{N+1}}\frac{\left(1-\frac{\zeta z}{N}\right)^N}{\left(1-\frac{\zeta v}{N}\right)^N}\frac{1}{v-z} \frac{N}{\zeta}.
\end{align*}
Consider the ratio of characteristic polynomials
\begin{equation}\label{RationChar}
 \frac{p_{y_2+N}\left(1-\frac{N}{\zeta v}\right)}{p_{y_1+N+1}\left(1-\frac{N}{\zeta z}\right)}=\frac{\prod_{k=0}^{N-1}\left(1-\frac{\left(1-\frac{N}{\zeta v}\right)}{a_k}\right)}{\prod_{k=0}^{N-1}\left(1-\frac{\left(1-\frac{N}{\zeta z}\right)}{a_k}\right)}\times \frac{\prod_{k=N}^{y_2+N-1}\left(1-\frac{\left(1-\frac{N}{\zeta v}\right)}{a_k}\right)}{\prod_{k=N}^{y_1+N}\left(1-\frac{\left(1-\frac{N}{\zeta z}\right)}{a_k}\right)}.
\end{equation}
The second factor above satisfies, where as usual $\mathfrak{f}_N \sim \mathfrak{g}_N$ means $\lim_{N\to \infty} \frac{\mathfrak{f}_N}{\mathfrak{g}_N}=1$, 
\begin{equation*}
   \frac{\prod_{k=N}^{y_2+N-1}\left(1-\frac{\left(1-\frac{N}{\zeta v}\right)}{a_k}\right)}{\prod_{k=N}^{y_1+N}\left(1-\frac{\left(1-\frac{N}{\zeta z}\right)}{a_k}\right)} \sim \frac{\left(\frac{N}{\zeta v}\right)^{y_2}}{\left(\frac{N}{\zeta z}\right)^{y_1+1}} \frac{\prod_{k=N}^{y_1+N}a_k}{\prod_{k=N}^{y_2+N-1}a_k}=\frac{\left(\frac{N}{\zeta v}\right)^{y_2}}{\left(\frac{N}{\zeta z}\right)^{y_1+1}} \frac{\prod_{k=0}^{y_1+N}a_k}{\prod_{k=0}^{y_2+N-1}a_k}.
\end{equation*}
Recall that $\{a_x\}_{x=0}^\infty$ is uniformly bounded. If we write $a_i^{(N)}=N^{-1}\left(\zeta a_i-N\right)$, for $1\le i \le N$, then we have $a_{i}^{(N)} \to 0$ for all $i$, $\sum_{i=1}^N a_i^{(N)} \to \zeta (\bar{a}-1)$ and $\sum_{i=1}^N (a_i^{(N)})^2 \to 0$. We then note that, uniformly on compact sets in $\mathbb{C}$, we have as $N \to \infty$,
\begin{equation}\label{CharPolyConv}
\prod_{k=0}^{N-1}\left(1+v\left(\frac{\zeta a_k-\zeta}{N}\right)\right)=\prod_{k=0}^{N-1}\left(1+va_i^{(N)}\right)\to e^{\zeta(\bar{a}-1)v}.
\end{equation}
Hence, we obtain that the first factor in (\ref{RationChar}) satisfies, as $N \to \infty$,
\begin{equation*}
  \frac{\prod_{k=0}^{N-1}\left(1-\frac{\left(1-\frac{N}{\zeta v}\right)}{a_k}\right)}{\prod_{k=0}^{N-1}\left(1-\frac{\left(1-\frac{N}{\zeta z}\right)}{a_k}\right)}=\frac{z^N}{v^{N}}\frac{\prod_{k=0}^{N-1}\left(1+v\left(\frac{\zeta a_k-\zeta}{N}\right)\right)}{\prod_{k=0}^{N-1}\left(1+z\left(\frac{\zeta a_k-\zeta}{N}\right)\right)}  \sim \frac{z^N}{v^N}\frac{e^{\zeta(\bar{a}-1)v}}{e^{\zeta(\bar{a}-1)z}}.
\end{equation*}
Thus, we get the asymptotics as $N \to \infty$,
\begin{equation*}
\frac{v^{N-1}}{z^{N+1}}\frac{p_{y_2+N}\left(1-\frac{N}{\zeta v}\right)}{p_{y_1+N+1}\left(1-\frac{N}{\zeta z}\right)}\frac{N}{\zeta}\sim \frac{z^{y_1}}{v^{y_2+1}}\frac{e^{\zeta(\bar{a}-1)v}}{e^{\zeta(\bar{a-1})z}} \left(\frac{N}{\zeta}\right)^{y_2-y_1} \frac{\prod_{k=0}^{y_1+N}a_k}{\prod_{k=0}^{y_2+N-1}a_k}.
\end{equation*}
Hence, putting everything together we obtain, as $N \to \infty$,
\begin{align*}
   \tilde{K}_{\frac{\zeta}{N}}^{(N)}\left(y_1,y_2\right) \sim \left(\frac{N}{\zeta}\right)^{y_2-y_1}\frac{\prod_{k=0}^{y_1+N-1}a_k}{\prod_{k=0}^{y_2+N-1}a_k} \frac{1}{\left(2\pi \textnormal{i}\right)^2}\oint_{|z|=1}dz\oint_{|v|=2}dv \frac{z^{y_1}}{v^{y_2+1}}e^{z^{-1}-v^{-1}+\zeta\bar{a}v-\zeta \bar{a}z}\frac{1}{v-z}.
\end{align*}
Observe that the ratio of factors in front of the contour integral cancel out when we take the determinant. Moreover, we can deform the $z$ and $v$ contours as long as the $z$-contour is contained in the $v$-contour and contains $0$. Hence, for any $m\ge 1$,
\begin{align*}
\det\left(\tilde{K}_{\frac{\zeta}{N}}^{(N)}\left(y_i,y_j\right)\right)_{i,j=1}^m\overset{N \to \infty}{\longrightarrow} \det\left(\mathbf{J}_{\zeta \bar{a}}\left(y_i,y_j\right)\right)_{i,j=1}^m.
\end{align*}
This implies the convergence of all the correlation functions and since the state space is discrete it implies convergence in distribution of the corresponding point processes, and thus of the maximal particles, and gives the statement of the theorem.

\end{proof}

\bibliographystyle{acm}
\bibliography{References}

\bigskip 

\noindent{\sc School of Mathematics, University of Edinburgh, James Clerk Maxwell Building, Peter Guthrie Tait Rd, Edinburgh EH9 3FD, U.K.}\newline
\href{mailto:theo.assiotis@ed.ac.uk}{\small theo.assiotis@ed.ac.uk}

\end{document}